\documentclass[
11pt,  reqno]{amsart}
\usepackage[margin=1.2in,marginparwidth=1.5cm, marginparsep=0.5cm]{geometry}
\pdfoutput=1

\usepackage{booktabs} 
\usepackage{microtype}
\usepackage{amssymb}
\usepackage{mathrsfs}

\usepackage{upgreek} 

\usepackage{color}
\usepackage[implicit=true]{hyperref}

\usepackage{xsavebox}

\usepackage{cases}

\allowdisplaybreaks[2]

\definecolor{gr}{rgb}   {0.,   0.69,   0.23 }
\definecolor{bl}{rgb}   {0.,   0.5,   1. }
\definecolor{mg}{rgb}   {0.85,  0.,    0.85}
\definecolor{yl}{rgb}   {0.8,  0.7,   0.}
\definecolor{or}{rgb}  {0.7,0.2,0.2}

\usepackage{tikz} 
\usepackage{marginnote}
\usepackage{scalerel} 

\usetikzlibrary{shapes.misc}
\usetikzlibrary{shapes.symbols}
\usetikzlibrary{decorations}
\usetikzlibrary{decorations.markings}

\tikzset{
	dot/.style={circle,fill=black,draw=black,inner sep=0pt,minimum size=0.5mm},
	>=stealth,
	}

\tikzset{
	dot2/.style={circle,fill=black,draw=black,inner sep=0pt,minimum size=0.2mm},
	>=stealth,
	}

\tikzset{
	ddot/.style={circle,fill=black,draw=black,inner sep=0pt,minimum size=0.8mm},
	>=stealth,
	}

\tikzset{decision/.style={ 
        draw,
        diamond,
        aspect=1.5
    }}

\tikzset{dia2/.style
={diamond,fill=white,draw=black,inner sep=0pt,minimum size=1mm},
	>=stealth,
	}

\tikzset{dia/.style
={star,fill=black,draw=black,inner sep=0pt,minimum size=1mm},
	>=stealth,
	}

\tikzset{dia/.style
={diamond,fill=black,draw=black,inner sep=0pt,minimum size=1.3mm},
	>=stealth,
	}

\makeatletter
\def\DeclareSymbol#1#2#3{\xsavebox{#1}{\tikz[baseline=#2,scale=0.15]{#3}}}
\def\<#1>{\xusebox{#1}}
\makeatother

\newsavebox{\peA}
\newsavebox{\pneA}
\newsavebox{\plA}
\newsavebox{\pgA}
\newsavebox{\pleA}
\newsavebox{\pgeA}
\newsavebox{\pezA}

\savebox{\peA}{\tikz \draw (0,0) node[shape=circle,draw,inner sep=0pt,minimum size=8.5pt] {\scriptsize  $=$};}
\savebox{\pneA}{\tikz \draw (0,0) node[shape=circle,draw,inner sep=0pt,minimum size=8.5pt] {\footnotesize $\neq$};}
\savebox{\plA}{\tikz \draw (0,0) node[shape=circle,draw,inner sep=0pt,minimum size=8.5pt] {\scriptsize $<$};}
\savebox{\pgA}{\tikz \draw (0,0) node[shape=circle,draw,inner sep=0pt,minimum size=8.5pt] {\scriptsize $>$};}
\savebox{\pleA}{\tikz \draw (0,0) node[shape=circle,draw,inner sep=0pt,minimum size=8.5pt] {\scriptsize $\leqslant$};}
\savebox{\pgeA}{\tikz \draw (0,0) node[shape=circle,draw,inner sep=0pt,minimum size=8.5pt] {\scriptsize $\geqslant$};}
\savebox{\pezA}{\tikz \draw (0,0) node[shape=circle,draw,
fill=white, 
inner sep=0pt,minimum size=8.5pt]{} ;}

\def \peB{\mathchoice
{\scalebox{.7}{{\usebox{\peA}}}}
{\scalebox{.7}{{\usebox{\peA}}}}
{\scalebox{.7}{{\usebox{\peA}}}}
{}
}

\def \plB{\mathchoice
{\scalebox{.7}{{\usebox{\plA}}}}
{\scalebox{.7}{{\usebox{\plA}}}}
{\scalebox{.7}{{\usebox{\plA}}}}
{}
}
\def \pgB{\mathchoice
{\scalebox{.7}{{\usebox{\pgA}}}}
{\scalebox{.7}{{\usebox{\pgA}}}}
{\scalebox{.7}{{\usebox{\pgA}}}}
{}
}

\def \pgeB{\mathchoice
{\scalebox{.7}{{\usebox{\pgeA}}}}
{\scalebox{.7}{{\usebox{\pgeA}}}}
{\scalebox{.7}{{\usebox{\pgeA}}}}
{}
}
\def \pezB{\mathchoice
{\scalebox{.7}{{\usebox{\pezA}}}}
{\scalebox{.7}{{\usebox{\pezA}}}}
{\scalebox{.7}{{\usebox{\pezA}}}}
{}
}

\newcommand{\pe}{\mathbin{{\peB}}}

\newcommand{\pl}{\mathbin{{\plB}}}
\newcommand{\pg}{\mathbin{{\pgB}}}

\newcommand{\pge}{\mathbin{{\pgeB}}}
\newcommand{\pez}{\mathbin{{\pezB}}}

\usetikzlibrary{decorations.pathmorphing, patterns,shapes}

\DeclareSymbol{dot}{-2.7}{\draw (0,0) node[dot]{};}

 \DeclareSymbol{1}{0}{\clip (-.7,0) rectangle (.7,1.6); \draw (0,0)  -- (0,1.2) node[dot] {};}

\DeclareSymbol{2}{0}{\draw (-0.5,1.2) node[dot] {} -- (0,0) -- (0.5,1.2) node[dot] {};}
\DeclareSymbol{3}{0}{\draw (0,0) -- (0,1.2) node[dot] {}; \draw (-.7,1) node[dot] {} -- (0,0) -- (.7,1) node[dot] {};}
\DeclareSymbol{30}{-3}{\draw (0,0) -- (0,-1); \draw (0,0) -- (0,1.2) node[dot] {}; \draw (-.7,1) node[dot] {} -- (0,0) -- (.7,1) node[dot] {};}
\DeclareSymbol{31}{-3}{\draw (0,0) -- (0,-1) -- (1,0) node[dot] {}; \draw (0,0) -- (0,1.2) node[dot] {}; \draw (-.7,1) node[dot] {} -- (0,0) -- (.7,1) node[dot] {};}
\DeclareSymbol{32}{-3}{\draw (0,0) -- (0,-1) -- (1,0) node[dot] {}; \draw (0,0) -- (0,-1) -- (-1,0) node[dot] {}; \draw (0,0) -- (0,1.2) node[dot] {}; \draw (-.7,1) node[dot] {} -- (0,0) -- (.7,1) node[dot] {};}
\DeclareSymbol{20}{-3}{\draw (0,0) -- (0,-1);\draw (-.7,1) node[dot] {} -- (0,0) -- (.7,1) node[dot] {};}
\DeclareSymbol{22}{-3}{\draw (0,0.3) -- (0,-1) -- (1,0) node[dot] {}; \draw (0,0.3) -- (0,-1) -- (-1,0) node[dot] {};\draw (-.7,1) node[dot] {} -- (0,0.3) -- (.7,1) node[dot] {};}
\DeclareSymbol{31p}{-3}{\draw (0,0) -- (0,-1) -- (1,0) node[dot] {}; \draw (0,0) -- (0,1.2) node[dot] {}; \draw (-.7,1) node[dot] {} -- (0,0) -- (.7,1) node[dot] {}; 
\draw (0,-1) node{\scaleobj{0.5}{\pez}}; 
\draw (0,-1) node{\scaleobj{0.5}{\pe}}; }
\DeclareSymbol{32p}{-3}{\draw (0,0) -- (0,-1) -- (1,0) node[dot] {}; \draw (0,0) -- (0,-1) -- (-1,0) node[dot] {}; \draw (0,0) -- (0,1.2) node[dot] {}; \draw (-.7,1) node[dot] {} -- (0,0) -- (.7,1) node[dot] {}; 
\draw (0,-1) node{\scaleobj{0.5}{\pez}};
\draw (0,-1) node{\scaleobj{0.5}{\pe}};}
\DeclareSymbol{22p}{-3}{\draw (0,0.3) -- (0,-1) -- (1,0) node[dot] {}; \draw (0,0.3) -- (0,-1) -- (-1,0) node[dot] {};\draw (-.7,1) node[dot] {} -- (0,0.3) -- (.7,1) node[dot] {}; 
\draw (0,-1) node{\scaleobj{0.5}{\pez}};
\draw (0,-1) node{\scaleobj{0.5}{\pe}};}
\DeclareSymbol{21p}{-3}{\draw (0,0.3) -- (0,-1) -- (1,0) node[dot] {}; 
\draw (-.7,1) node[dot] {} -- (0,0.3) -- (.7,1) node[dot] {}; 
\draw (0,-1) node{\scaleobj{0.5}{\pez}};
\draw (0,-1) node{\scaleobj{0.5}{\pe}};}

\DeclareSymbol{21}{-3}{\draw (0,0.3) -- (0,-1) -- (1,0) node[dot] {}; 
\draw (-.7,1) node[dot] {} -- (0,0.3) -- (.7,1) node[dot] {}; 
}

\DeclareSymbol{11p}{-3}{\draw (0,0) node[dot2] {} -- (0,-1) -- (1,0) node[dot] {}; 
\draw (0,0) -- (0,1.2) node[dot] {};  
\draw (0,-1) node{\scaleobj{0.5}{\pez}}; 
\draw (0,-1) node{\scaleobj{0.5}{\pe}}; }
\DeclareSymbol{100}{-3}
{\draw[white] (-.4,0) -- (.4,0); 
\draw (0,0) node[dot2] {} -- (0,-1)  ; 
\draw (0,0) -- (0,1.2) node[dot] {};}  
\usetikzlibrary{mindmap,backgrounds,trees,arrows,decorations.pathmorphing,decorations.pathreplacing,positioning,shapes.geometric,shapes.misc,decorations.markings,decorations.fractals,calc,patterns}

\tikzset{>=stealth',
         cvertex/.style={circle,draw=black,inner sep=1pt,outer sep=3pt},
         vertex/.style={circle,fill=black,inner sep=1pt,outer sep=3pt},
         star/.style={circle,fill=yellow,inner sep=0.75pt,outer sep=0.75pt},
         tvertex/.style={inner sep=1pt,font=\scriptsize},
         gap/.style={inner sep=0.5pt,fill=white}}
\tikzstyle{mybox} = [draw=black, fill=blue!10, very thick,
    rectangle, rounded corners, inner sep=10pt, inner ysep=20pt]
\tikzstyle{boxtitle} =[fill=blue!50, text=white,rectangle,rounded corners]

\tikzstyle{decision} = [diamond, draw, fill=blue!20,
    text width=4.5em, text badly centered, node distance=2.5cm, inner sep=0pt]
\tikzstyle{block} = [rectangle, draw, fill=blue!20,
    text width=2.7cm, text centered, rounded corners, minimum height=3em]

\tikzstyle{line} = [draw, very thick, color=black!50, -latex']

\tikzstyle{cloud} = [draw, ellipse,fill=red!40, 
node distance=2.5cm,
    minimum height=1em]

\tikzstyle{cloud2} = [draw, ellipse,fill=red!30, text=white,text width=10em, node distance=2.5cm, text centered, minimum height=4em]

\tikzstyle{cloud3} = [draw, ellipse, fill=cyan!30, 
node distance=2.5cm,
    minimum height=3em, 
    text width=1.7cm]

\tikzstyle{cloud4} = [draw, ellipse,fill=orange!70, node distance=2.5cm,
   text width=2.1cm,  
           minimum height=3em]

\tikzstyle{cloud5} = [draw, ellipse,fill=red!20, node distance=2.5cm,
    text centered, 
    text width=3.0cm, 
    minimum height=3em]

\tikzstyle{cloud6} = [draw, ellipse,fill=red!20, node distance=2.5cm,
    text width=1.5cm, 
    text centered, 
    minimum height=3em]

\tikzset{
    position/.style args={#1:#2 from #3}{
        at=(#3.#1), anchor=#1+180, shift=(#1:#2)
    }
}

\newtheorem{theorem}{Theorem} [section]

\newtheorem{lemma}[theorem]{Lemma}
\newtheorem{proposition}[theorem]{Proposition}
\newtheorem{remark}[theorem]{Remark}

\newtheorem{corollary}[theorem]{Corollary}

\DeclareMathOperator*{\supp}{supp}

\DeclareMathOperator{\Id}{Id}
\DeclareMathOperator{\sgn}{sgn}

\newcommand{\1}{\hspace{0.2mm}\textup{I}\hspace{0.2mm}}
\newcommand{\II}{\text{I \hspace{-2.8mm} I} }
\newcommand{\III}{\text{I \hspace{-2.9mm} I \hspace{-2.9mm} I}}
\newcommand{\IV}{\text{I \hspace{-2.9mm} V}}

\newcommand{\5}{\text{V}}

\newcommand{\noi}{\noindent}
\newcommand{\Z}{\mathbb{Z}}
\newcommand{\R}{\mathbb{R}}
\newcommand{\T}{\mathbb{T}}

\let\Re=\undefined\DeclareMathOperator*{\Re}{Re}
\let\Im=\undefined\DeclareMathOperator*{\Im}{Im}

\let\P= \undefined
\newcommand{\P}{\mathbf{P}}

\newcommand{\E}{\mathbb{E}}

\renewcommand{\L}{\mathcal{L}}

\newcommand{\K}{\mathcal{K}}

\newcommand{\F}{\mathcal{F}}

\newcommand{\al}{\alpha}
\newcommand{\be}{\beta}
\newcommand{\dl}{\delta}

\newcommand{\nb}{\nabla}

\newcommand{\Dl}{\Delta}
\newcommand{\eps}{\varepsilon}
\newcommand{\kk}{\kappa}
\newcommand{\g}{\gamma}

\newcommand{\ld}{\lambda}
\newcommand{\Ld}{\Lambda}
\newcommand{\s}{\sigma}
\newcommand{\Si}{\Sigma}
\newcommand{\ft}{\widehat}

\newcommand{\wt}{\widetilde}
\newcommand{\cj}{\overline}

\newcommand{\dt}{\partial_t}
\newcommand{\dd}{\partial}

\newcommand{\ud}{\underline}

\newcommand{\ta}{\theta}

\renewcommand{\l}{\ell}
\renewcommand{\o}{\omega}
\renewcommand{\O}{\Omega}

\newcommand{\les}{\lesssim}
\newcommand{\ges}{\gtrsim}

\newcommand{\jb}[1]
{\langle #1 \rangle}

\newcommand{\jbb}[1]
{[\hspace{-0.6mm}[ #1 ]\hspace{-0.6mm}]}

\newcommand{\ind}{\mathbf 1}

\newcommand{\N}{\mathbb{N}}

\renewcommand{\H}{\mathcal{H}}

\DeclareMathOperator{\Lip}{Lip}

\newtheorem*{ackno}{Acknowledgements}

\newcommand{\I}{\mathcal{I}}

\newcommand{\If}{\mathfrak{I}}

\newcommand{\A}{\mathcal{A}}

\newcommand{\C}{\mathcal{C}}
\numberwithin{equation}{section}
\numberwithin{theorem}{section}

\newcommand{\Q}{\mathbb{Q}}
\newcommand{\PP}{\mathbb{P}}

\DeclareMathOperator{\Law}{Law}

\newcommand{\ZZ}{\mathfrak{Z}}

\newcommand{\muu}{\vec{\mu}}

\newcommand{\rhoo}{\vec{\rho}}

\newcommand{\W}{\mathcal{W}}
\newcommand{\U}{\mathcal{U}}

\newcommand{\dr}{\theta}
\newcommand{\Dr}{\Theta}

\newcommand{\Ha}{\mathbb{H}_a}
\newcommand{\Hc}{\mathbb{H}_c}

\newcommand{\NN}{\mathcal{N}}
\newcommand{\D}{\mathcal{D}}

\newcommand{\Res}{\mathfrak{R}}
\newcommand{\Qxy}{Q_{X,Y}}

\newcommand{\QxyN}{Q_{X_N,Y_N}}

\newcommand{\dia}{\diamond}

\newcommand{\too}{\longrightarrow}

\newcommand{\proj}{\Pi}

\newcommand{\Ups}{\Upsilon}
\newcommand{\UUps}{{\ud \Upsilon}}

\newcommand{\Ab}{\mathbb{A}}

\newcommand{\plan}{\mathfrak{p}}

\newcommand{\Xc}{\mathcal{X}}

\newcommand{\Zc}{\mathcal{Z}}

\usepackage{comment}

\makeatletter
\@namedef{subjclassname@2020}{%
  \textup{2020} Mathematics Subject Classification}
\makeatother

\begin{document}
\baselineskip = 14pt

\title[Stochastic quantization of $\Phi^3_3$]
{Stochastic quantization of the $\Phi^3_3$-model}

\author[T.~Oh, M.~Okamoto, and L.~Tolomeo]
{Tadahiro Oh, Mamoru Okamoto, and Leonardo Tolomeo}

\address{
Tadahiro Oh, School of Mathematics\\
The University of Edinburgh\\
and The Maxwell Institute for the Mathematical Sciences\\
James Clerk Maxwell Building\\
The King's Buildings\\
Peter Guthrie Tait Road\\
Edinburgh\\ 
EH9 3FD\\
  United Kingdom,
 and
 School of Mathematics and Statistics, Beijing Institute of Technology, Beijing 100081, China}

\email{hiro.oh@ed.ac.uk}

\address{
Mamoru Okamoto\\
Department of Mathematics\\
 Graduate School of Science\\ Osaka University\\
Toyonaka\\ Osaka\\ 560-0043\\ Japan}
\email{okamoto@math.sci.osaka-u.ac.jp}

\address{
Leonardo Tolomeo\\ 
Mathematical Institute, Hausdorff Center for Mathematics, Universit\"at Bonn, Bonn, Germany,
and
School of Mathematics\\
The University of Edinburgh\\
and The Maxwell Institute for the Mathematical Sciences\\
James Clerk Maxwell Building\\
The King's Buildings\\
Peter Guthrie Tait Road\\
Edinburgh\\ 
EH9 3FD\\
 United Kingdom}

\email{l.tolomeo@ed.ac.uk}



\begin{abstract}

We study the construction of the $\Phi^3_3$-measure
and  complete the 
 program
on the (non-)construction of the focusing Gibbs measures,
initiated by Lebowitz, Rose, and Speer (1988).
This problem turns out to be critical, 
exhibiting the following phase transition.
In the weakly nonlinear regime, 
we prove normalizability of the $\Phi^3_3$-measure
and show that it is singular with respect to the massive Gaussian free field.
Moreover, we show that there exists a shifted measure with respect to which 
the $\Phi^3_3$-measure is absolutely continuous.
In the strongly nonlinear regime, 
by further developing the 
machinery 
introduced  by the authors, 
we establish non-normalizability of 
the $\Phi^3_3$-measure.
Due to the singularity of the $\Phi^3_3$-measure with respect to the massive Gaussian free field, 
this  non-normalizability part poses a particular challenge
as compared to our previous works.
In order to overcome this issue, we first construct a $\s$-finite version
of the $\Phi^3_3$-measure and show that this measure is not normalizable.
Furthermore, we prove that the truncated $\Phi^3_3$-measures
have no weak limit in a natural space, even up to a subsequence.

We also study the 
dynamical problem 
for  the canonical stochastic quantization of the $\Phi^3_3$-measure, 
namely, 
  the three-dimensional stochastic damped nonlinear wave equation with 
a quadratic
nonlinearity
forced by an additive space-time white noise
(= the hyperbolic $\Phi^3_3$-model).
By adapting the paracontrolled approach, 
in particular from the  works by Gubinelli, Koch, and the first author (2018)
and by the authors (2020),  
we prove
almost sure global well-posedness
of the hyperbolic $\Phi^3_3$-model
and invariance of the Gibbs  measure in the weakly nonlinear regime.
In the globalization part, 
we introduce a new, 
  conceptually simple and straightforward approach, 
where we  directly work with the  (truncated) Gibbs measure, 
using the Bou\'e-Dupuis variational formula
and ideas from  theory of optimal transport.

\end{abstract}

%
\maketitle
\tableofcontents

\section{Introduction}
\label{SEC:1}


\subsection{Overview}

In this paper, we study
the $\Phi^3_3$-measure 
on 
the three-dimensional torus on $\T^3 = (\R/2\pi\Z)^3$, formally written as
\begin{align}
d\rho(u) = Z^{-1} \exp \bigg(\frac{\s}3 \int_{\T^3} u^3 dx\bigg) d\mu(u), 
\label{H1}
\end{align}

\noi
and its associated stochastic quantization.
Here, $\mu$ is the massive Gaussian free field on $\T^3$
and the coupling constant $\s \in \R\setminus \{0\}$ measures the strength
of the cubic interaction.
The associated energy functional 
for the $\Phi^3_3$-measure  $\rho$ in \eqref{H1} 
is given by
\begin{align}
E(u)
= \frac 12 \int_{\T^3} |\jb{\nabla} u|^2 dx 
- \frac \s3 \int_{\T^3} u^3 dx, 
\label{H2}
\end{align}

\noi
where $\jb{\nabla} = \sqrt{1-\Dl}$.
Since $u^3$ is not sign definite,
the sign of $\s$ does not play any role
and, in particular, the problem is not  defocusing even if $\s < 0$.

Our main goal in this paper is to study 
the construction of the $\Phi^3_3$-measure
and its associated dynamics, 
following the program 
on the (non-)construction of focusing\footnote{By ``focusing'', 
we also mean the  non-defocusing (non-repulsive) case, such as the cubic interaction 
appearing in \eqref{H1}, 
such that the interaction potential 
(for example, $\frac{\s}3 \int_{\T^3} u^3 dx$ in \eqref{H1}) is unbounded from above.} 
Gibbs measures, 
initiated  by Lebowitz, Rose, and Speer \cite{LRS}.
Let us first go over the known results.
In the seminal work \cite{LRS}, 
Lebowitz, Rose, and Speer  studied the one-dimensional case
and constructed
the one-dimensional focusing Gibbs measures\footnote{As pointed out by 
Carlen, Fr\"ohlich, and Lebowitz \cite[p.\,315]{CFL}, there is in fact an error in the Gibbs measure
construction in \cite{LRS}, 
which was amended by Bourgain \cite{BO94} 
(for $2 < p< 6$ with any $K > 0$
and $p = 6$ with $0 < K \ll 1$) and 
the first and third authors
with Sosoe \cite{OST} (for $p = 6$ and 
$K \leq  \|Q\|_{L^2(\R)}^2$).
See \cite{OST} for a further discussion.
}
in the $L^2$-(sub)critical setting (i.e.~$2 < p\leq 6$)
with an $L^2$-cutoff:
\begin{align}
d\rho(u) = Z^{-1} 
\ind_{\{\int_\T |u|^2 dx \le K \}} \exp \bigg(\frac{1}{p} \int_{\T} |u|^p dx\bigg) d\mu(u)
\label{AX1}
\end{align}

\noi
or with a taming by the $L^2$-norm:
\begin{align}
d\rho(u) = Z^{-1} 
 \exp \bigg(\frac{1}{p} \int_{\T} |u|^p dx - A \Big(\int_\T u^2 dx\Big)^{q}\bigg) d\mu(u)
\label{AX2}
\end{align}

\noi
for some appropriate $q = q(p)$, 
where $\mu$ denotes the periodic Wiener measure on $\T$.
See Remark 2.1 in \cite{LRS}.
Here, the parameter $A > 0$ denotes the so-called
(generalized) 
chemical potential
and the expression \eqref{AX2} is referred to as the generalized grand-canonical Gibbs measure.
See also  the work by 
Carlen, Fr\"ohlich, and Lebowitz \cite{CFL}
for a further discussion, where they describe the details
of the construction of the 
generalized grand-canonical Gibbs measure in \eqref{AX2}
in the $L^2$-subcritical setting ($2 < p < 6$).
In  \cite{LRS}, 
Lebowitz, Rose, and Speer  also proved
non-normalizability of the focusing Gibbs measure $\rho$ in \eqref{AX1}:
\begin{align*}
\E_\mu \bigg[
\ind_{\{\int_\T |u|^2 dx \le K \}} \exp \bigg(\frac{1}{p} \int_{\T} |u|^p dx\bigg) \bigg]
= \infty
\end{align*}

\noi
in (i) 
 the $L^2$-supercritical case  ($ p>6$) for any $K > 0$
and 
(ii) the  $L^2$-critical case ($ p>6$), 
provided that $K > \|Q\|_{L^2(\R)}^2$, 
where $Q$ is the (unique\footnote{Up to the symmetries.}) optimizer for the Gagliardo-Nirenberg-Sobolev inequality
on $\R$
such that $\|Q\|_{L^6(\R)}^6 = 3\|Q'\|_{L^2(\R)}^2$.
In  a recent work \cite{OST}, the first and third authors
with Sosoe proved that the focusing $L^2$-critical Gibbs measure $\rho$ in \eqref{AX1}
(with $p = 6$)
is indeed constructible 
at the optimal mass threshold $K = \|Q\|_{L^2(\R)}^2$, 
thus answering an open question posed by Lebowitz, Rose, and Speer \cite{LRS}
and completing the program in the one-dimensional case.

In the two-dimensional setting, 
Brydges and Slade \cite{BS}
continued the study on the focusing Gibbs measures
and showed that with the quartic interaction ($p = 4$), 
the focusing Gibbs measure $\rho$ in \eqref{AX1}
(even with proper renormalization on the potential energy
$\frac{1}{4} \int_{\T^2} |u|^4 dx$ and on the $L^2$-cutoff)
is not normalizable as a probability measure.
See also \cite{OS} for an alternative proof.
In view of 
\begin{align}
 \ind_{\{|\,\cdot \,| \le K\}}(x) \le \exp\big( -  A |x|^\gamma\big) \exp\big(A K^\g\big)
\label{AX4}
\end{align}

\noi
for any $K>0$, $\g>0$, and $A>0$, 
this non-normalizability result 
of the focusing Gibbs measure on $\T^2$
with the quartic interaction ($p = 4$)
also applies
to the generalized grand-canonical Gibbs measure in \eqref{AX2}.
Furthermore, the same non-normalizability applies
for higher order interaction (for an integer $p\geq 5$).

In \cite{BO95}, 
 Bourgain reported Jaffe's construction of a  $\Phi^3_2$-measure endowed with a Wick-ordered
 $ L^2$-cutoff:
\begin{align}
d\rho = Z^{-1}
 \ind_{ \{\int_{\T^2} :\,u^2: \, dx\leq K\}} 
 e^{ \frac 13 \int_{\T^2}  :u^3: \, dx  }d \mu(u) ,
\notag
\end{align}
where $:\! u^2 \!:$ and $:\! u^3 \!:$ denote the Wick powers of $u$, 
and  $\mu$ denotes the massive Gaussian free field on $\T^2$.
See also \cite{OS}.
We point out that such a Gibbs measure with a (Wick-ordered) $L^2$-cutoff 
is not suitable for stochastic quantization in the heat and wave settings
due to the lack of the $L^2$-conservation.
In \cite{BO95}, 
Bourgain instead 
constructed
the following
generalized grand-canonical formulation of 
 the $\Phi^3_2$-measure:
\begin{align}
d\rho(u) = Z^{-1}
 e^{ \frac 13 \int_{\T^2}  :u^3: \, dx  - A
\big(\int_{\T^2} :\,u^2: \, dx\big)^2} d \mu(u) 
\notag
\end{align}

\noi
for sufficiently large $A>0$.
See \cite{OTh2, GKO, OOR, GKOT}
for the associated (stochastic) nonlinear wave dynamics.

In this paper, we consider the three-dimensional case
and complete the focusing Gibbs measure construction program
initiated by Lebowitz, Rose, and Speer \cite{LRS}.
More precisely, we consider
the following generalized grand-canonical formulation
of 
 the $\Phi^3_3$-measure 
(namely, with a taming by the Wick-ordered $L^2$-norm):
\begin{align}
d\rho(u) = Z^{-1}
 \exp \bigg( \frac \s3 \int_{\T^3} :\! u^3\!: \, dx - A
\bigg|\int_{\T^3} :\! u^2 \! : \, dx\bigg|^\g\bigg) d \mu(u)
\label{H3}
\end{align}

\noi
for suitable $A, \g > 0$.
We now state our first main  result in a somewhat formal manner.
See Theorem \ref{THM:Gibbs}
for the precise statement.

\begin{theorem}\label{THM:Gibbs0}
The following phase transition holds
for the $\Phi^3_3$-measure in \eqref{H3}.

\begin{itemize}
\item[\textup{(i)}] \textup{(weakly nonlinear regime).}
Let $0 < |\s|\ll 1$ and $\g = 3$.
Then, 
by introducing a further renormalization, 
the $\Phi^3_3$-measure 
$\rho$ in~\eqref{H3} exists as a probability measure, 
provided that $A = A(\s)> 0$ is sufficiently large.
In this case, the resulting $\Phi^3_3$-measure
$\rho$ and the massive Gaussian free field $\mu$ on $\T^3$
are mutually singular.

\smallskip

\item[\textup{(ii)}] \textup{(strongly nonlinear regime).}
When $|\s| \gg1$,
the $\Phi^3_3$-measure in \eqref{H3} is not normalizable
for any $A> 0$ and $\g > 0$.
Furthermore, the truncated $\Phi^3_3$-measures $\rho_N$ 
\textup{(}see \eqref{GibbsN} below\textup{)}
do not have a weak limit, 
as measures on $\C^{-\frac 34}(\T^3)$, 
even up to a subsequence. 

\end{itemize}

\end{theorem}

Theorem \ref{THM:Gibbs0} shows that the 
$\Phi^3_3$-model is critical
in terms of the measure construction.
In the case 
of  a higher order focusing interaction on $\T^3$
(replacing $:\!u^3\!:$ by $:\!u^p\!:$ in \eqref{H3} for an integer $p \geq 4$ 
with $\s > 0$ when $p$ is even), 
or  the $\Phi^3_4$-model on the four-dimensional torus $\T^4$, 
the focusing nonlinear interaction gets only worse
and thus
we expect that 
the same approach would yield non-normalizability.
Hence, in view of the previous results
\cite{LRS, BO94, BS, OST, OS}, 
Theorem \ref{THM:Gibbs0}
completes
the focusing Gibbs measure construction program,
thus answering an open question 
posed by Lebowitz, Rose, and Speer 
(see ``Extension to higher dimensions'' in \cite[Section 5]{LRS}).
See also our companion paper \cite{OOTol1}, 
where we completed
 the program
on the (non-)construction of the focusing Hartree Gibbs measures in the three-dimensional
setting.
See Remark \ref{REM:Hart} for a further discussion.

We point out that in the weakly nonlinear regime, 
the $\Phi^3_3$-measure $\rho$
 is constructed only as a weak limit of the truncated $\Phi^3_3$-measures.
 Moreover, we prove that there exists a shifted measure
 with respect to which the $\Phi^3_3$-measure is
 absolutely continuous;
 see Appendix~\ref{SEC:AC}.
As for the non-normalizability result
in Theorem~\ref{THM:Gibbs0}\,(ii), 
our proof is based on 
a refined version of the machinery introduced by the authors \cite{OOTol1}
and the first and third authors with Seong~\cite{OS},
which was in turn inspired by the work 
of the third author and Weber \cite{TW} 
on the non-construction of the Gibbs measure 
for the focusing cubic nonlinear Schr\"odinger equation (NLS) on the real line,
giving an alternative proof of Rider's result \cite{Rider}. 
We, however,  point out that there is an additional difficulty
in proving Theorem~\ref{THM:Gibbs0}\,(ii) 
due to the singularity of the $\Phi^3_3$-measure
with respect to the base massive Gaussian free field $\mu$.
(Note that the focusing Gibbs measures considered
in \cite{OOTol1, OS}
are equivalent to the base  Gaussian measures.)
In order to overcome this difficulty, 
we first introduce  a reference measure\footnote{This reference measure
is introduced as a tamed version of the $\Phi^3_3$-measure
and is 
not to be confused with the shifted measure mentioned above.
See Proposition \ref{PROP:ref}.} 
$\nu_\dl$
and construct a $\sigma$-finite version of the $\Phi^3_3$-measure
(expressed in terms of the reference measure $\nu_\dl$).
We then show that this $\s$-finite version of the $\Phi^3_3$-measure is not normalizable.
See 
Section~\ref{SEC:non}.

\begin{remark}\label{REM:Bourgain}\rm

(i) 
As the name suggests, 
the $\Phi^3_3$-measure
is of interest 
from the point of view
of constructive  quantum field theory.
In the defocusing case ($\s < 0$)
with a quartic interaction ($u^4$ in place of $u^3$), 
the measure $\rho$ in \eqref{H1} 
corresponds
to the well-studied $\Phi^4_3$-measure.
The construction of the $\Phi^4_3$-measure
is one of the early achievements in constructive quantum field theory.
For an overview of the  constructive program, 
see the introductions in \cite{AK, GH18b}.

\smallskip

\noi
(ii)
In the one- and two-dimensional cases, 
the non-normalizability of the focusing Gibbs measures
 emerges 
   in the $L^2$-critical case ($p = 6$ when $d = 1$
and $p = 4$ when $d = 2$), suggesting
its close relation to the finite time blowup phenomena
of the associated focusing NLS.
See~\cite{OST} for a further discussion.
In the three-dimensional case, 
it is interesting to  note that the $\Phi^3_3$-model
is $L^2$-{\it subcritical}
and yet we have the non-normalizability (in the strongly nonlinear regime).
Thus,  the non-normalizability of 
the $\Phi^3_3$-measure is not related to a blowup phenomenon.
Note that, unlike the focusing $\Phi^6_1$-
and $\Phi^4_2$-models which make sense 
in the complex-valued setting, 
the $\Phi^3_3$-model makes sense only in the real-valued setting.
It seems  of interest to investigate a possible relation to 
the following Gagliardo-Nirenberg inequality:
\[\int_{\R^3}|u(x)|^3 dx \les  \| u \|_{L^2(\R^3)}^\frac 32 \|u\|_{\dot H^1(\R^3)}^\frac 32. \]

\smallskip

\noi
(iii) 
Consider 
a $\Phi^3_3$-measure with a Wick-ordered $L^2$-cutoff:\footnote{With a slight modification, 
one may also consider $\rho$ in \eqref{H4} with a slightly different cutoff
$\ind_{\{ \int_{\T^3} \, : \, u^2 :\, dx  \le K\}}$, i.e.~without an absolute value, 
and prove the same (non-)normalizability results.
See Remark 5.10 in \cite{OOTol1}.}
\begin{align}
d\rho(u) = Z^{-1}
\ind_{\{ |\int_{\T^3} \, : \, u^2 :\, dx | \le K\}}
 \exp \bigg( \frac \s3 \int_{\T^3} :\! u^3\!: \, dx 
\bigg) d \mu(u).
\label{H4}
\end{align}

\noi
Then, an analogue of 
Theorem \ref{THM:Gibbs0}
holds 
for the $\Phi^3_3$-measure in \eqref{H4}.
In view of 
\eqref{AX4}, 
Theorem~\ref{THM:Gibbs0}
implies normalizability of the $\Phi^3_3$-measure in \eqref{H4}
(with a further renormalization)
in the weakly nonlinear regime ($0 < |\s| \ll 1$).
On the other hand, in the strongly nonlinear regime
($|\s| \gg 1$), 
a modification of the proof of Theorem \ref{Gibbs0}\,(ii)  (see also  \cite{OOTol1, OS})
yields non-normalizability
of the $\Phi^3_3$-measure in \eqref{H4}
for any $K > 0$.

\end{remark}

\begin{remark}\label{REM:Hart}\rm

In  \cite{BO97}, 
Bourgain 
 studied
the invariant Gibbs dynamics
for the focusing
 Hartree NLS 
on~$\T^3$ (with $\s > 0$):
\begin{align}
i \dt u + (1- \Dl) u -  \s(V*|u|^2) u = 0, 
\label{NLS1}
\end{align}

\noi
where $V = \jb{\nb}^{-\be} $ is  the Bessel potential of order $\be>0$.
In \cite{BO97}, 
Bourgain first constructed the 
focusing Gibbs measure with a Hartree-type interaction
(for complex-valued $u$), 
endowed 
with a Wick-ordered $L^2$-cutoff:
\begin{align*}
d\rho(u) = Z^{-1} \ind_{\{\int_{\T^3} :\,|u|^2: \, dx\leq K\}} \, e^{\frac \s4 \int_{\T^3} (V*:|u|^2:)\,  :|u|^2 :\, dx} d\mu(u)
\end{align*}

\noi
for $\be > 2$
and then constructed the invariant Gibbs dynamics
for the associated dynamical problem.\footnote{By combining
the construction of the focusing Hartree Gibbs measure
in the critical case ($\be = 2$) with $0 < \s \ll 1$ in \cite{OOTol1}  and 
the well-posedness result in \cite{DNY4}, 
this result 
on the focusing
 Hartree NLS \eqref{NLS1}
 by Bourgain \cite{BO97} 
can be extended to the 
critical   case $\be = 2$
(in the weakly nonlinear regime $0 < \s \ll 1$).}
In \cite{OOTol1}, 
we continued the study of the focusing Hartree $\Phi^4_3$-measure
in the generalized grand-canonical formulation (with  $\s > 0$):
\begin{align}
d\rho(u) = Z^{-1}
 \exp \bigg( \frac \s4 \int_{\T^3} (V \ast :\! u^2 \!:) :\! u^2 \!: \, dx - A
\bigg|\int_{\T^3} :\! u^2 \! : \, dx\bigg|^\g\bigg) d \mu(u)
\label{GibbsH}
\end{align}

\noi
and  established a phase transition
in two respects
(i)  the 
 focusing Hartree $\Phi^4_3$-measure $\rho$ in~\eqref{GibbsH}
is constructible for $\be > 2$, 
while it is not for $\be < 2$
and (ii) when $\be = 2$, 
the  focusing Hartree $\Phi^4_3$-measure
is constructible for $0 < \s \ll 1$, 
while it is not for $\s \gg 1$.
See~\cite{OOTol1} 
for the precise statements.
These results in \cite{OOTol1} in particular show the critical nature
of the focusing Hartree $\Phi^4_3$-model
when $\be = 2$.
In the same work, we also constructed the invariant Gibbs dynamics
for the associated (canonical) stochastic quantization equation.
See also \cite{OOTol1, Bring1, Bring2} for the defocusing case ($\s < 0$).
Note that when $\be = 0$, 
the defocusing Hartree $\Phi^4_3$-measure reduces
to the usual $\Phi^4_3$-measure.

In terms of scaling,
the focusing Hartree $\Phi^4_3$-model 
with $\be=2$ corresponds to the $\Phi^3_3$-model
and as such, they share some common features.
For example, 
they are both critical with a phase transition,
depending on  the size of 
the coupling constant~$\s$. 
At the same time, however, 
there are some differences.
While
the focusing Hartree $\Phi^4_3$-measure  
with $\be=2$ is absolutely continuous with 
respect to the base massive Gaussian free field $\mu$, 
the $\Phi^3_3$-measure studied in this paper 
is singular with respect to the base massive Gaussian free field $\mu$.
As mentioned above, 
this singularity of the $\Phi^3_3$-measure 
causes an additional difficulty in proving non-normalizability
in the strongly nonlinear regime $|\s| \gg 1$.

\end{remark}

\medskip

Next, we discuss the dynamical problem associated with the $\Phi^3_3$-measure
constructed in Theorem \ref{THM:Gibbs0}.
In the following, 
we consider 
 the canonical stochastic quantization equation \cite{PW, RSS}
for the $\Phi^3_3$-measure in \eqref{H3} (with $\g = 3$).
More precisely, 
we study   the following stochastic damped nonlinear wave equation (SdNLW)  
with a quadratic nonlinearity, 
posed
on~$\T^3$:
\begin{align}
\dt^2 u + \dt u + (1 -  \Dl)  u  - \s u^2  = \sqrt{2} \xi,
\qquad (x, t) \in \T^3\times \R_+,
\label{SNLW0}
\end{align}
where 
$\s \in \R\setminus \{0\}$,  $u$ is an unknown function, 
and $\xi$ denotes a (Gaussian) space-time white noise on $\T^3\times \R_+$
with the space-time covariance given by
\[ \E\big[ \xi(x_1, t_1) \xi(x_2, t_2) \big]
= \dl(x_1 - x_2) \dl (t_1 - t_2).\]

\noi
In this introduction, 
we keep our discussion at a formal level and do not worry about various renormalizations 
required 
to give a  proper meaning to the equation \eqref{SNLW0}.

With  $\vec{u} = (u, \dt u)$, 
define the energy $\mathcal{E}(\vec u)$ by 
\begin{align*}
\mathcal{E}(\vec{u})
& = E(u) +  \frac 12 \int_{\T^3} (\dt u)^2 dx \\
& = \frac 12 \int_{\T^3} |\jb{\nabla} u|^2 dx + \frac 12 \int_{\T^3} (\dt u)^2 dx 
- \frac \s3 \int_{\T^3} u^3 dx, 
\end{align*}

\noi
where $E(u)$ is as in \eqref{H2}.
This is precisely the energy (= Hamiltonian) 
of  the  (deterministic) nonlinear wave equation (NLW) on $\T^3$
with a quadratic nonlinearity:
\begin{align}
\dt^2 u + (1 -  \Dl)  u - \s u^2  = 0.
\label{NLW1}
\end{align}

\noi
Then,
by letting $v = \dt u$, we can write \eqref{SNLW0} as
the first order system:
\begin{align*}
\dt   \begin{pmatrix}
u \\ v
\end{pmatrix}
=     
\begin{pmatrix} 
\frac{\dd\mathcal{E}}{\dd v}
\rule[-3mm]{0pt}{2mm}
\\
- \frac{\dd\mathcal{E}}{\dd u}
\end{pmatrix}
+      \begin{pmatrix} 
0 
\\
- v + \sqrt 2 \xi
\end{pmatrix},
\end{align*}

\noi
which shows that 
the SdNLW dynamics \eqref{SNLW0}
is given as a superposition
of the deterministic NLW dynamics \eqref{NLW1}
and the Ornstein-Uhlenbeck dynamics
for $v = \dt u$:
\begin{align*}
\dt v = - v + \sqrt 2 \xi.
\end{align*}

\noi
Now, consider 
the Gibbs measure $\rhoo$, formally given by 
\begin{align}
\begin{split}
d\rhoo(\vec u ) 
& = Z^{-1}e^{-\mathcal E(\vec u  )}d\vec u 
=  d\rho \otimes d\mu_0(\vec u )\\
& = Z^{-1} \exp \bigg( \frac \s 3 \int_{\T^3} u^3 dx \bigg) d (\mu\otimes \mu_0)(u, v) , 
\end{split}
\label{Gibbs0}
\end{align}

\noi 
where  $\rho$ is the $\Phi^3_3$-measure in~\eqref{H1}
and $\mu_0$ denotes the  white noise measure;
see \eqref{gauss0}.
See Remark \ref{REM:Gibbs} for the precise definition of 
the Gibbs measure $\rhoo$.
Then, the observation above
shows that $\rhoo$ is
expected to be invariant under the dynamics of the quadratic SdNLW~\eqref{SNLW0}.
Indeed, from the stochastic quantization point of view, 
the equation \eqref{SNLW0}
is the so-called canonical 
 stochastic quantization equation (namely, the Hamiltonian stochastic quantization)
 for the $\Phi^3_3$-measure; see
 \cite{RSS}.
For this reason, it is natural 
to refer to \eqref{SNLW0}
as the {\it hyperbolic $\Phi^3_3$-model}.

Let us now state our main dynamical result in a somewhat formal manner.
See Theorem~\ref{THM:GWP} for the precise statement.

\begin{theorem}\label{THM:GWP0}

Let $\g=3$ and
$0 < |\s|\ll 1$.
Suppose that  
 $A = A(\s)> 0$ is sufficiently large as in Theorem \ref{THM:Gibbs0}\,(i).
Then, the hyperbolic $\Phi^3_3$-model~\eqref{SNLW0} 
on the three-dimensional torus $\T^3$ 
\textup{(}with a proper renormalization\textup{)} is almost surely globally well-posed
with respect to the random initial data distributed
by the \textup{(}renormalized\textup{)} Gibbs measure $\rhoo = \rho \otimes \mu_0$ in \eqref{Gibbs0}.
Furthermore, the Gibbs measure $\rhoo$ is invariant under the resulting dynamics.

\end{theorem}

In view of the critical nature of the $\Phi^3_3$-measure, 
Theorem \ref{THM:GWP0} is sharp in the sense that 
almost sure global well-posedness does not extend
to SdNLW with a focusing nonlinearity of a higher order.
The construction of the $\Phi^3_3$-measure in 
Theorem \ref{THM:Gibbs0}
requires us to introduce several renormalizations
together with the taming by the Wick-ordered $L^2$-norm.
This introduces modifications
to the equation \eqref{SNLW0}.
See Subsection \ref{SUBSEC:1.3}
and 
Sections~\ref{SEC:LWP} and \ref{SEC:GWP} for the precise formulation of the problem.

Over the last five years, 
stochastic nonlinear wave equations (SNLW)
in the singular setting have been studied extensively
in various settings:\footnote{Some of the works mentioned below
are on SNLW without damping.}
\begin{align}
\dt^2 u + \dt u + (1 -  \Dl)  u + \NN(u)  =  \xi
\label{SNLW0a}
\end{align}

\noi
for a power-type nonlinearity
 \cite{GKO, GKO2, GKOT,   Deya1, Deya2, ORTz, OOR, OOcomp, Tolomeo2, 
 OOTol1, Bring2, OWZ}
and for trigonometric and exponential nonlinearities
\cite{ORSW,  ORW,  ORSW2}.
We also  mention the works 
 \cite{OTh2,   OPTz,  OOTz, Bring2}
 on
 nonlinear wave equations with  rough random initial data.
In \cite{GKO2}, 
by combining  the paracontrolled calculus, 
originally introduced in the parabolic setting~\cite{GIP, CC, MW1},  
with the multilinear harmonic analytic approach, 
more traditional in studying dispersive equations, 
Gubinelli, Koch, and the first author studied the quadratic SNLW
\eqref{SNLW0} (without the damping).
The paracontrolled approach
in the wave setting was also used in our previous work \cite{OOTol1}
and was further developed by Bringmann \cite{Bring2}.
In order to prove local well-posedness of the hyperbolic $\Phi^3_3$-model \eqref{SNLW0}, 
we also follow the paracontrolled approach, 
in particular combining the analysis in \cite{GKO2, OOTol1}.  See Section \ref{SEC:LWP}.
As for the globalization part, 
a naive approach would be to  apply Bourgain's invariant measure argument \cite{BO94, BO96}.
However, due to the singularity of the $\Phi^3_3$-measure $\rho$
with respect to the base massive Gaussian free field $\mu$
(and the fact that the truncated  $\Phi^3_3$-measure $\rho_N$  converges to $\rho$
only weakly), 
there is an additional difficulty to overcome for the hyperbolic $\Phi^3_3$-model.
Hence, Bourgain's invariant measure argument is not directly applicable.
In the context of the defocusing Hartree cubic NLW on $\T^3$, 
Bringmann~\cite{Bring2} encountered a similar difficulty
and developed a new globalization argument.
While it is possible to adapt Bringmann's analysis to our current setting, 
we instead introduce a new  alternative argument, 
which is conceptually simple and straightforward.
In particular, 
we extensively use the variational approach and also 
use ideas from theory of optimal transport to 
directly estimate a probability with respect to the limiting Gibbs measure $\rhoo$
(in particular, without going through shifted measures as in \cite{Bring2}).
See Subsection \ref{SUBSEC:1.3} and Section \ref{SEC:GWP} for details.

\begin{remark}\rm 
A slight modification of our proof of Theorem \ref{THM:GWP0}
yields the corresponding results
(namely, almost sure global well-posedness and invariance of the associated Gibbs measure)
for the (deterministic) quadratic NLW \eqref{NLW1}
 on $\T^3$
in the weakly nonlinear regime.

\end{remark}

\begin{remark}\label{REM:heat}
\rm
We point out that an analogue 
of Theorem \ref{THM:GWP0}
also holds for  the parabolic $\Phi^3_3$-model, 
namely, the stochastic nonlinear heat equation
with a quadratic nonlinearity:
\begin{align}
 \dt u + (1 -  \Dl)  u  - \s u^2  = \sqrt{2} \xi,
\qquad (x, t) \in \T^3\times \R_+.
\label{heat1}
\end{align}

\noi
Thanks to the strong smoothing of the heat propagator, 
the well-posedness of \eqref{heat1} follows
from elementary analysis based on the first order expansion
(also known as the Da Prato-Debussche trick \cite{DPD2}).
See for example \cite{EJS}.
While there is an extra term coming from 
the taming by the Wick-ordered $L^2$-norm
(see,  for example,  \eqref{SNLW1} in the hyperbolic case), 
this term does not cause any issue in the parabolic setting.

\end{remark}

\begin{remark}\rm

In \cite{Tolomeo1}, the third author introduced 
a new approach to establish unique
ergodicity of Gibbs measures
for stochastic dispersive/hyperbolic equations.
This was further developed in \cite{Tolomeo3}
 to prove 
ergodicity of the hyperbolic $\Phi^4_2$-model, 
namely~\eqref{SNLW0a} on $\T^2$ with $\NN(u) = u^3$.
See also 
\cite{FT} by the third author and Forlano on the asymptotic Feller property
of the invariant Gibbs dynamics 
for the cubic SNLW on $\T^2$ with a slightly smoothed noise.
The ergodic property of the hyperbolic $\Phi^3_3$-model
is a challenging problem, in particular  due to its non-defocusing nature.

\end{remark}

\subsection{Construction of the $\Phi^3_3$-measure}
\label{SUBSEC:Gibbs}

In this subsection, we describe
a renormalization procedure 
and also a taming by the Wick-ordered $L^2$-norm
required to construct
the $\Phi^3_3$-measure in \eqref{H3}
and make 
a precise  statement (Theorem \ref{THM:Gibbs}).
For this purpose,  we first fix some notations. 
Given $ s \in \R$, 
let $\mu_s$ denote
a Gaussian measure with the Cameron-Martin space $H^s(\T^3)$,   formally defined by
\begin{align}
 d \mu_s 
   = Z_s^{-1} e^{-\frac 12 \| u\|_{{H}^{s}}^2} du
& =  Z_s^{-1} \prod_{n \in \Z^3} 
 e^{-\frac 12 \jb{n}^{2s} |\ft u(n)|^2}   
 d\ft u(n) , 
\label{gauss0}
\end{align}

\noi
where 
  $\jb{\,\cdot\,} = (1+|\,\cdot\,|^2)^\frac{1}{2}$.
When $ s= 1$, 
the Gaussian measure $\mu_s$ corresponds to 
 the massive Gaussian free field, 
 while it corresponds to  the white noise measure $\mu_0$ when $s = 0$.
For simplicity, 
we set 
\begin{align}
\mu = \mu_1
\qquad \text{and}
\qquad 
\muu = \mu \otimes \mu_{0} .
\label{gauss1}
\end{align}

Define the index sets $\Ld$ and $\Ld_0$ by 
\begin{align}
\Ld = \bigcup_{j=0}^{2} \Z^j\times \N \times \{ 0 \}^{2-j}
\qquad \text{and}\qquad \Ld_0 = \Ld \cup\{(0, 0, 0)\}
\label{index}
\end{align}

\noi
such that $\Z^3 = \Ld \cup (-\Ld) \cup \{(0, 0, 0)\}$.
Then, 
let 
$\{ g_n \}_{n \in \Ld_0}$ and $\{ h_n \}_{n \in \Ld_0}$
 be sequences of mutually independent standard complex-valued\footnote
{This means that $g_0,h_0\sim\NN_\R(0,1)$
and  
$\Re g_n, \Im g_n, \Re h_n, \Im h_n \sim \NN_\R(0,\tfrac12)$
for $n \ne 0$.}
 Gaussian random variables 
 and 
set $g_{-n} := \cj{g_n}$ and $h_{-n} := \cj{h_n}$ for $n \in \Ld_0$.
Moreover, we assume that 
$\{ g_n \}_{n \in \Ld_0}$ and $\{ h_n \}_{n \in \Ld_0}$ are independent from the space-time white noise $\xi$ in \eqref{SNLW0}.
We now define random distributions $u= u^\o$ and $v = v^\o$ by 
the following  Gaussian Fourier series:\footnote{By convention, 
 we endow $\T^3$ with the normalized Lebesgue measure $dx_{\T^3}= (2\pi)^{-3} dx$.}
\begin{equation} 
u^\o = \sum_{n \in \Z^3 } \frac{ g_n(\o)}{\jb{n}} e_n
 \qquad
\text{and}
\qquad
v^\o = \sum_{n\in \Z^3}  h_n(\o) e_n, 
\label{IV2}
\end{equation}

\noi
where  $e_n=e^{i n\cdot x}$.
Denoting by $\Law(X)$ the law of a random variable $X$ 
(with respect to the underlying probability measure $\PP$), 
we then have
\[
\Law (u, v) = \muu = \mu \otimes \mu_0
\] 

\noi
for $(u, v)$ in \eqref{IV2}.
Note that  $\Law (u, v) = \muu$ is supported on
\begin{align*}
\H^{s}(\T^3): = H^{s}(\T^3)\times H^{s - 1}(\T^3)
\end{align*}

\noi
for $s < -\frac 12$ but not for $s \geq -\frac 12$
(and more generally in $W^{s, p}(\T^3) \times W^{s-1, p}(\T^3)$
for any $1 \le p \le \infty$
 and $s < -\frac 12$).

 We now consider the $\Phi^3_3$-measure formally given by \eqref{H1}.
Since $u$ in the support of 
the massive Gaussian free field $\mu$ is merely a distribution, 
the cubic potential energy in \eqref{H1} is not well defined
 and thus a proper renormalization is required to give a meaning
 to the potential energy.
In order to explain the renormalization process, we first study the regularized model.

Given $N \in \N$, 
we denote by 
$\pi_N = \pi_N^\text{cube}$  the  frequency projector
onto the (spatial) frequencies
$\{n   = (n_1, n_2, n_3) \in\Z^3: \max_{j = 1, 2, 3} |n_j|\leq N\}$, 
defined by 
\begin{align}
\pi_N f = 
\pi_N^\text{cube} f
= 
\sum_{ n \in \Z^3  } \chi_N(n) \ft f (n)  e_n, 
\label{pi}
\end{align}

\noi
associated with  a Fourier multiplier $\chi_N = \chi_N^\text{cube}$:
\begin{align}
\chi_N(n) = 
\chi_N^\text{cube}(n) = 
\ind_Q\big(N^{-1}n\big), 
\label{chi}
\end{align}

\noi
where $Q$ denotes the cube of side length $2$  in $\R^3$ centered at the origin:
\begin{align}
Q = \big\{\xi   = (\xi_1, \xi_2, \xi_3) \in\R^3: \max_{j = 1, 2, 3} |\xi_j|\leq 1\big\}.
\label{Q1}
\end{align}

\noi
It turns out that, 
due to the critical nature of the $\Phi^3_3$-measure, 
a choice of frequency projectors makes
a difference.
See Remark \ref{REM:freq1}
and Subsection \ref{SUBSEC:freq} below
for discussions on different frequency projectors.
In comparing different frequency projectors, 
 we refer to $\pi_N = \pi_N^\text{cube}$ in~\eqref{pi} as the cube frequency projector
in the following.

Let   $u$  be as in \eqref{IV2}
and 
set $u_N = \pi_N u$.
For each fixed $x \in \T^3$, 
 $u_N(x)$ is
a mean-zero real-valued Gaussian random variable with variance
\begin{align}
\s_N = \E\big[u_N^2(x)\big] = \sum _{n \in \Z^3 } \frac{\chi^2_N(n)}{\jb{n}^2}
\sim N \too \infty, 
\label{sigma1}
\end{align}

\noi
as $N \to \infty$. Note that $\s_N$ is independent of $x\in \T^3$
due to the stationarity of $\mu$.
We define the Wick powers $:\! u_N^2 \!: $ and $:\! u_N^3 \!: $ by setting
\begin{align*}
:\! u_N^2 \!: \, = H_2(u_N; \s_N) = u_N^2 - \s_N
\quad \text{and} \quad 
:\! u_N^3 \!: \, = H_3(u_N; \s_N) = u_N^3 - 3 \s_N u_N, 
\end{align*}

\noi
where  $H_k (x,\s)$ denotes the Hermite polynomial of degree $k$ with variance parameter $\s$
 defined by the generating function:
\begin{align*}
 e^{tx - \frac{1}{2}\s t^2} = \sum_{k = 0}^\infty \frac{t^k}{k!} H_k(x;\s).
\end{align*}

\noi
This suggests us to consider the following
 renormalized potential energy:
\begin{align}
\begin{split}
R_N (u)
&= - \frac \s 3 \int_{\T^3} :\! u_N^3 \!: dx
+ A \bigg| \int_{\T^3} :\! u_N^2 \!: dx \bigg|^\g.
\end{split}
\label{K1}
\end{align}

\noi
As in the case of the  $\Phi^4_3$-measure in \cite{BG},
the renormalized potential energy $R_N(u)$
in \eqref{K1} is divergent (as $N \to \infty$)
and thus we need to introduce a further  renormalization.
This leads to  the following renormalized potential energy:
\begin{align} 
\begin{split}
 R_N^\dia (u)
=
R_N(u)
+ \al_N,
\end{split}
\label{K1r}
\end{align}
where $\al_N$ is a diverging constant (as $N \to \infty$)  defined in \eqref{YZ14} below.
Finally, we define 
 the truncated (renormalized) $\Phi^3_3$-measure $\rho_N$ by
\begin{align}
d \rho_N (u) = Z_N^{-1} e^{-R_N^\dia(u)} d \mu(u),
\label{GibbsN}
\end{align}
where the partition function $Z_N$ is given by
\begin{align}
Z_N = \int e^{-R_N^\dia(u)} d \mu(u).
\label{pfN1}
\end{align}

\noi
Then, we have the following construction
and non-normalizability of the $\Phi^3_3$-measure.
Due to the singularity of the $\Phi^3_3$-measure with respect to the 
base Gaussian measure $\muu$, we need to state our non-normalizability result
in a careful manner.
Compare this with \cite[Theorem~1.15]{OOTol1}
and  \cite[Theorem 1.3]{OS}.
See  the beginning of Section \ref{SEC:non}
for a further discussion.

\begin{theorem} \label{THM:Gibbs}
There exist $\s_1 \geq \s_0 > 0$ such that 
the following statements hold.
\smallskip

\begin{itemize}
\item[\textup{(i)}] 
\textup{(weakly nonlinear regime)}. 
Let $0 < |\s| < \s_0$.
Then, by choosing $\g = 3$ and $A = A(\s)>0$ sufficiently large, 
we have 
 the  uniform exponential integrability of the density\textup{:}
\begin{equation}
\sup_{N\in \N}Z_N = \sup_{N\in \N} \Big\| e^{- R_N^{\dia}(u)}\Big\|_{L^1(\mu)}
< \infty
\label{exp1c}
\end{equation}

\noi
and 
the truncated $\Phi^3_3$-measure $\rho_N$ in \eqref{GibbsN} 
converges weakly to a unique limit $\rho$, formally given by\footnote{By hiding $\al_N$
in \eqref{GibbsN} into the partition function $Z_N$, we could also say that 
the limiting $\Phi^3_3$-measure $\rho$ is formally given by  \eqref{H3} (with $\g = 3$).}
\begin{align}
d\rho(u) = Z^{-1}
 \exp \bigg( \frac \s3 \int_{\T^3} :\! u^3\!: \, dx - A
\bigg|\int_{\T^3} :\! u^2 \! : \, dx\bigg|^3 - \infty \bigg) d \mu(u).
\label{H3a}
\end{align}

\noi
In this case, the resulting $\Phi^3_3$-measure $\rho$ and the base massive Gaussian free field $\mu$ are mutually singular.

\smallskip
\item[\textup{(ii)}] \textup{(strongly nonlinear regime)}.
Let  $|\s| > \s_1$ and $\g\ge 3$.
Then, the $\Phi^3_3$-measure is not normalizable in the following sense.

Fix $\dl > 0$. Given $N \in \N$, 
let $\nu_{N, \dl}$ be  
 the following tamed version  of the truncated $\Phi^3_3$-measure\textup{:}
\begin{align}
d \nu_{N, \dl} (u) = Z_{N, \dl}^{-1} \exp\Big(-\dl \|\pi_N u\|_{B^{-\frac 34}_{3, \infty}}^{20} - R^\dia_N(u)\Big)  d\mu(u).
\label{tame1}
\end{align}

\noi
Then, $\{\nu_{N, \dl}\}_{N \in \N}$ converges
weakly to some limiting probability measure  $\nu_\dl$
and the following $\sigma$-finite version of the $\Phi^3_3$-measure\textup{:}
\begin{align*}
 d \cj   \rho_\dl
 &  = \exp\Big(\dl \|u\|_{B^{-\frac 34}_{3, \infty}}^{20}\Big) d \nu_\dl\\
&  = \lim_{N \to \infty} Z_{N, \dl}^{-1} \, \exp\Big(\dl \| u\|_{B^{-\frac 34}_{3, \infty}}^{20} \Big)
\exp\Big(-\dl \|\pi_N u\|_{B^{-\frac 34}_{3, \infty}}^{20} - R_N^\dia(u)\Big) d\mu(u)
\end{align*}

\noi
is a well-defined measure on $\C^{-100}(\T^3)$.
Furthermore, this $\s$-finite version  $\cj \rho_\dl$ of the $\Phi^3_3$-measure 
is not normalizable\textup{:}
\begin{align*}
\int 1\, d\cj \rho_\dl = \infty.
\end{align*}

Under the same assumption, 
the sequence $\{\rho_N\}_{N \in \N}$
of the truncated $\Phi^3_3$-measures in~\eqref{GibbsN}
does not converge to any weak limit, even up to a subsequence, 
as measures on  the Besov   space $B^{-\frac 34}_{3, \infty}(\T^3) \supset \C^{-\frac 34}(\T^3)$. 

\end{itemize}
\end{theorem}

In the weakly nonlinear regime, 
we also prove that the $\Phi^3_3$-measure $\rho$ is absolutely continuous
with respect to the shifted measure
$\Law (Y(1) +\s \ZZ(1) + \W(1))$, 
where  $\Law(Y(1) ) = \mu$, 
 $\ZZ = \ZZ(Y)$ is the limit of the quadratic process $\ZZ^N$ defined in \eqref{YZ12}, 
 and the auxiliary quintic process $\W = \W(Y)$ is defined in \eqref{AC0}.
While we do not use this property in this paper,
we present the proof in Appendix \ref{SEC:AC} for completeness.

As in case of the  $\Phi^4_3$-measure in \cite{BG}, 
we can prove uniform exponential integrability
of  the truncated density 
$ e^{- R_N^\dia  (u)}$
in $L^p(\mu)$
only for $p = 1$ due to 
the second renormalization introduced in \eqref{K1r}.
See also \cite{OOTol1, Bring1}
for a similar phenomenon
in the case of the defocusing Hartree $\Phi^4_3$-measure. 
We point out that 
the renormalized potential energy $R_N^{\dia}(u)$ in \eqref{K1r}
does {\it not} converge to any limit
and neither does the density $e^{-R_N^{\dia}(u)}$, 
which is essentially the source of the singularity of the $\Phi^3_3$-measure
with respect to the massive Gaussian free field $\mu$.

As in \cite{OOTol1}, 
following the variational approach introduced by Barashkov and Gubinelli \cite{BG}, 
we use the Bou\'e-Dupuis variational formula (Lemma \ref{LEM:var3})
to prove Theorem~\ref{THM:Gibbs}.
In fact, we make use of the
Bou\'e-Dupuis variational formula
in almost every single step of the proof.
In proving Theorem 
\ref{THM:Gibbs}\,(i), 
we first use the variational formula 
to  establish the uniform exponential integrability~\eqref{exp1c}
of  the truncated density 
$ e^{- R_N^\dia  (u)}$, 
from which 
tightness of the truncated $\Phi^3_3$-measure~$\rho_N$ in \eqref{GibbsN}
follows.
See Subsection \ref{SUBSEC:tight}.
Due to the  singularity of the $\Phi^3_3$-measure,
we need to apply a change of variables (see~\eqref{YZ13}) in the variational formulation
and thus 
we need to treat the taming part more carefully than that for the focusing Hartree $\Phi^4_3$-measure 
studied in \cite{OOTol1}.
See Lemma \ref{LEM:Dr8} below.
This lemma also reflects the critical nature of 
the $\Phi^3_3$-measure.

In Subsection \ref{SUBSEC:wcon}, 
we prove uniqueness of the limiting $\Phi^3_3$-measure.
Our main strategy is to follow
the  approach introduced in our previous work~\cite{OOTol1}
and compare two (arbitrary) subsequences $\rho_{N_{k_1}}$
and $\rho_{N_{k_2}}$, using the variational formula.
We point out, however,  that, due to the 
 critical nature of the $\Phi^3_3$-measure, 
 our uniqueness argument becomes more involved than
that in \cite[Subsection 6.3]{OOTol1}
for the subcritical 
defocusing Hartree $\Phi^4_3$-measure.
In particular, 
we need to make use of a certain orthogonality  property
to eliminate a problematic term.
See Remark \ref{REM:uniq1}.
See also Subsection \ref{SUBSEC:freq}.

In proving the singularity of the $\Phi^3_3$-measure, 
we once again follow the direct approach introduced in \cite{OOTol1}, 
making use of the variational formula.
We point out that 
 the proof of the singularity of the $\Phi^4_3$-measure by Barashkov and Gubinelli  \cite{BG2} goes through the shifted measure.
 On the other hand, as in \cite{OOTol1}, 
our proof is based on a direct argument  without referring to shifted measures. 
See Subsection \ref{SUBSEC:notAC}.

Let us now turn to the strongly nonlinear regime 
considered in Theorem \ref{THM:Gibbs}\,(ii).
As mentioned above, 
due to the singularity of the $\Phi^3_3$-measure, 
our formulation of the non-normalizability result in  
Theorem \ref{THM:Gibbs}\,(ii) is rather subtle.
In the situation where the truncated density 
$ e^{- R_N^\dia  (u)}$ converges to the limiting density
(as in \cite{OOTol1, OS}), 
it would suffice to prove 
\begin{align}
\sup_{N \in \N} \E_\mu \Big[ e^{- R_N^\dia  (u)} \Big] = \infty, 
\label{QQ3}
\end{align}

\noi
since \eqref{QQ3} would imply that 
there is no normalization constant
which would make the limit of the measure
$e^{- R_N^\dia  (u)} d\mu(u)$ into a probability measure.
In the current problem, however, 
 the potential energy $R_N^\dia(u)$ in \eqref{K1r}
(and the corresponding density $e^{-R_N^\dia(u)}$)
does {\it not} converge to any limit.
Thus, even if we prove a statement of the form  \eqref{QQ3}, 
we may still choose a sequence of constants
 $\ft Z_N$ such that 
the measures $\ft Z_N^{-1} e^{-R_N^\dia(u)} d\mu$ have a weak limit.
A similar phenomenon happens for  the $\Phi^4_3$-measure, 
where one needs to introduce the second order renormalization;
see~\cite{BG}.
The non-convergence of the truncated $\Phi^3_3$-measures
claimed in Theorem \ref{THM:Gibbs}\,(ii) 
tells us that this can not happen for the $\Phi^3_3$-measure.
See also Remark \ref{REM:A} below.

Our strategy is to first construct a $\s$-finite version of the $\Phi^3_3$-measure
and then prove its non-normalizability.
As stated in Theorem \ref{THM:Gibbs}\,(ii), 
we first introduce a tamed version $\nu_{N, \dl}$ of the truncated $\Phi^3_3$-measure, 
by introducing an appropriate taming function $F$; see \eqref{RM-1} below.
The first step is to show that this tamed truncated $\Phi^3_3$-measure 
 $\nu_{N, \dl}$ converges weakly to some limit $\nu_\dl$
 (Proposition \ref{PROP:ref}).
We then define a $\s$-finite version $\cj \rho_\dl$ of the $\Phi^3_3$-measure by setting
\begin{align*}
 d \cj   \rho_\dl = e^{\dl F(u)} d \nu_\dl
\end{align*}

\noi
and prove that $\cj \rho_\dl$ is not normalizable
(Proposition \ref{PROP:Gibbsref}).
Here, the $\s$-finite version $\cj \rho_\dl$ of the $\Phi^3_3$-measure
clearly depends on the choice of a taming function $F$.
Our choice is quite natural since 
the $\s$-finite version $\cj \rho_\dl$ of the $\Phi^3_3$-measure
is absolutely continuous with respect to 
the shifted measure
$\Law (Y(1) +\s \ZZ(1) + \W(1))$,
just like the (normalizable) $\Phi^3_3$-measure in the weakly nonlinear regime 
discussed above.
See Remark \ref{REM:ac}.

Once we construct the $\s$-finite version $\cj \rho_\dl$
of the $\Phi^3_3$-measure, our argument 
follows closely the strategy 
introduced 
in \cite{OOTol1, OS} for establishing
non-normalizability, using the  
 Bou\'e-Dupuis variational formula.
For this approach, 
we need to construct a drift achieving the desired divergence, 
where (the antiderivative of) the drift is designed to look like
``$-Y(1)$ + a perturbation'', 
where $\Law(Y(1) ) = \mu$; see \eqref{paa0} below.
Here,  the perturbation term is bounded in $L^2(\T^3)$
but has a large $L^3$-norm, thus having a highly concentrated
profile, such  as a soliton
or a finite time blowup profile.
As compared to our previous works \cite{OOTol1, OS}, 
there is an additional difficulty in proving
the non-normalizability claim 
in 
 Theorem \ref{THM:Gibbs}\,(ii)
due to the singularity of the $\Phi^3_3$-measure, 
which forces us to use 
a change of variables (see~\eqref{YZ13}) in the variational formulation.
See Remark \ref{REM:diff}.
The non-convergence of the truncated $\Phi^3_3$-measures $\rho_N$
stated in Theorem \ref{THM:Gibbs}\,(ii) follows
as a corollary to the non-normalizability 
of the $\s$-finite version $\cj \rho_\dl$ of the $\Phi^3_3$-measure;
see Proposition \ref{PROP:nonconv} and Subsection \ref{SUBSEC:nonconv}.
If the $\Phi^3_3$-measure existed as a probability measure
in the strongly nonlinear regime, 
then we would expect its support to be contained in $ \C^{-\frac 12-\eps}(\T^3)$
for any  $\eps > 0$, 
just as in the weakly nonlinear regime (and the $\Phi^4_3$-measure).
For this reason, the Besov space $B^{-\frac 34}_{3, \infty}(\T^3) \supset \C^{-\frac 34}(\T^3)$
 is a quite natural space to consider.
The restriction $\g \ge 3$ in Theorem \ref{THM:Gibbs}\,(ii) 
comes from 
the construction of the tamed 
version $\nu_\dl$ of the $\Phi^3_3$-measure; see~\eqref{QQ4} below.
For $\g < 3$, the taming by the Wick-ordered $L^2$-norm 
in~\eqref{H3} becomes weaker and thus we expect 
an analogous non-normalizability result to hold.

\begin{remark}\label{REM:freq1} \rm
We prove  Theorem \ref{THM:Gibbs}
for the cube frequency projector $\pi_N = \pi_N^\text{cube}$
defined in~\eqref{pi}.
If we instead consider the ball frequency projector $\pi_N^\text{ball}$
defined in \eqref{pib1} below, 
then our argument for the non-convergence claim in the strongly nonlinear regime
(Proposition~\ref{PROP:nonconv})
breaks down, while the other claims in Theorem 
\ref{THM:Gibbs} remain true for 
the ball frequency projector $\pi_N^\text{ball}$.
If we consider the smooth frequency projector $\pi_N^\text{smooth}$
defined in \eqref{pis1} below, 
then 
our argument for the uniqueness
of the limiting $\Phi^3_3$-measure in the weakly  nonlinear regime
(Proposition \ref{PROP:uniq})
breaks down.
In particular, the latter issue is closely related to the
critical nature of the $\Phi^3_3$-model
and, while we believe that uniqueness
of the limiting $\Phi^3_3$-measure holds
even in the case of 
 the smooth frequency projector $\pi_N^\text{smooth}$, 
it seems  non-trivial to prove this claim  by a modification of our argument.
We point out that the same issue  also appears in 
showing uniqueness of the limit  $\nu_\dl$ of 
 the tamed version $\nu_{N, \dl}$ of the truncated 
$\Phi^3_3$-measure  in \eqref{tame1}
in the strongly nonlinear regime (Proposition \ref{PROP:ref})
and  in the dynamical part (Proposition \ref{PROP:plan}).
See Subsection~\ref{SUBSEC:freq}
for a further discussion.
See also 
Remarks \ref{REM:uniq1} and~\ref{REM:non1}.

\end{remark}

\begin{remark}\label{REM:A} \rm
In the strongly nonlinear regime, 
Theorem \ref{THM:Gibbs}\,(ii) tells us that the truncated $\Phi^3_3$-measures
$\rho_N$ do not converge weakly to any limit as measures on 
 $B^{-\frac 34}_{3, \infty}(\T^3)\supset \C^{-\frac 34}(\T^3)$.
It is, however, possible that 
 the truncated $\Phi^3_3$-measures converges weakly to some limit (say, 
 the Dirac delta measure $\dl_0$ on the trivial function)
 as measures on some space with a very weak topology, say $\C^{-100}(\T^3)$.
Theorem \ref{THM:Gibbs}\,(ii) shows that 
if such weak convergence takes place, 
it must do so in a very pathological manner.

\end{remark}

\begin{remark}\rm The second renormalization in \eqref{K1r}
(i.e.~the cancellation  of the diverging constant $\al_N$)
 appears only at the level of the measure.
The associated equation (see \eqref{SNLW3} below)
does not see this additional renormalization.

\end{remark}

\begin{remark}\rm
It is of interest to investigate
a threshold value  $\s_* >0$
such that the construction of the $\Phi^3_3$-measure (Theorem \ref{THM:Gibbs}\,(i))
holds
for $0 < |\s| < \s_*$, 
while 
the non-normalizability of the $\Phi^3_3$-measure (Theorem \ref{THM:Gibbs}\,(ii))
holds
for $ |\s| > \s_*$.  
If such a threshold value $\s_*$ could be determined, 
it would also be of interest to determine
whether the $\Phi^3_3$-measure is normalizable
at the threshold $|\s| = \s_*$.
Such a problem, however, requires optimizing all the estimates
in the proof of Theorem \ref{THM:Gibbs}
and is out of reach at this point.
See  a recent work~\cite{OST}
by Sosoe and the first and third authors
for such analysis in 
 the one-dimensional case.

\end{remark}

\begin{remark}\label{REM:Gibbs}\rm

Consider the truncated Gibbs measure
$\rhoo_N = \rho_N \otimes \mu_0$
for the  hyperbolic $\Phi^3_3$-model \eqref{SNLW0}
with the density:
\begin{align}
d \rhoo_N (u, v) = Z_N^{-1} e^{-R_N^\dia(u)} d \muu(u, v),
\label{GibbsN2}
\end{align}

\noi
where $R_N^\dia(u)$  and $\muu$
are as in \eqref{K1r} and \eqref{gauss1}, respectively.
Since the potential energy $R_N^\dia(u)$  is independent of 
the second component $v$, 
Theorem \ref{THM:Gibbs} directly applies to the truncated Gibbs measure $\rhoo_N$.
In particular, 
in the weakly nonlinear regime ($0 < |\s| < \s_0$), 
 the truncated Gibbs measure $\rhoo_N$
 converges weakly to the limiting Gibbs measure 
 \begin{align}
 \rhoo = \rho \otimes \mu_0, 
 \label{Gibbs2}
 \end{align}
 where $\rho$ is the limiting $\Phi^3_3$-measure
 constructed in Theorem \ref{THM:Gibbs}\,(i).
 Moreover, the limiting Gibbs measure $\rhoo$
 and the base Gaussian measure $\muu = \mu \otimes \mu_0$
 are  mutually singular.

\end{remark}

\subsection{Hyperbolic $\Phi^3_3$-model}
\label{SUBSEC:1.3}

In this subsection, 
we 
provide
 a precise meaning to the hyperbolic $\Phi^3_3$-model \eqref{SNLW0}
and make Theorem \ref{THM:GWP0} more precise.
By considering the Langevin equation
for the Gibbs measure $\rhoo = \rho \otimes \mu_0$ constructed in Remark \ref{REM:Gibbs}, 
we formally obtain the following 
quadratic SdNLW (= the hyperbolic $\Phi^3_3$-model):
\begin{align}
\dt^2 u + \dt u + (1 -  \Dl)  u - \s :\! u^2 \!: + \,M (\,:\! u^2 \!:\,) u = \sqrt{2} \xi,
\label{SNLW1}
\end{align}

\noi
where $M$ is defined by 
\begin{align}
M (w) = 6A \bigg| \int_{\T^3} w dx \bigg| \int_{\T^3} w dx
\label{addM}.
\end{align}

\noi
Here, the term  $M (\,:\! u^2 \!:\,) u$ in \eqref{SNLW1}
comes from the taming by the Wick-ordered $L^2$-norm
appearing in~\eqref{H3a}.
The term 
$:\! u^2 \!:$ denotes the Wick renormalization\footnote{In order to give a proper meaning to 
 $:\! u^2 \!:$, we need to assume a structure on $u$.
 We postpone this discussion to Section \ref{SEC:LWP}.} 
 of $u^2$, formally given by 
$:\! u^2 \!:  \, = u^2 - \infty$.
Namely, the equation~\eqref{SNLW1}
is just a formal expression at this point.
In the following, we provide the meaning
of the process $u$ in \eqref{SNLW1}
by a limiting procedure.
In Section \ref{SEC:LWP}, 
we use the paracontrolled calculus to give a more precise meaning
to \eqref{SNLW1} by rewriting it into a system for three unknowns.
See~\eqref{SNLW6} below.

Given $N \in \N$, we consider the following quadratic SdNLW with 
a truncated noise:
\begin{align}
\begin{split}
\dt^2 & u_N + \dt u_N  + (1 -  \Dl)  u_N  
-\s    :\!  u_N^2 \!: 
+  M (\,:\! u_N^2 \!:\,)  u_N 
= \sqrt{2} \pi_N \xi, 
\end{split}
\label{SNLW2}
\end{align}

\noi
where $\pi_N$ is as in \eqref{pi} and the renormalized nonlinearity is defined by 
\begin{align}
 :\!  u_N^2 \!:  \, = 
 u_N^2 -\s_N
\label{stoconv1}
 \end{align}

\noi
with $\s_N$ as in \eqref{sigma1}.  See also \eqref{sigma1a}.
In Section \ref{SEC:LWP}, 
we study SdNLW \eqref{SNLW2} with the truncated noise
and prove the following local well-posedness statement
for the hyperbolic $\Phi^3_3$-model.

\begin{theorem}\label{THM:LWP0}
Given $  s > \frac 12$, 
let $(u_0, u_1) \in \H^{s}(\T^3)$.
Let $(\phi_0^\o, \phi^\o_1)$ be a pair of the Gaussian random distributions
with $\Law (\phi_0^\o, \phi^\o_1) = \muu = \mu \otimes \mu_0$.
Then, the solution $(u_N, \dt u_N)$ to the quadratic SdNLW \eqref{SNLW2}
with the truncated noise  and  the initial data
\begin{align}
(u_N, \dt u_N)|_{t = 0} = (u_0, u_1) + \pi_N (\phi_0^\o, \phi^\o_1)
\label{IV3}
\end{align}

\noi
converges  to a stochastic process $(u, \dt u) \in C([0, T]; \H^{-\frac 12 -\eps} (\T^3))$ almost surely,
where $T = T(\o)$ is an almost surely positive stopping time.
\end{theorem}

The limit $(u, \dt u)$ formally satisfies 
the equation \eqref{SNLW1}.
Here, we took the initial data of the form \eqref{IV3} 
for simplicity of the presentation.
A slight modification of the proof
yields an analogue of Theorem \ref{THM:LWP0}
with deterministic initial data
$(u_N, \dt u_N)|_{t = 0} = (u_0, u_1)$.
In this case, we need to choose a diverging 
constant $\s_N$, depending on $t.$
See \cite{GKO, GKO2} for such an argument.

We follow the paracontrolled approach in \cite{GKO2}, 
where the quadratic SNLW on $\T^3$ was studied.
However, the additional term $M$ in \eqref{SNLW1} and \eqref{SNLW2}
contains an ill-defined product $:\!u^2\!:$
(or $:\!u_N^2\!:$ in the limiting sense).
In order to treat this term, 
the analysis in \cite{GKO2} is not sufficient 
and thus 
we also need to adapt the paracontrolled analysis in our previous work \cite{OOTol1}
and rewrite the equation into a system for three unknowns.
(Note that in \cite{GKO2}, the resulting system
was for two unknowns.)
We also point out that, unlike \cite{GKO2}
(see also \cite{MW1} in the context of the parabolic $\Phi^4_3$-model), 
the equation for a less regular,  paracontrolled component
in our system (see \eqref{SNLW6} below) is nonlinear in the unknowns.
We then construct a continuous map 
from the space of enhanced data sets
to solutions.
While the proof of Theorem \ref{THM:LWP0}
follows from  a slight modification of 
the arguments in \cite{GKO2, OOTol1},
we present details  in Section \ref{SEC:LWP} for readers' convenience.

In order to establish our main goal in the dynamical part of the program (Theorem \ref{THM:GWP0}), 
we need to study the hyperbolic $\Phi^3_3$-model with the Gibbs measure initial
data.
Since the Gibbs measure $\rhoo = \rho \otimes \mu_0$ in \eqref{Gibbs2}
and the Gaussian field $\muu = \mu \otimes \mu_0$
are mutually singular as shown in Theorem \ref{THM:Gibbs}, 
it may seem that the local well-posedness in Theorem \ref{THM:LWP0}
with the Gaussian initial data (plus smoother deterministic initial data)
is irrelevant.
However, as we see in Section \ref{SEC:GWP},  
the analysis  for proving Theorem \ref{THM:LWP0} provides
us with a good intuition of the well-posedness problem
for the hyperbolic $\Phi^3_3$-model with the Gibbs measure initial data.
Furthermore, one of  advantages of considering the Gaussian initial data
(as in \eqref{IV3}) is that it provides a clear reason why 
 $\s_N$ appears in the renormalization in \eqref{stoconv1},
since $\s_N$ is nothing but the variance of the first order
approximation (= the stochastic convolution defined in \eqref{W2})
to the solution to \eqref{SNLW2}; see \eqref{sigma1a}.
This is the main reason for considering the local-in-time problem 
with the Gaussian initial data.

\medskip

Next, we turn our attention to the globalization problem.
For this purpose, we need to consider a different approximating equation.
Given $N \in \N$, we consider the truncated hyperbolic $\Phi^3_3$-model:
\begin{align}
\begin{split}
\dt^2 & u_N + \dt u_N  + (1 -  \Dl)  u_N  \\
& 
-\s  \pi_N \big(  :\! (\pi_N u_N)^2 \!:  \big)
+  M (\,:\! (\pi_N u_N)^2 \!:\,) \pi_N u_N 
= \sqrt{2} \xi, 
\end{split}
\label{SNLW3}
\end{align}

\noi
where 
$:\! (\pi_N u_N)^2 \!:  \, = 
(\pi_N u_N)^2 -\s_N.$
A slight modification of the proof of Theorem \ref{THM:LWP0}
yields uniform (in $N$) local well-posedness
of the truncated equation \eqref{SNLW3}
(with the same limiting process $(u, \dt u)$ as in Theorem \ref{THM:LWP0})
for the initial data of the form \eqref{IV3}.
By exploiting (formal) invariance of the truncated Gibbs measure
$\rhoo_N$ in \eqref{GibbsN2},\footnote{This is essentially Bourgain's invariant measure argument
\cite{BO94} applied to the truncated hyperbolic $\Phi^3_3$-model~\eqref{SNLW3}, 
whose nonlinear part is finite dimensional.} 
we see that the truncated hyperbolic $\Phi^3_3$-model \eqref{SNLW3} is almost surely globally well-posed
with respect to the truncated Gibbs measure $\rhoo_N$
and, moreover, $\rhoo_N$ is invariant under the resulting dynamics;
see Lemma~\ref{LEM:GWP4}.

We now state almost sure global well-posedness
of the hyperbolic $\Phi^3_3$-model.

\begin{theorem}\label{THM:GWP}
Let $0 < |\s| < \s_0$ and $A = A(\s)> 0$ is sufficiently large
 as in Theorem~\ref{THM:Gibbs}\,(i).
Then, 
there exists a non-trivial stochastic process $(u, \dt u) \in C(\R_+; \H^{-\frac 12 -\eps}(\T^3))$ for any $\eps>0$ such that,
given any $T>0$, the solution $(u_N, \dt u_N)$ to the truncated hyperbolic $\Phi^3_3$-model~\eqref{SNLW3}
with the random initial data distributed by the truncated Gibbs measure 
$\rhoo_N = \rho_N \otimes \mu_0$ in \eqref{GibbsN2}
converges to $(u, \dt u)$ in $C([0,T]; \H^{-\frac 12-\eps}(\T^3))$.
Furthermore, 
we have $\Law\big((u(t), \dt u(t))\big) = \rhoo$ for any $t \in \R_+$.

\end{theorem}

The main difficulty in proving Theorem \ref{THM:GWP}
comes from the mutual singularity of 
the Gibbs measure $\rhoo$ and the base Gaussian measure $\muu$
(and the fact that
 the truncated  Gibbs measure $\rhoo_N$  converges to $\rhoo$
only weakly)
such that 
Bourgain's invariant measure argument \cite{BO94, BO96} is not directly applicable.
In the context of the defocusing Hartree NLW on $\T^3$, 
Bringmann~\cite{Bring2} encountered the same issue, 
and introduced a new globalization argument, 
where a large time stability theory (in the paracontrolled setting) plays a crucial role.
Bourgain's invariant measure argument is often described
(see \cite{Bring2})
as ``the probabilistic version of a deterministic global theory using a (sub-critical) conservation law".
In~\cite{Bring2}, Bringmann considers the quantity  $\vec \rho_M((u_N, \dt u_N)(t) \in A)$, 
where $(u_N, \dt u_N)$ is the solution to the truncated equation 
with a cutoff parameter $N$.
While such an expression  is not conserved for $M \ne N$, 
it should be close to being constant in time when $M, N \gg 1$.
For this reason, he describes
his new globalization argument as 
``the probabilistic version of a deterministic global theory using almost conservation laws''.
We also point out that Bringmann's analysis
 relies on 
the fact that 
 the (truncated) Gibbs measure
is absolutely continuous with respect to a shifted  measure
\cite{OOTol1, Bring1} (as in Appendix \ref{SEC:AC} below).

While it is possible to follow Bringmann's approach, 
we instead introduce a new simple alternative argument
to prove almost sure global well-posedness.
Our approach consists of the following four steps:

\begin{itemize}
\item[{\bf 1.}]
We first 
establish a uniform (in $N$)  exponential integrability of the truncated enhanced data set 
(see \eqref{data3x} below)
with respect to the truncated measure
(Proposition~\ref{PROP:exptail2}).
We directly achieve this by combining the variational approach
with space-time estimates
{\it without} any reference to (the truncated version of) the shifted measure 
 constructed in Appendix~\ref{SEC:AC}.

\smallskip

\item[{\bf 2.}] 
Next,  by  a slight modification 
of the local well-posedness argument, 
 we prove   a stability result (Proposition \ref{PROP:LWPv}).
This is done  by a simple contraction argument, 
with  an exponentially decaying weight in time.

\smallskip
\item[{\bf 3.}] 
Then,  using the invariance of the truncated Gibbs measure, 
we establish a uniform (in $N$) control
on the solution to the truncated system (see \eqref{SNLW9} below)
with a large probability.
The argument relies on
a discrete Gronwall argument but is very straightforward.

\smallskip
\item[{\bf 4.}]
In the last step, 
we study  the convergence property of 
the distributions of the truncated enhanced data sets, 
emanating  from the truncated Gibbs measures. 
In particular, 
we study the Wasserstein-1 distance
of such a distribution with the limiting distribution, 
using ideas from theory of optimal transport (the Kantorovich duality).
See Proposition~\ref{PROP:plan} below.

\end{itemize}

\smallskip

\noi
Once we establish these four steps, 
Theorem \ref{THM:GWP} follows in a straightforward manner.
We believe that our new globalization argument is 
very simple, at least at a conceptual level,
and is easy to implement. 
See Section \ref{SEC:GWP}
for further details.

\begin{remark}\rm 
(i)
In this paper, we treated the hyperbolic $\Phi^3_3$-model.
In the three-dimensional case, it is possible
to consider the defocusing quartic interaction potential,
namely the $\Phi^4_3$-measure.
This leads to the following
 hyperbolic $\Phi^4_3$-model on $\T^3$:
\begin{align}
\dt^2 u + \dt u + (1 -  \Dl)  u + u^3  = \sqrt{2} \xi.
\label{NLWd}
\end{align}

\noi
Over the last ten years,  the parabolic $\Phi^4_3$-model:
\begin{align}
 \dt u +  (1 -  \Dl)  u + u^3  =  \sqrt 2\xi, 
\label{heat0}
\end{align}

\noi
has been studied extensively by many authors.
See
 \cite{Hairer, GIP, CC, Kupi, MW1, MWX, AK, GH18} and references therein. 
Up to date, the well-posedness issue of the  hyperbolic $\Phi^4_3$-model \eqref{NLWd}
 remains as an important open problem.\footnote{In a recent preprint \cite{BDNY},  Bringmann, Deng, Nahmod, and Yue
resolved this open problem 
in the case of the Gibbsian initial data with no stochastic forcing.}
In \cite{OWZ}, using Bringmann's analysis \cite{Bring2}, 
Y.~Wang, Zine, and the first author recently proved local well-posedness
of the cubic stochastic NLW\footnote{In \cite{OWZ}, 
the authors considered the undamped SNLW but the same analysis applies
to the damped SNLW.} on $\T^3$ with an almost space-time white noise forcing
(i.e.~replacing $\xi$ by $\jb{\nb}^{-\al}\xi$ for any $\al >0$
  in~\eqref{NLWd}).

\smallskip

\noi
(ii) In the parabolic setting \eqref{heat1}, 
there is no issue is applying Bourgain's invariant measure argument
in the usual manner since it is possible to prove local well-posedness
with deterministic initial data at the regularity of the $\Phi^3_3$-measure.
See \cite{HM} in the case of the parabolic $\Phi^4_3$-model \eqref{heat0}.

\end{remark}

\subsection{On frequency projectors}
\label{SUBSEC:freq}

We conclude this introduction by 
discussing different frequency projectors.
Given $N \in \N$, 
define 
 the   ball frequency projector
 $\pi_N^\text{ball}$ 
onto the  frequencies  $\{n\in\Z^3:|n|\leq N\}$ by setting
\begin{align}
\pi_N^\text{ball} f = 
\sum_{ n\in \Z^3 }\chi^\text{ball}_N(n) \ft f (n)  e_n, 
\label{pib1}
\end{align}

\noi
associated with  a Fourier multiplier 
\begin{align*}
 \chi_N^\text{ball} (n) = \ind_B\big(N^{-1}n\big), 
\end{align*}

\noi
where $B$ denotes the unit ball in $\R^3$ centered at the origin:
\[B = \big\{\xi   = (\xi_1, \xi_2, \xi_3) \in\R^3: |\xi|\leq 1\big\}.\]

\noi
We also 
define 
  the  smooth frequency projector
  $ \pi_N^\text{smooth}$
onto the  frequencies  $\{n\in\Z^3:|n|\leq N\}$
by setting
\begin{align}
\pi_N^\text{smooth} f = 
\sum_{ n\in \Z^3 } \chi_N^\text{smooth}(n) \ft f (n)  e_n, 
\label{pis1}
\end{align}

\noi
associated with  a Fourier multiplier  
\begin{align*}
\chi_N^\text{smooth} (n) = \chi\big(N^{-1}n\big)
\end{align*}

\noi
for some fixed  even function $\chi \in C^\infty_c(\R^3; [0, 1])$
with $\supp \chi \subset \{\xi\in\R^3:|\xi|\leq 1\}$ and $\chi\equiv 1$ 
on $\{\xi\in\R^3:|\xi|\leq \tfrac12\}$.

In Subsections \ref{SUBSEC:Gibbs}
and \ref{SUBSEC:1.3}, 
we stated the (non-)construction of the $\Phi^3_3$-measure
(Theorem~\ref{THM:Gibbs})
and the dynamical results for the hyperbolic $\Phi^3_3$-model
(Theorems \ref{THM:LWP0} and \ref{THM:GWP}), 
using the cube frequency projector
$\pi_N = \pi_N^\text{cube}$ defined in \eqref{pi}.
In comparison with the ball frequency projector $\pi_N^\text{ball}$
and the smooth frequency projector $\pi_N^\text{smooth}$, 
there are two important properties
that  the cube frequency projector 
$\pi_N^\text{cube}$ possesses simultaneously.

\begin{itemize}
\item[(i)]
As a composition of  (modulated) Hilbert transforms
in different coordinate directions, 
 the cube frequency projector 
$\pi_N^\text{cube}$ is uniformly (in $N$) bounded in $L^p(\T^3)$
for any $ 1 < p < \infty$.  

\smallskip

\item[(ii)]
The cube frequency projector is indeed a projection, in particular satisfying 
$(\Id - \pi_N^\text{cube})  \pi_N^\text{cube} = 0 $.

\end{itemize}

\smallskip

\noi
We make use of both of these properties in a crucial manner.
Note that while the  ball frequency projector $\pi_N^\text{ball}$
satisfies the property (ii), it is bounded in $L^p(\T^3)$ 
only for $p = 2$~\cite{Feff} and thus the property (i) is not satisfied.
On the other hand, 
by Young's inequality, 
the smooth frequency projector $\pi_N^\text{smooth}$
is bounded on $L^p(\T^3)$ for any $ 1\le p \le \infty$ but
it does not satisfy the property (ii).

Roughly speaking, Theorem \ref{THM:Gibbs} on the (non-)construction of the $\Phi^3_3$-measure consists
of the following five results:

\begin{itemize}
\item[(1)]
 the uniform exponential integrability \eqref{exp1c} and tightness
of the truncated $\Phi^3_3$-measures $\rho_N$
in the weakly nonlinear regime,

\smallskip

\item[(2)] uniqueness of the limiting $\Phi^3_3$-measure
in the weakly nonlinear regime,

\smallskip

\item[(3)] mutual singularity of the $\Phi^3_3$-measure
and the base Gaussian free field
in the weakly nonlinear regime, 

\smallskip

\item[(4)] non-normalizability
of the $\Phi^3_3$-measure
in the strongly nonlinear regime, 

\smallskip

\item[(5)] non-convergence
of the truncated $\Phi^3_3$-measures $\rho_N$
in the strongly nonlinear regime.

\end{itemize}

\noi
Starting with the truncated $\Phi^3_3$-measures $\rho_N$
in \eqref{GibbsN} defined in terms of the cube frequency projector $\pi_N^\text{cube}$
in \eqref{pi}, we establish (1) - (5) in Sections \ref{SEC:Gibbs} and \ref{SEC:non}.
In proving (5), the property~(i) above plays an important role
and thus our argument does not apply to 
 the ball frequency projector $\pi_N^\text{ball}$.
See Remark \ref{REM:non1}.

In establishing (2), uniqueness of the limiting $\Phi^3_3$-measure
(Proposition \ref{PROP:uniq}), 
we crucially make use of the property (ii) to show that a certain problematic term vanishes;
see $\1_2$ in~\eqref{KK0a}.
It turns out that this problematic term reflects the critical nature of the problem, 
where there is no room to spare, not even logarithmically.
In the case of the cube frequency projector $\pi_N^\text{cube}$, 
the property (ii) allows us to conclude that this term in fact vanishes.
In the case of the smooth projector $\pi_N^\text{smooth}$, 
the property~(ii) does not hold
and thus we need to show by hand that this problematic  term tends to 0.
As mentioned above, however, there is no room to spare
and it seems rather non-trivial to prove such a convergence result
by a modification of our argument.
See Remark \ref{REM:uniq1}.
In establishing (4) and (5), 
we first construct a
reference measure $\nu_\dl$
as a limit  of the tamed version  $\nu_{N, \dl}$ of the truncated 
$\Phi^3_3$-measure in \eqref{tame1}
(Proposition~\ref{PROP:ref}).
With the smooth projector $\pi_N^\text{smooth}$, 
the same issue also appears in 
showing uniqueness of the limit~$\nu_\dl$.

While we believe that 
Theorem \ref{THM:Gibbs} holds
for both 
 the ball frequency projector $\pi_N^\text{ball}$ (in particular (5) above)
 and the smooth frequency projector $\pi_N^\text{smooth}$ (in particular (2) above), 
we do not pursue these issues further in this paper 
in order to keep the paper length under control.

Let us now turn to 
 the dynamical part.
 As for the smooth frequency projector $\pi_N^\text{smooth}$, 
 there is no modification needed for the local well-posedness part.
However, as mentioned above, 
there is no   uniqueness of the limiting $\Phi^3_3$-measure
in this case.
Furthermore, 
we point out that 
the proof of Proposition \ref{PROP:plan}
also breaks down for 
the smooth frequency projector $\pi_N^\text{smooth}$
since part of the argument relies on 
 the proof of Proposition \ref{PROP:uniq}; see \eqref{Ga8}.
On the other hand, as for the  ball frequency projector $\pi_N^\text{ball}$, 
both Theorems \ref{THM:LWP0} and \ref{THM:GWP} hold
as they are stated.
However, the proof of the local well-posedness part needs to be modified
in view of the unboundedness of the ball frequency projector $\pi_N^\text{ball}$
in the Strichartz spaces (see~\eqref{M0}).
Note that this issue can be easily remedied by using
the Fourier restriction norm method
via the ($L^2$-based) $X^{s, b}$-spaces as in \cite{OTh2, Bring2, OWZ}.

\section{Notations and basic lemmas}
\label{SEC:2}

In describing regularities of functions and distributions, 
we  use $\eps > 0$ to denote a small constant.
We usually  suppress the dependence on such $\eps > 0$ in an estimate.
For $a, b > 0$, we use $a\lesssim b$ to mean that
there exists $C>0$ such that $a \leq Cb$.
By $a\sim b$, we mean that $a\lesssim b$ and $b \lesssim a$.

In dealing with space-time functions, 
we use the following short-hand notation 
$L^q_TL^r_x$ = $L^q([0, T]; L^r(\T^3))$, etc.

\subsection{Sobolev 
 and Besov spaces}
\label{SUBSEC:21}

Let $s \in \R$ and $1 \leq p \leq \infty$.
We define the $L^2$-based Sobolev space $H^s(\T^d)$
by the norm:
\begin{align*}
\| f \|_{H^s} = \| \jb{n}^s \ft f (n) \|_{\l^2_n}.
\end{align*}

\noi
We also define the $L^p$-based Sobolev space $W^{s, p}(\T^d)$
by the norm:
\begin{align*}
\| f \|_{W^{s, p}} = \big\| \F^{-1} [\jb{n}^s \ft f(n)] \big\|_{L^p}.
\end{align*}

\noi
When $p = 2$, we have $H^s(\T^d) = W^{s, 2}(\T^d)$.

Let $\phi:\R \to [0, 1]$ be a smooth  bump function supported on $[-\frac{8}{5}, \frac{8}{5}]$ 
and $\phi\equiv 1$ on $\big[-\frac 54, \frac 54\big]$.
For $\xi \in \R^d$, we set $\varphi_0(\xi) = \phi(|\xi|)$
and 
\begin{align}
\varphi_{j}(\xi) = \phi\big(\tfrac{|\xi|}{2^j}\big)-\phi\big(\tfrac{|\xi|}{2^{j-1}}\big)
\label{phi1}
\end{align}

\noi
for $j \in \N$.
Then, for $j \in \Z_{\geq 0} := \N \cup\{0\}$, 
we define  the Littlewood-Paley projector  $\P_j$ 
as the Fourier multiplier operator with a symbol $\varphi_j$.
Note that we have 
\begin{align*}
\sum_{j = 0}^\infty \varphi_j (\xi) = 1
\end{align*}

\noi
 for each $\xi \in \R^d$. 
Thus, 
we have 
\[ f = \sum_{j = 0}^\infty \P_j f. \]

Let us now recall
 the definition and basic properties of  paraproducts
 introduced by Bony~\cite{Bony}.
 See~\cite{BCD, GIP} for further details.
Given two functions $f$ and $g$ on $\T^3$
of regularities $s_1$ and $s_2$, 
we write the product $fg$ as
\begin{align}
\begin{split}
fg & \hspace*{0.3mm}
= f\pl g + f \pe g + f \pg g\\
& := \sum_{j < k-2} \P_{j} f \, \P_k g
+ \sum_{|j - k|  \leq 2} \P_{j} f\,  \P_k g
+ \sum_{k < j-2} \P_{j} f\,  \P_k g.
\end{split}
\label{para1}
\end{align}

\noi
The first term 
$f\pl g$ (and the third term $f\pg g$) is called the paraproduct of $g$ by $f$
(the paraproduct of $f$ by $g$, respectively)
and it is always well defined as a distribution
of regularity $\min(s_2, s_1+ s_2)$.
On the other hand, 
the resonant product $f \pe g$ is well defined in general 
only if $s_1 + s_2 > 0$. See Lemma \ref{LEM:para} below.
In the following, we also use the notation $f\pge g := f\pg g + f\pe g$.
In studying a  nonlinear problem, 
main difficulty usually arises in making sense 
of a product. 
Since paraproducts are always well defined, 
such a problem comes from a resonant product.
In particular, when the sum of regularities is negative, 
we need to impose an extra structure
to make sense of a  (seemingly)  ill-defined resonant product.
See Section \ref{SEC:LWP}
for a further discussion on the paracontrolled approach in this direction.

Next, we  recall the basic properties of the Besov spaces $B^s_{p, q}(\T^d)$
defined by the norm:
\begin{equation*}
\| u \|_{B^s_{p,q}} = \Big\| 2^{s j} \| \P_{j} u \|_{L^p_x} \Big\|_{\l^q_j(\Z_{\geq 0})}.
\end{equation*}

\noi
We denote the H\"older-Besov space by  $\C^s (\T^d)= B^s_{\infty,\infty}(\T^d)$.
Note that (i)~the parameter $s$ measures differentiability and $p$ measures integrability, 
(ii)~$H^s (\T^d) = B^s_{2,2}(\T^d)$,
and (iii)~for $s > 0$ and not an integer, $\C^s(\T^d)$ coincides with the classical H\"older spaces $C^s(\T^d)$;
see \cite{Graf}.

We recall the basic estimates in Besov spaces.
See \cite{BCD, GOTW} for example.

\begin{lemma}\label{LEM:Bes}
The following estimates hold.

\noi
\textup{(i) (interpolation)} 
Let $s, s_1, s_2 \in \R$ and $p, p_1, p_2 \in (1,\infty)$
such that $s = \ta s_1 + (1-\ta) s_2$ and $\frac 1p = \frac \ta{p_1} + \frac{1-\ta}{p_2}$
for some $0< \ta < 1$.
Then, we have
\begin{equation}
\| u \|_{W^{s,  p}} \les \| u \|_{W^{s_1, p_1}}^\ta \| u \|_{W^{s_2, p_2}}^{1-\ta}.
\label{interp}
\end{equation}

\noi
\textup{(ii) (immediate  embeddings)}
Let $s_1, s_2 \in \R$ and $p_1, p_2, q_1, q_2 \in [1,\infty]$.
Then, we have
\begin{align} 
\begin{split}
\| u \|_{B^{s_1}_{p_1,q_1}} 
&\les \| u \|_{B^{s_2}_{p_2, q_2}} 
\qquad \text{for $s_1 \leq s_2$, $p_1 \leq p_2$,  and $q_1 \geq q_2$},  \\
\| u \|_{B^{s_1}_{p_1,q_1}} 
&\les \| u \|_{B^{s_2}_{p_1, \infty}}
\qquad \text{for $s_1 < s_2$},\\
\| u \|_{B^0_{p_1, \infty}}
 &  \les  \| u \|_{L^{p_1}}
 \les \| u \|_{B^0_{p_1, 1}}.
\end{split}
\label{embed}
\end{align}

%


\smallskip

\noi
\textup{(iii) (Besov embedding)}
Let $1\leq p_2 \leq p_1 \leq \infty$, $q \in [1,\infty]$,  and  $s_2 \ge s_1 + d \big(\frac{1}{p_2} - \frac{1}{p_1}\big)$. Then, we have
\begin{equation*}
 \| u \|_{B^{s_1}_{p_1,q}} \les \| u \|_{B^{s_2}_{p_2,q}}.
\end{equation*}

\smallskip

\noi
\textup{(iv) (duality)}
Let $s \in \mathbb{R}$
and  $p, p', q, q' \in [1,\infty]$ such that $\frac1p + \frac1{p'} = \frac1q + \frac1{q'} = 1$. Then, we have
\begin{equation}
\bigg| \int_{\T^d}  uv \, dx \bigg|
\le \| u \|_{B^{s}_{p,q}} \| v \|_{B^{-s}_{p',q'}},
\label{dual}
\end{equation}

\noi
where $\int_{\T^d} u v \, dx$ denotes  the duality pairing between $B^{s}_{p,q}(\T^d)$ and $B^{-s}_{p',q'}(\T^d)$.

\smallskip
	
\noi		
\textup{(v) (fractional Leibniz rule)} 
Let $p, p_1, p_2, p_3, p_4 \in [1,\infty]$ such that 
$\frac1{p_1} + \frac1{p_2} 
= \frac1{p_3} + \frac1{p_4} = \frac 1p$. 
Then, for every $s>0$, we have
\begin{equation}
\| uv \|_{B^{s}_{p,q}} \les  \| u \|_{B^{s}_{p_1,q}}\| v \|_{L^{p_2}} + \| u \|_{L^{p_3}} \| v \|_{B^s_{p_4,q}} .
\label{prod}
\end{equation}

\end{lemma}

The interpolation \eqref{interp} follows
from the Littlewood-Paley characterization of Sobolev norms via the square function
and H\"older's inequality.

\begin{lemma}[paraproduct and resonant product estimates]
\label{LEM:para}
Let $s_1, s_2 \in \R$ and $1 \leq p, p_1, p_2, q \leq \infty$ such that 
$\frac{1}{p} = \frac 1{p_1} + \frac 1{p_2}$.
Then, we have 
\begin{align}
\| f\pl g \|_{B^{s_2}_{p, q}} \les 
\|f \|_{L^{p_1}} 
\|  g \|_{B^{s_2}_{p_2, q}}.  
\label{para2a}
\end{align}

\noi
When $s_1 < 0$, we have
\begin{align}
\| f\pl g \|_{B^{s_1 + s_2}_{p, q}} \les 
\|f \|_{B^{s_1 }_{p_1, q}} 
\|  g \|_{B^{s_2}_{p_2, q}}.  
\label{para2}
\end{align}

\noi
When $s_1 + s_2 > 0$, we have
\begin{align}
\| f\pe g \|_{B^{s_1 + s_2}_{p, q}} \les 
\|f \|_{B^{s_1 }_{p_1, q}} 
\|  g \|_{B^{s_2}_{p_2, q}}  .
\label{para3}
\end{align}

\end{lemma}

The product estimates \eqref{para2a},  \eqref{para2},  and \eqref{para3}
follow easily from the definition \eqref{para1} of the paraproduct 
and the resonant product.
See \cite{BCD, MW2} for details of the proofs in the non-periodic case
(which can be easily extended to the current periodic setting).

We also recall the following product estimate from \cite{GKO, BOZ}.

\begin{lemma}\label{LEM:gko}
Let $s> 0$.

\smallskip

\noi
\textup{(i)}
Let  $1<p_j,q_j,r\le\infty$, $j=1,2$ such that $\frac{1}{r}=\frac{1}{p_j}+\frac{1}{q_j}$.
Then, we have 
$$
\|\jb{\nb}^s(fg)\|_{L^r(\T^3)}\lesssim\| \jb{\nb}^s f\|_{L^{p_1}(\T^3)} \|g\|_{L^{q_1}(\T^3)}+ \|f\|_{L^{p_2}(\T^3)} 
\|  \jb{\nb}^s g\|_{L^{q_2}(\T^3)}.
$$

\smallskip

\noi
\textup{(ii)}
Let   $1<p \le \infty$ and $1 < q,r<\infty$ such that $s \geq   3\big(\frac{1}{p}+\frac{1}{q}-\frac{1}{r}\big)$
and $q, r' \ge p'$.
Then, we have
$$
\|\jb{\nb}^{-s}(fg)\|_{L^r(\T^3)}
\lesssim\| \jb{\nb}^{-s} f\|_{L^{p}(\T^3)} \| \jb{\nb}^{s} g\|_{L^{q}(\T^3)} .
$$
\end{lemma}


\subsection{On discrete convolutions}

Next, we recall the following basic lemma on a discrete convolution.

\begin{lemma}\label{LEM:SUM}
Let $d \geq 1$ and $\al, \be \in \R$ satisfy
\[ \al+ \be > d  \qquad \text{and}\qquad  \al < d .\]
\noi
Then, we have
\[
 \sum_{n = n_1 + n_2} \frac{1}{\jb{n_1}^\al \jb{n_2}^\be}
\les \jb{n}^{- \al + \ld}\]

\noi
for any $n \in \Z^d$, 
where $\ld = 
\max( d- \be, 0)$ when $\be \ne d$ and $\ld = \eps$ when $\be = d$ for any $\eps > 0$.

%
%
%
%

\end{lemma}

Lemma \ref{LEM:SUM} follows
from elementary  computations.
See, for example,  
 \cite[Lemma 4.2]{GTV} and \cite[Lemma 4.1]{MWX}.

\subsection{Tools from stochastic analysis}

We conclude this section by recalling useful lemmas
from stochastic analysis.
See \cite{Shige, Nua} for basic definitions.
Let $(H, B, \mu)$ be an abstract Wiener space.
Namely, $\mu$ is a Gaussian measure on a separable Banach space $B$
with $H \subset B$ as its Cameron-Martin space.
Given  a complete orthonormal system $\{e_j \}_{ j \in \N} \subset B^*$ of $H^* = H$, 
we  define a polynomial chaos of order
$k$ to be an element of the form $\prod_{j = 1}^\infty H_{k_j}(\jb{x, e_j})$, 
where $x \in B$, $k_j \ne 0$ for only finitely many $j$'s, $k= \sum_{j = 1}^\infty k_j$, 
$H_{k_j}$ is the Hermite polynomial of degree $k_j$, 
and $\jb{\cdot, \cdot} = \vphantom{|}_B \jb{\cdot, \cdot}_{B^*}$ denotes the $B$-$B^*$ duality pairing.
We then 
denote the closure  of 
polynomial chaoses of order $k$ 
under $L^2(B, \mu)$ by $\mathcal{H}_k$.
The elements in $\H_k$ 
are called homogeneous Wiener chaoses of order $k$.
We also set
\begin{align}
 \H_{\leq k} = \bigoplus_{j = 0}^k \H_j
\notag
\end{align}

\noi
 for $k \in \N$.

As a consequence
of the  hypercontractivity of the Ornstein-Uhlenbeck
semigroup  due to Nelson \cite{Nelson2}, 
we have the following Wiener chaos estimate
\cite[Theorem~I.22]{Simon}.
See also \cite[Proposition~2.4]{TTz}.

\begin{lemma}\label{LEM:hyp}
Let $k \in \N$.
Then, we have
\begin{equation*}
\|X \|_{L^p(\O)} \leq (p-1)^\frac{k}{2} \|X\|_{L^2(\O)}
 \end{equation*}
 
 \noi
 for any finite $p \geq 2$
 and any $X \in \H_{\leq k}$.

\end{lemma}

Lastly, we recall the following orthogonality
relation for the Hermite polynomials. 
See  \cite[Lemma 1.1.1]{Nua}.

\begin{lemma}\label{LEM:Wick2}
Let $f$ and $g$ be jointly Gaussian random variables with mean zero 
and variances $\s_f$
and $\s_g$.
Then, we have 
\begin{align*}
\E\big[ H_k(f; \s_f) H_\l(g; \s_g)\big] = \dl_{k\l} k! \big\{\E[ f g] \big\}^k, 
\end{align*}

\noi
where
 $H_k (x,\s)$ denotes the Hermite polynomial of degree $k$ with variance parameter $\s$.

\end{lemma}

\section{Construction of the $\Phi^3_3$-measure
in the weakly  nonlinear regime}
\label{SEC:Gibbs}

In this section, we present the construction
of the $\Phi^3_3$-measure in the weakly nonlinear regime (Theorem \ref{THM:Gibbs}\,(i)).
Our proof is based on the variational approach
introduced by  Barashkov and Gubinelli \cite{BG}.
See the  Bou\'e-Dupuis variational formula (Lemma \ref{LEM:var3})
below.
In Subsection~\ref{SUBSEC:var}, 
we briefly go over the setup of the variational
formulation for a partition function.
In Subsection~\ref{SUBSEC:tight}, 
we first establish the uniform exponential integrability~\eqref{exp1c}
and then prove tightness of the truncated $\Phi^3_3$-measures $\rho_N$ in \eqref{GibbsN}, 
which implies weak convergence of a subsequence.
In Subsection~\ref{SUBSEC:wcon},
we follow the approach introduced in our previous work \cite{OOTol1}
and prove uniqueness of the limiting $\Phi^3_3$-measure, 
thus establishing weak convergence of the entire  sequence $\{ \rho_N \}_{N \in \N}$.
Finally, 
in Subsection~\ref{SUBSEC:notAC},
we show that the $\Phi^3_3$-measure 
and the base Gaussian free field $\mu$ in \eqref{gauss1} are mutually singular.
While our proof of singularity of the $\Phi^3_3$-measure is inspired by the discussion in Section 4 of \cite{BG2}, 
we directly prove singularity without referring to  a shifted measure.
In  Appendix \ref{SEC:AC}, we show that the $\Phi^3_3$-measure is indeed absolutely continuous with respect 
to the shifted measure
$\Law (Y(1) +\s \ZZ(1) + \W(1))$, 
where  $\Law(Y(1) ) = \mu$, 
 $\ZZ = \ZZ(Y)$ is the limit of the quadratic process $\ZZ^N$ defined in \eqref{YZ12}, 
 and the auxiliary quintic process $\W = \W(Y)$ is defined in \eqref{AC0}.

\subsection{Bou\'e-Dupuis variational formula}
\label{SUBSEC:var}

Let $W(t)$ be the cylindrical 
 Wiener process on $L^2(\T^3)$ (with respect to the underlying probability measure $\PP$):
\begin{align}
W(t)
 = \sum_{n \in \Z^3} B_n (t) e_n, 
\label{W1}
\end{align}

\noi
where
$\{ B_n \}_{n \in \Z^3}$ 
is defined by 
$B_n(t) = \jb{\xi, \ind_{[0, t]} \cdot e_n}_{ x, t}$.
Here, $\jb{\cdot, \cdot}_{x, t}$ denotes 
the duality pairing on $\T^3\times \R$.
Note that we have, for any $n \in \Z^3$,  
 \[\text{Var}(B_n(t)) = \E\Big[
 \jb{\xi, \ind_{[0, t]} \cdot e_n}_{x, t}\cj{\jb{\xi, \ind_{[0, t]} \cdot e_n}_{x, t}}
 \Big] = \|\ind_{[0, t]} \cdot e_n\|_{L^2_{x, t}}^2 = t.\]
As a result, 
we see that $\{ B_n \}_{n \in \Ld_0}$ is a family of mutually independent complex-valued
Brownian motions conditioned so that $B_{-n} = \cj{B_n}$, $n \in \Z^3$.\footnote
{In particular, $B_0$ is  a standard real-valued Brownian motion.}  
 We  then define a centered Gaussian process $Y(t)$
by 
\begin{align}
Y(t)
=  \jb{\nabla}^{-1}W(t).
\label{P2}
\end{align}

\noi
Then, 
we have $\Law (Y(1)) = \mu$. 
By setting  $Y_N = \pi_NY $, 
we have   $\Law (Y_N(1)) = (\pi_N)_\#\mu$. 
In particular, 
we have  $\E [Y_N(1)^2] = \s_N$,
where $\s_N$ is as in~\eqref{sigma1}.

Next, let $\Ha$ denote the space of drifts, which are the progressively measurable processes 
 belonging to
$L^2([0,1]; L^2(\T^3))$, $\PP$-almost surely. 
For later use, we also define
$\Ha^1$
to be  the space of drifts, which are the progressively measurable processes 
 belonging to
$L^2([0,1]; H^1(\T^3))$, $\PP$-almost surely. 
Namely, we have
\begin{align}
\Ha^1 = \jb{\nb}^{-1} \Ha.
\label{Ha}
\end{align}

\noi
We now state  the  Bou\'e-Dupuis variational formula \cite{BD, Ust};
in particular, see Theorem 7 in~\cite{Ust}.
See also Theorem 2 in \cite{BG}.

\begin{lemma}\label{LEM:var3}
Let $Y(t)
=  \jb{\nabla}^{-1}W(t)$ be as in \eqref{P2}.
Fix $N \in \N$.
Suppose that  $F:C^\infty(\T^3) \to \R$
is measurable such that $\E\big[|F(Y_N(1))|^p\big] < \infty$
and $\E\big[|e^{-F(Y_N(1))}|^q \big] < \infty$ for some $1 < p, q < \infty$ with $\frac 1p + \frac 1q = 1$.
Then, we have
\begin{align}
- \log \E\Big[e^{-F(Y_N(1))}\Big]
= \inf_{\dr \in \mathbb H_a}
\E\bigg[ F( Y_N(1) + \pi_N I(\dr)(1)) + \frac{1}{2} \int_0^1 \| \dr(t) \|_{L^2_x}^2 dt \bigg], 
\label{BD1}
\end{align}

\noi
where $I(\ta)$ is defined by 
\begin{align}
I(\dr)(t) = \int_0^t \jb{\nabla}^{-1} \dr(t') dt'.
\label{P3a}
\end{align}
\end{lemma}

Lemma \ref{LEM:var3} plays a fundamental role
in almost every step of the argument presented in this section 
and Section \ref{SEC:non}.

We state a useful lemma on the  pathwise regularity estimates  of 
$:\!Y^k(t)\!:$ and $I(\dr)(1)$.

\begin{lemma}  \label{LEM:Dr}

\textup{(i)} 
For $k=1,2$, any finite $p \ge 2$, and $\eps>0$,
$:\!  Y_N^k(t) \!:$ converges to $:\! Y^k(t) \!:$ in $L^p(\O; \C^{-\frac k2-\eps}(\T^3))$
and also almost surely in $\C^{-\frac k2-\eps}(\T^3)$.
Moreover, 
we have 
\begin{align}
\begin{split}
\E 
\Big[ \|:\! Y_N^k(t) \!:\|_{\C^{-\frac k2 - \eps}}^p
\Big]
& \les  p^\frac k2   <\infty, \\
 \end{split}
 \label{P4}
\end{align}
uniformly in $N \in \N$ and $t \in [0, 1]$.
We also have 
\begin{align}
\E 
\Big[ \|:\! Y_N^2(t) \!:\|_{H^{-1}}^2
\Big]
\sim  t^2 \log N
\label{XX2}
\end{align}

\noi
for any   $t \in [0, 1]$.

\smallskip

\noi
\textup{(ii)} 
For any $N \in \N$, we have 
\begin{align}
\begin{split}
\E 
\bigg[ \int_{\T^3} :\!  Y_N^3(1) \!: dx
\bigg]
= 0.
 \end{split}
\notag
\end{align}

\smallskip

\noi
\textup{(iii)} For any $\dr \in \Ha$, we have
\begin{align*}
\| I(\dr)(1) \|_{H^{1}}^2 \leq \int_0^1 \| \dr(t) \|_{L^2}^2dt.
\end{align*}
\end{lemma}

\begin{proof}
The  bound \eqref{P4}  for $\eps > 0$ follows
from the Wiener chaos estimate (Lemma \ref{LEM:hyp}), Lemma~\ref{LEM:Wick2}, 
and then carrying out summations, using Lemma \ref{LEM:SUM}.
See, for example, \cite{GKO, GKO2}.
As for \eqref{XX2},  
proceeding as in the proof of Lemma 2.5 in \cite{OTh}
with Lemma \ref{LEM:Wick2}, we have
\begin{align}
\begin{split}
\E 
& \Big[ \|:\! Y_N^2(t) \!:\|_{H^{-1}}^2
\Big]\\
& = \sum_{n \in \Z^3}\frac{1}{\jb{n}^2} \int_{\T^3_x \times \T^3_y}
 \E\Big[  H_2(Y_N(x, t); t \s_N) H_2(Y_N(y, t); t \s_N)\Big] e_n(y - x)dx dy\\
& =\sum_{n \in \Z^3}\frac{t^2}{\jb{n}^2}
\sum_{n_1, n_2\in \Z^3}
\frac{\chi_N^2(n_1)\chi_N^2(n_2)}{\jb{n_1}^2\jb{n_2}^2}
\int_{\T^3_x \times \T^3_y}
e_{n_1 + n_2 - n}(x-y) dx dy\\
& =\sum_{n \in \Z^3}\frac{t^2}{\jb{n}^2}
\sum_{n = n_1 + n_2}
\frac{\chi_N^2(n_1)\chi_N^2(n_2)}{\jb{n_1}^2\jb{n_2}^2}, 
\end{split}
\label{XX3}
\end{align}

\noi
where $\chi_N(n_j)$ is as in \eqref{chi}.
The upper bound in \eqref{XX2} follows from applying Lemma~\ref{LEM:SUM}
to~\eqref{XX3}.
As for the lower bound, we consider the contribution from $|n| \le \frac{2}{3}N$
and $\frac 14 |n| \leq |n_1| \leq \frac 12 |n|$
(which implies $|n_2| \sim |n|$ and $|n_j| \le N$, $j = 1, 2$).
Then, from \eqref{XX3}, we obtain 
\begin{align*}
\E 
 \Big[ \|:\! Y_N^2(t) \!:\|_{H^{-1}}^2
\Big]
\ges  \sum_{\substack{n \in \Z^3\\|n| \le \frac 23 N}}
\frac{t^2}{\jb{n}^3} \sim t^2 \log N, 
\end{align*}

\noi
which proves the lower bound in \eqref{XX2}.
As for (ii), it follows from 
recalling the definition $:\!  Y_N^3(1) \!: \, = H_3(Y_N(1); \s_N)$
(with $\s_N$ as in \eqref{sigma1})
and 
the orthogonality relation of the Hermite polynomials
(Lemma \ref{LEM:Wick2} with $k =3$ and $\l = 0$).
Lastly, 
the claim in  (iii) follows from Minkowski's integral inequality
and Cauchy-Schwarz inequality; see Lemma~4.7 in \cite{GOTW}.
\end{proof}

\begin{remark}\rm
In \cite{GOTW, ORSW2}, 
a  slightly different (and weaker) variational formula
was used.  See also Lemma 1 in \cite{BG}.
Given a drift $\dr \in \Ha$, 
we  define the measure $\Q_\dr$ 
whose Radon-Nikodym derivative with 
respect to $\PP$ is given by the following stochastic exponential:
\begin{align*}
\frac{d\Q_\dr}{d\PP} = e^{\int_0^1 \jb{\dr(t),  dW(t)} - \frac{1}{2} \int_0^1 \| \dr(t) \|_{L^2_x}^2dt}, 
\end{align*}

\noi
where $\jb{\cdot, \cdot}$ stands for the usual  inner product on $L^2(\T^3)$.
Let  $\Hc$ denote the subspace of $\Ha$ consisting of drifts such that $\Q_\dr(\O) = 1$.
Then, the (weaker) variational formula
used in \cite{GOTW, ORSW2}
is given by \eqref{BD1}, 
where the infimum is taken over $\Hc\subset \Ha$
and we replace 
$Y$ and $\E = \E_\PP$
by $Y_\ta = Y - I(\ta)$ and $\E_{\Q_\ta}$.
Here, 
 $\E = \E_\PP$ and $\E_{\Q_\ta}$
 denote 
expectations with respect to the underlying probability measure $\PP$
and the measure $\Q_\ta$, respectively.
In such a formulation, 
$Y_\ta$ and the measure $\Q_\ta$ depend on a drift $\ta$.
This, however,  is not suitable for our purpose, 
since we construct a drift $\ta$ in \eqref{BD1} depending on $Y$.

\end{remark}

\subsection{Uniform exponential integrability and tightness}
\label{SUBSEC:tight}

In this subsection, 
we first prove  the uniform exponential integrability \eqref{exp1c}
via the  Bou\'e-Dupuis variational formula (Lemma~\ref{LEM:var3}).
Then, 
we establish  tightness of the truncated $\Phi^3_3$-measures $\{\rho_N\}_{N \in \N}$.

As in the case of 
 the $\Phi^4_3$-measure 
 studied in \cite{BG}
 (see also Section 6 in \cite{OOTol1}),
we need to introduce a 
 further renormalization than the standard Wick renormalization (see \eqref{K1r}).
 As a result, 
 the resulting $\Phi^3_3$-measure 
 is singular  with respect to the base Gaussian free field $\mu$;
 see Subsection \ref{SUBSEC:notAC}.
  We point out that this extra renormalization
 appears only at the level of the measure
 and thus does not affect the dynamical problem, at least locally in time.\footnote{As mentioned in Section \ref{SEC:1}, 
this singularity of the $\Phi^3_3$-measure
 causes an additional difficulty for  the globalization problem.} 
  In the following, we use the following short-hand notations: $Y_N(t) = \pi_N Y(t)$,
$\Dr (t) = I(\dr)(t)$, and $\Dr_N(t) = \pi_N \Dr(t)$
with $Y_N = Y_N(1)$ and $\Dr_N = \Dr_N(1)$.
 We also use 
 $Y=Y(1)$ and $\Dr =\Dr (1)$.

Let us first explain 
the second renormalization introduced in \eqref{K1r}.
Let $R_N$ be as in \eqref{K1}
and set 
\begin{align*}
\wt Z_N = \int e^{-R_N(u)} d\mu (u). 
\end{align*}

\noi
By Lemma \ref{LEM:var3}, 
we can  express the partition function $\wt Z_N$ 
as
\begin{align*}
- \log \wt Z_N 
= \inf_{\dr \in \mathbb H_a}
\E\bigg[ R_N ( Y+\Dr ) + \frac{1}{2} \int_0^1 \| \dr(t) \|_{L^2_x}^2 dt \bigg].
\end{align*}

\noi
By expanding the cubic Wick power, we have
\begin{align}
\begin{split}
-  \frac\s3\int_{\T^3} :\! (Y_N + \Dr_N)^3\!: dx
& = - \frac\s3\int_{\T^3} :\! Y_N^3\!: dx
-\s \int_{\T^3} :\! Y_N ^2\!:  \Dr_Ndx\\
& \quad -\s\int_{\T^3}  Y_N   \Dr_N^2 dx 
- \frac \s3 \int_{\T^3}  \Dr_N^3 dx.
\end{split}
\label{BD3}
\end{align}

\noi
In view of 
Lemma \ref{LEM:Dr}, 
the first term on the right-hand side vanishes
under an expectation, 
while 
we can estimate the third and fourth terms on 
 the right-hand side of \eqref{BD3}
(see Lemma~\ref{LEM:Dr7}).
As we see below, 
 the second term turns out to be divergent
 (and does not vanish under an expectation).
From the Ito product formula, we have
\begin{align}
\E\bigg[\int_{\T^3} \!:\!Y_N^2\!:\, \Dr_N dx \bigg] 
 = \E\bigg[ \int_0^1 \int_{\T^3}  
 :\!Y_N^2(t)\!: \, \dot \Dr_N(t) dx dt \bigg], 
 \label{YZ11}
\end{align}
 
 \noi
 where we have 
 $\dot \Theta_N (t) = \jb{\nabla}^{-1} \pi_N \theta(t)$ 
 in view of \eqref{P3a}.
Define $\ZZ^N$ with $\ZZ^N(0) = 0$ by its time derivative:
\begin{align}
\dot \ZZ^N (t) = (1-\Delta)^{-1} \!:\!Y_N^2(t)\!:
\label{YZ12}
\end{align}

\noi
and set 
$\ZZ_N = \pi_N \ZZ^N$.
Then, we
perform a 
  change of variables: 
\begin{align}
\dot \Ups^N(t) = \dot \Dr(t)  - \s \dot \ZZ_N(t)
\label{YZ13}
\end{align}

\noi
and set $\Ups_N = \pi_N \Ups^N$.
From \eqref{YZ11}, \eqref{YZ12}, and \eqref{YZ13}, 
we have 
\begin{align}
\begin{split}
\E\bigg[ -\s \int_{\T^3} :\!Y_N^2\!:  \Dr_N dx 
+ \frac12 \int_0^1 \|\dr(t)\|_{L^2_x}^2 dt \bigg]
= \frac 12 \E\bigg[  \int_0^1 \|\dot \Ups^N (t)\|_{H^1_x}^2 dt
\bigg] - \al_N,
\end{split}
\label{YZ13a}
\end{align}

\noi
where the divergent constant $\al_N$ is given by 
\begin{align}
\al_N  = \frac{\s^2}2  \E\bigg[\int_0^1   \| \dot \ZZ_N(t) \|_{H^1_x}^2 dt \bigg]
\too \infty, 
\label{YZ14}
\end{align}

\noi
 as $N \to \infty$.
The divergence in \eqref{YZ14}
can be easily seen from the spatial regularity 
$1- \eps$ of $\dot \ZZ_N(t)  = (1-\Delta)^{-1} \!:\!Y_N^2(t)\!:$
(with a uniform bound in $N \in \N$).  See Lemma \ref{LEM:Dr}.

In view of the discussion above, 
we define $R_N^\dia$ as in \eqref{K1r}, 
which removes the divergent constant $\al_N$ in \eqref{YZ13a}.
Then, 
from \eqref{pfN1} and the Bou\'e-Dupuis variational formula (Lemma~\ref{LEM:var3}), 
we have
\begin{equation}
- \log Z_N = \inf_{\dr \in \Ha} \E
\bigg[ R_N^\dia (Y + \Dr) + \frac{1}{2} \int_0^1 \| \dr(t) \|_{L^2_x}^2 dt \bigg]
\label{K7}
\end{equation}

\noi
for any $N \in \N$.
By setting
\begin{equation}
\W_N (\dr) = \E
\bigg[ R_N^\dia (Y + \Dr) + \frac{1}{2} \int_0^1 \| \dr(t) \|_{L^2_x}^2 dt \bigg], 
\label{K8}
\end{equation}

\noi
it follows from
\eqref{K1} with $\g=3$,
\eqref{K1r}, \eqref{BD3}, 
\eqref{YZ13a}, and Lemma \ref{LEM:Dr}\,(ii)
that 
\begin{align}
\begin{split}
\W_N(\dr)
&=\E
\bigg[  
-\s \int_{\T^3} Y_N \Dr_N^2 dx
-\frac \s 3 \int_{\T^3} \Dr_N^3 dx
\\
&\hphantom{XXii}
+ A \bigg| \int_{\T^3} \Big( :\! Y_N^2 \!: + 2 Y_N \Dr_N + \Dr_N^2 \Big) dx \bigg|^3 
+ \frac{1}{2} \int_0^1 \| \dot \Ups^N(t) \|_{H^1_x}^2 dt
 \bigg].
\end{split}
\label{K9}
\end{align}

\noi
We also set 
\begin{align}
\Ups_N = \Ups_N(1)
= \pi_N  \Ups^N(1) \qquad \text{and} \qquad \ZZ_N = \ZZ_N(1) = \pi_N \ZZ^N(1).
\label{K9a}
\end{align}

\noi
In view of the change of variables \eqref{YZ13}, 
we have
\begin{align}
\Dr_N = \Ups_N + \s \pi_N \ZZ_N
=: \Ups_N + \s \wt \ZZ_N, \qquad 
\text{i.e.~}\wt \ZZ_N := \pi_N \ZZ_N.
\label{K9b}
\end{align}

\noi
Namely, the original  drift $\dr$ in \eqref{K7}
depends on $Y$.
By the definition \eqref{YZ12} and \eqref{K9a}, 
$\ZZ_N$ is determined by $Y_N$.
Hence, in the following, we view
$\dot \Ups^N$ as a drift
and 
study the minimization problem~\eqref{K7}
by first studying 
each term in \eqref{K9}
(where we now view $\W_N$ as a function of $\dot \Ups^N$)
and 
then taking an infimum in $\dot \Ups^N \in \Ha^1$, 
where $\Ha^1$ is as in \eqref{Ha}.
Our main goal is to show that 
$\W_N(\dot \Ups^N)$ in \eqref{K9} is bounded away from $-\infty$, 
uniformly in $N \in \N$ and  $\dot \Ups^N \in \Ha^1$.

\begin{remark}\rm
In this paper, 
we work with 
the cube frequency projector $\pi_N = \pi_N^\text{cube}$ defined in~\eqref{pi}, 
satisfying $\pi_N^2 = \pi_N$.
In view of \eqref{K9a} and \eqref{K9b}, we have 
$\wt \ZZ_N = \ZZ_N$.
Nonetheless, we introduce the notation 
$\wt \ZZ_N$ in \eqref{K9b}
to indicate the modifications necessary 
to consider the case of the smooth frequency projector $\pi_N^\text{smooth}$
defined in \eqref{pis1}, which does not satisfy  
$(\pi_N^\text{smooth})^2 = \pi_N^\text{smooth}$.
This comment applies to the remaining part of the paper.

\end{remark}

We first state two lemmas whose proofs are 
presented at the end of this subsection.
While the first lemma is elementary, 
the second lemma (Lemma \ref{LEM:Dr8}) requires
much more careful analysis, 
reflecting  the critical nature of the $\Phi^3_3$-measure.

\begin{lemma}\label{LEM:Dr7}
Let $A>0$ and  $0<|\s|<1$.
Then, there exist small $\eps>0$ and   a constant  $c  >0$ 
 such that, for any $\dl> 0$, there exists $C_\dl > 0$ such that 
\begin{align}
 \bigg|\int_{\T^3}  Y_N \Dr_N^2 dx \bigg|
 &\les
 1 +  C_\dl \| Y_N \|_{\C^{-\frac 12-\eps}}^c + \dl \| \Ups_N\|_{L^2}^6 + \dl \| \Ups_N \|_{H^1}^2
  + \|\ZZ_N\|_{\C^{1 - \eps}}^c,
\label{KZ1a} \\
 \bigg|\int_{\T^3}  \Dr_N^3 dx \bigg|
 &\les
1  + \| \Ups_N\|_{L^2}^6 +  \| \Ups_N \|_{H^1}^2 +  \|\ZZ_N\|_{\C^{1 - \eps}}^3,
\label{KZ1}
\end{align}

\noi
and 
\begin{align}
\begin{split}
A \bigg| \int_{\T^3}  \Big( :\! Y_N^2 \!:  & + 2 Y_N \Dr_N + \Dr_N^2 \Big) dx \bigg|^3
\ge
\frac A2 
\bigg| \int_{\T^3}  \Big( 2 Y_N \Ups_N + \Ups_N^2 \Big) dx \bigg|^3
- \dl \| \Ups_N \|_{L^2}^6 \\
&\quad
- C_{\dl, \s} \Bigg\{
\bigg| \int_{\T^3} :\! Y_N^2 \!: dx \bigg|^3
+ \| Y_N \|_{\C^{-\frac 12-\eps}}^6
+ \| \ZZ_N \|_{\C^{1-\eps}}^6
\Bigg\}, 
\end{split}
\label{KZ2}
\end{align}

\noi
uniformly in $N \in \N$, 
where 
$\Dr_N = \Ups_N + \s \wt \ZZ_N$ as in \eqref{K9b}.

\end{lemma}

The next lemma allows us to control the term 
$\| \Ups_N \|_{L^2}^6$ appearing in Lemma \ref{LEM:Dr7}.

\begin{lemma} \label{LEM:Dr8}
There exists
 a non-negative random variable 
 $B(\o)$ 
  with $\E [ B^p ] \le C_p < \infty$ for any finite $p \ge 1$
  such that 
\begin{align}
\| \Ups_N \|_{L^2}^6
\les
\bigg| \int_{\T^3} \Big( 2 Y_N \Ups_N + \Ups_N^2 \Big) dx \bigg|^3
+ \| \Ups_N \|_{H^1}^2
+ B(\o),
\label{DN2a}
\end{align}
uniformly in $N \in \N$.

\end{lemma}

By assuming Lemmas \ref{LEM:Dr7}
and \ref{LEM:Dr8}, we now prove
the uniform exponential integrability~\eqref{exp1c}
and tightness of the truncated $\Phi^3_3$-measures $\rho_N$.

\smallskip

\noi
$\bullet$ {\bf Uniform exponential integrability:}
In view of \eqref{K9} and Lemma \ref{LEM:Dr8}, 
define the positive part $\U_N$ of $\W_N$ by 
\begin{align}
\U_N(\dot \Ups^N )
= \E \bigg[
\frac A2 \bigg| \int_{\T^3} \Big( 2 Y_N \Ups_N + \Ups_N^2 \Big) dx \bigg|^3
+ \frac{1}{2} \int_0^1 \| \dot \Ups^N(t) \|_{H^1_x}^2 dt
\bigg].
\label{K10}
\end{align}

As a corollary to Lemma~\ref{LEM:Dr}\,(i) with \eqref{YZ12}, we have, 
for any finite $p \ge 1$, 
\begin{align}
\E\Big[\|\ZZ_N \|_{\C^{1-\eps}}^p \Big]
\le \int_0^1 \E\Big[\|\!:\!Y_N^2(t)\!:\|_{\C^{-1-\eps}}^p\Big] dt 
\les p < \infty, 
\label{YZ15}
\end{align}

\noi
uniformly in $N \in \N$.
Then, by applying 
Lemmas \ref{LEM:Dr7} and 
 \ref{LEM:Dr8} to \eqref{K9}
together with 
Lemma~\ref{LEM:Dr} and \eqref{YZ15},
we obtain
\begin{align}
\begin{split}
\W_N(\dot \Ups^N)
&\ge 
-C_0 +
\E
\bigg[
\Big(\frac A{2} -c|\s|\Big)\bigg| \int_{\T^3} \Big( 2 Y_N \Ups_N + \Ups_N^2 \Big) dx \bigg|^3
\\
&\hphantom{XXXXXXX}
+ \Big( \frac 12 - c|\s| \Big) \int_0^1 \| \dot \Ups^N(t) \|_{H^1_x}^2 dt
\bigg]\\
& \ge -C_0' + \frac 1{10} \U_N(\dot \Ups^N), 
\end{split}
\label{DN2}
\end{align}

\noi
for any $0 < |\s| < \s_0$, 
provided  $A = A(\s_0) >0$ is sufficiently large.
Noting that the estimate~\eqref{DN2}
is uniform in $N \in \N$
and $\dot \Ups^N \in \Ha^1$, we conclude that 
\begin{align}
\inf_{N \in \mathbb{N}} \inf_{\dot \Ups^N \in \Ha^1} \W_N (\dot \Ups^N) 
\geq 
\inf_{N \in \mathbb{N}} \inf_{\dot \Ups^N \in \Ha^1}
\bigg\{ -C_0' + \frac{1}{10}\U_N(\dot \Ups^N)\bigg\}
 \geq - C_0' >-\infty.
\label{KZ14a}
\end{align}

\noi
Therefore, 
the uniform exponential integrability \eqref{exp1c}
follows from \eqref{K7}, \eqref{K8}, and \eqref{KZ14a}.

\smallskip

\noi
$\bullet$ {\bf Tightness:}
Next, we prove 
 tightness of the truncated $\Phi^3_3$-measures  $\{\rho_N\}_{N\in \N}$.
Although it follows from a slight modification of the argument in our previous work \cite[Subsection 6.2]{OOTol1},
we present a proof here for readers' convenience.

As a preliminary step, 
we first prove that 
 $Z_N$ in \eqref{pfN1} is uniformly bounded away from 0:
\begin{align}
\inf_{N \in \N} Z_N > 0.
\label{TT1}
\end{align}

\noi
In view of \eqref{K7} and \eqref{K8}, 
it suffices to establish an upper bound on $\W_N$ in \eqref{K9}.
By Lemma \ref{LEM:Bes} and \eqref{K9b}, we have 
\begin{align*}
\bigg|\int_{\T^3}  2 Y_N \Dr_N  dx \bigg|^3
& \les \|Y_N\|_{\C^{-\frac 12-\eps}}^3 \|\Dr_N\|_{H^{\frac 12 + 2\eps}}^3\\
& \les 1 + \|Y_N\|_{\C^{-\frac 12-\eps}}^c + \|\ZZ_N \|_{\C^{1-\eps}}^c
 + \|\Ups_N\|_{H^1}^c.
\end{align*}

\noi
Thus, we have
\begin{align}
\begin{split}
A \bigg| \int_{\T^3} &  \Big( :\! Y_N^2 \!:   + 2 Y_N \Dr_N + \Dr_N^2 \Big) dx \bigg|^3\\
& \les 1 + \|:\! Y_N^2 \!:\|_{\C^{-1-\eps}}^3
+ 
\|Y_N\|_{\C^{-\frac 12-\eps}}^c + \|\ZZ_N \|_{\C^{1-\eps}}^c
 + \|\Ups_N\|_{H^1}^c.
\end{split}
\label{TT2}
\end{align}

\noi
Then, from \eqref{K9}, Lemma \ref{LEM:Dr7}, and \eqref{TT2}
with Lemma \ref{LEM:Dr} and \eqref{YZ15}, 
we obtain
\begin{align*}
\inf_{ \dot \Ups^N \in \Ha^1} \W_N 
& \les 1 + 
\inf_{ \dot \Ups^N \in \Ha^1}  \E\Bigg[\bigg( \int_0^1 \| \dot \Ups^N(t) \|_{H^1_x}^2 dt\bigg)^c\Bigg]
\les 1
\end{align*}

\noi
by taking $\dot \Ups^N \equiv 0$, for example.
This proves \eqref{TT1}.

We now prove tightness of the truncated $\Phi^3_3$-measures.
Fix small $\eps > 0$ and 
let 
$ B_R \subset H^{-\frac 12 - \eps}(\T^3)$ 
be the closed ball of radius $R> 0$ centered at the origin.
Then, by Rellich's compactness lemma, 
we see that  $B_R$ is compact in $H^{-\frac 12 - 2\eps}(\T^3)$.
In the following, we show that 
 given any small $\dl > 0$, there exists $R = R(\dl ) \gg1 $
such that 
\begin{align}
\sup_{N \in \N} \rho_N(B_R^c)< \dl.
\label{T0}
\end{align}
\noi

Given $M \gg1 $, 
let $F$ be a bounded smooth non-negative function such that 
\begin{align}
F(u) = 
\begin{cases}
 M, & \text{if }\|u\|_{H^{-\frac 12 - \eps}} \leq \frac R2,\\
0, & \text{if }\|u\|_{H^{-\frac 12 - \eps} }> R.
\end{cases}
\label{T0a}
\end{align}

\noi
Then, from \eqref{TT1}, we have 
\begin{align}
\rho_N(B_R^c) \leq Z_N^{-1} \int e^{-F(u)-R_N^{\dia}(u)} d\mu
\les \int e^{-F(u)-R_N^{\dia}(u)} d\mu
=: \ft Z_N, 
\label{T1}
\end{align}

\noi
uniformly in  $N \gg 1$.
Under  the change of variables \eqref{YZ13}
(see also \eqref{YZ13a}), 
define $\ft R^{\dia}_{N}(Y+ \Ups^{N}+ \s \ZZ_{N})$  by 
\begin{align}
\begin{split}
\ft R^{\dia}_{N}(Y+ \Ups^{N}+ \s \ZZ_{N})
& =  -\frac \s 3 \int_{\T^3} :\! Y_N^3 \!: dx
- \s \int_{\T^3} Y_N \Dr_N^2 dx
- \frac \s 3 \int_{\T^3} \Dr_N^3 dx
\\
&\hphantom{X}
+ A \bigg| \int_{\T^3} \Big( :\! Y_N^2 \!: + 2 Y_N \Dr_N + \Dr_N^2 \Big) dx \bigg|^3, 
\end{split}
\label{KZ16}
\end{align}

\noi
where 
$\Dr_N = \Ups_N + \s \wt \ZZ_N$
with $\wt \ZZ_N = \pi_N \ZZ_N $ as in \eqref{K9b}.
Then, by \eqref{T1} and
the Bou\'e-Dupuis  variational formula (Lemma~\ref{LEM:var3}),
we have
\begin{align}
\begin{split}
-\log \ft Z_{N}
= \inf_{\dot \Ups^{N}\in  \mathbb H_a^1}
\E \bigg[& F(Y+ \Ups^{N} + \s \ZZ_{N}) \\
& +  \ft R^{\dia}_{N}(Y+ \Ups^{N} + \s\ZZ_{N}) 
+ \frac12 \int_0^1 \| \dot \Ups^{N}(t) \|_{H^1_x}^2dt \bigg].
\end{split}
\label{T2}
\end{align}

Since
$Y+ \s \ZZ_N \in \H_{\leq 2}$, 
it follows from  Lemma \ref{LEM:Dr}, \eqref{YZ15}, 
Chebyshev's inequality, 
and choosing $R \gg 1$ that 
\begin{align}
\begin{split}
\PP \Big( &   \|Y+\Ups^{N} + \s \ZZ_{N}
\|_{H^{-\frac 12 - \eps}} >  \tfrac R2
\Big) \\
&  \leq  \PP \Big( \|Y+ \s \ZZ_N
\|_{H^{-\frac 12 - \eps}} >  \tfrac R4
\Big) 
+ \PP \Big( \|\Ups^{N}\|_{H^1} >  \tfrac R4
\Big) \\
& \leq \frac 12 
+ \frac {16}{R^2} \E\Big[\|\Ups^N\|_{H^1_x}^2\Big], 
\end{split}
\label{T3}
\end{align}

\noi
uniformly in $N \in \N$ and $R\gg1$.
Then, 
from 
\eqref{T0a}, \eqref{T3},  and Lemma \ref{LEM:Dr}, 
we obtain 
\begin{align}
\begin{split}
\E  \Big[ F(Y+ \Ups^{N}+ \s \ZZ_{N})
 \Big]
& \ge M \E \Big[  \ind_{ \big\{\|Y+ \Ups^{N}+ \s \ZZ_{N}
\|_{H^{-\frac 12 - \eps}} \leq  \tfrac R2\big\}}\Big]\\
&  \geq \frac{M}{2}
-  \frac {16M }{R^2} \E\Big[\|\Ups^N\|_{H^1_x}^2\Big]\\
&  \geq \frac{M}{2}
-  \frac {1}{4}\E\bigg[ \int_0^1 \| \dot \Ups^{N}(t) \|_{H^1_x}^2dt \bigg], 
\end{split}
\label{T4}
\end{align}

\noi
where we set  $M = \frac{1}{64} R^2$ in the last step.
Hence, from \eqref{T2}, \eqref{T4}, 
and repeating the computation leading to \eqref{KZ14a}
(by possibly making $\s_0$ smaller), we obtain
\begin{align}
\begin{split}
-\log \ft Z_{N}
& \geq \frac M{2}
+ 
 \inf_{\dot \Ups^{N}\in  \mathbb H_a^1}
\E \bigg[
   \ft R^{\dia}_{N}(Y+ \Ups^{N}+ \s \ZZ_{N}) 
+ \frac14 \int_0^1 \| \dot \Ups^{N}(t) \|_{H^1_x}^2dt \bigg]\\
& \geq \frac M4, 
\end{split}
\label{T5}
\end{align}

\noi
uniformly $N \in \N$ and $M = \frac{1}{64}  R^2 \gg 1$.
Therefore, given any small $\dl > 0$, by choosing $R = R(\dl ) \gg1 $
and setting $M = \frac 1{64} R^2\gg1$, 
the desired bound \eqref{T0} follows from 
\eqref{T1} and~\eqref{T5}. 
This proves 
tightness of the truncated $\Phi^3_3$-measures $\{\rho_N\}_{N\in \N}$.

\medskip

We conclude this subsection by presenting the proofs of 
 Lemmas \ref{LEM:Dr7} and  \ref{LEM:Dr8}.

\begin{proof}[Proof of Lemma \ref{LEM:Dr7}]

From \eqref{dual}, \eqref{prod}, \eqref{embed}, and \eqref{interp}
in Lemma \ref{LEM:Bes} followed by Young's inequality,
we have
\begin{align}
\begin{split}
\bigg|  \int_{\T^3} &   Y_N \Dr_N^2 dx \bigg|
\les
\| Y_N \|_{\C^{-\frac 12-\eps}} \| \Dr_N \|_{H^{\frac 12+2\eps}} \| \Dr_N \|_{L^2} \\
&\les
\| Y_N \|_{\C^{-\frac 12-\eps}}
\Big( \| \Ups_N \|_{H^{\frac 12+2\eps}} \big(\| \Ups_N \|_{L^2} + \| \ZZ_N \|_{\C^{1-\eps}} \big) + \| \ZZ_N \|_{\C^{1-\eps}}^2 \Big) \\
&\les
\| Y_N \|_{\C^{-\frac 12-\eps}}
\Big( \| \Ups_N \|_{L^2}^{\frac 12-2\eps} \| \Ups_N \|_{H^1}^{\frac 12+2\eps} \big(\| \Ups_N \|_{L^2} + \| \ZZ_N \|_{\C^{1-\eps}} \big)
+ \| \ZZ_N \|_{\C^{1-\eps}}^2 \Big) \\
&\les
1
+ C_\dl \| Y_N \|_{\C^{-\frac 12-\eps}}^c
+ \dl \| \Ups_N \|_{L^2}^6 + \dl \| \Ups_N \|_{H^1}^2
 + \| \ZZ_N \|_{\C^{1-\eps}}^c,
\end{split}
\label{KZ3}
\end{align}

\noi
which yields \eqref{KZ1a}.
As for 
the second estimate \eqref{KZ1}, 
it follows from Sobolev's inequality, the interpolation \eqref{interp}, 
and Young's inequality that 
\begin{align}
&\bigg| \int_{\T^3} \Ups_N^3 dx \bigg|
\les \| \Ups_N \|_{H^{\frac 12}}^3
\les \| \Ups_N \|_{L^2}^{\frac 32} \| \Ups_N \|_{H^1}^{\frac 32}
\les \| \Ups_N \|_{L^2}^6 + \| \Ups_N \|_{H^1}^2,
\label{KZ3a}
\end{align}

\noi
while H\"older's inequality with \eqref{embed} shows
\begin{align*}
\bigg| \int_{\T^3} \Ups_N^2 \wt \ZZ_N dx \bigg|
+
\bigg| \int_{\T^3} \Ups_N \wt \ZZ_N^2 dx \bigg|
+
\bigg| \int_{\T^3} \wt \ZZ_N^3 dx \bigg|
\les
1+ \| \Ups_N \|_{L^2}^6 + \| \ZZ_N \|_{\C^{1-\eps}}^3.
\end{align*}

Note that, given any $\g > 0$,  there exists a constant $C= C(J)>0$ such that
\begin{align}\label{YY9}
\bigg|\sum_{j = 1}^J a_j\bigg|^\g
\ge \frac 12 |a_1|^\g -C \bigg( \sum_{j = 2}^J |a_j|^\g\bigg)
\end{align}

\noi
for any $a_j\in \R$.  See Section 5 in \cite{OOTol1}.
Then, from \eqref{YY9} and Cauchy's inequality, 
we have
\begin{align*}
&A \bigg| \int_{\T^3}  \Big( :\! Y_N^2 \!:  + 2 Y_N \Dr_N + \Dr_N^2 \Big) dx \bigg|^3 \\
&\ge
\frac A2 
\bigg| \int_{\T^3}  \Big( 2 Y_N \Ups_N + \Ups_N^2 \Big) dx \bigg|^3
- C A \Bigg\{
\bigg| \int_{\T^3} :\! Y_N^2 \!: dx \bigg|^3
+ |\s|^3 \bigg| \int_{\T^3} Y_N \wt \ZZ_N dx \bigg|^3 \\
&\hphantom{XXX}
+ |\s|^3 \bigg| \int_{\T^3} \Ups_N \wt \ZZ_N dx \bigg|^3
+ \s^6 \bigg| \int_{\T^3} \wt \ZZ_N^2 dx \bigg|^3
\Bigg\} \\
&\ge
\frac A2 
\bigg| \int_{\T^3}  \Big( 2 Y_N \Ups_N + \Ups_N^2 \Big) dx \bigg|^3
- \dl \| \Ups_N \|_{L^2}^6 \\
&\hphantom{XXX}
- C_{\dl, \s} \Bigg\{
\bigg| \int_{\T^3} :\! Y_N^2 \!: dx \bigg|^3
+ \| Y_N \|_{\C^{-\frac 12-\eps}}^6
+ \| \ZZ_N \|_{\C^{1-\eps}}^6
\Bigg\}.
\end{align*}

\noi
This proves \eqref{KZ2}.
This completes the proof of Lemma \ref{LEM:Dr7}.
\end{proof}

Next, we present the proof of
 Lemma \ref{LEM:Dr8}.

\begin{proof}[Proof of Lemma \ref{LEM:Dr8}]
If we have 
\begin{align}
\| \Ups_N \|_{L^2}^2 \gg \bigg| \int_{\T^3} Y_N \Ups_N dx \bigg|, 
\label{DN9}
\end{align}

\noi
then, we have
\begin{align}
\begin{split}
\| \Ups_N \|_{L^2}^6
=
\bigg( \int_{\T^3} \Ups_N^2 dx\bigg)^3
\sim
\bigg| \int_{\T^3} \Big( 2 Y_N \Ups_N + \Ups_N^2 \Big) dx \bigg|^3,
\end{split}
\label{DNc1}
\end{align}
which shows \eqref{DN2a}.
Hence, we assume that 
\begin{align}
\| \Ups_N \|_{L^2}^2 \les \bigg| \int_{\T^3} Y_N \Ups_N dx \bigg|
\label{DN9a}
\end{align}

\noi
in the following.

Given $j \in \N$, 
define the  sharp frequency projections $\Pi_j$
with a Fourier multiplier 
$\ind_{\{|n|\le 2\}}$ when $j = 1$
and $\ind_{\{2^{j-1}<  |n|\le 2^j\}}$ when $j \ge 2$.
We also set $\Pi_{\le j} = \sum_{k = 1}^j \Pi_k$
and $\Pi_{> j} = \Id - \Pi_{\le j}$.
Then,
write  $\Ups_N$ as 
\begin{align}
\Ups_N  = \sum_{j=1}^\infty \Pi_j \Ups_N = \sum_{j=1}^\infty (\ld_j \proj_j Y_N + w_j),
\label{DN6}
\end{align}

\noi
where $\ld_j$ and $w_j$ are given by 
\begin{align}
\ld_j &:=
\begin{cases}
\frac{\jb{\Ups_N, \proj_j Y_N}}{\|\proj_j Y_N\|_{L^2}^2},  & \text{if } \| \proj_j Y_N \|_{L^2} \neq 0, \\
0, & \text{otherwise},
\end{cases}
\qquad
\text{and}
\qquad  
w_j :=
\proj_j \Ups_N - \ld_j \proj_j Y_N.
\label{DN6a}
\end{align}

\noi
By definition,  $w_j = \Pi_j w_j$ is orthogonal to $\proj_j Y_N$ (and also to $Y_N$) in $L^2(\T^3)$.
Thus, we have 
\begin{align}
\| \Ups_N \|_{L^2}^2
&= \sum_{j=1}^\infty \Big( \ld_j^2 \| \proj_j Y_N \|_{L^2}^2 + \| w_j \|_{L^2}^2 \Big), \label{DN7} \\
\int_{\T^3} Y_N \Ups_N dx
&= \sum_{j=1}^\infty \ld_j \| \proj_j Y_N \|_{L^2}^2.
\label{DN8}
\end{align}

\noi
Hence, 
from \eqref{DN9a}, \eqref{DN7}, and \eqref{DN8}, we have
\begin{align}
\sum_{j=1}^\infty \ld_j^2 \| \proj_j Y_N \|_{L^2}^2
\les \bigg| \sum_{j=1}^\infty \ld_j \| \proj_j Y_N \|_{L^2}^2 \bigg|.
\label{DN11}
\end{align}

Fix  $j_0 = j_0(\o) \in \N$ (to be chosen later).
By  Cauchy-Schwarz's inequality and \eqref{DN6a}, we have
\begin{align}
\begin{split}
\bigg| \sum_{j=j_0+1}^\infty \ld_j \| \proj_j Y_N \|_{L^2}^2 \bigg|
&\le \bigg( \sum_{j=1}^\infty \ld_j^2 2^{2j} \| \proj_j Y_N \|_{L^2}^2 \bigg)^{\frac 12}
\bigg( \sum_{j=j_0+1}^\infty 2^{-2j} \| \proj_j Y_N \|_{L^2}^2 \bigg)^{\frac 12} \\
&\le \bigg( \sum_{j=1}^\infty 2^{2j} \| \proj_j \Ups_N \|_{L^2}^2 \bigg)^{\frac 12}
\bigg( \sum_{j=j_0+1}^\infty 2^{-2j} \| \proj_j Y_N \|_{L^2}^2 \bigg)^{\frac 12} \\
&\sim\| \Ups_N \|_{H^1}
 \|\Pi_{> j_0} Y_N \|_{H^{-1}}.
\end{split}
\label{DN12}
\end{align}

\noi
On the other hand, it follows
from
 Cauchy-Schwarz's inequality, 
 \eqref{DN11}, and 
Cauchy's inequality that 
\begin{align}
\begin{split}
\bigg| \sum_{j=1}^{j_0} \ld_j \| \proj_j Y_N \|_{L^2}^2 \bigg|
&\le \bigg( \sum_{j=1}^\infty \ld_j^2 \| \proj_j Y_N \|_{L^2}^2 \bigg)^{\frac 12}
\bigg( \sum_{j=1}^{j_0} \| \proj_j Y_N \|_{L^2}^2 \bigg)^{\frac 12} \\
&\le C\bigg| \sum_{j=1}^\infty \ld_j \| \proj_j Y_N \|_{L^2}^2 \bigg|^{\frac 12}
\bigg( \sum_{j=1}^{j_0} \| \proj_j Y_N \|_{L^2}^2 \bigg)^{\frac 12} \\
&\le
\frac 12 \bigg| \sum_{j=1}^\infty \ld_j \| \proj_j Y_N \|_{L^2}^2 \bigg|
+ C' \|\Pi_{\le j_0}Y_N \|_{L^2}^2.
\end{split}
\label{DN14}
\end{align}

\noi
Hence,  from \eqref{DN12} and \eqref{DN14}, 
we obtain 
\begin{align}
\begin{split}
\bigg| \sum_{j=1}^\infty \ld_j \| \proj_j Y_N \|_{L^2}^2 \bigg|
\les
\| \Ups_N \|_{H^1}
 \|\Pi_{> j_0} Y_N \|_{H^{-1}}
+
\|\Pi_{\le j_0}Y_N \|_{L^2}^2.
\end{split}
\label{DN15}
\end{align}

Since $Y_N$ is spatially homogeneous, we have
\begin{align}
 \|\Pi_{> j_0} Y_N \|_{H^{-1}}^2
&=  \int_{\T^3} :\! ( \jb{\nb}^{-1}\Pi_{> j_0} Y_N )^2 \!: dx
+  \E \big[( \jb{\nb}^{-1}\Pi_{> j_0} Y_N )^2 \big].
\label{DN15x}
\end{align}

\noi
Recalling \eqref{P2}, 
we can bound  
the second term  by 
\begin{align}
\wt \s_{j_0} : =  \E \big[( \jb{\nb}^{-1}\Pi_{> j_0} Y_N )^2 \big]
 = \sum_{\substack{n\in \Z^3\\|n| > 2^{j_0}}}\frac{\chi_N^2(n)}{\jb{n}^4}
 \les 2^{-j_0}.
\label{DN15xx}
\end{align}

\noi
Let $Z_{N, j_0} =  \jb{\nb}^{-1}\Pi_{> j_0} Y_N $.
Proceeding as in the proof of Lemma 2.5 in \cite{OTh}
with Lemma \ref{LEM:Wick2}, we have
\begin{align}
\begin{split}
\E \Bigg[ \bigg(
 \int_{\T^3} :\! Z_{N, j_0}^2\!: dx\bigg)^2 \Bigg]
& = \int_{\T^3_x \times \T^3_y}
 \E\Big[  H_2( Z_{N, j_0}(x); \wt \s_{j_0})
 H_2( Z_{N, j_0} (y); \wt \s_{j_0})\Big] dx dy\\
& =  2
\sum_{\substack{n_1, n_2\in \Z^3\\|n_j| > 2^{j_0}}}\frac{\chi_N^2(n_1)\chi_N^2(n_2)}{\jb{n_1}^4\jb{n_2}^4}
\int_{\T^3_x \times \T^3_y}
e_{n_1 + n_2}(x-y) dx dy\\
& 
=  2 \sum_{\substack{n\in \Z^3\\|n| > 2^{j_0}}}\frac{\chi_N^4(n)}{\jb{n}^8}
\sim 2^{-5j_0}.
\end{split}
\label{DN15y}
\end{align}

\noi
Now, define a non-negative random variable $B_1(\o)$ by 
\begin{align}
B_1(\o) = 
\bigg(
\sum_{j=1}^\infty
2^{4 j}
\Big(
 \int_{\T^3} :\! Z_{N, j}^2\!: dx \Big)^2
\bigg)^{\frac 12}.
\label{DN15z}
\end{align}

\noi
By Minkowski's integral inequality, the Wiener chaos estimate (Lemma \ref{LEM:hyp}), and \eqref{DN15y}, 
we have 
\begin{align}
\E \big[B_1^p\big] \leq p^p
\Bigg(
\sum_{j=1}^\infty
2^{4 j}
\bigg\|
 \int_{\T^3} :\! Z_{N, j}^2\!: dx \bigg\|_{L^2(\O)}^2
\Bigg)^{\frac p2}
\les p^p < \infty
\label{DN100}
\end{align}

\noi
for any finite $p \geq 2$ (and hence for any finite $p \ge 1$).
Hence, from 
\eqref{DN15x}, \eqref{DN15xx},  and \eqref{DN15z}, we obtain
\begin{align}
 \|\Pi_{> j_0} Y_N \|_{H^{-1}}^2
\les  2^{-2 j_0} B_1(\o)+ 2^{-j_0}.
\label{DN15a}
\end{align}

\noi
Next, define a non-negative random variable $B_2(\o)$ by 
\[ B_2(\o) = \sum_{j = 1}^\infty\bigg|\int_{\T^3} :\! (\prod_{j} Y_N)^2 \!: dx\bigg|.\]
Then, a similar computation shows 
\begin{align}
\begin{split}
\|\Pi_{\le j_0}Y_N \|_{L^2}^2
&=  \int_{\T^3} :\! (\proj_{\le j_0} Y_N)^2 \!: dx
+  \E \big[ (\proj_{\le j_0} Y_N)^2 \big] \\
&\les B_2(\o) + 2^{j_0}
\end{split}
\label{DN15b}
\end{align}

\noi
and  $\E \big[B_2^p\big] \leq C_p < \infty$
for any finite $p \geq 1$.

Therefore, putting 
\eqref{DN9a}, \eqref{DN8} 
 \eqref{DN15},  \eqref{DN15a}, and \eqref{DN15b} 
together, 
choosing $2^{j_0} \sim 1 + \| \Ups_N \|_{H^1}^{\frac 23}$, 
and applying Cauchy's inequality, 
 we obtain
\begin{align}
\begin{split}
\| \Ups_N \|_{L^2}^6
& \les \bigg| \int_{\T^3} Y_N \Ups_N dx \bigg|^3
= \bigg| \sum_{j=0}^\infty \ld_j \| \proj_j Y_N \|_{L^2}^2 \bigg|^3\\
&\les \Big(   2^{-3j_0} B_1(\o)^{\frac 32} + 2^{-\frac 32 j_0}  \Big) \| \Ups_N \|_{H^1}^3 +   B_2^3(\o)
+ 2^{3j_0}  
\\
& \les
\| \Ups_N \|_{H^1}^2 + 
B_1^3(\o)
+ B_2^3(\o) + 1.
\end{split}
\label{DNc2}
\end{align}

\noi
This proves \eqref{DN2a} in the case \eqref{DN9a} holds.
This concludes the proof of Lemma \ref{LEM:Dr8}.
\end{proof}

\begin{remark}\label{REM:bd}
\rm

From the proof of Lemma \ref{LEM:Dr8}
(see \eqref{DN9} and \eqref{DNc2})
with Lemma \ref{LEM:Dr8},
we also have
\begin{align}
\begin{split}
\E\Bigg[\bigg| \int_{\T^3} Y_N \Ups_N dx \bigg|^3\Bigg]
& \les
\E\Big[\| \Ups_N \|_{L^2}^6
+ \| \Ups_N \|_{H^1}^{2}\Big]
+ 1\\
& \les
\U_N + 1, 
\end{split}
\label{DNc3}
\end{align}

\noi
where  $\U_N$ is as in \eqref{K10}.

\end{remark}

\subsection{Uniqueness of the limiting $\Phi^3_3$-measure}
\label{SUBSEC:wcon}

The tightness of the truncated Gibbs measures $\{\rho_N\}_{N\in \N}$, proven in the previous subsection,
together with Prokhorov's theorem implies
existence of a weakly convergent subsequence.
In this subsection, we prove uniqueness of 
the limiting $\Phi^3_3$-measure, which  allows us to conclude
the weak convergence of the entire sequence 
$\{\rho_N\}_{N\in \N}$.
While we  follow the uniqueness argument in our previous work \cite[Subsection 6.3]{OOTol1}, 
there are extra terms to control due to the focusing nature of the problem under consideration.

\begin{proposition}\label{PROP:uniq}
Let $\{ \rho_{N^1_k}\}_{k = 1}^\infty$
and $\{ \rho_{N^2_k}\}_{k = 1}^\infty$ be 
two weakly convergent  subsequences
of the truncated $\Phi^3_3$-measures
$\{\rho_N\}_{N\in \N}$ defined in \eqref{GibbsN}, 
converging weakly to $\rho^{(1)}$ and $\rho^{(2)}$ as $k \to \infty$,  respectively.
 Then, we have 
$\rho^{(1)} = \rho^{(2)}$.

\end{proposition}

\begin{proof}

\noi
$\bullet$ {\bf Step 1:}
We first show that 
\begin{align}
\lim_{k \to \infty} Z_{N^1_k} = \lim_{k \to \infty} Z_{N^2_k}, 
\label{KZ15}
\end{align}

\noi
where $Z_N$ is as in \eqref{pfN1}.
By taking a further subsequence, we may assume that 
 $N^1_k \ge N^2_k$, $k \in \N$.
Recall the change of variables \eqref{YZ13} 
and
let  $\ft R^{\dia}_{N}(Y+ \Ups^{N}+ \s \ZZ_{N})$ be as in \eqref{KZ16}.
Then, by 
the Bou\'e-Dupuis variational formula (Lemma~\ref{LEM:var3}), we have
\begin{align}
-\log Z_{N^j_k}
= \inf_{\dot \Ups^{N^j_k}\in  \mathbb H_a^1}\E \bigg[ \ft R^{\dia}_{N^j_k}(Y+ \Ups^{N^j_k}+ \s \ZZ_{N^j_k}) + \frac12 \int_0^1 \| \dot \Ups^{N^j_k}(t) \|_{H^1_x}^2dt \bigg]
\label{KZ17}
\end{align}

\noi
for $j= 1, 2$ and $k \in \N$.
We point out that  $Y$ and $\ZZ_N$ do not depend
on the drift $\dot \Ups^N$
in~\eqref{KZ17}.

Given $\dl > 0$, 
let  $ \UUps^{N^2_k}$ be an almost optimizer for \eqref{KZ17} with $j = 2$:
\begin{align}
-\log Z_{N^2_k}
& \ge 
\E \bigg[ \ft R^{\dia}_{N^2_k}(Y+  \UUps^{N^2_k}+ \s \ZZ_{N^2_k}) 
+ \frac12 \int_0^1 \| \dot \UUps^{N^2_k}(t) \|_{H^1_x}^2dt \bigg] - \dl.
\label{KZ17a}
\end{align}

\noi
By setting $\UUps_{N^2_k}: =\pi_{N^2_k}\UUps^{N^2_k}$, we have 
 \begin{align}
\pi_{N^1_k}\UUps_{N^2_k} = \UUps_{N^2_k}
 \label{chi2}
 \end{align}

\noi
since $N^1_k \ge N^2_k$.
Then, by choosing 
 $\Ups^{N^1_k} = \UUps_{N^2_k}$, 
 it follows from \eqref{KZ17a} and \eqref{chi2} that 
 \begin{align}
- & \log  Z_{N^1_k}
+\log Z_{N^2_k} \notag \\
&\le \inf_{\dot \Ups^{N^1_k}\in  \mathbb H_a^1} \E \bigg[ \ft R^{\dia}_{N^1_k}(Y+ \Ups^{N^1_k}+ \s \ZZ_{N^1_k}) + \frac12 \int_0^1 \| \dot \Ups^{N^1_k}(t) \|_{H^1_x}^2dt \bigg]\notag \\
& \quad \quad 
- \E \bigg[ \ft R^{\dia}_{N^2_k}(Y+ \UUps^{N^2_k}+ \s \ZZ_{N^2_k}) 
+ \frac12 \int_0^1 \| \dot  \UUps^{N^2_k}(t) \|_{H^1_x}^2dt \bigg] + \dl\notag \\
& \le 
\E \bigg[ \ft R^{\dia}_{N^1_k}(Y+ \UUps_{N^2_k}+ \s \ZZ_{N^1_k}) + \frac12 \int_0^1 \| \dot \UUps_{N^2_k}(t) \|_{H^1_x}^2dt \bigg]\notag \\
& \quad\quad 
- \E \bigg[ \ft R^{\dia}_{N^2_k}(Y+ \UUps^{N^2_k}+ \s \ZZ_{N^2_k}) 
+ \frac12 \int_0^1 \| \dot  \UUps^{N^2_k}(t) \|_{H^1_x}^2dt \bigg] + \dl\notag \\
& \leq   \E\Big[
\ft R^{\dia}(Y_{N^1_k}+ \UUps_{N^2_k}+ \s \wt \ZZ_{N^1_k})
- \ft R^{\dia}(Y_{N^2_k}+ \UUps_{N^2_k}+ \s \wt \ZZ_{N^2_k}) \Big]+ \dl, 
\label{KZ18}
\end{align}

\noi
where  $\wt \ZZ_{N^j_k}  = \pi_{N^j_k}  \ZZ_{N^j_k}  $ is as in \eqref{K9b}.
Here, 
$ \ft R^{\dia}$ is defined by 
\begin{align}
\begin{split}
\ft R^{\dia}(Y+ \Ups+ \s \ZZ)
& =  
- \s \int_{\T^3} Y \Dr^2 dx
- \frac \s 3 \int_{\T^3} \Dr^3 dx
\\
&\hphantom{X}
+ A \bigg| \int_{\T^3} \Big( :\! Y^2 \!: + 2 Y \Dr + \Dr^2 \Big) dx \bigg|^3, 
\end{split}
\label{KZ18a}
\end{align}

\noi
where  $\Dr = \Ups + \s \ZZ$.

We now estimate the right-hand side of \eqref{KZ18}.
The main point is that in  the difference
\begin{align}
 \E\Big[\ft R^{\dia}(Y_{N^1_k}+ \UUps_{N^2_k}+ \s \wt \ZZ_{N^1_k})
- \ft R^{\dia}(Y_{N^2_k}+ \UUps_{N^2_k}+ \s \wt \ZZ_{N^2_k})\Big], 
\label{KZ18b}
\end{align}

\noi
we only have  differences in  $Y$-terms and $\ZZ$-terms,
which allows us to gain  a negative power of $N^2_k$.
The contribution from 
the first  term on the right-hand side
in \eqref{KZ18a}  is given by 
\begin{align}
\begin{split}
-\s \E& \bigg[ \int_{\T^3}
( Y_{N^1_k} - Y_{N^2_k} ) \UUps_{N^2_k}^2 dx\bigg] \\
&-\s^2 \E \bigg[ \int_{\T^3}
( Y_{N^1_k} - Y_{N^2_k} ) (2 \UUps_{N^2_k} + \s \wt \ZZ_{N^1_k}) \wt \ZZ_{N^1_k} dx\bigg] \\
&-\s^2 \E \bigg[ \int_{\T^3}
Y_{N^2_k} (\wt \ZZ_{N^1_k} -\wt  \ZZ_{N^2_k}) ( 2 \UUps_{N^2_k}+ \s \wt \ZZ_{N^1_k} + \s \wt \ZZ_{N^2_k} ) dx\bigg].
\end{split}
\label{KZ19}
\end{align}

\noi
Let $\U _{N_k^2}= \U _{N_k^2}(\dot \UUps^{N_k^2})$ be as in \eqref{K10} with $\Ups_N = \UUps_{N_k^2}$ and $\Ups^N = \UUps^{N_k^2}$.
Then, from Lemmas~\ref{LEM:Dr} and \ref{LEM:Dr8}, we have 
\begin{align*}
\E\Big[\| \UUps_{N^2_k} \|_{ H^1}^2+ \| \UUps_{N^2_k} \|_{ L^2}^6\Big]
\les 1 + \U _{N_k^2}.
\end{align*}

\noi
Now, proceeding as in \eqref{KZ3}
together with H\"older's inequality in $\o$
and Young's inequality, 
we bound 
the first term in \eqref{KZ19} by
\begin{align}
\begin{split}
 \E & \Big[\| Y_{N^1_k} - Y_{N^2_k} \|_{\C^{-\frac 12-\eps}} 
\| \UUps_{N^2_k} \|_{L^2}^{\frac 32 -2\eps}\|\UUps_{N^2_k} \|_{H^1}^{\frac 12+2\eps} 
\Big] \\
&\le \| Y_{N^1_k} - Y_{N^2_k} \|_{L^\frac 6{3-4\eps}_\o \C^{-\frac 12-\eps}_x} 
\| \UUps_{N^2_k} \|_{L^6_\o L^2_x}^{\frac 32 - 2\eps } \|\UUps_{N^2_k} \|_{L^2_\o H^1_x}^{\frac 12+ 2\eps}
 \\
&\les (N^2_k)^{-a}
\| \UUps_{N^2_k} \|_{L^6_\o L^2_x}^{\frac 32 - 2\eps } \|\UUps_{N^2_k} \|_{L^2_\o H^1_x}^{\frac 12+ 2\eps}
\\
&\les (N^2_k)^{-a}\Big( 1+ 
\U_{N_k^2}
\Big), 
\end{split}
\label{KZ19b}
\end{align}

\noi
where the second inequality follows from
a modification of the proof of Lemma \ref{LEM:Dr}\,(i)
and noting that the Fourier transform of 
$Y_{N^1_k} - Y_{N^2_k}$ is supported on the frequencies $\{|n| \ges N^2_k\}$,
which allows us to gain a small negative power of $N^2_k$.
Note that 
the implicit constants in~\eqref{KZ19b}
depend on $A>0$ and $\s$.
However, the sizes of $A$ and $|\s|$ do not
play any role in the subsequent analysis and thus 
we suppress the dependence on  $A$ and $\s$
in the following.
The same comment applies to Subsections \ref{SUBSEC:wcon} and \ref{SUBSEC:notAC}.

The second and third terms in \eqref{KZ19}
and the second term on the right-hand side of \eqref{KZ18a}
can be handled in a similar manner (with \eqref{YZ15}
to control the $\wt \ZZ_{N^j_k}$-terms). 
As a result, we can bound the first two terms on the right-hand side of \eqref{KZ18a}
by 
\begin{align}
\begin{split}
 (N^2_k)^{-a}\Big( 
&C(Y_{N^1_k}, Y_{N^2_k},  \ZZ_{N^1_k}, \ZZ_{N^2_k})
+ \U _{N_k^2}
\Big)
\les  (N^2_k)^{-a}\Big( 
1 
+
\U _{N_k^2}
\Big)
\end{split}
\label{KZ20}
\end{align}

\noi
for some small $a > 0$,
where
 $C(Y_{N^1_k}, Y_{N^2_k},  \ZZ_{N^1_k}, \ZZ_{N^2_k})$ denotes
certain high moments of various stochastic terms involving
$Y_{N^j_k}$ and $\ZZ_{N^j_k}$, $j = 1, 2$, 
which are bounded by some constant,  independent of 
$N^j_k$, $j = 1, 2$, in view of Lemma \ref{LEM:Dr} and \eqref{YZ15}.

It remains to treat the difference coming from the last term in \eqref{KZ18a}.
By Young's and H\"older's inequalities, we have
\begin{align}
 & \E \Bigg[ \bigg| \int_{\T^3} \Big( :\! Y_{N_k^1}^2 \!: + 2 Y_{N_k^1} (\UUps_{N_k^2} + \s \wt \ZZ_{N_k^1}) + (\UUps_{N_k^2} + \s\wt  \ZZ_{N_k^1})^2 \Big) dx \bigg|^3
\notag \\
&\hphantom{XX}
- \bigg| \int_{\T^3} \Big( :\! Y_{N_k^2}^2 \!: + 2 Y_{N_k^2} (\UUps_{N_k^2} + \s \wt \ZZ_{N_k^2}) + (\UUps_{N_k^2} + \s\wt  \ZZ_{N_k^2})^2 \Big) dx \bigg|^3 \bigg] \notag \\
&\les \Bigg\{ \bigg\| \int_{\T^3} \Big( :\! Y_{N_k^1}^2 \!: - :\! Y_{N_k^2}^2 \!: \Big) dx \bigg\|_{L^3_\o}
+ \bigg\|\int_{\T^3} (Y_{N_k^1} - Y_{N_k^2}) \UUps_{N_k^2} dx\bigg\|_{L^3_\o} \notag\\
&\hphantom{XX}
+\bigg\| \int_{\T^3}  (Y_{N_k^1} - Y_{N_k^2}) \wt \ZZ_{N_k^1} dx \bigg\|_{L^3_\o} 
+ \bigg\| \int_{\T^3} Y_{N_k^2} (\wt \ZZ_{N_k^1} - \wt \ZZ_{N_k^2}) dx \bigg\|_{L^3_\o} 
\notag\\
&\hphantom{XX}
+\bigg\| \int_{\T^3} (\wt \ZZ_{N_k^1} -\wt  \ZZ_{N_k^2}) (
2\UUps_{N^2_k} + \s\wt \ZZ_{N_k^1}+ \s\wt  \ZZ_{N_k^2}) dx \bigg\|_{L^3_\o}
\Bigg\}
\notag\\
&\quad
\times \Bigg\{
\bigg\| \int_{\T^3} \Big( :\! Y_{N_k^1}^2 \!: + 2 Y_{N_k^1} (\UUps_{N_k^2} + \s \wt \ZZ_{N_k^1}) + (\UUps_{N_k^2} + \s \wt \ZZ_{N_k^1})^2 \Big) dx \bigg\|_{L^3_\o}^2 \notag\\
&\hphantom{XX}
+
\bigg\| \int_{\T^3} \Big( :\! Y_{N_k^2}^2 \!: + 2 Y_{N_k^2} (\UUps_{N_k^2} + \s \wt \ZZ_{N_k^2}) + (\UUps_{N_k^2} + \s \wt \ZZ_{N_k^2})^2 \Big) dx \bigg\|_{L^3_\o}^2
\Bigg\} \notag\\
&=: \1 \times \II.
\label{KZa1}
\end{align}

\noi
We divide $\1$ into two groups:
\begin{align}
\begin{split}
\1&  = \Bigg(\1 -  \bigg\|\int_{\T^3} (Y_{N_k^1} - Y_{N_k^2}) \UUps_{N_k^2} dx\bigg\|_{L^3_\o}\Bigg)
+   \bigg\|\int_{\T^3} (Y_{N_k^1} - Y_{N_k^2}) \UUps_{N_k^2} dx\bigg\|_{L^3_\o}\\
& =: \1_1 + \1_2.
\end{split}
\label{KK0a}
\end{align}

\noi
By the definition \eqref{pi}
of the cube frequency projector $\pi_N = \pi_N^\text{cube}$, we have 
\begin{align}
\int_{\T^3} (Y_{N_k^1} - Y_{N_k^2}) \UUps_{N_k^2} dx 
\int_{\T^3} \pi_{N^2_k}(Y_{N_k^1} - Y_{N_k^2}) \cdot \UUps_{N_k^2} dx = 0
\label{ortho1}
\end{align}

\noi
and thus $\1_2 = 0$.

By 
Lemma \ref{LEM:Bes}, 
H\"older's inequality in $\o$,
and Young's inequality, 
 followed by Lemma \ref{LEM:Dr8} with~\eqref{K10}, 
we can estimate $\1_1$ in \eqref{KK0a} by
\begin{align}
\begin{split}
\1_1
&\les
\|  :\! Y_{N_k^1}^2 \!: - :\! Y_{N_k^2}^2 \!:  \|_{L^3_\o \C^{-1-\eps}_x} \\
&\hphantom{X}
+ \| Y_{N_k^1} - Y_{N_k^2} \|_{L^6_\o \C^{-\frac12-\eps}_x}
 \| \wt \ZZ_{N_k^1} \|_{L^6_\o \C^{1-\eps}_x}
+ \| Y_{N_k^2} \|_{L^6_\o \C^{-\frac12-\eps}_x} \| \wt \ZZ_{N_k^1} -\wt  \ZZ_{N_k^2} \|_{L^6_\o \C^{1-\eps}_x}
\\
&\hphantom{X}
+\| \wt \ZZ_{N_k^1} - \wt \ZZ_{N_k^2} \|_{L^6_\o \C^{1-\eps}_x} \Big( \| \UUps_{N_k^2} \|_{L^6_\o L^2_x} + \| \wt \ZZ_{N_k^1} \|_{L^6_\o \C^{1-\eps}_x} + \| \wt \ZZ_{N_k^2} \|_{L^6_\o \C^{1-\eps}_x} \Big) \\
&\les
(N_k^2)^{-a} \Big( 1+ \U_{N_k^2} \Big)^{\frac 16}, 
\end{split}
\label{KZa4}
\end{align}

\noi
where we used Lemma \ref{LEM:Dr} and \eqref{YZ15} 
in bounding the terms involving $Y_{N_k^j}$ and $\wt \ZZ_{N_k^j} = \pi_{N_k^j} \ZZ_{N_k^j}$.  
As for  $\II$ in \eqref{KZa1},
it follows from \eqref{DNc3}, Lemma \ref{LEM:Dr8}, and \eqref{K10} that
\begin{align}
\begin{split}
\bigg\| & \int_{\T^3} \Big( :\! Y_{N_k^j}^2 \!: + 2 Y_{N_k^j} (\UUps_{N_k^2} + \s\wt  \ZZ_{N_k^j}) + (\UUps_{N_k^2} + \s \wt \ZZ_{N_k^j})^2 \Big) dx \bigg\|_{L^3_\o}
\\
&\les
1+ \bigg\| \int_{\T^3} Y_{N_k^j} \UUps_{N_k^2} dx \bigg\|_{L^3_\o} + \| \UUps_{N_k^2} \|_{L^6_\o L^2_x}^2 \\
&\les 1+ \U _{N_k^2}^{\frac 13}.
\end{split}
\label{KZa2}
\end{align}

 From 
 \eqref{DN2}, 
 \eqref{K8}, \eqref{K9}, 
 \eqref{KZ16}, 
and replacing $\UUps^{N^2_k}$ by 0 in view of 
 \eqref{KZ17},  
  we have
\begin{align}
\begin{split}
\sup_{k \in \N} \, 
\U _{N_k^2}(\dot \UUps^{N^2_k})
& \leq 10 C_0' + 
10 \sup_{k \in \N}\E \bigg[ \ft R^{\dia}_{N^2_k}(Y+ \UUps^{N^2_k}+ \s \ZZ_{N^2_k}) 
+ \frac12 \int_0^1 \| \dot  \UUps^{N^2_k}(t) \|_{H^1_x}^2dt \bigg] \\
& \les 1 + \dl 
+ \sup_{k \in \N}
\E \big[ \ft R^{\dia}_{N^2_k}(Y+ 0+ \s \ZZ_{N^2_k})  \big]\\
& \les 1.
\end{split}
\label{KZ20a}
\end{align}

\noi
Hence, 
from \eqref{ortho1}, \eqref{KZa4}, \eqref{KZa2}, 
and \eqref{KZ20a}, we obtain that 
\begin{align}
\1 \cdot \II 
\les  (N^2_k)^{-a}
\too 0, 
\label{KK0b}
\end{align}

\noi
as $k \to \infty$.
Therefore, from  \eqref{KZ20} and \eqref{KK0b}, 
we conclude 
that 
\begin{align}
 \E\bigg[\ft R^{\dia}(Y_{N^1_k}+ \UUps_{N^2_k}+ \s\wt \ZZ_{N^1_k})
- \ft R^{\dia}(Y_{N^2_k}+ \UUps_{N^2_k}+ \s \wt \ZZ_{N^2_k})\bigg]
\too 0, 
\label{KZ21}
\end{align}

\noi
as $k \to \infty$.
 Since the choice of $\dl > 0$ was arbitrary, 
it follows from \eqref{KZ18} and \eqref{KZ21} that 
\begin{align}
\lim_{k \to \infty} Z_{N^1_k} \geq  \lim_{k \to \infty} Z_{N^2_k}. 
\label{KZ21a}
\end{align}

\noi
By taking a subsequence of $\{N^2_k\}_{k \in \N}$, 
still denoted by  $\{N^2_k\}_{k \in \N}$, 
we may assume that  $ N^1_k \leq  N^2_k$.
By repeating the computation above, 
we then obtain 
\begin{align}
\lim_{k \to \infty} Z_{N^1_k} \leq  \lim_{k \to \infty} Z_{N^2_k}. 
\label{KZ21b}
\end{align}

\noi
Therefore, 
\eqref{KZ15} follows from \eqref{KZ21a} and \eqref{KZ21b}.

\smallskip

\noi
$\bullet$ {\bf Step 2:}
Next, we prove $\rho^{(1)} = \rho^{(2)}$.
This claim follows from a small modification  of Step~1.
For this purpose, we need to prove that for every bounded Lipschitz continuous function $F: \C^{-100}(\T^3) \to \R$,
we have  
\begin{align}
 \lim_{k\to \infty} \int \exp(F(u)) d \rho_{N^1_k} 
 \ge  \lim_{k\to \infty} \int \exp(F(u)) d \rho_{N^2_k}
\notag
\end{align}

\noi
under the condition   $N^1_k \ge N^2_k$, $k \in \N$
(which can be always satisfied by taking a subsequence of  $\{N^1_k\}_{k \in \N}$).
In view of \eqref{pfN1} and \eqref{KZ15}, it suffices to show
\begin{align}
\begin{split}
\limsup_{k\to \infty} \bigg[  -\log\Big(\int  & \exp(  F(u)
 -R_{N^1_k}^{\dia}(u))
d\mu \Big) \\
& + \log\Big(\int \exp(F(u) -R_{N^2_k}^{\dia}(u)) d \mu\Big)\bigg] \le 0.
 \label{KZ23}
\end{split}
\end{align}

\noi

By the Bou\'e-Dupuis variational formula (Lemma \ref{LEM:var3}), we have 
\begin{align}
\begin{split}
- & \log  \Big(\int \exp(F(u)
-R_{N^j_k}^{\dia}(u))
 d \mu  \Big)
\\
& = \inf_{\dot \Ups^{N^j_k}\in  \mathbb H_a^1}\E \bigg[ 
- F(Y+ \Ups^{N^j_k}+ \s \ZZ_{N^j_k}) \\
&\hphantom{XXXXXXXX}
+ \ft R^{\dia}_{N^j_k}(Y+ \Ups^{N^j_k}+ \s \ZZ_{N^j_k}) + \frac12 \int_0^1 \| \dot \Ups^{N^j_k}(t) \|_{H^1_x}^2dt \bigg], 
\end{split}
\label{KZ24}
\end{align}

\noi
where $\ft R^{\dia}_{N^j_k}(Y+ \Ups^{N^j_k}+ \s \ZZ_{N^j_k})$ is as in \eqref{KZ16}.
Given $\dl > 0$, 
let  $\UUps^{N^2_k}$ be an almost optimizer for~\eqref{KZ24} with $j = 2$:
\begin{align*}
-  \log   &  \Big(\int \exp(F(u)
-R_{N^2_k}^{\dia}(u))
 d \mu  \Big)\\
& \ge 
\E \bigg[ 
- F(Y+ \UUps^{N^2_k}+ \s \ZZ_{N^2_k})\\
& \quad  + 
\ft R^{\dia}_{N^2_k}(Y+ \UUps^{N^2_k}+ \s \ZZ_{N^2_k})  
+ \frac12 \int_0^1 \| \dot  \UUps^{N^2_k}(t) \|_{H^1_x}^2dt \bigg] - \dl.
\end{align*}

\noi
Then, by choosing $\Ups^{N^1_k} = \UUps_{N^2_k} = \pi_{N^2_k}\UUps^{N^2_k}$
and proceeding as in \eqref{KZ18}, 
 we have 
\begin{align}
  - & \log\Big(\int   \exp(  F(u)
 -R_{N^1_k}^{\dia}(u))
d\mu \Big) 
 + \log\Big(\int \exp(F(u) -R_{N^2_k}^{\dia}(u)) d \mu\Big) \notag \\
& \le 
\E \bigg[ 
- F(Y+ \UUps_{N^2_k}+ \s \ZZ_{N^1_k})\notag \\
& \hphantom{XXXi}
+ \ft R^{\dia}_{N^1_k}(Y+ \UUps_{N^2_k}+ \s \ZZ_{N^1_k}) + \frac12 \int_0^1 \| \dot \UUps_{N^2_k}(t) \|_{H^1_x}^2dt \bigg]\notag \\
& \quad
- \E \bigg[ 
- F(Y+ \UUps^{N^2_k}+ \s \ZZ_{N^2_k})\notag \\
& \hphantom{XXXi}
+ \ft R^{\dia}_{N^2_k}(Y+ \UUps^{N^2_k}+ \s \ZZ_{N^2_k}) 
+ \frac12 \int_0^1 \| \dot  \UUps^{N^2_k}(t) \|_{H^1_x}^2dt \bigg] + \dl\notag \\
& \le  
\Lip (F)\cdot 
 \E\Big[ \|  \pi_{N^2_k}^\perp\UUps^{N^2_k}
 - \s ( \ZZ_{N^1_k} -  \ZZ_{N^2_k}) \|_{\C^{-100}} \Big] \notag \\
 & \quad +  \E\Big[
\ft R^{\dia}(Y_{N^1_k}+ \UUps_{N^2_k}+ \s \wt \ZZ_{N^1_k})
- \ft R^{\dia}(Y_{N^2_k}+ \UUps_{N^2_k}+ \s \wt \ZZ_{N^2_k}) \Big]+ \dl, 
\label{KZ26}
\end{align}

\noi
where $\pi_N^\perp = \text{Id} - \pi_N$ and $\ft R^{\dia}$ is as in \eqref{KZ18a}.
We can proceed as in Step 1 to show that  the second term
on the right-hand side of \eqref{KZ26}
satisfies \eqref{KZ21}.
Here, we need to use the boundedness of $F$ in showing an analogue
of \eqref{KZ20a} in the current context
(with an almost optimizer
$\UUps^{N^2_k}$ for \eqref{KZ24}).

Finally, we estimate the first term on the right-hand side of \eqref{KZ26}.
Write 
\begin{align*}
 \E\Big[ &  \|  \pi_{N^2_k}^\perp\UUps^{N^2_k}
 - \s ( \ZZ_{N^1_k} -  \ZZ_{N^2_k}) \|_{\C^{-100}} \Big] \\
&  \les
 \E\Big[ \|  \pi_{N^2_k}^\perp\Ups^{N^2_k}\|_{\C^{-100}} \Big] 
 +  
 \E\Big[ \| 
  \ZZ_{N^1_k}  -   \ZZ_{N^2_k}\|_{\C^{-100}} \Big].
\notag
\end{align*}

\noi
A standard computation 
with \eqref{YZ12}
shows  that the second term on the right-hand side tends
to 0 as $k \to \infty$.
As for the first term, 
from Lemma \ref{LEM:Dr} and (an analogue of) \eqref{KZ20a}, we obtain
\begin{align*}
 \E\Big[ \|  \pi_{N^2_k}^\perp\UUps^{N^2_k}\|_{\C^{-100}} \Big] 
 \les
(N^2_k)^{-a} \|  \UUps^{N^2_k}\|_{L^2_\o H^1_x} 
 \les
(N^2_k)^{-a} \Big( \sup_{k \in \N} \, \U_{N^2_k} \Big)^{\frac 12}
\too 0,  
\end{align*}

\noi
as $k \to \infty$.
Since the choice of $\dl > 0$ was arbitrary, 
we conclude \eqref{KZ23}
and hence $\rho^{(1)} = \rho^{(2)}$.
This completes the proof of Proposition \ref{PROP:uniq}.
\end{proof}

\begin{remark} \label{REM:uniq1}\rm
In the proof of Proposition \ref{PROP:uniq}, 
we used the orthogonality relation \eqref{ortho1}
to conclude that $\1_2 = 0$.
While the same orthogonality holds
for the ball frequency projector $\pi_N^\text{ball}$ in \eqref{pib1},
such an orthogonality relation is false for the smooth 
frequency projector $\pi_N^\text{smooth}$ in~\eqref{pis1}.
As seen from the proof of Lemma \ref{LEM:Dr8}
and the uniform bound \eqref{KZ20a}
on $\U_{N^2_k}(\dot \UUps^{N^2_k})$, the quantity 
$\1_2$ in \eqref{KK0a} is critical (with respect to the spatial regularity/integrability and 
also with respect to the $\o$-integrability).
From Remark \ref{REM:bd}
and \eqref{KZ20a}, we see that the quantity $\1_2$ is bounded, 
uniformly in $k \in \N$.
In the absence of   the orthogonality \eqref{ortho1}, however, we do 
not know how to show that this term tends to $0$ as $k \to \infty$
in the case of the smooth 
frequency projector $\pi_N^\text{smooth}$.
We point out that the same issue  also appears in 
the proofs of 
Propositions \ref{PROP:ref}
and  \ref{PROP:plan}
in the case of the smooth 
frequency projector $\pi_N^\text{smooth}$.

\end{remark}

\subsection{Singularity of the $\Phi^3_3$-measure}
\label{SUBSEC:notAC}
We conclude this section 
by proving mutual singularity of the $\Phi^3_3$-measure $\rho$, 
constructed in the previous subsections, 
and the base  Gaussian free field $\mu$ in \eqref{gauss1}.
In Section 4 of \cite{BG2}, 
Barashkov and Gubinelli proved the singularity of the $\Phi^4_3$-measure
by making use of the shifted measure.
In the following, we 
follow our previous work  \cite{OOTol1}
and present a direct proof of
singularity of the $\Phi^3_3$-measure
without referring to  a shifted measure.
See also Appendix \ref{SEC:AC}, where we construct 
a shifted measure with respect to which the $\Phi^3_3$-measure 
is absolutely continuous.

\begin{proposition}\label{PROP:sing}
Let  $R_N$ be as  in \eqref{K1} with $\g = 3$, 
and $\eps>0$.
Then, 
there exists a strictly increasing sequence $\{ N_k \}_{k \in \N} \subset \N$ such that
the set
\begin{align}
S := \big\{ u \in H^{-\frac 12-\eps}(\T^3) : \lim_{k \to \infty} (\log N_k)^{-\frac 34} R_{N_k}(u) =0 \big\}
\notag
\end{align}
satisfies
\begin{align}
 \mu (S) = 1\qquad \text{but}\qquad \rho (S) =0.
\label{Ks2}
\end{align}

\noi
In particular, the $\Phi^3_3$-measure $\rho$ and 
the massive Gaussian free field $\mu$ in \eqref{gauss1} are mutually singular.
\end{proposition}

\begin{proof}

From \eqref{K1} with $\g=3$, the Wiener chaos estimate (Lemma \ref{LEM:hyp}), 
Lemma \ref{LEM:Wick2}, 
and Lemma \ref{LEM:SUM},
we have
\begin{align*}
\| R_N(u) \|_{L^2(\mu)}^2
&\les \bigg\| \int_{\T^3} :\! u_N^3 \!: dx \bigg\|_{L^2(\mu)}^2
+ \bigg\| \int_{\T^3} :\! u_N^2 \!: dx \bigg\|_{L^6(\mu)}^6 \\
&\les \bigg\| \int_{\T^3} :\! u_N^3 \!: dx \bigg\|_{L^2(\mu)}^2
+ \bigg\| \int_{\T^3} :\! u_N^2 \!: dx \bigg\|_{L^2(\mu)}^6 \\
&\les \sum_{\substack{n_1+n_2+n_3=0 \\ n_j \in NQ}} \jb{n_1}^{-2} \jb{n_2}^{-2} \jb{n_3}^{-2}
+ \bigg( \sum_{\substack{n_1+n_2=0 \\ n_j \in NQ}} \jb{n_1}^{-2} \jb{n_2}^{-2} \bigg)^3
\\
&\les
\sum_{\substack{|n_1| , |n- n_1| \les N}} \jb{n_1}^{-2} \jb{n-n_1}^{-1}
+ 1 \\
&\les \log N, 
\end{align*}

\noi
where 
$Q$ denotes the  cube of side length $2$ in $\R^3$ centered at the origin 
 as in \eqref{Q1}.
Thus, we have 
\[
\lim_{N \to \infty} (\log N)^{-\frac 34}
\| R_N(u) \|_{L^2(\mu)}
\les \lim_{N \to \infty}
(\log N)^{-\frac 14}
= 0.
\]

\noi
Hence, there exists a subsequence such that 
\[
\lim_{k \to \infty} (\log N_k)^{-\frac 34} R_{N_k}(u) =0, 
\]
almost surely with respect to $\mu$.
This proves $\mu(S) = 1$ in \eqref{Ks2}.

Given $k \in \N$, define $G_k(u)$ by 
\begin{align}
 G_k(u) = (\log N_k)^{-\frac 34} R_{N_k} (u).
 \label{Ks3a}
\end{align}

\noi
In the following, 
 we show that 
$e^{G_k(u)}$ tends to $0$ in $L^1(\rho)$.
This will imply that there exists a subsequence
of $G_k(u)$ 
 tending to $- \infty$, 
almost surely with respect to the $\Phi^3_3$-measure $\rho$,
 which in turn yields the second claim in~\eqref{Ks2}:
 $\rho(S) = 0$.

Let $\phi$ be a smooth bump function as in Subsection \ref{SUBSEC:21}.
By Fatou's lemma, the weak convergence of $\rho_M$ to $\rho$, 
 the boundedness of $\phi$, and \eqref{GibbsN}, we have
\begin{align}
\begin{split}
\int  e^{G_k(u)}d\rho(u)
& \leq \liminf_{K \to \infty}
\int \phi \bigg(\frac{G_k(u) }{K}\bigg)e^{G_k(u)}d\rho(u)\\
& = \liminf_{K \to \infty}
\lim_{M \to \infty}
\int \phi \bigg(\frac{G_k(u)}{K}\bigg)e^{G_k(u)}d\rho_M(u)\\
& \leq 
\lim_{M \to \infty}
\int e^{G_k(u)}d\rho_M(u)
= Z^{-1} \lim_{M \to \infty}
\int e^{G_k(u) - R_M^\dia(u)}d\mu(u)\\
& =: Z^{-1} \lim_{M\to \infty} C_{M, k}, 
\end{split}
\label{Ks4}
\end{align}

\noi
provided that  $\lim_{M\to \infty} C_{M, k} $ exists.
Here,  $Z = \lim_{M \to \infty} Z_M$ denotes the partition function for~$\rho$.

Our main goal is to  show that the right-hand side of~\eqref{Ks4}
tends to $0$ as $k \to \infty$.
As in the previous subsections, we proceed with 
 the change of variables \eqref{YZ13}: 
\[ \dot \Ups^M(t)  = \dot \Dr(t) - \s \dot \ZZ_M(t).\]

\noi
Then, 
by the Bou\'e-Dupuis variational formula (Lemma \ref{LEM:var3})
and \eqref{Ks3a}, we have 
\begin{align}
\begin{aligned}
- \log C_{M, k} 
& 
= \inf_{\dot \Ups^{M}\in  \Ha^1}
\E \bigg[ 
- (\log N_k)^{-\frac 34} R_{N_k}(Y+ \Ups^{M} +\s \ZZ_{M}) 
\\
& \hphantom{XXXXX}+ 
\ft R^{\dia}_{M}(Y+ \Ups^{M}+\s \ZZ_{M}) + \frac12 \int_0^1 \| \dot \Ups^{M}(t) \|_{H^1_x}^2dt \bigg]\\
& 
=:  \inf_{\dot \Ups^{M}\in  \Ha^1}
\ft \W_{M, k}(\dot \Ups^M), 
\end{aligned}
\label{Ks5a}
\end{align}

\noi
where 
$\ft R^{\dia}_{N}$ is as in \eqref{KZ16}.
In the following, we prove that the right-hand side (and hence the left-hand side)
of \eqref{Ks5a} 
diverges to $\infty$ as $k \to \infty$.

Proceeding   as in Subsection \ref{SUBSEC:tight}
(see \eqref{DN2}), 
we bound  the last two terms
on the right-hand side of \eqref{Ks5a} as
\begin{align}
\E \bigg[ 
\ft R^{\dia}_{M}(Y+ \Ups^{M} + \s \ZZ_{M}) + \frac12 \int_0^1 \| \dot \Ups^{M}(t) \|_{H^1_x}^2dt \bigg]
\geq 
 -C_0 + \frac{1}{10}\U_M, 
\label{Ks62}
\end{align}

\noi
where
 $\U_M = \U_M(\dot \Ups^M)$
is given by  \eqref{K10} with $\Ups_N = \pi_M \Ups^M$
and 
 $\Ups^N = \Ups^M$:
 \begin{align}
\U_M
= \E \bigg[
\frac A2 \bigg| \int_{\T^3} \Big( 2 Y_M \pi_M \Ups^M + (\pi_M \Ups^M)^2 \Big) dx \bigg|^3
+ \frac{1}{2} \int_0^1 \| \dot \Ups^M(t) \|_{H^1_x}^2 dt
\bigg].
\label{Ks62a}
\end{align}

Next, we study the first term on the right-hand side of \eqref{Ks5a}, 
which gives the main (divergent) contribution.
From \eqref{K1} with $\g=3$,
we have
\begin{align}
\begin{split}
R_{N_k} (Y + \Ups^M +\s \ZZ_M)
&=
-\frac \s 3 \int_{\T^3} :\! Y_{N_k}^3 \!:  dx
- \s \int_{\T^3} :\! Y_{N_k}^2 \!: \Dr_{N_k} dx \\
&\hphantom{X}
- \s \int_{\T^3} Y_{N_k} \Dr_{N_k}^2 dx
- \frac \s 3 \int_{\T^3} \Dr_{N_k}^3 dx
\\
&\hphantom{X}
+ A \bigg| \int_{\T^3} \Big( :\! Y_{N_k}^2 \!: + 2 Y_{N_k} \Dr_{N_k} + \Dr_{N_k}^2 \Big) dx \bigg|^3\\
& =: \1 + \II + \III + \IV + \5
\end{split}
\label{Ks6}
\end{align}

\noi
for $N_k \leq M$,
where
$\Dr_{N_k}$ is given by 
\begin{align}
\Dr_{N_k} := \pi_{N_k} \Dr = \pi_{N_k} \Ups^M + \s \pi_{N_k} \ZZ_M.
\label{Ks61}
\end{align}

\noi
As 
we see below, 
 under an expectation, 
the second term $\II$ on the right-hand side of \eqref{Ks6} 
(which is precisely the term removed by the second renormalization) gives
a divergent contribution; see \eqref{Ks8} below.
From Lemma \ref{LEM:Dr}, 
the first term $\1$ on the right-hand side of \eqref{Ks6} gives $0$
under an expectation.
As for the last three terms, 
 we proceed  as in Subsection \ref{SUBSEC:tight}
 (see also the proof of Proposition \ref{PROP:uniq})
 and obtain
\begin{align}
\begin{split}
\big|\E \big[ \III + \IV + \5 \big] \big|
 \les 
C(Y_{N_k},   \pi_{N_k} \ZZ_{M})
+ 
\U_{N_k}
 \les 
1+ 
\U_{N_k}
\end{split}
\label{Ks63}
\end{align}

\noi
where 
$C(Y_{N_k}, \pi_{N_k}  \ZZ_{M})$ denotes
certain high moments of various stochastic terms involving
$Y_{N_k}$ and $ \pi_{N_k} \ZZ_M$
and 
$\U_{N_k} = \U_{N_k}( \dt \pi_{N_k} \Ups^M)$ 
is given by  \eqref{K10} with $\Ups_N = 
\Ups^N =  \pi_{N_k} \Ups^M $:
\begin{equation}
\begin{split}
\U_{N_k} = \E\bigg[ & \frac A2 \bigg| \int_{\T^3} \Big( 2 Y_{N_k} 
\pi_{N_k}\Ups^M + (\pi_{N_k} \Ups^M)^2 \Big) dx \bigg|^3
+ \frac{1}{2} \int_0^1 \| \dt (\pi_{N_k} \Ups^M) (t) \|_{H^1_x}^2 dt\bigg].
\end{split}
\label{Ks64}
\end{equation}

In view of the smallness of $(\log N_k)^{-\frac 34} $
in \eqref{Ks5a},
the second term in \eqref{Ks64}
can be controlled by the positive terms $\U_M$ in \eqref{Ks62}
(in particular by the second term in \eqref{Ks62a}).
As for  the first term in \eqref{Ks64},
it follows from \eqref{DNc3}, 
$\pi_{N_k}  \Ups^M = \pi_{N_k} \pi_M \Ups^M$ for $N_k \le M$, 
and Lemma~\ref{LEM:Dr8}
with \eqref{Ks62a} that 
\begin{align}
\begin{split}
 \E \Bigg[ \bigg| \int_{\T^3} \Big( 2 Y_{N_k} \Ups^M + (\pi_{N_k} \Ups^M)^2 \Big) dx \bigg|^3 \Bigg]
&
\les 
\bigg\| \int_{\T^3} Y_{N_k} \pi_{N_k} \Ups^M dx \bigg\|_{L^3_\o}^3
+ \| \pi_{N_k} \Ups^M \|_{L^6_\o L^2_x}^6 \\
&
\les 
1+ \| \pi_M \Ups^M \|_{L^6_\o L^2_x}^6
+ \| \Ups^M \|_{L^2_\o H^1_x}^2 \\
&\les
1+\U_M
\end{split}
\notag
\end{align}
for $N_k \le M$.
Hence, 
$\U_{N_k}$ in \eqref{Ks64}
can be controlled by  $\U_M$ in \eqref{Ks62a}:
\begin{align}
\U_{N_k}
\les 
1+\U_M.
\label{Ks69}
\end{align}

\noi
Hence, from \eqref{Ks5a}, \eqref{Ks62}, \eqref{Ks6}, \eqref{Ks63}, 
and \eqref{Ks69},
we obtain
\begin{align}
\ft \W_{M, k}(\dot \Ups^M)
\ge \s (\log N_k)^{-\frac 34} \E \bigg[ 
 \int_{\T^3} :\! Y_{N_k}^2 \!:  \Dr_{N_k} dx \bigg]
 - C_1 + \frac{1}{20}\U_M
\label{Ks66}
\end{align}

\noi
for any $M\ge  N_k \gg 1$.

Therefore, it remains to estimate the contribution from the second term
on the right-hand side of~\eqref{Ks6}.
Let us first state a lemma whose proof is presented at the end of this subsection.

\begin{lemma}\label{LEM:Ks3}
We have
\begin{align}
 \E \bigg[  \int_0^1 \jb{ \dot \ZZ_N (t), \dot \ZZ_M (t)}_{H^1_x}dt \bigg] 
\sim  \log N
\label{Ks7}
\end{align}

\noi
for  any $1\leq N \le M$, where
 $\dot \ZZ_N = \pi_N \dot \ZZ^N$.
\end{lemma}

By assuming Lemma \ref{LEM:Ks3}, 
we complete the proof of Proposition \ref{PROP:sing}.
By \eqref{YZ11}, \eqref{YZ12} with $\ZZ_{N_k} = \pi_{N_k}\ZZ^{N_k}$, \eqref{Ks61}, 
Lemma \ref{LEM:Ks3},  Cauchy's inequality
(with small $\eps_0 > 0$), 
and Lemma \ref{LEM:Dr} (see \eqref{XX2}), 
we have 
\begin{align}
\begin{aligned}
 \s \E & \bigg[\int_{\T^3} \!:\!Y^2_{N_k}\!: \Dr_{N_k} dx \bigg] 
 = \s \E\bigg[ \int_0^1 \int_{\T^3}  
 \! :\!Y_{N_k}^2(t)\!: \dot \Dr_{N_k}(t) dt \bigg]\\
& 
 = 
  \s^2 \E \bigg[ \int_0^1  \jb{ \dot \ZZ_{N_k}(t),  \dot \ZZ_{M} (t)}_{H^1_x}  dt \bigg]
+ \s \E\bigg[   \int_0^1   \jb{\dot \ZZ_{N_k} (t), \dot \Ups^M (t)}_{H^1_x} dt \bigg]\\
& 
 \geq  c\log N_k
- \eps_0 \E\bigg[ \int_0^1 \| :\!  Y_{N_k}^2(t)\!:\|_{H^{-1}_x}^2 dt   \bigg]
- C_{\eps_0}\E\bigg[ \int_0^1 \| \dot \Ups^M (t) \|_{H^1_x}^2 dt   \bigg]\\
& \ge  
\frac{c}{2} \log N_k
- C_{\eps_0} \E\bigg[ \int_0^1 \| \dot \Ups^M (t) \|_{H^1_x}^2 dt \bigg]
\end{aligned}
\label{Ks8}
\end{align}

\noi
for $M\ge  N_k \gg 1$.
Thus, putting 
\eqref{Ks5a}, 
\eqref{Ks66}, 
and 
\eqref{Ks8} together, 
we have 
\begin{align}
\begin{aligned}
- \log C_{M, k} 
& 
\geq \inf_{\dot \Ups^{M}\in  \Ha^1}\Big\{ 
c (\log N_k)^{\frac 14} - C_2 + \frac{1}{40}\U_M\Big\}
\geq 
c (\log N_k)^{\frac 14} - C_2
\end{aligned}
\label{Ks9}
\end{align}

\noi
for any sufficiently large $k \gg1 $ (such that $N_k \gg 1$).
Hence, from \eqref{Ks9}, 
we obtain 
\begin{align}
C_{M, k}
\les
\exp \Big( -c (\log N_k)^{\frac 14} \Big)
\label{Ks10}
\end{align}
for $M\ge  N_k \gg 1$,  uniformly in $M \in \N$.
Therefore, 
by taking limits in $M \to \infty$ and then $k \to \infty$, 
we conclude from \eqref{Ks4} and \eqref{Ks10} that
\[ 
\lim_{k \to \infty} \int  e^{G_k(u)}d\rho(u) = 0\]

\noi
as desired.
This completes the proof of Proposition \ref{PROP:sing}.
\end{proof}

We conclude this section  by presenting the proof of Lemma \ref{LEM:Ks3}.

\begin{proof}[Proof of Lemma \ref{LEM:Ks3}]
For simplicity, we suppress the time dependence in the following.
From~\eqref{YZ12}, we have 
\begin{align}
\ft {\dot \ZZ}_N (n)
&= \jb{n}^{-2} 
\sum_{\substack{n_1, n_2 \in \Z^3 \\ n=n_1+n_2 \neq 0}}
\ft Y_N(n_1) \ft Y_N(n_2)
\label{Kt2}
\end{align}

\noi
for $n \ne 0$.
On the other hand, when $n = 0$, 
it follows from Lemma \ref{LEM:Wick2} that 
\begin{align}
\E \Big[ |\ft {\dot \ZZ}_{N} (0)|^2 \Big]
= \E\bigg[\Big(\sum_{\substack{n_1 \in \Z^3\\ n_1\in NQ}} \big(|\ft Y_N(n_1)|^2 - \jb{n_1}^{-2}\big)\Big)^2\bigg]
\les \sum_{n_1 \in \Z^3} \jb{n_1}^{-4}
\les 1, 
\label{Kt3}
\end{align}

\noi
where $Q$ is as in \eqref{Q1}.
Hence, from \eqref{Kt2} and \eqref{Kt3},
we have
\begin{align*}
\E \bigg[ \int_0^1 & \jb{ \dot \ZZ_N(t), \dot\ZZ_M(t)}_{H^1_x}dt \bigg] 
= 
\int_0^1 \E \Bigg[ \sum_{n \in \Z^3} \jb{n}^2 
\ft {\dot \ZZ}_N (n, t)\cj{\ft {\dot\ZZ}_M (n, t)}  \bigg] dt \\
&=
\int_0^1 \E \Bigg[ \sum_{n \in \Z^3\setminus \{0\}} \jb{n}^2 
\ft {\dot \ZZ}_N (n, t)\cj{\ft {\dot\ZZ}_M (n, t)}  \bigg] dt 
+ O (1).
\end{align*}

\noi
We now proceed as in the proof
of \eqref{XX2} in Lemma \ref{LEM:Dr}\,(i).
By  applying \eqref{YZ12} and   Lemma~\ref{LEM:Wick2}, 
and summing over  $\big\{|n| \le \frac{2}{3}N, 
\ \frac 14 |n| \leq |n_1| \leq \frac 12 |n|\big\}$
(which implies $|n_2| \sim |n|$ and $|n_j| \le  N$, $j = 1, 2$), 
we have 
\begin{align*}
\E & \Bigg[ \sum_{n \in \Z^3\setminus \{0\}} \jb{n}^2 
\ft {\dot \ZZ}_N (n, t)\cj{\ft {\dot\ZZ}_M (n, t)}  \bigg] \\
& = \sum_{n \in \Z^3}\frac{\chi_N(n)\chi_M(n)}{\jb{n}^2} \\
& \hphantom{XX}\times \int_{\T^3_x \times \T^3_y}
 \E\Big[  H_2(Y_N(x, t); t \s_N) H_2(Y_N(y, t); t \s_N)\Big] e_n(y - x)dx dy\\
& =\sum_{n \in \Z^3}\frac{t^2\chi_N(n)\chi_M(n)}{\jb{n}^2}
\sum_{n_1, n_2\in \Z^3}
\frac{\chi_N^2(n_1)\chi_N^2(n_2)}{\jb{n_1}^2\jb{n_2}^2}
\int_{\T^3_x \times \T^3_y}
e_{n_1 + n_2 - n}(x-y) dx dy\\
& =\sum_{n \in \Z^3}\frac{t^2\chi_N(n)\chi_M(n)}{\jb{n}^2}
\sum_{n = n_1 + n_2}
\frac{\chi_N^2(n_1)\chi_N^2(n_2)}{\jb{n_1}^2\jb{n_2}^2}
\sim t^2 \log N, 
\end{align*}

\noi
where $\chi_N(n_j)$ is as in \eqref{chi}.
By integrating on $[0, 1]$, 
we obtain
the desired bound \eqref{Ks7}.
\end{proof}

\section{Non-normalizability in the strongly nonlinear regime}
\label{SEC:non}

\subsection{Reference measures and the $\s$-finite $\Phi^3_3$-measure}
In this section, we prove non-normalizability 
of the $\Phi^3_3$-measure in the strongly nonlinear regime
(Theorem \ref{THM:Gibbs}\,(ii)).
In \cite{OOTol1}, 
we introduced a strategy for establishing
non-normalizability
in the context of the focusing Hartree $\Phi^4_3$-measures on $\T^3$, using the  
 Bou\'e-Dupuis variational formula.
We point out that, in~\cite{OOTol1}, 
the focusing Hartree $\Phi^4_3$-measures
were absolutely continuous with respect to the base Gaussian free field $\mu$.
Moreover, the truncated  potential energy $R_N^\text{Hartree}(u)$ and the 
corresponding density
$e^{- R_N^\text{Hartree}(u)}$
of the truncated focusing Hartree $\Phi^4_3$-measures
formed convergent sequences.
In \cite{OOTol1}, we proved the following version of the non-normalizability
 of the focusing Hartree $\Phi^4_3$-measure:
\begin{align}
\sup_{N \in \N} \E_\mu \Big[ e^{- R_N^\text{Hartree} (u)} \Big] = \infty.
\label{XX4}
\end{align}

\noi
Denoting the limiting density by $e^{-R^\text{Hartree} (u)}$, 
this result says that the $\s$-finite version of the focusing Hartree $\Phi^4_3$-measure:
\[ e^{- R^\text{Hartree} (u)}d\mu(u)\]

\noi
is not normalizable (i.e.~there is no normalization constant
to make this into a probability measure).
See also \cite{OS} for an analogous non-normalizability result
for the log-correlated focusing Gibbs measures with a quartic interaction potential.

The main new difficulty in our current problem
is the singularity of the $\Phi^3_3$-measure.
In particular, the potential energy $R_N^\dia(u)$ in \eqref{K1r}
(and the corresponding density $e^{-R_N^\dia(u)}$)
does {\it not} converge to any limit.
Hence, even if we prove a non-normalizability statement
of the form~\eqref{XX4}, 
it might still be possible that 
by choosing a sequence of  constants $\ft Z_N$ appropriately, 
the measure $\ft Z_N^{-1} e^{-R_N^\dia(u)} d\mu$ has a weak limit.
This is precisely the case for the $\Phi^4_3$-measure;
see~\cite{BG}.
The non-convergence claim in Theorem \ref{THM:Gibbs}\,(ii)
for the truncated $\Phi^3_3$-measures (see Proposition \ref{PROP:nonconv} below)
tells us that this is not the case for the $\Phi^3_3$-measure.

In order to overcome this issue, 
we first construct a reference measure $\nu_\dl$
as a weak limit of the following tamed version  of the truncated $\Phi^3_3$-measure
(with $\dl > 0$):
\begin{align*}
d \nu_{N, \dl} (u) = Z_{N, \dl}^{-1} \exp\Big(-\dl F(\pi_N u) - R^\dia_N(u)\Big)  d\mu(u)
\end{align*}

\noi
for some appropriate taming function $F $; see \eqref{RM-1}.
See Proposition \ref{PROP:ref}.
We also show that $F(u)$, without the frequency projection $\pi_N$ on $u$, is well defined almost surely 
with respect to the limiting reference measure $\nu_\dl = \lim_{N \to \infty} \nu_{N, \dl}$.
This  allows us to construct a $\sigma$-finite version of the $\Phi^3_3$-measure:
\begin{align}
 d \cj   \rho_\dl = e^{\dl F(u)} d \nu_\dl
 = \lim_{N \to \infty} Z_{N, \dl}^{-1} \, e^{\dl F( u)} 
 e^{-\dl F(\pi_N u) - R_N^\dia(u)} d\mu(u).
\label{rho1}
\end{align}

\noi
The main point is that while the truncated $\Phi^3_3$-measure $\rho_N$
(= $\nu_{N, \dl}$ with $\dl = 0$) may not be convergent, 
the tamed version $\nu_{N, \dl}$ of the truncated $\Phi^3_3$-measure
converges to the limit $\nu_\dl$, thus allowing us to define 
a $\s$-finite version of the $\Phi^3_3$-measure.
We then show that this  $\s$-finite version $\cj \rho_\dl$ 
of the $\Phi^3_3$-measure in \eqref{rho1} is not normalizable
in the strongly nonlinear regime.
See Proposition \ref{PROP:Gibbsref}.
Furthermore, as a corollary to this non-normalizability 
result of the  $\s$-finite version $\cj \rho_\dl$ 
of the $\Phi^3_3$-measure, 
we also show that 
the sequence $\{\rho_N\}_{N \in \N}$
of the truncated $\Phi^3_3$-measures defined in \eqref{GibbsN}
does not converge weakly in a natural space\footnote{For example, 
in the weakly nonlinear regime, the support of the limiting $\Phi^3_3$-measure  
constructed in Theorem \ref{THM:Gibbs}\,(i)
is contained in the space $\A(\T^3) \supset \C^{-\frac 34}(\T^3)$.}
$\A(\T^3)$ (see~\eqref{RM-2} below)
for the $\Phi^3_3$-measure.
See Proposition \ref{PROP:nonconv}.

\medskip

We first state the construction of the reference measure.
Let $p_t$ be the kernel of the heat semigroup $e^{t\Dl}$.
Then, define the space $\A = \A(\T^3)$ via the norm:
\begin{align}
\|  u \|_{\A} := \sup_{0 < t \le 1} \Big( t^{\frac38}\| p_t \ast u \|_{L^3(\T^3)} \Big).
\label{RM-2}
\end{align}

\noi
Recall from \cite[Theorem 5.3]{LR}\footnote{The discussion in \cite{LR}
is on $\R^d$, but a slight modification yields the corresponding result on $\T^d$.}
(see also  \cite[(2.41)]{Tri} and \cite[Theorem 2.34]{BCD})
that 
\begin{align}
 \A = B^{-\frac{3}{4}}_{3, \infty}(\T^3).
 \label{LR1}
\end{align}

\noi
In particular, the space $\A$ contains the support of the massive Gaussian 
free field $\mu$ on $\T^3$ and thus we have 
$\|  u \|_{\A} < \infty$, $\mu$-almost surely.
See Lemma \ref{LEM:emb2} below.
In the following, 
 for simplicity of notation, 
we use  $\A$ rather than $B^{-\frac{3}{4}}_{3, \infty}(\T^3)$.
Moreover, the notation~$\A$ is suitable for our purpose, since we make use of the characterization~\eqref{RM-2}
extensively via the Schauder estimate, 
which we recall now
(see for example \cite{ORW}):
\begin{align}
\| p_t*  u \|_{L^q (\T^3)} \le C_{\al, p, q} \,  t^{-\frac{\al}2-\frac 32 (\frac1p-\frac1q)} \|\jb{\nabla}^{-\al} u \|_{L^p(\T^3)}
\label{Sch}
\end{align}

\noi
for any  $\al \geq 0$ and $1\le p\le q\le \infty$.
From the Schauder estimate \eqref{Sch}
(or directly from~\eqref{LR1}), 
we see that $W^{-\frac{3}{4}, 3}(\T^3) \subset \A$.

Given $N \in \N$, we set 
$u_N = \pi_N u$.
Then, given $\dl > 0$ and $N \in \N$, we define the measure
$\nu_{N,\delta}$ by 
\begin{align}
d \nu_{N,\delta} (u)
=Z_{N,\delta}^{-1} \exp \Big( -\delta \| u_N \|_{\A}^{20} - R_N^{\dia}(u) \Big) d \mu (u)
\label{RM-1}
\end{align}
for $N \in \N$ and $\dl>0$,
where $R_N^\dia$ is as in \eqref{K1r} and
\begin{align}
Z_{N,\delta}
= \int \exp \Big( -\delta \| u_N \|_{\A}^{20} - R_N^{\dia}(u) \Big) d \mu (u).
\label{RM-1a}
\end{align}

\noi
Namely, $ \nu_{N,\delta} $ is a tamed version of the truncated $\Phi^3_3$-measure $\rho_N$ in \eqref{GibbsN}.
We prove that the sequence $\{\nu_{N, \dl}\}_{N\in \N}$
converges weakly to some limiting probability measure  $\nu_\dl$.

\begin{proposition} \label{PROP:ref}
Let $\sigma \neq 0$ and $\g \ge 3$.
Then, given any  $\delta > 0$, the sequence of measures $\{ \nu_{N,\delta} \}_{N \in \N}$ 
defined in \eqref{RM-1} 
converges weakly to a unique probability measure $\nu_\delta$, 
and similarly $Z_{N,\delta}$ converges to $Z_\delta$.
Moreover, $\| u \|_{\A}$ is finite $\nu_\delta$-almost surely, and we have
\begin{equation} \label{ref1}
d \nu_\delta (u)
= \frac{\exp(-(\delta-\delta') \| u \|_{\A}^{20})}{\int \exp(-(\delta-\delta') \| u \|_{\A}^{20}) d \nu_{\delta'}(u)} d \nu_{\delta'} (u)
\end{equation}
for $\dl>\dl'>0$.
\end{proposition}

This proposition  allows us to define a $\s$-finite version of the $\Phi^3_3$-measure by 
\begin{align}
 d \cj   \rho_\dl  =  e^{\dl \|u\|_{\A}^{20} }d \nu_\dl
\label{ref1a}
\end{align}

\noi
for any $\dl > 0$.
At a very {\it formal} level, 
$\dl \|u\|_{\A}^{20}$ in the exponent of \eqref{ref1a}
and $-\delta \| u_N \|_{\A}^{20}$ in the exponent of \eqref{RM-1}
cancel each other in the limit as $N \to \infty$, 
and thus the right-hand side of~\eqref{ref1}
 formally looks like
$Z_\dl^{-1} \lim_{N\to \infty} e^{-R^\dia_N(u)} d\mu$.
While this discussion is merely formal, 
it explains why we refer to the measure $\cj \rho_\dl$
as 
a $\s$-finite version of the $\Phi^3_3$-measure.
The identity~\eqref{ref1}
shows how $\nu_\dl$'s for different values of $\dl > 0$ are related.
When $\dl = 0$, the expression $Z_\dl  \cj \rho_\dl$ 
would formally correspond
to a limit of $e^{-R_N^\dia(u)} d\mu$, but 
in order to achieve the weak convergence claimed in Proposition \ref{PROP:ref}
and construct a $\s$-finite version of the $\Phi^3_3$-measure, 
we need to start with a tamed version (i.e.~$\dl > 0$) of the truncated $\Phi^3_3$-measure.
For the sake of concreteness, we chose a taming via the  $\A$-norm
but it is possible to consider a different taming (say, based on some other norm)
and obtain the same result.

The next proposition shows that the $\s$-finite version $\cj \rho_\dl$
of the $\Phi^3_3$-measure defined in \eqref{ref1a} is not normalizable
in the strongly nonlinear regime.

\begin{proposition} \label{PROP:Gibbsref}
Let $\s \gg 1$ and $\g \ge 3$.
Given  $\delta > 0$, let  $\nu_\delta$ be the measure constructed  in Proposition~\ref{PROP:ref}
and let $\cj \rho_\dl$ be as in \eqref{ref1a}.
Then, we have
\begin{align}
\int 1\, d \cj \rho_\dl = \int \exp\Big( \delta  \|u \|_{\A}^{20} \Big) d \nu_\delta = \infty.
\label{pa00a}
\end{align}
\end{proposition}

\begin{remark} \label{REM:ac2}\rm
(i) 
A slight modification
of  the computation in Subsection \ref{SUBSEC:notAC}
combined with  the analysis in Subsection~\ref{SUBSEC:refm} presented below
(Step 1 of the proof of Proposition \ref{PROP:ref})
shows that 
 the tamed version~$\nu_\dl$
of the $\Phi^3_3$-measure, constructed in Proposition \ref{PROP:ref}, 
and the massive Gaussian free field $\mu$ 
are mutually singular, just like the $\Phi^3_3$-measure
 in the weakly nonlinear regime, constructed in Section \ref{SEC:Gibbs}.
As a consequence, 
 the $\s$-finite version $\cj \rho_\dl$
of the $\Phi^3_3$-measure defined in \eqref{ref1a}
and the massive Gaussian free field $\mu$ 
are mutually singular.

\smallskip

\noi
(ii) 
In Appendix \ref{SEC:AC}, 
we show that the limiting $\Phi^3_3$-measure is absolutely continuous
with respect to the shifted measure
$\Law (Y(1) +\s \ZZ(1) + \W(1))$
in the weakly nonlinear regime.
A slight modification of the argument 
in Appendix \ref{SEC:AC}
also shows that the tamed version~$\nu_\dl$
of the $\Phi^3_3$-measure
constructed in Proposition \ref{PROP:ref}
and the $\s$-finite version $\cj \rho_\dl$
of the $\Phi^3_3$-measure in~\eqref{ref1a}
are also  absolutely continuous
with respect to the same shifted measure, 
even in the strongly nonlinear regime.
See Remark \ref{REM:ac}.
This shows that 
 the measure $\cj \rho_\dl$
in \eqref{ref1a} is a quite natural candidate
to consider 
as a  $\s$-finite version 
of the $\Phi^3_3$-measure.

\end{remark}

As a corollary to (the proofs of)
Propositions \ref{PROP:ref} and \ref{PROP:Gibbsref}, 
we show the following non-convergence result
for the truncated $\Phi^3_3$-measure $\rho_N$ in \eqref{GibbsN}.

\begin{proposition}\label{PROP:nonconv}
Let $\s \gg 1$,  $\g \ge 3$, 
and  $\A = \A(\T^3)$ be as in \eqref{RM-2}.
Then, 
the sequence $\{\rho_N\}_{N \in \N}$
of the truncated $\Phi^3_3$-measures defined in \eqref{GibbsN}
does not converge weakly to any limit 
as probability measures on $\A$.
The same claim holds for any subsequence 
 $\{\rho_{N_k}\}_{k \in \N}$.
\end{proposition}

In Subsection \ref{SUBSEC:refm}, 
we present the proof of  Proposition \ref{PROP:ref}.
In Subsection \ref{SUBSEC:NN}, 
we then prove the non-normalizability (Proposition \ref{PROP:Gibbsref}).
Finally, we present the proof of Proposition \ref{PROP:nonconv}
in Subsection \ref{SUBSEC:nonconv}.

\subsection{Construction of the reference measure}
\label{SUBSEC:refm}

In this subsection, we present the proof of  Proposition \ref{PROP:ref}
on the construction of the reference measure $\nu_\dl$.
We first establish several  preliminary lemmas.

\begin{lemma} \label{LEM:emb1}
Let the $\A$-norm be as in \eqref{RM-2}.
Then, we have 
\begin{align*}
 \|u\|_{\A} \les \|u\|_{H^{-\frac 14}}.
\end{align*}

\end{lemma}

\begin{proof}
This is immediate from  the Schauder estimate \eqref{Sch}.
\end{proof}

\begin{lemma} \label{LEM:emb2}
We have $W^{-\frac{3}{4}, 3}(\T^3) \subset \A$ and thus the quantity $\| u \|_{\A}$ 
is finite $\mu$-almost surely.
Moreover, given any  $1 \le p < \infty$, we have
\begin{align}
\E_\mu \Big[ \| \pi_N u \|_{\A}^p \Big] \le C_p < \infty, 
\label{RM0}
\end{align}
uniformly in $N \in \N \cup \{\infty\}$ with the understanding that $\pi_\infty = \Id$.
\end{lemma}

\begin{proof}
As we already mentioned,  the first claim follows
from the Schauder estimate~\eqref{Sch}
 (or from \eqref{LR1}).
As for the bound \eqref{RM0}, 
from the Schauder estimate \eqref{Sch}, 
Minkowski's integral inequality, and the Wiener chaos estimate (Lemma \ref{LEM:hyp}) with \eqref{IV2}, we have
\begin{align*}
\E_\mu \Big[ \|  \pi_N u\|_{\A}^p \Big] 
& \les \E_\mu \Big[ \| u \|_{W^{-\frac 34, 3}}^p \Big] 
\les \Big\| \| \jb{\nb}^{-\frac{3}{4}} u (x)\|_{L^p(\mu)}\Big\|_{L^3_x}^p \\
& \le p^\frac{p}{2} \Big\| \| \jb{\nb}^{-\frac{3}{4}} u (x)\|_{L^2(\mu)}\Big\|_{L^3_x}^p\\
& \le p^\frac{p}{2}\bigg( \sum_{n \in \Z^3} \frac{1}{\jb{n}^\frac 72 }\bigg)^p < \infty.
\end{align*}

\noi
This proves \eqref{RM0}.
\end{proof}

We now  present the proof of Proposition \ref{PROP:ref}.

\begin{proof}[Proof of Proposition \ref{PROP:ref}]
$\bullet$ {\bf Step 1:}
In this first part, we prove that $Z_{N,\delta}$ in \eqref{RM-1a} is uniformly bounded in $N\in \N$.
As for the tightness of $\{ \nu_{N,\dl} \}_{N \in \N}$ and the uniqueness of $\nu_\delta$ 
claimed in the statement, 
we can repeat arguments analogous to those in Subsections \ref{SUBSEC:tight} and \ref{SUBSEC:wcon}
and thus we omit details. 

From  \eqref{RM-1a} and
the Bou\'e-Dupuis variational formula (Lemma \ref{LEM:var3}) with 
the change of variables \eqref{YZ13},
we have 
\begin{equation}
\begin{split}
-\log Z_{N,\delta}
&= \inf_{\dot \Ups^N \in \mathbb H_a^1 }\E
\bigg[  \delta \| Y_N + \Dr_N \|_\A^{20} 
-\s \int_{\T^3} Y_N \Dr_N^2 dx
-\frac \s 3 \int_{\T^3} \Dr_N^3 dx
\\
&\hphantom{XXXXX}
+ A \bigg| \int_{\T^3} \Big( :\! Y_N^2 \!: + 2 Y_N \Dr_N + \Dr_N^2 \Big) dx \bigg|^\g \\
&\hphantom{XXXXX}
+ \frac{1}{2} \int_0^1 \| \dot \Ups^N(t) \|_{H^1_x}^2 dt
 \bigg],
\end{split}
\label{RM4}
\end{equation}
where 
$\Dr_N = \Ups_N + \s \wt \ZZ_N$
with $\wt \ZZ_N = \pi_N\ZZ_N$ as in \eqref{K9b}.
Our goal is to establish a uniform lower bound on the right-hand side 
of~\eqref{RM4}.
Unlike Subsection \ref{SUBSEC:tight}, 
we do not assume smallness  on $|\s|$.
In this case, a rescue comes from the extra positive term 
$\delta \| Y_N + \Dr_N \|_\A^{20}$ as compared to \eqref{K9}.

Given any $0 < c_0 < 1$, 
it follows from Young's inequality \eqref{YY9} with $\g \ge 3$ that 
\begin{align}
\bigg| \int_{\T^3} \Big( :\! Y_N^2 \!: + 2 Y_N \Dr_N + \Dr_N^2 \Big) dx \bigg|^\g
\ge c_0
\bigg| \int_{\T^3} \Big( :\! Y_N^2 \!: + 2 Y_N \Dr_N + \Dr_N^2 \Big) dx \bigg|^3 - C.
\label{QQ4}
\end{align}

\noi
Then, taking an expectation and applying
Lemmas \ref{LEM:Dr7} and \ref{LEM:Dr8}
 with 
 Lemma \ref{LEM:Dr} and \eqref{YZ15}, we have
\begin{align}
\E\Bigg[A \bigg| \int_{\T^3} \Big( :\! Y_N^2 \!: + 2 Y_N \Dr_N + \Dr_N^2 \Big) dx \bigg|^\g\Bigg]
\ge C_0 \E  \Big[ \| \Ups_N \|_{L^2}^6\Big]
 - C_1 \E \Big[\| \Ups_N\|_{H^1}^2\Big]- C
\label{RM4a}
\end{align}

\noi
for some $C_0 > 0$ and $0 < C_1 \le \frac 14$.
Hence, 
it follows  from~\eqref{RM4}, \eqref{RM4a}, 
and Lemma~\ref{LEM:Dr7} together with 
 Lemma \ref{LEM:Dr} and \eqref{YZ15} that 
there exists $C_2 > 0$ such that
\begin{align}
\begin{split}
-\log Z_{N,\delta}
&\ge \inf_{\dot \Ups^N \in \mathbb H_a^1 }\E\bigg[
\delta \| Y_N + \Ups_N + \s \wt \ZZ_N \|_\A^{20}
 -\frac \s 3 \int_{\T^3} (\Ups_N + \s \wt \ZZ_N)^3 dx\\
&\hphantom{XXXXXX} 
+ C_2   \| \Ups_N \|_{L^2}^6 
+ C_2\| \Ups_N \|_{H^1}^2    \bigg] 
- C.
\end{split}
\label{RM5}
\end{align}

\noi
By Young's inequality, we have 
\begin{align}
\begin{split}
\bigg| \int_{\T^3} \Ups_N^2 \wt \ZZ_N dx \bigg|
+ \bigg| \int_{\T^3} \Ups_N\wt  \ZZ_N^2 dx \bigg|
&\le \| \Ups_N \|_{L^2}^2  \| \ZZ_N \|_{\C^{1-\eps}} + \| \Ups_N \|_{L^2} \| \ZZ_N \|_{\C^{1-\eps}}^2 \\
& \le
 \frac{C_2}{2|\s|} \| \Ups_N \|_{L^2}^6 
+ \| \ZZ_N \|_{\C^{1-\eps}}^c + C_\s.
\end{split}
\label{RM6}
\end{align}

\noi
Hence, from \eqref{RM5} and \eqref{RM6}
with  \eqref{YY9} (with $\g = 20$) and Lemma  \ref{LEM:emb2},
we obtain  
\begin{equation} 
\begin{split}
-\log Z_{N,\delta}
\ge \inf_{\dot \Ups_N \in \mathbb H_a^1 }
\E\bigg [ & \, \frac \delta2 \| \Ups_N \|_\A^{20}  - \frac{|\s|}3 \|\Ups_N\|_{L^3}^3\\
&  + \frac {C_2} 2   \| \Ups_N \|_{L^2}^6 + C_2\| \Ups_N \|_{H^1}^2  
 \bigg] - C.
 \end{split}
 \label{ZND1}
\end{equation}

\noi
Now, we need to estimate the $L^3$-norm of $\Ups_N$.
From \eqref{RM-2}, Sobolev's inequality, 
and the mean value theorem: $|1 - e^{-t|n|^2}| \les 
(t |n|^2)^\ta$ for any $0 \le \ta \leq 1$,  
 we have 
\begin{align*}
\|\Ups_N\|_{L^3}^3
&\les t^{-\frac 98} \| \Ups_N \|_\A^3 + \|\Ups_N - p_t \ast \Ups_N\|_{H^\frac 12 }^3 \\
&\les t^{-\frac 98} \| \Ups_N \|_\A^3 + t^\frac 34 \|\Ups_N \|_{H^1}^3
\end{align*}

\noi
for $0<t \le 1$.
By choosing $ t^\frac 34  \sim  \big(1 + \frac{|\s|}{C_2} \| \Ups_N \|_{H^1}\big)^{-1}$
and applying Young's inequality, we obtain
\begin{equation} \label{ZND2}
\begin{aligned}
|\s| \|\Ups_N\|_{L^3}^3
&\le C_{C_2, |\s|}
\| \Ups_N \|_{H^1}^\frac 32\| \Ups_N \|_\A^3
+ \frac {C_2}4  \| \Ups_N \|_{H^1}^2 + 1\\
&\le C_{C_2, |\s|, \dl} + \frac \delta 4\| \Ups_N \|_\A^{20}
 + \frac {C_2}2\| \Ups_N \|_{H^1}^2 .
\end{aligned}
\end{equation}

\noi
Therefore, 
from  \eqref{ZND1} and \eqref{ZND2},
we conclude that 
\[  Z_{N,\delta}\le C_\dl < \infty, \]

\noi
uniformly in $N \in \N$.

\smallskip

\noi
$\bullet$ {\bf Step 2:}
Next, we show  that $\| u \|_{\A}$ is finite $\nu_\delta$-almost surely.
Let $ \eta$ be a smooth function with compact support with 
$\int_{\R^3} | \eta  (\xi)|^2  d\xi = 1$
and set 
\[ \ft \rho(\xi)= \int_{\R^3} \eta (\xi - \xi_1) \cj{\eta (-\xi_1)}d\xi_1.\]

\noi
Given $\eps > 0$, 
 define $\rho_\eps$ by 
\begin{align}
 \rho_\eps(x) = \sum_{n \in \Z^3} \ft \rho (\eps n) e^{i n \cdot x}. 
 \label{XX5a}
\end{align}

\noi
Since the support of $\ft \rho$ is compact,
the sum on the right-hand side is over finitely many frequencies.
Thus, given  any $\eps>0$, there exists $N_0(\eps) \in \N$ such that 
\begin{align}
\rho_\eps \ast u = \rho_\eps \ast u_N
\label{XX6}
\end{align}

\noi
for any $N \ge N_0(\eps)$.
From the Poisson summation formula, we have 
\[ \rho_\eps (x) = \sum_{n \in \Z^3} \eps^{-3} \big|\F_{\R^3}^{-1}(\eta)(\eps^{-1} x + 2\pi n)\big|^2
\ge 0,  \]

\noi
where $\F_{\R^3}^{-1}$ denotes the inverse Fourier transform on $\R^3$.
Noting that 
\[
\| \rho_\eps\|_{L^1(\T^3)} = 
\int_{\T^3} \rho_\eps (x) dx
= \ft \rho (0)
= \| \eta \|_{L^2(\R^3)}^2
=1,
\]

\noi
we have, from Young's inequality, that 
\begin{align}
\| \rho_\eps \ast u \|_{\A} \le \| u \|_{\A}.
\label{XX7}
\end{align}

\noi
Moreover, $\{\rho_\eps\}$ defined above
is an approximation to the identity on $\T^3$
and thus for any distribution $u$ on $\T^3$, 
$\rho_\eps \ast u \to u$ in the $\A$-norm,  as $\eps \to 0$.

Let $\dl>\delta'>0$.
By Fatou's lemma, the weak convergence of $\{ \nu_{N, \dl} \}_{N \in \N}$
from Step 1
 with \eqref{XX6},  \eqref{XX7}, 
 and the definition \eqref{RM-1} of $\nu_{N, \dl}$, 
we have 
\begin{align*}
\int \exp\Big( (\delta - \delta') \|u \|_{\A}^{20} \Big) d \nu_\delta
&\le \liminf_{\eps \to 0} \int \exp\Big( (\delta - \delta') \|\rho_\eps \ast u \|_{\A}^{20} \Big) d \nu_\delta \\
&= \liminf_{\eps \to 0} \lim_{N \to \infty}  \int \exp\Big( (\delta - \delta') \|\rho_\eps \ast u_N \|_{\A}^{20} \Big) d \nu_{N,\delta}\\
&\le \lim_{N \to \infty}  \int \exp\Big( (\delta - \delta') \| u_N \|_{\A}^{20} \Big) d \nu_{N,\delta}\\
&= \lim_{N \to \infty} \frac{Z_{N,\delta'}}{Z_{N,\delta}} \int 1 \, d \nu_{N,\delta'}\\
&= \frac{Z_{\delta'}}{Z_\delta}.
\end{align*}

\noi
Hence, we have 
\begin{align*}
\int \exp\Big((\delta - \delta') \|u \|_{\A}^{20} \Big) d \nu_\delta < \infty
\end{align*}

\noi
for any $\dl>\dl'>0$.
By choosing $\delta' = \frac \delta2$,
we obtain
\begin{align*}
\int \exp\Big(\frac \delta2 \|u \|_{\A}^{20} \Big) d \nu_\delta < \infty, 
\end{align*}

\noi
which shows
 that $\|u\|_{\A}$ is finite almost surely
with respect to $\nu_\dl$.

\smallskip

\noi
$\bullet$ {\bf Step 3:}
Finally, we prove the relation  \eqref{ref1}.
We first note that it suffices to show that
\begin{align}
 \frac{Z_{\delta}}{Z_{\delta'}} d \nu_\delta = \exp\Big( - (\delta - \delta') \|u \|_{\A}^{20} \Big)   d \nu_{\delta'}
\label{RM8a}
\end{align}
for any $\delta > \delta'>0$.
In fact, once we have \eqref{RM8a}, 
by integration, we obtain 
\begin{align}
 \frac{Z_{\delta}}{Z_{\delta'}} = \int \exp\Big( - (\delta - \delta') \|u \|_{\A}^{20} \Big)   d \nu_{\delta'}
\label{RM8b}
\end{align}

\noi
and thus \eqref{ref1} follows from \eqref{RM8a} and \eqref{RM8b}.

Let $F: \C^{-100}(\T^3) \to \R$ be a bounded Lipschitz function with $F \ge 0$.
The dominated convergence theorem, the weak convergence of $\{ \nu_{N, \dl} \}_{N \in \N}$
from Step 1, and \eqref{RM-1} yield that
\begin{align*}
& \frac{Z_{\delta}}{Z_{\delta'}} \int F(u) d \nu_\delta  -  \int F(u) \exp\Big( - (\delta - \delta') \|u \|_{\A}^{20} \Big) d \nu_{\delta'} \\
& = \lim_{\eps \to 0} \bigg( \frac{Z_{\delta}}{Z_{\delta'}} \int F(u) d \nu_\delta  -  \int F(u) \exp\Big( - (\delta - \delta') \|\rho_\eps \ast u \|_{\A}^{20} \Big) d \nu_{\delta'} \bigg) \\
&= \lim_{\eps \to 0} \lim_{N \to \infty}
\bigg(
\frac{Z_{N,\delta}}{Z_{N,\delta'}} \int F(u) d \nu_{N,\delta}  -  \int F(u) \exp\Big( - (\delta - \delta') \|\rho_\eps \ast u_N \|_{\A}^{20} \Big) d \nu_{N,\delta'} \bigg) \\
&= \lim_{\eps \to 0} \lim_{N \to \infty}
\int F(u) \Big[ \exp\Big( - (\delta - \delta') \|u_N \|_{\A}^{20}\Big)  - \exp\Big(-(\delta-\delta') \|\rho_\eps \ast u_N \|_{\A}^{20} \Big)\Big] d \nu_{N, \delta'}.
\end{align*}
Therefore, we have 
\begin{align}
\begin{split}
 \bigg|  \frac{Z_{\delta}}{Z_{\delta'}} & \int F(u) d \nu_\delta  -  \int F(u) \exp\Big( - (\delta - \delta') \|u \|_{\A}^{20} \Big) d \nu_{\delta'} \bigg|
\\
&\les \limsup_{\eps \to 0} \limsup_{N \to \infty}
\int \Big|
\exp\Big( - (\delta - \delta') \|u_N \|_{\A}^{20}\Big)\\
& \hphantom{XXXXXXXXXX}- \exp\Big(-(\delta-\delta') \|\rho_\eps \ast u_N \|_{\A}^{20} \Big)
 \Big| \, d \nu_{N,\delta'}(u)\\
&\les \limsup_{\eps \to 0} \limsup_{N \to \infty}
\int \Big|
\exp\Big( - (\delta - \delta') \|\pi_N u^N (\o)\|_{\A}^{20}\Big)\\
& \hphantom{XXXXXXXXXX}- \exp\Big(-(\delta-\delta') \|\rho_\eps \ast \pi_N u^N (\o)\|_{\A}^{20} \Big)
 \Big| \, d\PP(\o), 
\end{split}
\label{RM8c}
\end{align}

\noi
where $u^N$ is a random variable with $\Law(u^N) = \nu_{N, \dl'}$.
Noting that the integrand is uniformly bounded by $2$, 
it follows from the bounded convergence theorem that 
 the right-hand side of~\eqref{RM8c} tends to 0 
 once we show that 
 $\|\rho_\eps \ast \pi_N u^N (\o) - \pi_N u^N(\o) \|_{\A}$ tends to 0 
in measure (with respect to $\PP$).
Namely, it suffices to show 
\begin{align*}
 \lim_{\eps \to 0} & \lim_{N \to \infty}
 \PP\big( \{ \o \in \O: \|\rho_\eps \ast \pi_N u^N (\o) - \pi_N u^N(\o)\|_{\A} 
 > \al \}\big)\\
& =   \lim_{\eps \to 0} \lim_{N \to \infty}
\nu_{N,\delta'}\big( \big\{ \|u_N - \rho_\eps \ast u_N \|_{\A} > \al \}\big) = 0
\end{align*}
for any $\al>0$.

From  \eqref{RM-2} and \eqref{Sch}, we have 
\begin{align}
\| u_N - \rho_\eps \ast u_N \|_{\A}
 \les \| u_N - \rho_\eps \ast u_N \|_{W^{-\frac 34, 3}} 
\les \eps^{\frac 18} \| u_N\|_{W^{-\frac 58, 3}}.
\label{ref2a}
\end{align}

\noi
Hence, from Chebyshev's inequality and \eqref{ref2a}, 
it suffices to prove
\begin{equation}
\int  \| u_N\|_{W^{-\frac 58, 3}}  d \nu_{N,\delta'} 
\les \int \exp \big(   \| u_N\|_{W^{-\frac 58, 3}} \big) d \nu_{N,\delta'} \le C_{\dl'} < \infty, 
\label{ref3}
\end{equation}
uniformly in $N \in \N$.
We use the variational formulation as in \eqref{RM4}, and write  
\begin{align*}
-  & \log   \bigg(\int \exp \big(  \| u_N\|_{W^{-\frac 58, 3}} \big) d \nu_{N,\delta'} \bigg) \\
&= \inf_{\dot \Ups^N \in \mathbb H_a^1 }\E
\bigg[ 
 \delta' \| Y_N + \Dr_N \|_\A^{20} 
- \|Y_N + \Dr_N\|_{W^{-\frac 58, 3}}-\s \int_{\T^3} Y_N \Dr_N^2 dx \\ 
&\hphantom{XXXXXX} 
-\frac \s 3 \int_{\T^3} \Dr_N^3 dx
+ A \bigg| \int_{\T^3} \Big( :\! Y_N^2 \!: + 2 Y_N \Dr_N + \Dr_N^2 \Big) dx \bigg|^\g \\
&\hphantom{XXXXXX}
+ \frac{1}{2} \int_0^1 \| \dot \Ups^N(t) \|_{H^1_x}^2 dt
 \bigg]\\
 &\phantom{X} + \log Z_{N,\delta'},
\end{align*}
where 
$\Dr_N = \Ups_N + \s \wt \ZZ_N$.
From 
Lemma \ref{LEM:Dr}
and \eqref{YZ15}, 
we have, for any finite $p \geq 1$,  
\begin{align}
\E\Big[\| Y_N\|_{W^{-\frac 58, 3}}^p
+ \| \ZZ_N\|_{W^{-\frac 58, 3}}^p\Big] < \infty,
\label{ref4}
\end{align}

\noi
uniformly in $N \in \N$.  See also the proof of Lemma \ref{LEM:emb2}.
Then, 
arguing as in  \eqref{ZND1} and \eqref{ZND2} with Young's inequality, 
Sobolev's inequality, and \eqref{ref4}, we obtain
\begin{align*}
-  & \log \bigg(\int \exp \big( \| u_N\|_{W^{-\frac 58, 3}} \big) d \nu_{N,\delta'} \bigg) \\
&\ge \inf_{\dot \Ups^N \in \mathbb H_a^1 }\E\bigg[-\|\Ups_N\|_{W^{-\frac 58,3}} 
+ C_0 \big(  \| \Ups_N \|_{L^2}^6 + \| \Ups_N \|_{H^1}^2\big)  + \frac {\delta'}4 \| \Ups_N \|_\A^{20} \bigg]
 - C_{C_0,\delta'}\\
&\gtrsim -1.
\end{align*}

\noi
This proves
 \eqref{ref3}
 and hence
 concludes the proof of Proposition \ref{PROP:ref}.
\end{proof}

\subsection{Non-normalizability
of the $\s$-finite measure $\cj \rho_\dl$}
\label{SUBSEC:NN}

In this subsection, we present the proof of Proposition \ref{PROP:Gibbsref}
on the non-normalizability of the $\s$-finite version $\cj \rho_\dl$
of the $\Phi^3_3$-measure defined in \eqref{ref1a}.

Given $\eps > 0$, let $\rho_\eps$ be as in \eqref{XX5a}.
Then, by  \eqref{XX7}, the weak convergence of $\{ \nu_{N,\dl} \}_{N \in \N}$
(Proposition \ref{PROP:ref}),  \eqref{XX6}, and \eqref{RM-1},
we have
\begin{align*}
\int & \exp\Big( \delta  \|u \|_{\A}^{20} \Big) d \nu_\delta
\ge \int \exp\Big( \delta  \|\rho_\eps \ast u  \|_{\A}^{20} \Big) d \nu_\delta\\
&\ge \limsup_{L \to \infty} \int \exp\Big( \dl \min\big(  \|\rho_\eps \ast u \|_{\A}^{20}, L\big) \Big) d \nu_\delta \\
&=  \limsup_{L \to \infty} \lim_{N \to \infty}
\int \exp\Big( \dl \min\big(  \|\rho_\eps \ast u_N \|_{\A}^{20}, L\big) \Big) d \nu_{N,\delta}\\
&=  \limsup_{L \to \infty} \lim_{N \to \infty}
Z_{N,\delta}^{-1} \int \exp\Big(\delta  \min\big(\|\rho_\eps \ast u_N \|_{\A}^{20},L\big) -\dl \| u_N \|_{\A}^{20}  - R_N^{\dia}(u)\Big) d \mu (u).
\end{align*}

\noi
Hence, \eqref{pa00a} is reduced to showing that
\begin{align}
\limsup_{L \to \infty} \lim_{N \to \infty}
\E_\mu \Big[ \exp\Big(\delta   \min\big(\|\rho_\eps \ast u_N \|_{\A}^{20},L\big) - \dl\| u_N \|_{\A}^{20}
- R_N^{\dia}(u)\Big) \Big]
= \infty.
\label{pa00}
\end{align}

Let $Y=Y(1)$ be as in \eqref{P2}.
By the Bou\'e-Dupuis variational formula (Lemma \ref{LEM:var3})
with the  change of variables \eqref{YZ13},
we have
\begin{align}
\begin{split}
-& \log \E \Big[ \exp\Big(\delta   \min\big(\|\rho_\eps \ast u_N \|_{\A}^{20},L\big) - \dl \| u_N \|_{\A}^{20} 
- R_N^{\dia}(u)\Big) \Big] \\
&= \inf_{\dot \Ups^N \in \mathbb H_a^1}
\E\bigg[
-\delta   \min\big(\|\rho_\eps \ast (Y_N + \Ups_N +\s \wt \ZZ_N) \|_{\A}^{20},L\big) 
+ \dl \| Y_N + \Ups_N +\s\wt  \ZZ_N \|_{\A}^{20} 
\\
&\hphantom{XXXXXXX}
+ \ft R_N^\dia (Y + \Ups^N +\s\ZZ_N)
+ \frac 12 \int_0^1 \| \dot \Ups^N (t)\|_{H^1_x} ^2 dt \bigg],
\label{DPf}
\end{split}
\end{align}

\noi
where $\ft R_N^\dia$ is as in \eqref{KZ16} with the third power in the last term replaced by the $\g$th power.
With $\Dr_N = \Ups_N + \s \wt \ZZ_N$,
a slight modification of \eqref{KZ3} yields
\begin{align}
\begin{split}
\bigg| & \int_{\T^3} Y_N \Dr_N^2 dx \bigg|
=
\bigg| \int_{\T^3} Y_N ( \Ups_N^2 + 2 \s \Ups_N \wt \ZZ_N + \s^2 \wt \ZZ_N^2) dx \bigg| \\
&\le
C_\s\Big( 1
+ \| Y_N \|_{\C^{-\frac 12-\eps}}^c +  \| \ZZ_N \|_{\C^{1-\eps}}^c\Big)
+ \frac 1{100 |\s|} \Big( \| \Ups_N \|_{L^2}^3
+ \| \Ups_N \|_{H^1}^2 \Big).
\end{split}
\label{E4}
\end{align}

\noi
By Young's inequality, 
we have
\begin{align}
\begin{split}
\bigg|  \int_{\T^3} \Dr_N^3 dx - \int_{\T^3} \Ups_N^3 dx \bigg|
& =
\bigg| \int_{\T^3} \Big( 3 \s \Ups_N^2 \wt\ZZ_N + 3 \s^2 \Ups_N \wt\ZZ_N^2 + \s^3\wt \ZZ_N^3 \Big) dx \bigg| \\
&\le
C_\s \| \ZZ_N \|_{\C^{1-\eps}}^3
+ \frac 1{100 |\s|}  \| \Ups_N \|_{L^2}^3.
\end{split}
\label{E5}
\end{align}

\noi
Then, applying \eqref{E4} and \eqref{E5}
with 
Lemma \ref{LEM:Dr} and 
\eqref{YZ15} to \eqref{DPf}, 
we obtain
\begin{align}
\begin{split}
-& \log \E \Big[ \exp\Big(\delta   \min\big(\|\rho_\eps \ast u \|_{\A}^{20},L\big) - \dl \| u_N \|_{\A}^{20}
- R_N^{\dia}(u)\Big) \Big] \\
&\le \inf_{\dot \Ups^N \in \mathbb H_a^1}
\E \bigg[
-\delta   \min\big(\|\rho_\eps \ast (Y_N + \Ups_N +\s \wt\ZZ_N) \|_{\A}^{20},L\big) 
+ \dl \| Y_N + \Ups_N +\s \wt\ZZ_N \|_{\A}^{20} \\
&\hphantom{XXXXX}
- \frac \s3 \int_{\T^3} \Ups_N^3 dx
+ \| \Ups_N \|_{L^2}^3
+ A \bigg| \int_{\T^3} \Big( :\! Y_N^2 \!: + 2 Y_N \Dr_N + \Dr_N^2 \Big) dx \bigg|^\g
\\
&\hphantom{XXXXX}
+ \frac 3{4} \int_0^1 \| \dot \Ups^N (t)\|_{H^1_x} ^2 dt \bigg]
+C_\s, 
\end{split}
\label{E5a}
\end{align}

\noi
where  $\Dr_N = \Ups_N + \s \wt\ZZ_N$.

In the following, we show that 
the right-hand side of \eqref{E5a} tends to $-\infty$
as $N, L \to \infty$, provided that $|\s| >0$ is sufficiently large.
By following the strategy introduced in our previous works
\cite{OOTol1, OS}, 
we construct a drift $\dot \Ups^N$, achieving this goal.
The main idea is to construct a drift $\dot \Ups^N$
such that $ \Ups^N$
 looks like ``$- Y(1) + $ 
a perturbation'' (see \eqref{paa0}), where the perturbation term is bounded in $L^2(\T^3)$
but has a large cubic integral (see \eqref{E10b} below).
While we do not make use of solitons in this paper, 
one should think of this perturbation as something like a soliton
or a finite blowup solution (at a fixed time)
with a highly concentrated profile.

\begin{remark}\label{REM:diff}\rm
While our construction of the drift
follows that in \cite{OOTol1}, 
we need to proceed more carefully in our current problem
in handling the first two terms under the expectation in \eqref{E5a}.
If we simply apply
 \eqref{YY9} (with $\g = 20$)
 to separate $\Ups_N$ from $Y_N$ and $\s \wt \ZZ_N$, 
 we end up with an expression like
 \[  -\dl   \min\Big(\frac 12 \|\rho_\eps \ast \Ups_N \|_{\A}^{20},L\Big) 
+ 2\dl \|  \Ups_N  \|_{\A}^{20} \]

\noi
such that the coefficients of 
$\|\rho_\eps \ast \Ups_N \|_{\A}^{20}$ and 
$\| \Ups_N \|_{\A}^{20}$ no longer agree, which causes a serious trouble.
We instead need to keep the same coefficient
for  the first two terms under the expectation in \eqref{E5a}
and make use of the difference structure.
Compare this with the analysis in \cite{OOTol1, OS}, 
where no such cancellation was needed.

\end{remark}

\medskip

Fix a parameter $M \gg 1$.
Let $f: \R^3 \to \R$ be a real-valued Schwartz function
such that 
the Fourier transform $\ft f$ is a smooth even non-negative function
supported  $\big\{\frac 12 <  |\xi| \le 1 \}$ such that  $\int_{\R^3} |\ft f (\xi)|^2 d\xi = 1$.
Define a function $f_M$  on $\T^3$ by 
\begin{align}
f_M(x) &:= M^{-\frac 32} \sum_{n \in \Z^3} \ft f\Big( \frac n M\Big) e_n, \label{fMdef} 
\end{align}

\noi
where $\ft f$ denotes the Fourier transform on $\R^3$ defined by 
\begin{align}
\ft f(\xi) = \frac{1}{(2\pi)^\frac{3}{2}} \int_{\R^3} f(x) e^{-in\cdot x} dx.
\notag
\end{align}

\noi
Then, a direct calculation shows the following lemma.

\begin{lemma} \label{LEM:fmb}
For any $M \in \N$ and $\al>0$, we have
\begin{align}
\int_{\T^3} f_M^2 dx &= 1 + O(M^{-\alpha}), \label{fM0} \\
\int_{\T^3} (\jb{\nabla}^{-1} f_M)^2 dx &\les M^{-2}, \label{fm2} \\
\int_{\T^3} |f_M|^3 dx  \sim \int_{\T^3} f_M^3 dx &
\sim M^{\frac 32}.
\label{E10b}
\end{align}
\end{lemma}


\begin{proof}
As for \eqref{fM0} and \eqref{fm2}, 
see the proof of Lemma 5.13 in \cite{OOTol1}. 
From  \eqref{fMdef} and the fact that $\ft f$ is supported on $\{ \frac 12 < |\xi| \le 1 \}$, 
we have 
\begin{align}
\int_{\T^3} f_M^3 dx
= M^{-\frac 92} \sum_{n_1,n_2 \in \Z^3}
\ft f \Big( \frac{n_1} M \Big)
\ft f \Big( \frac{n_2} M \Big)
\ft f \Big( -\frac{n_1+n_2} M \Big)
\sim M^{\frac 32}.
\label{E10c}
\end{align}

\noi
The bound $\| f_M\|_{L^3}^3 \ges M^\frac{3}{2}$
follows from \eqref{E10c}, 
while $\| f_M\|_{L^3}^3 \les M^\frac{3}{2}$
follows from Hausdorff-Young's inequality.
This proves \eqref{E10b}.
\end{proof}

We define $Z_M$ and $\al_M$ by 
\begin{align}
Z_M := \sum_{|n| \le M} \widehat{Y\big(\tfrac 12)} (n) e_n
\qquad \text{and}\qquad 
\al_M := \E \big[ Z_M^2(x) \big]. 
\label{fmb1}
\end{align}

\noi
Note that $\al_M$ is independent of $x \in \T^3$
thanks to 
 the spatial translation invariance of $Z_M$. 
Then, we have the following lemma.
See Lemma 5.14 in \cite{OOTol1} for the proof.

\begin{lemma} \label{LEM:leo2}
Let $ M\gg 1$ and  $ 1 \le p < \infty$.
Then, we have
\begin{align}
&\al_M \sim M,\label{NRZ0}\\
&\E\bigg[\int_{\T^3} |Z_M|^p dx \bigg] \le C(p) M^\frac p2,
\notag\\
&\E\bigg[\Big(\int_{\T^3} Z_M^2 dx - \al_M \Big)^2\bigg]
+\E\bigg[ \Big( \int_{\T^3} Y_N Z_M dx - \int_{\T^3} Z_M^2 dx \Big)^2  \bigg] \les 1, 
\notag\\
&\E\bigg[\Big( \int_{\T^3} Y_N  f_M dx \Big)^2\bigg] 
+ \E\bigg[\Big( \int_{\T^3} Z_M  f_M dx \Big)^2\bigg] \les M^{-2}
\notag
\end{align}

\noi
for any $N \ge M$.

\end{lemma}

We now present the proof of Proposition \ref{PROP:Gibbsref}.

\begin{proof}[Proof of Proposition \ref{PROP:Gibbsref}]

As described above, our main goal is to prove \eqref{pa00}.

Fix $N \in \N$, appearing in \eqref{E5a}.
For $M \gg 1$,
we set $f_M$, $Z_M$, and $\al_M$ as in \eqref{fMdef} and~\eqref{fmb1}.
We now choose  a drift $\dot \Ups^N$ for~\eqref{E5a}
by setting
\begin{align}
\dot \Ups^N (t) = 2 \cdot \ind_{t > \frac 12} \jb{\nabla} \Big( -Z_M + \sgn(\s) \sqrt{\al_M} f_M \Big), 
\label{paax}
\end{align}

\noi
where $\sgn(\s)$ is the sign of $\s \ne 0$.
Then, we have 
\begin{align}
\Ups^N := I(\dot \Ups^N)(1) 
= \int_0^1 \jb{\nb}^{-1} \dot \Ups^N(t) dt
= - Z_M + \sgn(\s) \sqrt{\al_M} f_M.
\label{paa0}
\end{align}

\noi
Note that for $N \ge M \ge 1$, 
we have $\Ups_N = \pi_N \Ups^N = \Ups^N$, 
since $Z_M$ and $f_M$ are supported on the frequencies $\{|n|\le M\}$.

Let us first make some preliminary computations.
We start with the first two terms under the expectation in \eqref{E5a}:
\begin{align}
\begin{split}
- & \delta   \min\big(\|\rho_\eps \ast (Y_N + \Ups_N +\s \wt\ZZ_N) \|_{\A}^{20},L\big) 
+ \dl \| Y_N + \Ups_N +\s\wt \ZZ_N \|_{\A}^{20} \\
& = -\delta   \min\big(\|\rho_\eps \ast (Y_N + \Ups_N +\s\wt \ZZ_N) \|_{\A}^{20}
- \| Y_N + \Ups_N +\s\wt \ZZ_N \|_{\A}^{20} ,\\
& \hphantom{XXXXXk} L-  \| Y_N + \Ups_N +\s \wt\ZZ_N \|_{\A}^{20} \big) \\
& =:  -\dl \min(\1, \II).
\end{split}
\label{XX8}
\end{align}

\noi
We first consider $\II$.
From Lemma \ref{LEM:emb1}, \eqref{interp}, and Lemma \ref{LEM:fmb}, 
we have 
\begin{align}
\| f_M \|_\A
\les \| f_M \|_{H^{-\frac 14}}
\les \| f_M \|_{L^2}^{\frac 34} \| f_M \|_{H^{-1}}^{\frac 14}
\les M^{-\frac 14}.
\label{pa001}
\end{align}

\noi
From \eqref{paa0}, \eqref{NRZ0} in  Lemma \ref{LEM:leo2},
and  \eqref{pa001}, we have
\begin{align}
\begin{split}
\II & \geq  L-  2 \al_M^{10} \|  f_M \|_{\A}^{20} 
- C\Big( \| Y_N\|_{\A}^{20}  + \|Z_M\|_{\A}^{20} + |\s| \| \ZZ_N \|_{\A}^{20} \Big)\\
 & \geq  L-  C_0 M^5 
- C\Big( \| Y_N\|_{\A}^{20}  + \|Z_M\|_{\A}^{20} + |\s| \| \ZZ_N \|_{\A}^{20} \Big)\\
 & \geq  \frac 12 L
- C\Big( \| Y_N\|_{\A}^{20}  + \|Z_M\|_{\A}^{20} + |\s| \| \ZZ_N \|_{\A}^{20} \Big)
\end{split}
\label{pa001a}
\end{align}

\noi
for $L \gg M^5$.
Note that the second term on the right-hand side is harmless
since it is  bounded under an expectation.
Next, we turn to $\1$ in \eqref{XX8}.
Let $\dl_0$ denote the Dirac delta on $\T^3$. 
Then,  by applying \eqref{paa0},  Young's inequality, 
 Lemma \ref{LEM:emb1}, 
 \eqref{NRZ0}, and \eqref{fM0} in Lemma~\ref{LEM:fmb}
 and by choosing $\eps = \eps(M) > 0$ sufficiently small, 
 we have 
\begin{align}
\begin{split}
\1 & \ge - \Big| \|\rho_\eps \ast (Y_N + \Ups_N +\s \wt\ZZ_N) \|_{\A}^{20}
- \| Y_N + \Ups_N +\s \wt\ZZ_N \|_{\A}^{20} \Big|\\
& \ge - C \| (\rho_\eps - \dl_0)  \ast (Y_N + \Ups_N +\s\wt \ZZ_N) \|_{\A}
 \| Y_N + \Ups_N +\s\wt \ZZ_N \|_{\A}^{19} \\
& \ge - C \al_M^{10} \| (\rho_\eps - \dl_0)  \ast f_M \|_{H^{-\frac 14}}^{20}
- C\Big( \| Y_N\|_{\A}^{20}  + \|Z_M\|_{\A}^{20} + |\s| \| \ZZ_N \|_{\A}^{20} \Big)\\
& \ge - C \eps^5  M^{10} 
- C\Big( \| Y_N\|_{\A}^{20}  + \|Z_M\|_{\A}^{20} + |\s| \| \ZZ_N \|_{\A}^{20} \Big)\\
& =  - C_0
- C\Big( \| Y_N\|_{\A}^{20}  + \|Z_M\|_{\A}^{20} + |\s| \| \ZZ_N \|_{\A}^{20} \Big).
\end{split}
\label{pa001b}
\end{align}

\noi
Therefore, from \eqref{XX8}, \eqref{pa001a}, and \eqref{pa001b}
together with \eqref{fmb1}, Lemma \ref{LEM:emb2} and \eqref{YZ15}, we obtain 
\begin{align}
\E\Big[-\dl \min(\1, \II)\Big] \le C(\dl, \s).
\label{pa002}
\end{align}

Next, we treat the third term under the expectation in \eqref{E5a}.
This term gives the main contribution.
From 
\eqref{paa0} and Young's inequality with  Lemma \ref{LEM:fmb}, 
we have 
\begin{align}
\begin{split}
& \s \int_{\T^3} \Ups_N^3 dx - |\s| \al_M^{\frac 32} \int_{\T^3} f_M^3 dx \\
&= -\s \int_{\T^3} Z_M^3 dx
+3 |\s| \int_{\T^3} Z_M^2 \sqrt{\al_M} f_M dx
- 3\s \int_{\T^3} Z_M \al_M f_M^2 dx \\
&\ge - \eta |\s| \al_M^{\frac 32} \int_{\T^3} f_M^3 dx
- C_{\eta} |\s| \int_{\T^3} |Z_M|^3 dx
\end{split}
\label{pa1}
\end{align}
for any $0<\eta<1$.
Then,
it follows from 
 \eqref{pa1} with $\eta = \frac 12$ and Lemmas \ref{LEM:fmb} and \ref{LEM:leo2}
 that 
\begin{align}
\begin{split}
\E \bigg[ \s \int_{\T^3} \Ups_N^3 dx \bigg]
&\ge (1-\eta) |\s| \al_M^{\frac 32} \int_{\T^3} f_M^3 dx
- C_\eta |\s| \E \bigg[ \int_{\T^3} |Z_M|^3 dx \bigg] \\
&\ges |\s| M^3 - |\s| M^{\frac 32} \\
&\ges |\s| M^3
\end{split}
\label{paa2}
\end{align}
for $M \gg 1$.

We now treat the fourth and sixth terms  under the expectation in \eqref{E5a}.
From \eqref{paa0}, we have $\Ups_N \in \H_{\le 1}$.
Then, by the Wiener chaos estimate (Lemma \ref{LEM:hyp})
and \eqref{paa0} with Lemmas~\ref{LEM:fmb} and \ref{LEM:leo2}, we have 
\begin{align}
\E \Big[ \| \Ups_N \|_{L^2}^3 \Big]
\les \E \Big[ \| \Ups_N \|_{L^2}^2 \Big]^{\frac 32}
\les M^{\frac 32}.
\label{p2b}
\end{align}

\noi
Recall that  both $\ft Z_M$ and $\ft f_M$ are supported on $\{|n|\leq M\}$. 
Then, 
from  \eqref{paax},  \eqref{paa0}, 
and Lemmas \ref{LEM:fmb} and \ref{LEM:leo2} as above,  we have 
\begin{align}
\E \bigg[ \int_0^1 \| \dot \Ups^N (t)\|_{H^1_x}^2  dt \bigg]
\les M^2 \E \Big[ \| \Ups^N \|_{L^2}^2 \Big]
\les M^3. \label{pa4}
\end{align}

We state a lemma which controls
the fifth  term  under the expectation in \eqref{E5a}.
We present the proof of this lemma at the end of this subsection.

\begin{lemma}\label{LEM:L2}
Let $\g > 0$.  Then, we have 
\begin{align}
\E \bigg[ \Big|\int_{\T^3} :\! (Y_N + \Ups_N +\s \wt \ZZ_N)^2 \!: dx \Big|^\g \bigg]  
\le C(\s, \g) < \infty, 
\label{p2a}
\end{align}

\noi
uniformly in $N \ge M\ge 1$.\footnote{Recall from \eqref{paa0} that the definition of 
$\Ups_N$ depends on $M$.}

\end{lemma}

Therefore, putting 
\eqref{E5a}, \eqref{pa002}, \eqref{paa2}, 
\eqref{p2b},  \eqref{pa4}, and 
Lemma \ref{LEM:L2}  together, 
we obtain 
\begin{align}
\begin{split}
-& \log \E \Big[ \exp\Big(\delta   \min\big(\|\rho_\eps \ast u \|_{\A}^{20},L\big) - \dl \| u_N \|_{\A}^{20}
- R_N^{\dia}(u)\Big) \Big] \\
&\le
-  C_1 |\s| M^3 + C_2 M^3 +C(\dl, \s, \g)
\end{split}
\label{E14}
\end{align}

\noi
for some  $C_1, C_2 > 0 $, 
provided that $L \gg M^5 \gg 1$ and $\eps = \eps(M) > 0$ sufficiently small.
By taking the limits in $N$ and $L$, we 
 conclude from \eqref{E14} that
\begin{align}
\limsup_{L \to \infty} & \lim_{N \to \infty}
\E_\mu \Big[ \exp\Big(\delta   \min\big(\|\rho_\eps \ast u_N \|_{\A}^{20},L\big) - \dl\| u_N \|_{\A}^{20}
- R_N^{\dia}(u)\Big) \Big]
 \notag\\
&
\ge \exp \Big(  C_1 |\s| M^3 - C_2 M^3  - C_0(\s) \Big)
\too \infty, 
\notag
\end{align}
as $M \to \infty$,
provided that
$|\s|$ is sufficiently large.
This proves \eqref{pa00}
and thus we 
 conclude the proof of Proposition \ref{PROP:Gibbsref}.
\end{proof}

We conclude this subsection by presenting the proof of Lemma \ref{LEM:L2}.

\begin{proof}[Proof of Lemma \ref{LEM:L2}]

From  \eqref{YZ12} and \eqref{K9b},
we have
\begin{align}
\begin{split}
\int_{\T^3} \Ups_N \wt \ZZ_N dx
& = \int_0^1 \int_{\T^3} \jb{\nb}^{-\frac 34} \Ups_N \cdot \jb{\nb}^{-\frac 54}
\pi_N^2( : \! Y_N^2(t) \!: ) dxdt\\
& \le
\| \Ups_N \|_{H^{-\frac 34}} \int_0^1 \| : \! Y_N^2(t) \!: \|_{H^{-\frac 54}} dt.
\end{split}
\label{Pabb3}
\end{align}

\noi
As for the first factor, 
it follows from \eqref{paa0}, \eqref{interp}, \eqref{NRZ0}, and Lemma \ref{LEM:fmb} that
\begin{align}
\begin{split}
\| \Ups_N \|_{H^{-\frac 34}}
&\les \| Z_M \|_{H^{-\frac 34}}
+ \sqrt{\al_M} \| f_M \|_{H^{-\frac 34}} \\
&\les \| Z_M \|_{H^{-\frac 34}}
+ \sqrt{\al_M} \| f_M \|_{H^{-1}}^{\frac 34} \| f_M \|_{L^2}^{\frac 14} \\
&\les \| Z_M \|_{H^{-\frac 34}}
+ M^{-\frac 14}.
\label{Pabb4}
\end{split}
\end{align}
Hence, from \eqref{Pabb3}, \eqref{Pabb4}, \eqref{fmb1}, and Lemma \ref{LEM:Dr},
we obtain
\begin{align}
\E \bigg[ \Big| \int_{\T^3} \Ups_N  \wt \ZZ_N dx \Big|^2 \bigg]
\les \E \Big[ \| \Ups_N \|_{H^{-\frac 34}}^2 \Big]
+ \E \Big[\| : \! Y_N^2 (t)\!: \|_{L^1_t([0, 1]; H^{-\frac 54}_x)}^2 \Big]
\les 1.
\label{Pabb5}
\end{align}

From \eqref{paa0}, we have 
\begin{align}
\begin{split}
\Ups_N^2 & + 2 Y_N \Ups^N\\
&=
Z_M^2 -2 \sgn(\s)  \sqrt{\al_M} Z_M f_M + \al_M f_M^2\\
& \quad   -2Y_N Z_M + 2\sgn(\s)  \sqrt{\al_M} Y_N f_M \\
&= (Z_M^2-\al_M) -2 \sgn(\s)  \sqrt{\al_M} Z_M f_M + \al_M (-1 + f_M^2)
+ 2 \al_M\\
&\quad  -2(Y_N Z_M - Z_M^2) - 2(Z_M^2 - \al_M) -2 \al_M + 
2\sgn(\s)  \sqrt{\al_M} Y_N f_M\\
& = 
- (Z_M^2-\al_M) -2 \sgn(\s)  \sqrt{\al_M} Z_M f_M + \al_M (-1 + f_M^2)\\
&\quad  -2(Y_N Z_M - Z_M^2) + 
2\sgn(\s)  \sqrt{\al_M} Y_N f_M.
\end{split}
\label{pabb2}
\end{align}

\noi
Note from \eqref{YZ12} and \eqref{paa0} that 
$\int_{\T^3} :\! (Y_N + \Ups_N +\s \wt \ZZ_N)^2 \!: dx  \in \H_{\le 4}$.
Then, 
from the Wiener chaos estimate (Lemma \ref{LEM:hyp}), \eqref{paa0}, \eqref{Pabb5}, \eqref{pabb2},
and 
Lemmas \ref{LEM:Dr} and \ref{LEM:leo2} with \eqref{fM0}, 
we have 
\begin{align*}
 \E & \bigg[ \Big|\int_{\T^3} :\! (Y_N + \Ups_N +\s \wt \ZZ_N)^2 \!: dx \Big|^\g \bigg]  \\
& \le C(\g) \Bigg\{\E \bigg[ \Big|\int_{\T^3} :\! (Y_N + \Ups_N +\s \wt \ZZ_N)^2 \!: dx \Big|^2 \bigg] 
\Bigg\}^\frac \g2 \\
&=  C(\g) \Bigg\{\E \bigg[ \Big|\int_{\T^3} :{Y_N^2}: dx
+ \int_{\T^3} (\Ups_N^2 + 2Y_N \Ups_N) dx
+\s^2 \int_{\T^3}\wt  \ZZ_N^2 dx \\
&\hphantom{XXXXXX}
+ 2 \s \int_{\T^3} \Ups_N \wt \ZZ_N dx
+ 2 \s \int_{\T^3}Y_N \wt  \ZZ_N  dx
\Big|^2 \bigg] \Bigg\}^\frac \g2 
\\
&\le  C(\g) \Bigg\{ \E \bigg[ \Big(\int_{\T^3} :{Y_N^2}: dx \Big)^2 \bigg]
+ \s^4 \E \bigg[ \Big( \int_{\T^3} \wt \ZZ_N^2 dx \Big)^2 \bigg] \\
&\hphantom{XXXX}
+ \s^2 \E \bigg[ \Big( \int_{\T^3} \Ups_N \wt \ZZ_N dx \Big)^2 \bigg]
+ \s^2 \E \bigg[ \Big( \int_{\T^3} Y_N\wt  \ZZ_N dx \Big)^2 \bigg] \\
&\hphantom{XXXX}
+ \E \bigg[ \Big( -\int_{\T^3} Y_N Z_M dx + \int_{\T^3} Z_M^2 dx \Big)^2 \bigg] \\
&\hphantom{XXXX}
 + \E \bigg[ \Big( \int_{\T^3} Z_M^2 dx - \al_M \Big)^2 \bigg] 
+ \al_M^2 \bigg( - 1 + \int_{\T^3} f_M^2 dx \bigg)^2 \\
&\hphantom{XXXX}
 + \al_M  \E \bigg[ \Big( \int_{\T^3} Y_N f_M dx \Big)^2 \bigg]
 + \al_M \E \bigg[ \Big( \int_{\T^3} Z_Mf_M dx \Big)^2 \bigg]  \Bigg\}^\frac \g2 \\
&\le C(\s, \g), 
\end{align*}

\noi
which yields the bound \eqref{p2a}.
\end{proof}

\subsection{Non-convergence of the truncated $\Phi^3_3$-measures}
\label{SUBSEC:nonconv}

In this subsection, we present the proof of 
Proposition \ref{PROP:nonconv}
on non-convergence of the truncated $\Phi^3_3$-measures $\{\rho_N\}_{N \in \N}$.

We first define a slightly different  tamed version of the truncated $\Phi^3_3$-measure
by setting
\begin{align}
d \nu_{\delta}^{(N)} (u)
=(Z_{\delta}^{(N)})^{-1} \exp \Big( -\delta \| u \|_{\A}^{20} - R_N^{\dia}(u) \Big) d \mu (u)
\label{rho3}
\end{align}
for $N \in \N$ and $\dl>0$,
where the $\A$-norm and $R_N^\dia$ are as in \eqref{RM-2} and \eqref{K1r},
respectively,  and
\begin{align*}
Z_{\delta}^{(N)}
= \int \exp \Big( -\delta \| u \|_{\A}^{20} - R_N^{\dia}(u) \Big) d \mu (u).
\end{align*}

\noi
As compared to  $\nu_{N, \dl}$ in \eqref{RM-1}, 
there is no frequency cutoff $\pi_N$ in  
the taming $-\delta \| u \|_{\A}^{20}$ in \eqref{rho3}. 
As a corollary to the proof of Proposition \ref{PROP:ref}, 
we obtain the following convergence result
for~$\nu_\dl^{(N)}$.

\begin{lemma}\label{LEM:conv1}
Let  $\dl > 0$, 
Then,  as measures on $\C^{-100}(\T^3)$, 
the sequence of measures $\{ \nu_{\delta}^{(N)} \}_{N \in \N}$ defined in \eqref{rho3} converges weakly to the limiting measure 
 $\nu_\dl$
constructed in Proposition \ref{PROP:ref}.

\end{lemma}

\begin{proof}

By the definitions \eqref{RM-1} and \eqref{rho3} of $\nu_{N, \dl}$ and $\nu_\dl^{(N)}$, it suffices to prove
\begin{align*}
\lim_{N \to \infty}
\bigg\{& \int F(u) \exp \Big( -\delta \| u \|_{\A}^{20}  - R_N^{\dia}(u) \Big) d \mu (u)\\
&  - \int F(u)\exp \Big(  -\delta \| u_N \|_{\A}^{20} - R_N^{\dia}(u) \Big) d \mu (u)
\bigg\} = 0
\end{align*}

\noi
for any bounded continuous function 
 $F: \C^{-100}(\T^3) \to \R$.
In the following, we prove 
\begin{align}
\lim_{N\to \infty} \int \Big| \exp \Big( -\delta \| u \|_{\A}^{20} - R_N^{\dia}(u) \Big)  - \exp \Big( -\delta \| u_N \|_{\A}^{20} - R_N^{\dia}(u) \Big)\Big| d \mu (u) = 0.
 \label{rho4} 
\end{align}

By the uniform boundedness of the frequency projector $\pi_N$ on $\A$, 
we have  
\begin{align}
\| u_N \|_\A \les \| u \|_\A,
\label{rho5} 
\end{align}

\noi
uniformly in $N \in \N$.
Then, it follows from the mean-value theorem, \eqref{rho5}, 
and the Schauder estimate \eqref{Sch}
that there exists $c_0 > 0$ such that 
\begin{align}
\begin{split}
\int &  \Big| \exp \Big( -\delta \| u \|_{\A}^{20} - R_N^{\dia}(u) \Big)  - \exp \Big( -\delta \| u_N \|_{\A}^{20} - R_N^{\dia}(u) \Big)\Big| \, d \mu (u)\\
& \les  \dl \int \exp \Big( -\delta \min\big(\| u \|_{\A}^{20}, \| u_N\|_{\A}^{20}\big) - R_N^{\dia}(u) \Big) 
\big|\|u\|_{\A}^{20}-\|u_N\|_{\A}^{20}\big| \, d \mu (u)\\
& \les \dl  \int \exp \Big( -\delta c_0 \| u_N\|_{\A}^{20} - R_N^{\dia}(u) \Big) \|u-u_N\|_{\A}\|u\|_{\A}^{19} d \mu (u)\\
& \les  \dl \int \exp \Big( -\delta c_0 \| u_N\|_{\A}^{20} - R_N^{\dia}(u) \Big) N^{-\frac 18} \|u\|_{W^{-\frac 58, 3}}^{20} d \mu (u).
\end{split}
\label{rho6a}
\end{align}

\noi
In the last step, we used the following bound:
\[ \| u - u_N \|_{\A} \les \|\pi_N^\perp u\|_{W^{-\frac 34, 3}} \les N^{-\frac 18} \| u\|_{W^{-\frac 58, 3}}, \]

\noi
which follows from 
\eqref{RM-2}, \eqref{Sch}, and the fact that 
$\pi_N^\perp u = u - u_N$ has the frequency support $\{|n|\ges N\}$.
Therefore, by \eqref{RM-1}, Proposition \ref{PROP:ref}, and \eqref{ref3}, we obtain
\begin{align*}
\limsup_{N\to \infty}&  \int \Big| \exp \Big( -\delta \| u \|_{\A}^{20} - R_N^{\dia}(u) \Big)  - \exp \Big( -\delta \| u_N \|_{\A}^{20} - R_N^{\dia}(u) \Big)\Big| d \mu (u) \\
&\les \dl\lim_{N\to \infty}
 \int \exp \Big( -\delta c_0 \| u_N\|_{\A}^{20} - R_N^{\dia}(u) \Big) N^{-\frac 18} \|u\|_{W^{-\frac 58, 3}}^{20} d \mu (u) \\
&= \dl \lim_{N\to \infty} N^{-\frac 18} Z_{N, c_0\dl} \int \|u\|_{W^{-\frac 58, 3}}^{20}  d \nu_{N, c_0\dl}\\
&= 0.
\end{align*}

\noi
This proves \eqref{rho4}.
\end{proof}

\begin{remark}\label{REM:non1}\rm
In the penultimate step of 
\eqref{rho6a}, we used the boundedness of the cube frequency projector $\pi_N = \pi_N^\text{cube}$
on $L^3(\T^3)$ and hence
this argument 
does not work 
for the ball frequency projector $\pi_N^\text{ball}$
defined in \eqref{pib1}.

\end{remark}

We conclude this section by presenting the proof of Proposition \ref{PROP:nonconv}.

\begin{proof}[Proof of Proposition \ref{PROP:nonconv}]
Suppose by contradiction that, as probability measures on  $\A$, 
 $\{\rho_{N_k}\}_{k \in \N}$ 
has a weak limit $\nu_0$. 
Then, given any 
 $\delta > 0$,  from  Lemma \ref{LEM:conv1}
 with \eqref{rho3} and  \eqref{GibbsN}, we have 
\begin{align}
\begin{split}
d\nu_\delta 
&= \lim_{k \to \infty} \frac{\exp \Big( -\delta \| u \|_{\A}^{20} - R_{N_k}^{\dia}(u) \Big)}
{\int \exp \Big( -\delta \| v \|_{\A}^{20} - R_{N_k}^{\dia}(v) \Big) d \mu (v)} d \mu (u) \\
&=\lim_{k\to \infty} \frac{\exp \big( -\delta \| u \|_{\A}^{20} \big) }
{\int \exp \big( -\delta \| v \|_{\A}^{20} \big) d \rho_{N_k} (v)} d \rho_{N_k} (u) \\
&= \frac{\exp \big( -\delta \| u \|_{\A}^{20} \big)}{\int \exp \big( -\delta \| v \|_{\A}^{20} \big) d \nu_0 (v)} d \nu_0 (u),
\end{split}
\label{rho7}
\end{align}

\noi
where the limits are interpreted as weak limits of measures on $\C^{-100}(\T^3)$. 
Note that, in the last step,  we used the weak convergence in $\A$ of the 
truncated $\Phi^3_3$-measures $\rho_{N_k}$, since $\exp ( -\delta \| u \|_{\A}^{20} )$ is continuous on $\A$, but not on $\C^{-100}(\T^3)$.
Therefore, from \eqref{rho7} and \eqref{ref1a}, we obtain
\begin{align}
 d\nu_0(u)
= \bigg(\int \exp \big( -\delta \| v \|_{\A}^{20} \big) d \nu_0 (v)\bigg) 
d\cj \rho_\dl (u).
\label{rho8}
\end{align}

\noi
By assumption, $\nu_0$ is a probability measure on $\A$
and thus $\| u \|_{\A} < \infty$, $\nu_0$-almost surely.
By the fact that  $\nu_0$ is a probability measure, \eqref{rho8}, 
 and Proposition \ref{PROP:Gibbsref}, 
 we obtain
\begin{align*}
1 &= \int 1 \, d \nu_0\\
&= \int \exp \Big( -\delta \| u \|_{\A}^{20} \Big) d \nu_0 (u) 
\int 1 \,  d \cj \rho_\dl (u)\\
&= \infty,
\end{align*}

\noi
which yields a 
contradiction.
Therefore, 
no subsequence of the truncated $\Phi^3_3$-measures $\rho_N$ has a weak limit
 as probability measures on  $\A$.
\end{proof}

\section{Local well-posedness}
\label{SEC:LWP}

In this section,
we present the proof of Theorem \ref{THM:LWP0} on 
  local well-posedness
of the (renormalized) hyperbolic $\Phi^3_3$-model \eqref{SNLW1}:
\begin{align}
\dt^2 u + \dt u + (1 -  \Dl)  u - \s :\! u^2 \!: + M (\,:\! u^2 \!:\,) u = \sqrt{2} \xi,
\label{SNLWA2B}
\end{align}

\noi
where $M$ is defined as in \eqref{addM}.
For the local theory, the size of $\s \ne 0$ does not play any role
and hence we set $\s = 1$ in the remaining part of this section.
As mentioned in Section \ref{SEC:1}, 
local well-posedness of \eqref{SNLWA2B}
follows from a slight modification of the argument in \cite{GKO2, OOTol1}.
We, however, point out that the argument in \cite{GKO2} 
on the quadratic SNLW alone
is not sufficient due to the additional term 
$M (\,:\! u^2 \!:\,) u$, coming from the taming in constructing
the $\Phi^3_3$-measure.

\subsection{Paracontrolled approach}

In this subsection, we 
go over a paracontrolled approach to rewrite the equation \eqref{SNLWA2B}
into a system of three unknowns.
While our presentation closely follows those in \cite{GKO2, OOTol1}, 
we present some details for readers' convenience.
Proceeding 
 in the spirit of 
\cite{CC, MW1, GKO2, OOTol1}, 
we 
transform the quadratic SdNLW \eqref{SNLWA2B} to a system of PDEs.
In order to treat the additional term $M (\,:\! u^2 \!:\,) u$ in \eqref{SNLWA2B},
which contains an ill-defined product in $:\! u^2 \!:\,$, 
we follow the approach
 in our previous work \cite{OOTol1}
on the focusing Hartree $\Phi^4_3$-model,
which leads to 
 the system of three equations; see \eqref{SNLW6} below.
Compare this with  \cite{CC, MW1, GKO2}, 
where the resulting  systems consist of two equations.
At the end of this subsection, 
we state a local well-posedness result
 of the resulting system.

The main difficulty in studying the hyperbolic $\Phi^3_3$-model \eqref{SNLWA2B}
comes from the roughness of the space-time white noise.
This is already manifested at the level of the linear equation.
Let $\Psi$  denote the stochastic convolution, satisfying
 the following linear stochastic damped wave equation:
\begin{align}
\notag
\begin{cases}
\dt^2 \Psi + \dt\Psi +(1-\Dl)\Psi  = \sqrt{2}\xi\\
(\Psi,\dt\Psi)|_{t=0}=(\phi_0,\phi_1),
\end{cases}
\end{align}

\noi
where  $(\phi_0, \phi_1) = (\phi_0^\o, \phi_1^\o)$ 
is a pair of 
 the Gaussian random distributions with 
$\Law (\phi_0^\o, \phi_1^\o) = \muu = \mu \otimes \mu_0$ in \eqref{gauss1}.
Define the linear damped wave propagator $\D(t)$ by 
\begin{equation} \notag 
\D(t) = e^{-\frac t2 }\frac{\sin\Big(t\sqrt{\tfrac34-\Dl}\Big)}{\sqrt{\tfrac34-\Dl}}
\label{D1}
\end{equation}

\noi
viewed as a Fourier multiplier operator.
By setting
\begin{align}
\jbb{n} = \sqrt{\frac 34+|n|^2},
\label{jbbn}
\end{align}

\noi
we have
\begin{align}
\D(t) f  =  e^{-\frac t2}\sum_{n\in \Z^3} 
\frac{\sin (t \jbb{n})}{\jbb{n}}
 \ft f (n) e_n.
\label{W3}
\end{align}

\noi
Then, the stochastic convolution $\Psi$ can be expressed as 
\begin{equation}
\Psi (t) =  S(t)(\phi_0, \phi_1) + \sqrt 2 \int_0^t\D(t-t')dW(t'), 
\label{W2}
\end{equation}

\noi
where  $S(t)$ is defined by 
\begin{align}
S(t) (f, g) = \dt\D(t)f +  \D(t) (f + g)
\label{St0}
\end{align}

\noi
and 
$W$ denotes a cylindrical Wiener process on $L^2(\T^3)$ defined in \eqref{W1}.
It is easy to see that $\Psi$ almost surely lies in 
$C (\R_+;W^{-\frac{1}{2}-\eps, \infty}(\T^3))$
for any $\eps > 0$; 
see Lemma \ref{LEM:stoconv0} below.
In the following, we use $\eps > 0$ to denote
a small positive constant, which can be arbitrarily small.

In the following,  we adopt Hairer's convention to denote the stochastic terms by trees;
the vertex ``\,$\<dot>$\,''  corresponds to the space-time white noise $\xi$,
while the edge denotes the Duhamel integral operator 
$\I$ given by 
\begin{align}
\I (F) (t) = \int_0^t \D(t - t') F(t') dt'
= \int_0^t e^{-\frac {t-t'}2 }\frac{\sin\Big((t-t')\sqrt{\tfrac34-\Dl}\Big)}{\sqrt{\tfrac34-\Dl}}
F(t') dt'.
\label{lin1}
\end{align}

\noi
With a slight abuse of notation,
we set  
\begin{align}
 \<1> : = 
 \Psi,  
\label{W2a}
\end{align}

\noi
where $\Psi$ is as in \eqref{W2}, 
with the understanding that $\<1>$ in \eqref{W2a} includes the 
random linear solution $S(t)(\phi_0, \phi_1)$.
As mentioned above, $\<1>$
has (spatial) regularity\footnote{We only discuss spatial regularities of various stochastic objects
in this part. Hereafter,   we use $a-$ 
 to denote $a- \eps$ 
for arbitrarily small $\eps > 0$.}
$- \frac 12 -$.

Given $N \in \N$,
we  define the truncated stochastic terms  $\<1>_N$ 
and   $\<20>_N$ by 
\begin{align}
\<1>_N := \pi_N \<1>
\qquad \text{and}
\qquad 
  \<20>_N :=  \I (\<2>_N)
    =  \int_{0}^t \D(t-t') \<2>_N(t') dt', 
\label{so4a}
\end{align}

\noi
where
$\pi_N$ is the frequency projector defined in \eqref{pi} and
$\<2>_N$ is the Wick power
defined by 
\begin{align}
\<2>_N   := \<1>_N^2 - \s_N
\label{so4b}
\end{align}

\noi
 with 
\begin{align}
\s_N &  = \E \big[ \<1>_N^2(x, t)\big]
 =  \sum_{n\in \Z^3} \frac{\chi^2_N(n)}{ \jb{n}^2}  \sim  N
\too \infty, 
\label{sigma1a}
\end{align}

\noi
as $N \to \infty$.
Note that $\s_N$ in \eqref{sigma1a} is independent\footnote{This comes
from the space-time translation invariance of the truncated stochastic convolution $\<1>_N$.} of $(x, t) \in \T^3 \times \R_+$
and  agrees with $\s_N$ defined in \eqref{sigma1}.
Note that we have 
$\<2> = \lim_{N \to \infty} \<2>_N$ in $C([0,T];W^{-1 - ,\infty}(\T^3))$ almost surely.
See Lemma \ref{LEM:stoconv0}.

Next, we define the second order stochastic term $\<20>$:
\begin{align}
  \<20> := 
   \I (\<2>)
  =  \int_{0}^t \D(t-t') \<2>(t') dt', 
\notag
\end{align}

\noi
as a limit of $\<20>_N$ defined in \eqref{so4a}.
With a naive regularity counting, with one degree of smoothing
from  the damped wave Duhamel integral operator $\I$ in \eqref{lin1},  one may expect that 
$\<20>$ has regularity $0- = 2(-\frac 12 -)+1$.
However,
by exploiting the multilinear dispersive smoothing effect, 
Gubinelli, Koch, and the first author showed that 
there is an extra $\frac 12$-smoothing
for $\<20>$ and that $\<20>$ has regularity $\frac 12-$.
See Lemma \ref{LEM:sto1} below.
See also \cite{OOcomp, Bring2, OWZ}
for analogous multilinear dispersive smoothing 
for the random wave equations.
In particular, see \cite{Bring2, OWZ}, 
where multilinear smoothing has been studied extensively
for higher order stochastic objects in the cubic case.

If we proceed with the second order expansion as in \cite{GKO2}:
\begin{align*}
u = \<1> + \<20> + v, 
\end{align*}

\noi
the residual term $v$ satisfies the equation of the form:
\begin{align*}
(\dt^2 + \dt  +1 -  \Dl)  v 
& =    2 v \<1> + 2 \<1>\<20> + \text{other terms}.
\end{align*}

\noi
Inheriting the worse regularity $-\frac 12-$ of $\<1>$, 
the second  term  $\<1>\<20>$ 
has regularity $-\frac 12 -$.
Hence, we expect $v$ to have regularity at most $\frac 12- = (-\frac 12-) + 1$.
In particular, the product 
$v \<1>$ is not well defined since $(\frac 12-) + (-\frac 12-) < 0$.

In order to overcome this problem, 
we  now introduce a paracontrolled ansatz
as in \cite{MW1, GKO2}:
\begin{align}
u = \<1> + \<20> + X+ Y,
\label{decomp3a}
\end{align}

\noi
where $X$ and $Y$ satisfy
\begin{align}
\begin{split}
 (\dt^2 +\dt +1 - \Dl) X  & =  2 (X+Y+\<20>)\pl \<1>  - M( \,:\! u^2 \!:\, ) \<1>,
\end{split}
\label{SNLW5}\\
\begin{split}
 (\dt^2 +\dt +1  - \Dl) Y 
&  = (X+Y+\<20>)^2 +  2(X+Y+\<20>)\pge \<1> \\
& \quad   - M( \,:\! u^2 \!:\, )(X+Y + \<20>)
\end{split}
\label{SNLW5y}
\end{align}

\noi
with the understanding that
\begin{align}
:\! u^2 \!:\,  = (X+Y+ \<20>)^2 + 2 (X+Y) \<1> + 2 \<1> \<20> + \<2> .
\label{u2}
\end{align}

\noi
Here, $\pge = \pg + \pe$.
Note that, in the $X$-equation \eqref{SNLW5}, we collected the worst terms from the $v$-equation, 
while all the terms in the $Y$-equation \eqref{SNLW5y} are expected to behave better
(that is, if the resonant product in \eqref{SNLW5y} can be given a meaning).
We point out that the problematic term $M( \,:\! u^2 \!:\, )$
appears in {\it both} equations, unlike the situation in \cite{GKO2}.

There are two resonant products in the system \eqref{SNLW5} - \eqref{SNLW5y}, 
which do not a priori  make sense:
$\<20>\pe \<1>$ and $X\pe \<1>$.
We can use stochastic analysis and multilinear harmonic analysis to 
give a meaning to 
 the first resonant product:
\begin{align*}
\<21p> := \<20>\pe \<1>
\end{align*}

\noi
as a distribution of regularity $0- = (\frac{1}{2}-) + (-\frac 12 -)$
(without renormalization).
See Lemma~\ref{LEM:sto2} below.
This in particular says that $Y$ has expected regularity $1-$.

In view of Lemma \ref{LEM:para}, 
the right-hand side of \eqref{SNLW5} has regularity $-\frac 12-$
(if we pretend that $M( \,:\! u^2 \!:\, )$ makes sense), 
and thus 
we expect that 
$X$ has regularity $\frac 12-$.
In particular, 
the resonant product    $X \pe \<1>$ in the $Y$-equation is not well defined since 
 the sum of the regularities is negative.
In \cite{GKO2}, this issue was overcome
by substituting the Duhamel formulation of the $X$-equation
into the resonant product
$X \pe \<1>$ and then introducing certain paracontrolled operators
(see \eqref{X2}, \eqref{X3}, and \eqref{X6} below).
This was possible in~\cite{GKO2} since
there was no additional term $ M( \,:\! u^2 \!:\, )$ in the system, in particular
in the $X$-equation.
In our current problem, 
the problematic resonant product 
$X\pe \<1>$ also appears in $ M( \,:\! u^2 \!:\, )$, in particular, in the $X$-equation.
Thus, a strategy in \cite{MW1, GKO2} of substituting
the Duhamel formulation of the $X$-equation into $X\pe \<1>$ would lead
to an infinite iteration of such substitutions.
We point out that such an infinite
iteration of the Duhamel formulation works in certain situations
but we choose an alternative approach which is simpler.

The main idea is to follow the strategy in our previous work \cite{OOTol1}
and introduce a new unknown, representing the problematic resonant product:
\begin{align}
\text{``}\,\Res = X \pe \<1>\,\text{''}
\label{Res1}
\end{align}

\noi 
which leads to a system of three
unknowns $(X, Y, \Res)$.

We now turn our attention to $:\!u^2\!:$ in \eqref{u2}.
Let $\Qxy$ to 
denote
 a good part of $:\! u^2 \!:$ defined by 
\begin{align}
\Qxy =
(X+Y)^2 + 2 (X+Y) \<20> +2 X \pl \<1> + 2 X \pg \<1> + 2 Y \<1>.
\label{Pxy}
\end{align}
In view of $X \pl \<1>$ and $ Y \<1>$, 
$\Qxy$ has (expected) regularity 
 $ - \frac 12-$
From \eqref{decomp3a}, \eqref{Res1}, and \eqref{Pxy}, 
we can write $:\!u^2\!:$ as
\begin{align}
:\! u^2 \!: \,
= \Qxy + 2 \Res + \<20>^2 + 2 \<21> + \<2>, 
\label{u3}
\end{align}

\noi 
where $\<21>$ denotes the product of  $\<20>$ and  $\<1>$ given by 
\[
\<21>
=  \<20> \pl \<1> +  \<21p> 
+  \<20> \pg \<1>.
\]

\noi
By substituting the Duhamel formulation of
the $X$-equation \eqref{SNLW5} and \eqref{u3}
into \eqref{Res1}, we obtain
\begin{align}
\Res &=
2 \I \Big( (X+Y+\<20>)\pl \<1> \Big)\pe \<1>
- \I \Big( M \big( \Qxy + 2 \Res + \<20>^2 + 2 \<21> + \<2> \big) \<1> \Big) \pe \<1>.
\label{Res2}
\end{align}

\noi
As we see below, both resonant products on the right-hand side
are not well defined at this point.

Let us consider the  first term on the right-hand side of  \eqref{Res2}:
\begin{align}
\I \Big( (X+Y+\<20>)\pl \<1> \Big)\pe \<1>.
\label{u4}
\end{align}

\noi
Due to the paraproduct structure (with the high frequency part given by $\<1>$)
under  the Duhamel integral operator $\I$, 
we see that the resonant product in \eqref{u4} is not well defined at this point
since a term   $\I (w\pl \<1>)$ has (at best) regularity $\frac 12-$.
In order to give a precise meaning to the right-hand side of \eqref{Res2}, 
we now recall the paracontrolled operators introduced in~\cite{GKO2}.\footnote
{Strictly speaking, the paracontrolled operators introduced in \cite{GKO2}
are for the undamped wave equation. 
Since the local-in-time mapping property remains unchanged, 
we ignore this minor point.}
We point out that
in the parabolic setting, 
it is at this step where
one would introduce commutators  and exploit their smoothing properties.
For our dispersive problem, however, 
one of the commutators does not provide any smoothing
and thus such an argument does not seem to work.
See \cite[Remark~1.17]{GKO2}.

Given a function $w$ on $\T^3 \times \R_+$, 
 define
\begin{align}
\begin{split}
 \If_{\pl}(w) (t)
:  \! &  =  \I (w\pl \<1>)(t) \\
& =       \sum_{n \in \Z^3}
e_n  \sum_{\substack{n =  n_1 +  n_2\\  |n_1| \ll |n_2|}}
\int_0^t e^{-\frac{t-t'}2} \frac{\sin ((t-t')\jbb{n})}{\jbb{n}}
\ft w(n_1, t')\,  \ft{\<1>}(n_2, t') dt',
\end{split}
\label{X2}
\end{align}

\noi
where $\jbb{n}$ is as in \eqref{jbbn}.
Here, $|n_1| \ll |n_2|$ signifies the paraproduct $\pl$
in the definition of $\If_{\pl}$.\footnote{For simplicity of the presentation, 
we use the less precise definitions of paracontrolled operators.
For example, see~\eqref{XX2a}
for the precise definition of 
the paracontrolled operator $ \If_{\pl}^{(1)}$.}
As mentioned above, 
the regularity of 
$ \If_{\pl}(w)$ is (at best) $\frac 12-$
and thus the resonant product 
$ \If_{\pl}(w) \pe \<1>$ does not make sense
in terms of deterministic analysis.
Proceeding as in  \cite{GKO2}, 
we divide the paracontrolled operator $\If_{\pl}$
into two parts.
Fix small $\theta  > 0$.
Denoting by $n_1$ and $n_2$ the spatial frequencies
of $w$ and $\<1>$ as in \eqref{X2}, 
we  define $ \If_{\pl}^{(1)}$ and $ \If_{\pl}^{(2)}$ 
as the restrictions of $\If_{\pl}$ onto $\{ |n_1| \ges |n_2|^{\theta}\}$
and  $\{ |n_1| \ll |n_2|^{\theta}\}$.
More concretely, we set 
\begin{align}
 \If_{\pl}^{(1)}(w)(t) 
:= \sum_{n \in \Z^3}
e_n  \sum_{\substack{n =  n_1 +  n_2\\    |n_2|^\theta \les |n_1| \ll |n_2|}}
\int_0^t e^{-\frac{t-t'}2} \frac{\sin ((t-t') \jbb{n})}{\jbb{n}}
\ft w(n_1, t')\,  \ft{\<1>}(n_2, t') dt'
\label{X3}
\end{align}

\noi
and 
\begin{align}
 \If_{\pl}^{(2)}(w)   := \If_{\pl}(w) - \If_{\pl}^{(1)}(w).
 \label{X3a}
\end{align}

\noi
As for the first paracontrolled operator $ \If_{\pl}^{(1)}$, 
 the lower bound  $|n_1| \ges |n_2|^\theta$ 
 and the positive regularity of $w$
 allow us to prove 
a smoothing property
such that 
 the resonant product 
$ \If_{\pl}^{(1)}(w) \pe \<1>$ is well defined.
See Lemma~\ref{LEM:sto3}
below.

As noted in \cite{GKO2},
the second paracontrolled operator $ \If_{\pl}^{(2)}$
does not seem to possess a (deterministic) smoothing property.
One of the main novelties in \cite{GKO2}
was then to directly 
 study the  random operator $\If_{\pl, \pe}$ defined by
\begin{align}
\begin{split}
\If_{\pl, \pe}(w) (t)
: \!  & =
  \If_{\pl}^{(2)}(w)\pe \<1> (t) \\
&  = \sum_{n \in \Z^3}e_n 
    \int_0^{t} 
\sum_{n_1 \in \Z^3}
\ft w(n_1, t') \A_{n, n_1} (t, t') dt', 
\end{split}
\label{X6}
\end{align}

\noi
where $\A_{n, n_1} (t, t')$ is given by 
\begin{align}
\A_{n, n_1} (t, t')
= \ind_{[0 , t]}(t') \sum_{\substack{n - n_1 =  n_2 + n_3\\ |n_1| \ll |n_2|^\theta \\ |n_1 + n_2|\sim |n_3|}}
 e^{-\frac{t-t'}{2}} \frac{\sin ( (t - t') \jbb{n_1+n_2})}{\jbb{n_1+n_2}}
   \ft{\<1>}(n_2, t')\,   \ft{\<1>}(n_3, t) .
\label{X7}
\end{align}

\noi
Here, 
the condition  $ |n_1 + n_2|\sim |n_3|$
is used to denote the  spectral multiplier corresponding 
to the resonant product $\pe$ in \eqref{X6}.
See \eqref{A0a} and \eqref{A0b} for the precise definitions.
The almost sure bounded property of 
the  random operator $\If_{\pl, \pe}$
was studied in \cite{GKO2, OOTol1}.
See Lemma \ref{LEM:sto4J} below.

Next, we consider 
the  second  term on the right-hand side of  \eqref{Res2}:
\begin{align}
\I \Big( M \big( \Qxy + 2 \Res + \<20>^2 + 2 \<21> + \<2> \big) \<1> \Big) \pe \<1>.
\label{u5}
\end{align}

\noi
Once again, the resonant product is not well defined since
the sum of regularities is negative.
The term \eqref{u5} appeared in our previous work 
\cite{OOTol1} on the focusing Hartree $\Phi^4_3$-model, 
where we introduced 
 the following stochastic term:
\begin{align}
\Ab(x, t, t') = \sum_{n \in \Z^3} e_n(x) \sum_{\substack{n = n_1 + n_2\\|n_1|\sim| n_2|}}
 e^{-\frac{t-t'}{2}} \frac{\sin ( (t - t') \jbb{n_1})}{\jbb{n_1}}
 \ft{\<1>} (n_1, t') \ft{\<1>} (n_2, t)
\label{sto3}
\end{align}

\noi
for $t \geq t' \geq 0$, 
where $|n_1|\sim| n_2|$ signifies the resonant product.
Then, we have
\begin{align}
\Big( \I\big( M(w) \<1> \big) \pe \<1> \Big)(t)
= \int_0^t 
M(w) (t') \Ab(t, t') dt'.
\label{sto1a}
\end{align}

\noi
We point out that the Fourier transform $\ft \Ab(n, t, t')$
corresponds to $\A_{n, 0}(t, t')$ defined in \eqref{X7}
and thus the analysis for $\Ab$ is closely related to 
that for the paracontrolled
operator $\If_{\pl, \pe}$ in~\eqref{X6}.
See  Lemma \ref{LEM:stoA} below for 
the almost sure regularity of $\Ab$.

\medskip
Finally, 
we are  ready to present the full system for the three unknowns $(X, Y, \Res)$.
Putting together \eqref{SNLW5}, \eqref{SNLW5y}, \eqref{Pxy},  \eqref{Res2},
\eqref{X3}, \eqref{X6}, and \eqref{sto1a}, 
we arrive at the following system:
\begin{align}
\begin{split}
 (\dt^2 + \dt  +1 - \Dl) X  
& =  2 (X+Y+\<20>)\pl \<1>\\
& \hphantom{X}
- M \big( \Qxy + 2 \Res + \<20>^2 + 2 \<21> + \<2> \big) \<1>,\\
 (\dt^2 + \dt +1  - \Dl) Y
&  = (X+Y+\<20>)^2 + 2 (\Res + Y\pe \<1> + \<21p>) + 2 (X+Y+\<20>) \pg \<1> \\
& \hphantom{X}
- 
M \big( \Qxy + 2 \Res + \<20>^2 + 2 \<21> + \<2> \big) (X+Y+\<20>),\\
\Res 
&= 2\If_{\pl}^{(1)}
 \big( X+Y+\<20> \big)\pe \<1>
+2   \If_{\pl, \pe}
 \big( X+Y+\<20>\big)\\
& \hphantom{X}
-\int_0^t  M \big( \Qxy + 2 \Res + \<20>^2 + 2 \<21> + \<2> \big) 
\Ab(t, t') dt',\\
(X, \dt X, Y, \dt Y)|_{t = 0} & = (X_0, X_1, Y_0, Y_1).
\end{split}
\label{SNLW6}
\end{align}

\noi
By  viewing the following random distributions and operator
in the system above:
\begin{align}
 \<1>,  \quad \<2>, \quad \<20>, \quad  \<21p>, \quad
\Ab,
\quad  \text{and}\quad 
\If_{\pl, \pe},
\label{Xi1}
\end{align}

\noi
as predefined deterministic data
with certain regularity\,/\,mapping properties, 
we prove the following local well-posedness 
 of the system \eqref{SNLW6}.

\begin{theorem} \label{THM:1}
Let 
 $ \frac 14<s_1 < \frac 12 < s_2 \le s_1+\frac 14$
and $s_2 - 1 \le s_3 < 0$. 
Then,
there exist $\ta = \ta(s_3) > 0$ 
and $\eps = \eps(s_1, s_2, s_3 ) > 0$
such that 
if 
\begin{itemize}
\item   $\<1>$ is a distribution-valued function belonging to
 $C([0, 1]; W^{-\frac 12 - \eps, \infty}(\T^3))\cap C^1([0, 1]; W^{-\frac 32 - \eps, \infty}(\T^3))$, 

\smallskip
\item
$\<2>$ is a distribution-valued function belonging to $ C([0, 1]; W^{-1 - \eps, \infty}(\T^3))$, 

\smallskip
\item
$\<20>$ is a distribution-valued function belonging to 
$ C([0, 1]; W^{\frac 12 - \eps, \infty}(\T^3))\cap C^1([0, 1]; W^{-1 - \eps, \infty}(\T^3))$,

\smallskip
\item
$\<21p>$ is a distribution-valued function belonging to $C([0, 1]; H^{- \eps}(\T^3))$, 

\smallskip
\item
$\Ab(t, t')$ 
is a distribution-valued function belonging to $ L^\infty_{t'} L^3_{t}(\Dl_2(1); 
H^{ - \eps}(\T^3))$, 
where $\Dl_2(T) \subset [0, T]^2 $ is defined by 
\begin{align}
\Dl_2(T) = \{(t, t') \in \R_+^2: 0 \leq t' \leq t \leq T\}, 
\label{sto2}
\end{align}

\smallskip

\item 
 the operator
 $\If_{\pl, \pe}$ 
belongs to  the class
 $ \L_2\big(\tfrac32, 1\big)$, 
 where  $ \L_2(q, T)$ is defined by 
 \begin{align}
\L_2(q, T) := \L(L^q([0, T]; L^2(\T^3) )
\, ;\, 
 L^\infty([0, T]; H^{s_3}(\T^3))), 
\label{L1ast}
\end{align}

\smallskip

\end{itemize}

\noi
then the system \eqref{SNLW6} is locally well-posed in 
$\H^{s_1}(\T^3) \times \H^{s_2}(\T^3)$.
More precisely, 
given any $(X_0, X_1, Y_0, Y_1) \in \H^{s_1}(\T^3)\times \H^{s_2}(\T^3)$, 
there exist $T > 0$ 
and  a unique solution $(X, Y, \Res) $ to the 
hyperbolic $\Phi^3_3$-system \eqref{SNLW6} on $[0, T]$
in the class\textup{:}
\begin{align}
 Z^{s_1, s_2, s_3}(T)
 =  X^{s_1}(T) \times Y^{s_2}(T) \times 
 L^3([0, T ];H^{s_3}(\T^3)). 
\label{Z1}
\end{align}

\noi
Here,
$X^{s_1} (T)$ and $Y^{s_2}(T)$ are 
the energy spaces at the regularities $s_1$ and $s_2$
intersected with appropriate Strichartz spaces
defined in \eqref{M0}  below.
Furthermore, the solution $(X, Y, \Res)$
depends  Lipschitz-continuously 
on the enhanced data set\textup{:}
\begin{align}
\big(X_0, X_1, Y_0, Y_1, 
\<1>, \,  \<2>,  \,  \<20>, \, \<21p>, \,
\Ab, \,  \If_{\pl, \pe}\big)
\label{data1}
\end{align}

\noi
in the class\textup{:}
\begin{align*}
\mathcal{X}^{s_1, s_2, \eps}_T
& = \H^{s_1}(\T^3) \times 
\H^{s_2}(\T^3)\\
& \hphantom{X}
\times 
\big(C([0,T]; W^{-\frac 12 - \eps, \infty}(\T^3))\cap C^1([0,T]; W^{-\frac 32 - \eps, \infty}(\T^3))\big)\\
& \hphantom{X}
\times 
C([0,T]; W^{-1 - \eps, \infty}(\T^3))\\
& \hphantom{X}
\times 
\big(C([0,T]; W^{\frac 12 - \eps, \infty}(\T^3)) \cap C^1([0, T]; W^{-1 - \eps, \infty}(\T^3))\big)\\
& \hphantom{X}
\times 
C([0,T]; H^{- \eps}(\T^3))
\times  L^\infty_{t'}L^3_t(\Dl_2(T); H^{ - \eps}(\T^3))
\times
 \L_2\big(\tfrac32, T\big).
\end{align*}

\end{theorem}

Given the a priori regularities of the enhanced data, 
Theorem \ref{THM:1} follows
from the standard energy and Strichartz estimates for the wave equation.
While the proof is a slight modification of those  in \cite{GKO2, OOTol1}, 
we present the proof of Theorem \ref{THM:1} in Subsection \ref{SUBSEC:LWP3} for readers' convenience.
The local well-posedness of the hyperbolic $\Phi^3_3$-model
(Theorem \ref{THM:LWP0})
follows from  Theorem \ref{THM:1}
and the almost sure convergence of the truncated stochastic objects:
\begin{align}
 \<1>_N,  \quad \<2>_N, \quad \<20>_N, \quad  \<21p>_N, \quad
\Ab_N,
\quad  \text{and}\quad 
\If_{\pl, \pe}^N
\label{enhz}
\end{align}

\noi
to the elements in the enhanced data set in \eqref{Xi1};
see Lemmas \ref{LEM:stoconv0}, \ref{LEM:sto1}, \ref{LEM:sto2}, 
\ref{LEM:sto3}, \ref{LEM:sto4J}, and \ref{LEM:stoA}
in Subsection \ref{SUBSEC:para}.
See Remark \ref{REM:conv0} below.

\begin{remark}\label{REM:conv0}\rm
(i)  For the sake of the well-posedness
of the system \eqref{SNLW6}, 
we considered 
general initial data $(X_0, X_1, Y_0, Y_1) \in \H^{s_1}(\T^3)\times \H^{s_2}(\T^3)$
in Theorem \ref{THM:1}.
However, in order to go back from the system~\eqref{SNLW6}
to the hyperbolic $\Phi^3_3$-model \eqref{SNLWA2B}
with the identification~\eqref{Res1}
(in the limiting sense), 
we need to set $(X_0, X_1) = (0, 0)$
since  the resonant product
of the linear solution  
$ S(t) (X_0, X_1)$ and $\<1>$ is not well defined
in general.
As we see in Section \ref{SEC:GWP}, 
we simply use the zero initial data for the system \eqref{SNLW6}
in constructing global-in-time invariant Gibbs dynamics
for  the hyperbolic $\Phi^3_3$-model \eqref{SNLWA2B}.

\smallskip

\noi
(ii) Our choice of the norms for $\<21p>$ 
is crucial in the globalization argument.
See Proposition~\ref{PROP:exptail2}
and Remark \ref{REM:3}.

\smallskip

\noi
(iii) In proving  the local well-posedness result 
of  the system \eqref{SNLW6} stated in Theorem \ref{THM:1}, 
we do not need to use the $C^1_T$-norms
for $\<1>$ and $\<20>$.
However, we will need
these $C^1_T$-norms for $\<1>$ and $\<20>$
in the globalization argument presented in Section~\ref{SEC:GWP}
and thus have included them in the hypothesis
and the definition of Theorem \ref{THM:1}
of the space $\mathcal{X}^{s_1, s_2, \eps}_T$.
See also \eqref{data3} and Remark \ref{REM:data}.

Furthermore, with this definition of the space $\mathcal{X}^{s_1, s_2, \eps}_T$, 
the map 
from an enhanced data set in \eqref{data1} (with $(X_0, X_1, Y_0, Y_1) = (0, 0, u_0, u_1)$)
to $(u, \dt u)$, where $ u = \<1> + \<20> + X + Y$ as in~\eqref{decomp3a}
becomes a continuous map
from $\mathcal{X}^{s_1, s_2, \eps}_T$
to $ C([0, T]; \H^{-\frac{1}{2}-\eps}(\T^3))$.

\end{remark}

\subsection{Strichartz estimates}
Given  $0 \leq s \leq 1$, 
we say that a pair $(q, r)$ is $s$-admissible 
(a pair $(\wt q, \wt r)$ is dual $s$-admissible,\footnote{Here, we define
the notion of dual $s$-admissibility for the convenience of the presentation.
Note that $(\wt q, \wt r)$ is dual $s$-admissible
if and only if $(\wt q', \wt r')$ is $(1-s)$-admissible.}
 respectively)
if $1 \leq \wt q < 2 < q \leq \infty$, 
 $1< \wt r \le 2 \leq r <\infty$, 
\begin{align*}
 \frac{1}{q} + \frac 3r  = \frac{3}{2}-  s =  \frac1{\wt q}+ \frac3{\wt r} -2, 
\qquad
\frac 1q + \frac{1}{r} \leq \frac 1 2, 
\qquad \text{and} 
\qquad  
\frac1{\wt q}+\frac1{\wt r} \geq \frac32   .
\end{align*}


We say that  $u$ is a solution to the following nonhomogeneous linear damped wave equation:
\begin{align}
\begin{cases}
(\dt^2 + \dt + 1-\Dl) u = F \\
( u, \dt u) |_{t = 0}=(u_0,  u_1) 
\end{cases}
\label{NLW2}
\end{align}

\noi
on a time interval containing $t= 0$, 
if $u$ satisfies the following Duhamel formulation:
\[ u = 
S(t) (u_0, u_1) + \int_0^t\D(t-t') F(t') dt',
\]  
where $S(t)$ and $\D(t)$ are  as in \eqref{St0} and \eqref{W3}, respectively.
We now recall the  Strichartz estimates
for solutions to the nonhomogeneous linear damped wave equation~\eqref{NLW2}.

\begin{lemma}\label{LEM:Str}
Given $0 \leq s \leq 1$,
let $(q, r)$ and $(\wt q,\wt r)$
be $s$-admissible and dual $s$-admissible pairs, respectively. 
Then, a solution $u$ to the nonhomogeneous linear damped wave equation~\eqref{NLW2}
satisfies
\begin{align}
\| (u, \dt u) \|_ {L^\infty_T \H^s_x } + 
 \| u  \|_{L^q_TL^r_x}
\lesssim 
\|(u_0, u_1) \|_{\H^s}  +  \| F\|_{L^{\wt q}_TL^{\wt r}_x}
\label{Str1}
\end{align}

\noi
for all $0 < T \leq 1$. 
The following estimate also holds:
\begin{align}
\| (u, \dt u) \|_ {L^\infty_T \H^s_x } 
+  \| u  \|_{L^q_TL^r_x}
\lesssim 
\|(u_0, u_1) \|_{\H^s}  +  \| F \|_{L^{1}_T H^{s-1}_x}
\label{Str2}
\end{align}

\noi
for all $0 < T \leq 1$.
The same estimates also hold for any finite $T > 1$ but 
with the implicit constants depending on $T$.

\end{lemma}

The Strichartz estimates on $\R^d$ are well known; see \cite{GV, LS, KeelTao}
in the context of the undamped wave equation
(with the linear part $\dt^2 - \Dl$).
For the undamped Klein-Gordon equation (with the linear part
$\dt^2 + 1- \Dl$), 
see \cite{KSV}.
Thanks to the finite speed of propagation, 
these estimates on $\T^3$ follow from the corresponding
estimates on $\R^3$.

As for the current damped case, 
by setting $v(t) = e^{\frac{t}{2}}u(t)$, 
the damped wave equation \eqref{NLW2} becomes 
\begin{align*}
\begin{cases}
(\dt^2 +  \frac{3}{4}-\Dl) v = e^{\frac{t}{2}}F \\
( v, \dt v) |_{t = 0}=(u_0,  u_1),  
\end{cases}
\end{align*}

\noi
to which the Strichartz estimates for the Klein-Gordon equation apply.
By undoing the transformation, 
we then obtain the Strichartz estimates
for  the damped equation \eqref{NLW2}
on finite time intervals $[0, T]$, 
where the implicit constants depend on $T$.

In proving Theorem \ref{THM:1}, we use the fact that 
$\big(8, \frac{8}{3}\big)$
and $(4, 4)$ are $\frac 14$-admissible
and 
  $\frac 12$-admissible, respectively.
We also use
a dual $\frac12$-admissible pair  $\big(\frac 43, \frac 43\big)$.

\subsection{Stochastic terms and paracontrolled operators}
\label{SUBSEC:para}

In this subsection, we collect regularity properties of stochastic terms and 
the paracontrolled operators.  See \cite{GKO2, OOTol1}
for the proofs.
Note that the stochastic objects are constructed
from the stochastic convolution $\<1> = \Psi$ in \eqref{W2}.
In particular, in the following, probabilities of various events 
are measured with respect to the Gaussian initial data
and the space-time white noise.\footnote{With the notation in Section \ref{SEC:GWP}
(see \eqref{PPP}),
this is equivalent to saying that we measure various events with respect to $\muu \otimes \PP_2$.}

First,
we state the regularity properties of $\<1>$ and $\<2>$.
See Lemma 3.1 in \cite{GKO2} and Lemma 4.1 in \cite{OOTol1}.

\begin{lemma}\label{LEM:stoconv0}
Let  $T >0$.

\smallskip

\noi
\textup{(i)}
For any $\eps > 0$, 
$ \<1>_N $  in \eqref{so4a} converges to $\<1>$
in $C([0,T];W^{- \frac 12 -\eps,\infty}(\T^3))\cap C^1([0,T];W^{- \frac 32 -\eps,\infty}(\T^3))$
 almost surely.
 In particular, we have
  \[\<1> \in C([0,T];W^{- \frac 12 -\eps,\infty}(\T^3))
  \cap C^1([0,T];W^{- \frac 32 -\eps,\infty}(\T^3))
  \]
  
  \noi
  almost surely.
Moreover, we have 
the following tail estimate:
\begin{align}
\PP\Big( \|\<1>_N\|_{C_T W^{-\frac 12-\eps, \infty}_x\cap \, C^1_T W^{-\frac 32-\eps, \infty}_x}> \ld\Big) 
\leq C(1+T)\exp\big(-c \ld^{2}\big)
\label{ex1}
\end{align}

\noi
for any $T>0$  and $\ld > 0$, 
uniformly in $N \in \N \cup\{\infty\}$
with the understanding that $\<1>_\infty = \<1>$.

\smallskip

\noi
\textup{(ii)}
For any $\eps > 0$, 
$ \<2>_N $  in \eqref{so4b} converges to $\<2>$
 in 
$C([0,T];W^{- 1-\eps,\infty}(\T^3))$ almost surely.
 In particular, we have
 \[\<2> \in C([0,T];W^{-1-\eps,\infty}(\T^3))\]
 
 \noi
 almost surely.
Moreover, we have 
the following tail estimate:
\begin{align*}
\PP\Big( \|\<2>_N\|_{C_T W^{- 1-\eps, \infty}_x}> \ld\Big) 
\leq C(1+T)\exp\big(-c \ld\big)
\end{align*}

\noi
for any $T>0$  and $\ld > 0$, 
uniformly in $N \in \N \cup\{\infty\}$
with the understanding that $\<2>_\infty = \<2>$.

\end{lemma}

\begin{remark}\label{REM:decay} \rm

A slight modification of the proof of the exponential tail estimate \eqref{ex1} shows that 
there exists small $\dl > 0$ such that 
\begin{align*}
\PP\Big( N_2^\dl \| \<1>_{N_1} - \<1>_{N_2}\|_{C_T W^{-\frac 12-\eps, \infty}_x
\cap \, C^1_T W^{-\frac 32-\eps, \infty}_x}> \ld\Big) 
\leq C(1+T)\exp\big(-c \ld^{2}\big)
\end{align*}

\noi
for any $T>0$  and $\ld > 0$, 
uniformly in $N_1 \ge N_2 \ge 1$.
A similar comment applies to the other elements
$ \<2>_N$, $\<20>_N$, $\<21p>_N$, 
$\Ab_N$, and 
$\If_{\pl, \pe}^N$
 in the truncated enhanced data set
in \eqref{enhz}.

\end{remark}

The next two lemmas 
treat  $\<20>$ 
and the resonant product $\<21p>$, 
exhibiting  an extra $\frac 12$-smoothing.
See Propositions 1.6 and 1.8 in \cite{GKO2}.
While the exponential tail estimates \eqref{stox1}
and \eqref{stox2} were not proven in \cite{GKO2}, 
they follow from the second moment bounds
on the Fourier coefficients of $\<20>_N$ and $\<21p>_N$
obtained in \cite{GKO2}
and arguing as in the proof of Lemma 2.3 in \cite{GKOT}, 
using a version of the 
Garsia-Rodemich-Rumsey inequality (see Lemma 2.2 in \cite{GKOT})
with the fact that  $\<20>_N \in \H_2$ and $\<21p>_N\in \H_{\le 3}$.
Since the required argument is verbatim from \cite{GKOT}, 
we omit details.

\begin{lemma}\label{LEM:sto1}
Let   $T >0$.
Then, 
$ \<20>_N $ converges to 
$\<20>$
in 
$C([0,T];W^{\frac 12 -\eps,\infty}(\T^3)) \cap C^1([0,T];W^{-1 -\eps,\infty}(\T^3))$
almost surely
for any  $\eps > 0$.
 In particular, we have
  \[\<20> \in C([0,T];W^{\frac 12 -\eps,\infty}(\T^3))
  \cap C^1([0,T];W^{-1 -\eps,\infty}(\T^3))\]
  
  \noi
  almost surely for any  $\eps > 0$.
Moreover, we have 
the following tail estimate:
\begin{align}
\PP\Big( \|\<20>_N\|_{C_T W^{\frac 12-\eps, \infty}_x
\cap\,  C^1_TW^{-1 -\eps,\infty}_x}> \ld\Big) 
\leq C(1+T)\exp\big(-c \ld\big)
\label{stox1}
\end{align}

\noi
for any $T>0$  and $\ld > 0$, 
uniformly in $N \in \N \cup\{\infty\}$
with the understanding that $\<20>_\infty = \<20>$.

\end{lemma}

\begin{lemma}\label{LEM:sto2}
Let   $T >0$.
Then, 
$\<21p>_N := \<20>_N\pe \<1>_N$
 converges to $\<21p>$
 in 
$C([0,T];W^{ -\eps,\infty}(\T^3))$
almost surely
for any  $\eps > 0$.
 In particular, we have
  \[\<21p> \in C([0,T];W^{ -\eps,\infty}(\T^3))\]
  
  \noi
  almost surely  for any  $\eps > 0$.
Moreover, we have 
the following tail estimate:
\begin{align}
\PP\Big( \|\<21p>_N\|_{C_T W^{-\eps, \infty}_x}> \ld\Big) 
\leq C(1+T)\exp\big(-c \ld^\frac{2}{3}\big)
\label{stox2}
\end{align}

\noi
for any $T>0$  and $\ld > 0$, 
uniformly in $N \in \N \cup\{\infty\}$
with the understanding that $\<21p>_\infty = \<21p>$.

\end{lemma}

Next, we state the almost sure mapping  properties of the paracontrolled operators.
We first consider the paracontrolled operator
$\If_{\pl}^{(1)}$ defined in~\eqref{X3}.
By writing out the frequency relation  $|n_2|^\theta \les |n_1| \ll |n_2|$
in a more precise manner, we have
\begin{align}
\begin{split}
 \If_{\pl}^{(1)} (w) (t)
 &   =  \sum_{n \in \Z^3}
e_n  \sum_{n =  n_1 +  n_2}
\sum_{ \ta k + c_0 \leq j < k-2}
\varphi_j(n_1) \varphi_k(n_2)  \\
& \hphantom{XXXXX}
\times 
\int_0^t e^{-\frac{t-t'}2} \frac{\sin ((t - t') \jbb{n})}{\jbb{n}} 
\ft w(n_1, t')\,  \ft{\<1>}(n_2, t') dt', 
\end{split}
\label{XX2a}
\end{align}

\noi
where $\varphi_j$ is as in \eqref{phi1} and $c_0 \in \R$ is some fixed constant.
Given a pathwise regularity of $\<1>$, 
the mapping property of 
$\If_{\pl}^{(1)}$ can be established in a deterministic manner.
See Lemma 7.1 in~\cite{OOTol1}.
See also Corollary 5.2 in \cite{GKO2}.

\begin{lemma}\label{LEM:sto3}

Let  $s>0$ and $T > 0$.
Then, given small $\theta > 0$, 
there exists small $\eps= \eps(s, \ta)  > 0$
such that 
the following deterministic estimate holds
the paracontrolled operator $ \If_{\pl}^{(1)}$ defined in~\eqref{X3}\textup{:}
\begin{align}
\| \If_{\pl}^{(1)}(w) \|_{L^\infty_T H^{\frac12+3\eps}_x}
\les \|w\|_{L^2_T H_x^{s}}
\|\<1>\|_{L^2_T W_x^{-\frac 12 - \eps, \infty}}.
\label{A00}
\end{align}

\noi
In particular, 
 $ \If_{\pl}^{(1)}$ belongs almost surely
to the class
\begin{align}
 \L_1(T) = \L ( L^2([0, T]; H^{s}(\T^3) )
 \, ; \, 
C([0, T]; H^{\frac 12 + 3\eps}(\T^3))) .
\notag
\end{align}

\noi
Moreover, by letting
$ \If_{\pl}^{(1), N}$, $N \in \N$, denote the paracontrolled operator
in \eqref{X3} with $\<1>$ replaced by the truncated stochastic convolution $\<1>_N$ in \eqref{so4a}, 
the truncated paracontrolled operator $ \If_{\pl}^{(1), N}$ converges almost surely to $ \If_{\pl}^{(1)}$ 
in $\L_1 (T)$.

\end{lemma}

Next, we consider 
 the random operator 
$\If_{\pl, \pe}$ defined in \eqref{X6}.
By writing out the frequency relations
more carefully
as in \eqref{XX2a},
we have
\begin{align}
\begin{split}
\If_{\pl, \pe}(w) (t)
&  = \sum_{n \in \Z^3}e_n 
    \int_0^{t} 
\sum_{j = 0}^\infty
\sum_{n_1 \in \Z^3}
\varphi_j(n_1)
\ft w(n_1, t') \A_{n, n_1} (t, t') dt', 
\end{split}
\label{A0a}
\end{align}

\noi
where $\A_{n, n_1} (t, t')$ is given by 
\begin{align}
\begin{split}
\A_{n, n_1} (t, t')
& = \ind_{[0 , t]}(t')
\sum_{\substack{k = 0\\ 0 \le j < \ta k + c_0}}^\infty
\sum_{\substack{\l, m = 0\\|\l-m|\leq 2}}^\infty
\sum_{n - n_1 =  n_2 + n_3} 
 \varphi_k(n_2)  
 \varphi_\l(n_1 + n_2) 
 \varphi_m(n_3) \\
& \hphantom{XXXX}
\times
e^{-\frac{t-t'}2}
\frac{\sin ((t - t') \jbb{n_1+n_2})}{\jbb{n_1+n_2}} 
   \ft{\<1>}(n_2, t')\,   \ft{\<1>}(n_3, t)  .
\end{split}
\label{A0b}
\end{align}

\noi
Then, we have 
the  following almost sure mapping property
of the random operator 
$\If_{\pl, \pe}$.
See Proposition 2.5 in \cite{OOTol1}.
See also Proposition 1.11 in \cite{GKO2}.

\begin{lemma}\label{LEM:sto4J}
Let  $ s_3 < 0$ and $T>0$.
Then, there exists small $\theta = \theta (s_3) > 0$
such that, for any finite $ q >1$, 
the paracontrolled operator $\If_{\pl, \pe}$
defined by \eqref{X6} and \eqref{X7}
belongs to $\L_2(q, T)$ defined in \eqref{L1ast}, 
almost surely.
Furthermore the following tail estimate holds
for some  $C,c >0$:
\begin{align}
\PP \Big( \| \If_{\pl, \pe} \|_{\L_2 (q, T)} > \ld \Big)
\le
C (1+T) \exp ( -  \ld)
\label{pote2}
\end{align}

\noi
for any $\ld\gg 1$.

If we define  the truncated paracontrolled operator $\If_{\pl, \pe}^N$, $N \in \N$, 
by replacing 
$\<1>$ in~\eqref{X6} and~\eqref{X7}
with the truncated stochastic convolution $\<1>_N$ in \eqref{so4a}, 
then
the truncated paracontrolled operators $\If_{\pl, \pe}^N$ converge almost surely to $\If_{\pl, \pe}$ 
in $\L_2 (q, T)$.
Furthermore, the tail estimate~\eqref{pote2}
holds for the truncated paracontrolled operators $\If_{\pl, \pe}^N$,
uniformly in  $N\in \N$.

\end{lemma}

Finally, we state the regularity property of
$\Ab$ defined in \eqref{sto3}.
See Lemma 7.2 in \cite{OOTol1}.
Given $N \in \N$, 
we define the  truncated version $\Ab_N$:
\begin{align}
\Ab_N(x, t, t') = \sum_{n \in \Z^3} e_n(x) \sum_{\substack{n = n_1 + n_2\\|n_1|\sim| n_2|}}
 e^{-\frac{t-t'}{2}} \frac{\sin ( (t - t') \jbb{n_1})}{\jbb{n_1}}
 \ft{\<1>}_N(n_1, t') \ft {\<1>}_N (n_2, t)
\label{sto3a}
\end{align}

\noi
by replacing $\<1>$ by $\<1>_N$  in \eqref{sto3}.

\begin{lemma}\label{LEM:stoA}

Fix finite $q \geq 2$.
Then, 
given any  $T,\eps>0$ and finite $p \geq 1$, 
 $\{ \Ab_N \}_{N\in \N}$ is a Cauchy sequence in 
 $L^p(\O;L^\infty_{t'}L^q_t(\Dl_2(T); H^{-\eps}(\T^3)))$, 
 converging to some limit $\Ab$ \textup{(}formally defined by \eqref{sto3}\textup{)}
 in $L^p(\O;L^\infty_{t'}L^q_t(\Dl_2(T);H^{-\eps}(\T^3)))$, 
 where $\Dl_2(T)$ is as in \eqref{sto2}.
Moreover,  $\Ab_N$  converges almost surely to the same  limit in $L^\infty_{t'}L^q_t(\Dl_2(T); 
H^{ -\eps}(\T^3))$.
Furthermore, 
 we   have 
the following uniform tail estimate\textup{:}
\begin{align*}
\PP\Big( \| \Ab_N\|_{L^\infty_{t'}L^q_t(\Dl_2(T);  H^{-\eps}_x)} > \ld\Big) 
\leq 
C (1+T) \exp ( -  \ld)
\end{align*}

\noi
for any  $\ld \gg 1$, and $N \in \N \cup \{\infty\}$, 
where $\Ab_\infty = \Ab$.

\end{lemma}

\subsection{Proof of local well-posedness}
\label{SUBSEC:LWP3}

In this subsection,
we present the proof of Theorem~\ref{THM:1}.
In the following, we assume that $s_3 < 0 < s_1< s_2 < 1$.
Recall that 
$\big(8, \frac{8}{3}\big)$ 
and $(4, 4)$ are $\frac 14$-admissible
and  $\frac 12$-admissible, respectively.
Given $0 < T \leq 1$, 
we define $X^{s_1} (T)$
(and $Y^{s_2} (T)$) as the intersection of the energy spaces of regularity  $s_1$ 
(and $s_2$, respectively) and the Strichartz space:
\begin{align}
\begin{split}
 X^{s_1} (T)
 & = C([0,T];H^{s_1}(\T^3))\cap C^1([0,T]; H^{s_1-1}(\T^3))
 \cap L^8([0, T]; W^{s_1-\frac 14, \frac{8}{3}}(\T^3)),\\
 Y^{s_2} (T)
 & = C([0,T];H^{s_2}(\T^3))\cap C^1([0,T]; H^{s_2-1}(\T^3))
 \cap L^4([0, T]; W^{s_2 - \frac 12, 4}(\T^3)), 
\end{split}
 \label{M0}
\end{align}

\noi
and set 
\begin{align*}
Z^{s_1, s_2, s_3} (T) =  X^{s_1}(T) \times  Y^{s_2}(T) \times L^3([0,T]; H^{s_3}(\T^3)).
\end{align*}

\noi
By writing  \eqref{SNLW6} in the Duhamel formulation,
 we have
\begin{align}
X &=  \Phi_1(X, Y,  \Res) \notag \\
:\! & = 
S(t)(X_0, X_1)
+ 2 \I \Big( (X+Y+\<20>)\pl \<1> \Big) \notag \\
& \hphantom{X}
- \I \Big( M \big( \Qxy + 2 \Res + \<20>^2 + 2 \<21> + \<2> \big) \<1>\Big),\notag \\
Y
& = \Phi_2(X, Y,  \Res) \notag \\
:\! & = 
S(t)(Y_0, Y_1)
+\I \Big( (X+Y+\<20>)^2\Big)  \label{SNLW8} \\
& \hphantom{X}
+ 2 \I \big( \Res + Y\pe \<1>+\<21p> \big) +2 \I \Big( (X+Y+\<20>) \pg \<1> \Big)\notag \\
& \hphantom{X}
 - \I \Big(M \big( \Qxy + 2 \Res + \<20>^2 + 2 \<21> + \<2> \big) (X+Y+\<20>) \Big), \notag \\
\Res 
&= \Phi_3(X, Y,  \Res)\notag \\ 
:\! &=
2\If_{\pl}^{(1)}
 \big( X+Y+\<20> \big)\pe \<1>
+2   \If_{\pl, \pe}
 \big( X+Y+\<20>\big)\notag \\
& \hphantom{X}
-\int_0^t  M \big( \Qxy + 2 \Res + \<20>^2 + 2 \<21> + \<2> \big) 
\Ab(t, t') dt'.\notag 
%
\end{align}

%

\noi
In the following, 
we use $\eps = \eps(s_1, s_2, s_3 ) > 0$ to denote a small positive number.
Given an enhanced data set as in~\eqref{data1}, 
we set
\begin{align*}
\Xi = \big(
\<1>, \,  \<2>,  \,  \<20>, \<21p>,\, \Ab, \If_{\pl, \pe}\big)
\end{align*}

\noi
and 
\begin{align}
\begin{split}
\| \Xi \|_{\mathcal{X}^{\eps}_T}
&=
\| \<1> \|_{C_T  W^{-\frac 12-\eps,\infty}_x \cap \, C_T^1  W^{-\frac 32-\eps,\infty}_x}
+ \| \<2> \|_{C_T  W^{-1-\eps,\infty}_x}\\
&\hphantom{X}
+ \| \<20> \|_{C_T W^{ \frac 12 -\eps, \infty}_x\cap \, C_T^1 W^{ -1 -\eps, \infty}_x}  
+ \| \<21p> \|_{C_T H^{ -\eps}_x} \\
&\hphantom{X}
+ \| \Ab \|_{L^\infty_{t'}L^3_t(\Dl_2; H^{ - \eps}_x)}
+  \| \If_{\pl, \pe} \|_{\L_2 (\frac32, T)} 
\label{data3}
\end{split}
\end{align}

\noi
for some small $\eps = \eps(s_1, s_2, s_3)> 0$.
Moreover, we assume that
\begin{align}
\| (X_0,X_1) \|_{\H^{s_1}}
+ \| (Y_0,Y_1) \|_{\H^{s_2}}
+ \| \Xi \|_{\mathcal{X}^{\eps}_1}
\leq K
\label{data4}
\end{align}

\noi
for some $K \geq 1$.
Here, we assume the bound on $\Xi$ 
for
the time interval $[0, 1]$.

\begin{remark}\label{REM:data} \rm
As for proving  local well-posedness stated in Theorem \ref{THM:1}, 
we do not need to use the $C^1_T  W^{-\frac 32-\eps,\infty}_x$-norm
for $\<1>$
and the $C^1_T  W^{-1 -\eps,\infty}_x$-norm
for $\<20>$.
However, in  constructing global-in-time dynamics, 
we need to make use of these norms
and thus we have included them in the definition of 
the $\mathcal{X}^{ \eps}_T$-norm in \eqref{data3}.
\end{remark}

We first establish  preliminary estimates.
By Sobolev's inequality, we have
\begin{align}
\| f^2 \|_{H^{-a}}
\les \| f^2 \|_{L^{\frac 6{3+2a}}}
= \| f \|_{L^{\frac{12}{3+2a}}}^2
\les \| f \|_{H^{\frac{3-2a}4}}^2
\label{M1-1}
\end{align}
for any $0\leq a< \frac 32$.
By \eqref{Pxy}, \eqref{M1-1}, Lemma \ref{LEM:para}, 
Lemma \ref{LEM:gko}\,(ii), and H\"older's inequality
with~\eqref{data4}, 
we have
\begin{align}
\| \Qxy \|_{L^\infty_{T} H^{-100}_x}
& \les \| (X+Y)^2 \|_{L^\infty_{T} H^{-100}_x}
+ \| X \<20> \|_{L^\infty_{T} H^{-100}_x}
+ \| Y \<20> \|_{L^\infty_{T} H^{-100}_x} \notag \\
&\quad+ \| X \pl \<1> \|_{L^\infty_{T} H^{-100}_x}
+ \| X \pg \<1> \|_{L^\infty_{T} H^{-100}_x}
+ \| Y \<1> \|_{L^\infty_{T} H^{-100}_x} \notag \\
&\les
\| X \|_{L^\infty_{T} H^\eps_x}^2
+ \| Y \|_{L^\infty_{T} H^\eps_x}^2 \notag \\
&\quad +
 \Big( \| X \|_{L^\infty_{T} L^2_x}
+ \| Y \|_{L^\infty_{T} L^2_x} \Big)
\| \<20> \|_{L^\infty_{T} L^{\infty}_x}  \label{M1}\\
&\quad +
\| X \|_{L^\infty_{T} L^2_x}
\| \<1> \|_{L^\infty_{T} W^{-\frac 12-\eps,\infty}_x}
+ \| Y \|_{L^\infty_{T} H^{\frac 12+ \eps}_x} 
\| \<1> \|_{L^\infty_{T} W^{-\frac 12-\eps,\infty}_x}
\notag \\
&\les
\|(X, Y, \Res)\|_{Z^{s_1, s_2, s_3}(T)}^2
+
K^2, \notag 
\end{align}

\noi
provided that
$s_1 \geq \eps$ 
and $s_2 \geq \frac 12 + \eps$.

We now estimate $\Phi_1 (X, Y, \Res)$  in \eqref{SNLW8}.
By \eqref{M0}, Lemmas \ref{LEM:Str} and \ref{LEM:para}, \eqref{addM}, and \eqref{M1}
with \eqref{data4},  we have
\begin{align}
\| \Phi_1  (X, Y, \Res) \|_{X^{s_1}(T)}
&\les \| (X_0, X_1)\|_{\H^{s_1}}
+ 
\big\|  (X+Y+\<20>)\pl \<1> \big\|_{L^1_{T} H^{s_1-1}_x} \notag \\
&\quad
+ \big\| M \big( \Qxy + 2 \Res + \<20>^2 + 2 \<21> + \<2> \big) \<1>\big\|_{L^1_{T} H^{s_1-1}_x}
\notag \\
&\les
\| (X_0, X_1)\|_{\H^{s_1}}
+ T
 \| X+Y+\<20> \|_{L^\infty_{T}L^2_x}
\| \<1> \|_{L^\infty_{T} W_x^{-\frac 12 -\eps, \infty}} \label{M1a}\\
&\quad+
T^{\frac 13}
\|\Qxy + 2 \Res + \<20>^2 + 2 \<21> + \<2> \|_{L^3_TH^{-100}_x}^{2} 
\|\<1>\|_{L^\infty_TH^{s_1-1}_x}
\notag \\
&\les
\| (X_0, X_1)\|_{\H^{s_1}}
+ 
T^{\frac 13}K  \Big( \|(X, Y, \Res)\|_{Z^{s_1, s_2, s_3}(T)}^4 +K^4 \Big),
\notag 
\end{align}

\noi
provided that $\eps \le s_1 <  \frac 12-\eps$, $s_2 \ge \frac 12+\eps$, and $s_3\ge-100$.

Next,  we estimate  $\Phi_2 (X, Y, \Res)$  in \eqref{SNLW8}.
By \eqref{M0} and Lemma \ref{LEM:Str}
with the fractional Leibniz rule (Lemma \ref{LEM:gko}\,(i)), we have
\begin{align}
\big\|  & \I \big( (X+Y+\<20>)^2 \big) \big\|_{Y^{s_2} (T)}
 \les \| \jb{\nb}^{s_2 - \frac 12}(X+Y+\<20>)^2\|_{L^\frac{4}{3}_{T, x}} \notag \\
& \les T^\frac{1}{4} 
\Big(\| \jb{\nb}^{s_2 - \frac 12}X\|_{L^8_T L^\frac{8}{3}_x}^2
+ \| \jb{\nb}^{s_2 - \frac 12}Y\|_{L^4_{T, x}}^2
+ \| \jb{\nb}^{s_2 - \frac 12}\<20>\|_{L^\infty_{T, x}}^2\Big) \label{M4a}\\
& \les T^\frac{1}{4} \Big( \|(X, Y, \Res)\|_{Z^{s_1, s_2, s_3}(T)}^2 +K^2 \Big), 
\notag
\end{align}

\noi
provided that $\frac 12 \le s_2 \leq \min(1 -\eps, s_1 + \frac 14)$.
By 
Lemmas \ref{LEM:Str} and \ref{LEM:para},  \eqref{M4a}, and \eqref{M1}
with \eqref{data4},   we have
\begin{align}
\| & \Phi_2  (X, Y, \Res) \|_{Y^{s_2}(T)}   \notag \\
&\les
\| (Y_0, Y_1)\|_{\H^{s_2}}
+ 
\big\| \I \big( (X+Y+\<20>)^2 \big) \big\|_{Y^{s_2} (T)}
+ \| \Res \|_{L^1_T H^{s_2-1}_x} \notag \\
& \hphantom{X}
+ \| Y \pe \<1> \|_{L^1_T H^{s_2-1}_x}
+ \| \<21p> \|_{L^1_T H^{s_2-1}_x}
+ \| (X+Y+\<20>) \pg \<1> \|_{L^1_T H^{s_2-1}_x} \notag \\
& \hphantom{X}
+ 
\big\| M \big( \Qxy + 2 \Res + \<20>^2 + 2 \<21> + \<2> \big) (X+Y+\<20>) \big\|_{L^1_T H^{s_2-1}_x}
\notag \\
&\les
\| (Y_0, Y_1)\|_{\H^{s_2}}
+ 
T^\frac{1}{4} \Big( \|(X, Y, \Res)\|_{Z^{s_1, s_2, s_3}(T)}^2 +K^2 \Big)
+ T^{\frac{2}{3}} \| \Res \|_{L^3_T H^{s_3}_x} \notag \\
&\hphantom{X}
+
T \| \<21p> \|_{L^\infty_T H^{-\eps}_x}
+ T \Big( \| X \|_{L_T^\infty H_x^{s_1}} + \| Y \|_{L_T^\infty H_x^{s_2}} + \| \<20> \|_{L^\infty_T W^{\frac 12-\eps}_x} \Big)
\| \<1> \|_{L^\infty_T W^{-\frac 12-\eps}_x}
\label{M4} \\
&\hphantom{X}
+T^{\frac 13} \| \Qxy + 2 \Res + \<20>^2 + 2 \<21> + \<2> \|_{L_T^3 H_x^{-100}}^2
\notag \\
&\quad
\hphantom{XXX}
\times 
\Big( \| X \|_{L_T^\infty H_x^{s_1}} + \| Y \|_{L_T^\infty H_x^{s_2}} + \| \<20> \|_{L^\infty_T W_x^{\frac 12-\eps, \infty}} \Big)\notag \\
&\les
\| (Y_0, Y_1)\|_{\H^{s_2}}
+ 
T^{\frac 14}   \Big( \|(X, Y, \Res)\|_{Z^{s_1, s_2, s_3}(T)}^5 +K^5 \Big),
\notag 
\end{align}

\noi
provided that $s_1 \ge \eps$, 
$\frac 12+\eps < s_2 \le \min (1-3\eps, s_1+\frac 14, s_3+1)$, 
and $s_3 \ge -100$.

Finally, we estimate $\Phi_3(X, Y,  \Res)$ in \eqref{SNLW8}.
By 
Lemma \ref{LEM:para}, 
Lemma \ref{LEM:sto3} (in particular \eqref{A00}), 
and
\eqref{M1}
with \eqref{data4}, 
we have
\begin{align}
\| & \Phi_3  (X, Y,  \Res)\|_{L^3_T H^{s_3}_x}  \notag \\
 &\les \big\|
 \If_{\pl}^{(1)}
 \big( X+Y+\<20> \big)\pe \<1> \big\|_{L^3_T H^{s_3}_x}
+ \| \If_{\pl, \pe}
 \big( X+Y+\<20>\big) \big\|_{L^3_T H^{s_3}_x}  \notag \\
& \hphantom{X}
+ \bigg\| 
\int_0^t  M \big( \Qxy + 2 \Res + \<20>^2 + 2 \<21> + \<2> \big) 
\Ab(t, t') dt' \bigg\|_{L^3_T H^{s_3}_x} \notag \\
&\les
T^{\frac 13}
\| \If_{\pl}^{(1)}
 \big( X+Y+\<20> \big) \big\|_{L^\infty_TH^{\frac 12 + 3\eps}_x}
\| \<1> \|_{L^\infty_T W^{-\frac 12 - \eps, \infty}_x} 
+ T^{\frac 13} K \| X+Y+\<20> \|_{  L^\infty_T  L^2_x}  \notag \\
& \hphantom{X}
+
\int_0^T 
| M(\Qxy + 2 \Res + \<20>^2 + 2 \<21> + \<2>) (t')| \cdot \| \Ab(t, t')\|_{L^3_t([t', T]; H^{s_3}_x)} dt'
 \label{M5}\\
&\les
T^\frac{1}{3}
K^2 \Big( \| X \|_{L_T^\infty H_x^{s_1}} + \| Y \|_{L_T^\infty H_x^{s_2}} + \| \<20> \|_{L^\infty_T W_x^{\frac 12-\eps, \infty}} \Big) \notag \\
& \hphantom{X}
+
T^{\frac 13} K \| \Qxy + 2 \Res + \<20>^2 + 2 \<21> + \<2> \|_{L_T^3 H_x^{-100}}^2
 \notag \\
& \les T^\frac{1}{3}
K  \Big( \|(X, Y, \Res)\|_{Z^{s_1, s_2, s_3}(T)}^4 +K^4 \Big)
 \notag 
\end{align}
provided that 
$s_1 > 0 $ with sufficiently small $\eps = \eps(s_1) > 0$
(in view of Lemma \ref{LEM:sto3}),     
 $s_2 \geq \frac 12 + \eps$,  and $- 100 \le s_3\le-\eps$.

Note that  $|x|x$ is differentiable with a locally bounded derivative.
In view of \eqref{addM}, this allows us to estimate the difference
$M(w_1) - M(w_2)$.
By repeating a similar computation, we also obtain the difference estimate:
\begin{align}
\begin{split}
\| \vec \Phi &  (X, Y, \Res)
- \vec \Phi   (\wt X, \wt Y, \wt \Res)
 \|_{Z^{s_1, s_2, s_3}(T)} \\
&\les
T^{\frac 1{4}} \Big(\|(X, Y, \Res)\|_{Z^{s_1, s_2, s_3}(T)}^4 + K^4 \Big)
\|(X , Y, \Res)  - (\wt X , \wt Y, \wt \Res)\|_{Z^{s_1, s_2, s_3}(T)}, 
\end{split}
\label{M6}
\end{align}

\noi
where 
\begin{align*}
\vec \Phi := (\Phi_1, \Phi_2, \Phi_3).
\end{align*}
Therefore, by choosing $T= T(K)>0$ sufficiently small, 
we conclude from \eqref{M1a}, \eqref{M4}, \eqref{M5}, and \eqref{M6} that 
$\vec \Phi = (\Phi_1, \Phi_2, \Phi_3)$ 
is a contraction on the closed ball $B_R \subset Z^{s_1, s_2, s_3}(T)$
of radius $R\sim 1+
\| (X_0, X_1) \|_{\H^{s_1}}
+ \| (Y_0, Y_1) \|_{\H^{s_2}}$ centered at the origin.
A similar computation yields 
Lipschitz continuous dependence of the solution $(X, Y, \Res)$
on the enhanced data set $(X_0,X_1,Y_0,Y_1,\Xi)$
measured in the $\mathcal{X}^{s_1, s_2,\eps}_T$-norm
by possibly making $T > 0$ smaller.
This concludes the proof of 
Theorem~\ref{THM:1}.

\section{Invariant Gibbs dynamics}
\label{SEC:GWP}

In this section, we present the proof of Theorem \ref{THM:GWP}.
In the remaining part of this section, 
we work in the weakly nonlinear regime.
Namely, we fix $\s \ne 0$ such that $|\s| \le \s_0$,
where $\s_0$ is as in 
 Theorem~\ref{THM:Gibbs}\,(i).
We also fix sufficiently large $A\gg1 $ as in 
 Theorem~\ref{THM:Gibbs}\,(i)
 such that the $\Phi^3_3$-measure 
$\rho$ is constructed as the limit of
the truncated $\Phi^3_3$-measures $\rho_N$
in \eqref{GibbsN}.
With these parameters,
consider
the truncated Gibbs measure $\rhoo_N$:
\begin{align}
\rhoo_N = \rho_N \otimes \mu_0
\label{rhooN}
\end{align}

\noi
for $N \in \N$, 
where $\mu_0$ is the white noise measure; see \eqref{gauss0} with $s = 0$.
A standard argument \cite{GKOT, ORTz, OOTol1}
shows that 
the truncated Gibbs measure $\rhoo_N$ is invariant under 
the truncated hyperbolic $\Phi^3_3$-model \eqref{SNLW3}:
\begin{align}
\begin{split}
\dt^2 & u_N + \dt u_N  + (1 -  \Dl)  u_N  \\
& 
-\s  \pi_N \big(   :\! (\pi_N u_N)^2 \!:\, \big)
+  M (\,:\! (\pi_N u_N)^2 \!:\,) \pi_N u_N 
= \sqrt{2} \xi, 
\end{split}
\label{SNLW3a}
\end{align}

\noi
where 
$:\! (\pi_N u_N)^2 \!:  \, = 
(\pi_N u_N)^2 -\s_N$
and $\pi_N$ and $\s_N$ are as in \eqref{pi} and \eqref{sigma1}, respectively.
See Lemma \ref{LEM:GWP4} below.
Moreover, as a corollary to Theorem \ref{THM:Gibbs}\,(i), 
the truncated Gibbs measure $\rhoo_N$ in~\eqref{rhooN} converges weakly to the Gibbs measure
$\rhoo = \rho \otimes \mu_0 $ in \eqref{Gibbs2}.

Our main goal is to construct global-in-time 
dynamics for the limiting hyperbolic $\Phi^3_3$-model~\eqref{SNLW1}
almost surely with respect to the Gibbs measure $\rhoo$,
and prove invariance of the Gibbs measure $\rhoo$ under the limiting 
hyperbolic $\Phi^3_3$-dynamics.
A naive approach would be 
 to apply Bourgain's invariant measure
argument \cite{BO94, BO96}, by exploiting the invariance
of the truncated Gibbs measure $\rhoo_N$ 
under the truncated hyperbolic $\Phi^3_3$-dynamics,  
and to 
try to construct global-in-time limiting dynamics
for 
the limiting process $u = \lim_{N \to \infty} u_N$.
There are, however, two issues in the current situation:
(i) the truncated Gibbs measure $\rhoo_N$ converges  to the limiting Gibbs measure
$\rhoo$ {\it only weakly} and (ii) 
the Gibbs measure $\rhoo$ and the base Gaussian measure 
$\muu = \mu \otimes \mu_0$ in \eqref{gauss1} are {\it mutually singular}.
Moreover, our local theory relies on the paracontrolled approach, 
which gives additional difficulty.
As a result, Bourgain's invariant measure
argument \cite{BO94, BO96} is not directly applicable to our problem.
In~\cite{Bring2}, Bringmann encountered a similar problem
in the context of the defocusing Hartree NLW on $\T^3$, 
where he overcame this issue by introducing a new globalization argument, 
by using  the fact that the (truncated)  Gibbs measure
is absolutely continuous with respect to a shifted  measure (as in Appendix \ref{SEC:AC} below)
\cite{OOTol1, Bring1} in a uniform manner
and 
 establishing a (rather involved) large time stability theory, 
 where sets of large probabilities are characterized
 via the shifted measures.

In the following, we introduce a new  alternative globalization
argument.
This new argument has the advantage of being  
 conceptually  simple  and straightforward.
Our approach consists of several steps:

\smallskip
\begin{itemize}
\item[{\bf 1.}]
In the first step, we
establish a uniform (in $N$)  exponential integrability of the truncated enhanced data set $\Xi_N$
(see \eqref{data3x} below)
with respect to the truncated measure $\rhoo_N\otimes \PP_2$
(Proposition~\ref{PROP:exptail2}).
Here, 
 $\PP_2$ is the measure for the stochastic forcing
defined in \eqref{PPP} below. 
By combining the variational approach
with space-time estimates, 
we prove this uniform exponential integrability 
{\it without} any reference to (the truncated version of) the shifted measure 
 $\Law (Y(1) +\s \ZZ(1) + \W(1))$  constructed in Appendix~\ref{SEC:AC}.
As a corollary, 
we construct the limiting enhanced data set $\Xi$ associated with the Gibbs measure $\rhoo$
(see~\eqref{enh1} below)
by establishing convergence of the truncated enhanced data set $\Xi_N$
almost surely with respect to the limiting measure $\rhoo \otimes \PP_2$.

\smallskip

\item[{\bf 2.}] In the second step, we establish  a stability result (Proposition \ref{PROP:LWPv}).
We prove this stability result by a simple contraction argument, 
where we use a norm with an exponentially decaying weight in time.
As a result, the proof follows from  a small modification of 
that of the local well-posedness (Theorem \ref{THM:1}).
As compared to~\cite{Bring2}, 
our stability argument is very simple (both in terms of the statements
and the proofs).

\smallskip
\item[{\bf 3.}] In the third step, 
we establish a uniform (in $N$) control
on the solution $(X_N, Y_N, \Res_N)$
to the truncated system (see \eqref{SNLW9} below)
with respect to the truncated measure $\rhoo_N \otimes \PP_2$
 (Proposition \ref{PROP:tail2}).
The proof is based on 
the invariance of the truncated Gibbs measure $\rhoo_N$
and a discrete Gronwall argument.

\smallskip
\item[{\bf 4.}]
In the fourth step, 
we study  the pushforward measures
$(\Xi_N)_\#(\rhoo_N\otimes \PP_2)$
and 
$(\Xi)_\#(\rhoo\otimes \PP_2)$.
In particular, 
by using ideas from theory of optimal transport (the Kantorovich duality)
and the Bou\'e-Dupuis variational formula, 
we prove that the pushforward measure
$(\Xi_N)_\#(\rhoo_N\otimes \PP_2)$
converges to 
$(\Xi)_\#(\rhoo\otimes \PP_2)$
in the Wasserstein-1 distance, as $N \to \infty$;
see Proposition~\ref{PROP:plan} below.

\end{itemize}

\smallskip

Once we establish Steps 1 - 4, the proof of Theorem \ref{THM:GWP}
follows in a straightforward manner.
In Subsection \ref{SUBSEC:GWP1}, 
we first study the truncated dynamics \eqref{SNLW3a}
and briefly go over almost sure global well-posedness of \eqref{SNLW3a}
and invariance of the truncated Gibbs measure $\rhoo_N$ (Lemma~\ref{LEM:GWP4}).
We then discuss the details of Step 1 above.
In Subsection \ref{SUBSEC:GWP2}, 
we first go over the details of Steps 2, 3, and 4
and then present the proof of 
 Theorem \ref{THM:GWP}.

\medskip

\noi
{\bf Notations:}
By assumption,  the Gaussian field $\muu = \mu \otimes \mu_0$ in \eqref{gauss1}
and hence the (truncated) Gibbs measure
are independent of (the distribution of) the space-time white noise $\xi$ in \eqref{SNLW1}
and \eqref{SNLW3a}.
Hence, we can write the probability space $\Omega$
as 
\begin{align}
\O = \O_1 \times \O_2
\label{O1}
\end{align}
such that the random Fourier series in \eqref{IV2}  depend only on $\o_1 \in \O_1$, 
while the cylindrical Wiener process $W$ in \eqref{W1}
depends only on $\o_2 \in \O_2$.
In view of \eqref{O1}, 
we also write the underlying probability measure $\PP$ on $\O$
as 
\begin{align}
\PP = \PP_1 \otimes \PP_2,
\label{PPP} 
\end{align}
where $\PP_j$ is the marginal probability measure on $\O_j$, $j = 1, 2$.

With the decomposition \eqref{O1} in mind, 
we set
\begin{align}
\<1> (t;\vec u_0, \o_2)
=
S(t) \vec u_0
+ \sqrt 2 \int_0^t \D(t-t') dW(t', \o_2)
\label{enh0a}
\end{align}

\noi
for $\vec u_0 =(u_0,u_1) \in \H^{-\frac 12-\eps}(\T^3)$ and $\o_2 \in \O_2$,
where $S(t)$ and $\D(t)$ are as in \eqref{St0} and \eqref{W3}, respectively.
When it is clear from the context, 
we may suppress the dependence on $\vec u_0$ and/or $\o_2$.
Given $N \in \N$, we set
\begin{align}
\<1>_N (\vec u_0, \o_2) = \pi_N \<1> (\vec u_0, \o_2), 
\label{Ba1x}
\end{align}

\noi
where $\pi_N$ is as in \eqref{pi}. 
We also set
\begin{align}
\begin{split}
\<2>_N (\vec u_0, \o_2) &= \<1>_N^2 (\vec u_0, \o_2) - \s_N, \\
\<20>_N (\vec u_0, \o_2) &= \pi_N \I (\<2>_N (\vec u_0,\o_2)), \\
\<21p>_N (\vec u_0, \o_2)& = \<20>_N(\vec u_0, \o_2) \pe \<1>_N(\vec u_0, \o_2), 
\end{split}
\label{enh0}
\end{align}

\noi
and define $\Ab_N(\vec u_0, \o_2)$ as in \eqref{sto3a}
by replacing $\<1>_N$ with $\<1>_N (\vec u_0, \o_2)$.
We define 
the paracontrolled operator
$ \wt \If_{\pl, \pe}^N = \wt \If_{\pl, \pe}^N(\vec u_0, \o_2)$
 in a manner
analogous to $ \If_{\pl, \pe}^N$
in 
Lemma  \ref{LEM:sto4J}, 
but with an extra frequency cutoff $\pi_N$. 
Namely, instead of \eqref{X2}, 
we first define
$\wt \If^N_{\pl}$ by 
\begin{align}
\wt \If^N_{\pl}(w) (t)
    =  \I (\pi_N(w\pl \<1>_N))(t) , 
\label{enh0b}
\end{align}

\noi
where $\<1>_N = \<1>_N(\vec u_0, \o_2)$ is as in \eqref{Ba1x}.
We then define
$\wt \If^{(1), N}_{\pl}$ and $\wt \If^{(2), N}_{\pl}$
as in \eqref{X3} and \eqref{X3a} with an extra frequency cutoff  
$\chi_N(n)$, 
depending on $|n_1|\ges |n_2|^\ta$
or $|n_1|\ll |n_2|^\ta$.
Note that the conclusion of Lemma~\ref{LEM:sto3}
(in particular the estimate \eqref{A00})
holds for $\wt \If^{(1), N}_{\pl}$, 
uniformly in $N \in \N$.
Finally, we  define 
$\wt \If^N_{\pl, \pe}$ by 
\begin{align}
\wt \If^N_{\pl, \pe}(w) (t)
   =
\wt  \If_{\pl}^{(2), N}(w)\pe \<1>_N(t) , 
\label{enh0c}
\end{align}

\noi
namely, by inserting 
a  frequency cutoff  $\chi_N(n_1 + n_2)$
and replacing $\<1>$ by $\<1>_N = \<1>_N(\vec u_0, \o_2)$ in~\eqref{X7}.
We then define
the truncated enhanced data set $\Xi_N (\vec u_0, \o_2)$  by 
\begin{align}
\Xi_N (\vec u_0, \o_2)
  = \big(   \<1>_N , \,  \<2>_N, \,
   \<20>_N, \, \<21p>_N, \, \Ab_N, 
   \, \wt \If_{\pl, \pe}^N\big), 
\label{data3x}
\end{align}

\noi
where, on the right-hand side,  we suppressed the dependence on 
$(\vec u_0, \o_2)$ for notational simplicity.
Note that, 
given  $\vec u_0  \in \H^{-\frac 12-\eps}(\T^3)$, 
 the enhanced data set 
$\Xi_N (\vec u_0, \o_2)$  does not converge
in general.
Nonetheless, for the notational purpose, let us 
 {\it formally} define the (untruncated) enhanced data set $\Xi(\vec u_0, \o_2)$ by setting
\begin{align}
\Xi(\vec u_0, \o_2) = \big(
\<1>, \,  \<2>,  \,  \<20>, \<21p>,\, \Ab, \If_{\pl, \pe}\big), 
\label{enh1}
\end{align}

\noi
where each term on the right-hand side
is a limit of the corresponding term in \eqref{data3x}
(if it exists).
In Corollary \ref{COR:lim}, we will construct 
the enhanced data set $\Xi(\vec u_0, \o_2)$ in \eqref{enh1} 
as a limit of 
the truncated enhanced data set $\Xi_N(\vec u_0, \o_2)$ in \eqref{data3x}
almost surely with respect to $\rhoo \otimes \PP_2$.

In the remaining part of this section, we fix $s_1, s_2, s_3 \in \R$
satisfying 
\begin{align}
 \frac 14<s_1 < \frac 12 < s_2  <   s_1 + \frac 14
 \qquad \text{and}\qquad  
s_2 - 1 <   s_3 < 0.
\label{S1}
\end{align}

\noi
Furthermore, we take both $s_1$
and  $s_2$ to be sufficiently close to $\frac 12$
(such that the conditions in~\eqref{S2}
are satisfied, say with $r_ 1= r_2 = 3$).

\begin{remark}\label{REM:decay2}\rm
(i) In view of \eqref{Ba1x} with \eqref{pi}, we have
$\<1>_N(\vec u_0, \o_2)
= \<1>_N(\pi_{N} \vec u_0, \o_2)$
and thus
 \[\Xi_N (\vec u_0, \o_2)
= \Xi_N (\pi_{N} \vec u_0, \o_2).\]

\noi
Namely, the truncated enhanced data set $\Xi_N (\vec u_0, \o_2)$  
in \eqref{data3x} depends only on the low frequency part
$\pi_{N} \vec u_0$
of the initial data.

\smallskip

\noi
(ii)  Note that 
the terms
$ \<20>_N$,  $\<21p>_N$, and $ \wt \If_{\pl, \pe}^N$
in \eqref{data3x} 
come with an extra frequency cutoff
as compared to the corresponding terms studied  in Section \ref{SEC:LWP}.
When $\Law (\vec u_0) = \muu$, 
the results in Lemmas \ref{LEM:sto1}, 
\ref{LEM:sto2},  and \ref{LEM:sto4J}, and Remark \ref{REM:decay}
from  Subsection~\ref{SUBSEC:para}
also apply to 
$ \<20>_N(\vec u_0, \o_2)$,  $\<21p>_N(\vec u_0, \o_2)$, and $ \wt \If_{\pl, \pe}^N(\vec u_0, \o_2)$.

\smallskip

\noi
(iii) Note that the $\Xc^\eps_T$-norm
for enhanced data sets defined in  \eqref{data3}
also measures 
the time derivatives of $\<1>_N$ and $\<20>_N$
in appropriate space-time norms.
In view of \eqref{enh0} and \eqref{lin1}, 
the time derivative of $\<20>_N(\vec u_0, \o_2)$
is given by 
\[ \dt \<20>_N (t; \vec u_0, \o_2)= \pi_N \int_0^t \dt \D(t-t') \<2>_N(t'; \vec u_0, \o_2) dt'.\]

\noi
As for the stochastic convolution, 
recall that, unlike the heat or Schr\"odinger case, 
the stochastic convolution
for the damped wave equation is differentiable in time
and the time derivative of $\<1>_N(\vec u_0, \o_2)$
is given by 
\begin{align}
\dt \<1>_N (t;\vec u_0, \o_2)
=
\pi_N \dt S(t) \vec u_0
+ \sqrt 2 \pi_N \int_0^t \dt \D(t-t') dW(t', \o_2).
\label{enh9}
\end{align}

\noi
The formula \eqref{enh9} easily follows 
from  viewing  the stochastic integral in \eqref{enh0a} (with an extra frequency cutoff $\pi_N$)
 as a Paley-Wiener-Zygmund integral
and taking a time derivative.

\end{remark}

\subsection{On the truncated dynamics}
\label{SUBSEC:GWP1}

In this subsection, we 
study the truncated hyperbolic $\Phi^3_3$-model~\eqref{SNLW3a}.
We first go over  local well-posedness of the truncated equation \eqref{SNLW3a}
and then almost sure global well-posedness
and  invariance of the truncated Gibbs  measure $\rhoo_N$;
see Lemmas \ref{LEM:LWPode} and \ref{LEM:GWP4}.
Then, 
by combining  the  Bou\'e-Dupuis variational formula 
(Lemma \ref{LEM:var3})
and space-time estimates, 
we prove 
 uniform (in $N$) exponential integrability of the truncated enhanced data set 
$ \Xi_N (\vec u_0, \o_2)$
 with respect to  $\rhoo_N \otimes \PP_2$ on $(\vec u_0, \o_2)$;
see  Proposition \ref{PROP:exptail2}.
As a corollary, we prove that the truncated 
enhanced data set $\Xi_N(\vec u_0, \o_2)$ in \eqref{data3x}
converges to the limiting enhanced data set $\Xi(\vec u_0, \o_2)$ in \eqref{enh1}
almost surely with respect to the limiting measure $\rhoo\otimes \PP_2$
(Corollary \ref{COR:lim}).

Given  $N \in \N$, 
let 
 $\vec u_0 = (u_0, u_1)$ 
be a pair of random distributions such that 
$\Law( (u_0, u_1)) = \rhoo_N = \rho_N \otimes \mu_0$.
Let $u_N$ be a solution 
to the truncated equation \eqref{SNLW3a} with $(u_N, \dt  u_N) |_{t = 0} 
= \vec u_0$.
With 
$:\! (\pi_N u_N)^2 \!:  \, = 
(\pi_N u_N)^2 -\s_N$, we write \eqref{SNLW3a} as 
\begin{align}
\begin{cases}
\begin{aligned}
&\dt^2 u_N  + \dt u_N + (1 - \Dl) u_N \\
&\quad - \s \pi_N \big( ( \pi_N u_N)^2 - \s_N\big)  
+ M \big( ( \pi_N u_N)^2 -\s_N \big)
\pi_N u_N =\sqrt{2} \xi
\end{aligned}
\\
(u_N,\dt u_N)|_{ t= 0}  =\vec u_0 , 
\end{cases}
\label{SNLWode2}
\end{align}

\noi
where 
$M$ is  as in  \eqref{addM}.
Note that, due to the presence
of the frequency projector $\pi_N$,  
 the dynamics \eqref{SNLWode2} on high frequencies $\{|n| \ges N\}$
and low frequencies $\{|n|\les N\}$
are decoupled.
The high frequency part of the dynamics \eqref{SNLWode2} is given by 
\begin{align}
\begin{cases}
\begin{aligned}
&\dt^2 \pi_{N}^\perp u_N  + \dt \pi_{N}^\perp u_N + (1 - \Dl)\pi_{N}^\perp  u_N =\sqrt{2} \pi_{N}^\perp \xi
\end{aligned}
\\
(\pi_{N}^\perp u_N,\dt \pi_{N}^\perp u_N)|_{ t= 0}  =\pi_{N}^\perp  \vec u_0.
\end{cases}
\label{high1}
\end{align}

\noi 
The solution $\pi_{N}^\perp u_N$ to \eqref{high1} is given by
\begin{align}
\pi_{N}^\perp u_N 
= \pi_{N}^\perp  \<1>( \vec u_0), 
\label{high2}
\end{align}

\noi
where $\<1>(\vec u_0)$ is as in \eqref{enh0a}
with the $\o_2$-dependence suppressed.
With  $v_N = \pi_{N} u_N$, 
the low frequency part of the dynamics \eqref{SNLWode2} is given by 
\begin{align}
\begin{cases}
\begin{aligned}
&\dt^2 v_N  + \dt v_N + (1 - \Dl) v_N \\
&\quad - \s \pi_N \big( (\pi_N v_N)^2 - \s_N\big)  
+ M \big(   (\pi_N v_N)^2 -\s_N \big)
\pi_N v_N =\sqrt{2} \pi_{N}\xi
\end{aligned}
\\
(v_N,\dt v_N)|_{ t= 0}  = \pi_{N} \vec u_0 , 
\end{cases}
\label{SNLWode2p}
\end{align}

\noi
where  we kept $\pi_N$ in several places to emphasize that 
\eqref{SNLWode2p} depends only on finite many
frequencies $\{n \in NQ\}$ with $Q$ as in \eqref{Q1}.
By writing~\eqref{SNLWode2p} in the Duhamel formulation, 
we have
\begin{align}
v_N (t) = \pi_{N} S(t) \vec u_0 + \int_0^t \D(t - t') \NN_N(v_N)(t') dt'
 + \<1>_{N}(t; 0), 
\label{high3}
\end{align}

\noi
where the truncated nonlinearity $\NN_N(v_N)$ is given by 
\begin{align}
\NN_N(v_N) = \s \pi_N \big(  (\pi_N v_N)^2 - \s_N\big)  
- M \big(  (\pi_N v_N)^2 -\s_N \big)\pi_N v_N.
\label{high4}
\end{align}

\noi
and  $\<1>_{N}(t; 0)$ is as in \eqref{Ba1x} with $\vec u_0 = 0$:
\begin{align*}
\<1>_{N}(t; 0, \o_2) = \sqrt 2 \int_0^t \D(t-t')  \pi_{N} dW(t', \o_2).
\end{align*}

\noi
For each fixed $N \in \N$, 
we have 
 $\<1>_{N} (t; 0) = \pi_{N} \<1> (t; 0) \in C^1(\R_+; C^\infty(\T^3))$; see Remark \ref{REM:decay2}.
 By viewing   $\<1>_{N} (t; 0)$  in \eqref{high3} as a perturbation, 
it suffices to study 
the following damped NLW with a deterministic perturbation:
\begin{align}
v_N (t) = \pi_{N} S(t) (v_0, v_1) + \int_0^t \D(t - t') \NN_N(v_N)(t') dt'
 +  F, 
\label{high6}
\end{align}

\noi
where $(v_0,v_1) \in \H^1(\T^3)$, $\s_N$ is as in \eqref{sigma1}, and 
$F \in C^1(\R_+; C^\infty(\T^3))$ is a given deterministic function.

A standard contraction argument
with the one degree of smoothing from the Duhamel integral operator $\I$ in \eqref{lin1}
and Sobolev's inequality yields the following local well-posedness
of \eqref{high6}.
Since the argument is standard, we omit details.
See, for example, the proof of Lemma 9.1 in \cite{OOTol1}.

\begin{lemma} \label{LEM:LWPode}
Let  $N \in \N$.
Given any 
$(v_0,v_1) \in \H^1(\T^3)$ and $F  \in C^1([0,1]; H^1 (\T^3))$ with
\[
\| (v_0,v_1) \|_{\H^1} \le R
\qquad \text{and}\qquad 
\| F \|_{C^1([0, 1]; H^1)} \le K
\]
for some $R,K\ge 1$, 
 there exist $\tau= \tau (R,K, N)>0$ and a unique solution $v_N$ 
 to \eqref{high6} on $[0,\tau]$,
 satisfying the bound\textup{:}
 \[ \| v_N\|_{\wt X^1(\tau)} \les R + K, \]

\noi
where
\begin{align}
\wt X^1(\tau) =  C([0,\tau];H^1(\T^3))\cap C^1([0,\tau]; L^2(\T^3)).
\notag
\end{align}
Moreover, the solution $v_N$ is unique 
in $\wt X^1(\tau)$.

\end{lemma}

\begin{remark}\label{REM:uniq}\rm
(i) A standard contraction argument gives 
 $\tau= \tau (R,K, N) 
 \sim( R + K + N)^{-\ta}$ 
 for some $\ta > 0$, 
 in particular the local existence 
 depends on $N \in \N$.

\smallskip

\noi
(ii) 
We also point out that the uniqueness statement
for $v_N$  in Lemma \ref{LEM:LWPode} is unconditional, 
namely, the uniqueness of the solution $v_N$ holds in the entire class
$\wt X^1(\tau)$.
Then, from~\eqref{high2} and 
the unconditional uniqueness of the solution $v_N = v_N(\pi_{N} \vec u_0)$ to \eqref{SNLWode2p}, 
we obtain the {\it unique} representation of $u_N$:
\[ u_N =  \pi_{N}^\perp  \<1>( \vec u_0) +  v_N(\pi_{N} \vec u_0).\]

\noi
See for example \eqref{sol1} below, 
where we use a different representation of $u_N$.

\end{remark}

Before proceeding further, let us introduce some notations.
Given the cylindrical Wiener process  $W$ in~\eqref{W1}, 
by possibly enlarging the probability space $\Omega_2$, 
there exists a family of translations $\tau_{t_0} : \O_2 \to \O_2$ such that
\[
W(t, \tau_{t_0}(\o_2))
=  W(t+t_0, \o_2) - W(t_0,\o_2)
\]

\noi
for $t, t_0 \ge 0$ and $\o_2 \in \O_2$.
Denote by $\Phi^N(t)$ the stochastic flow map 
to the truncated hyperbolic $\Phi^3_3$-model~\eqref{SNLW3a}
constructed in Lemma \ref{LEM:LWPode}
(which is not necessarily global at this point).
Namely, 
\begin{align}
\begin{split}
\vec u_N(t) 
 = ( u_N(t), \dt u_N(t))
&  = \Phi^N(t) (\vec u_0,\o_2)\\
& =  \big(\Phi^N_1(t)(\vec u_0, \o_2),  \Phi^N_2(t)(\vec u_0, \o_2)\big)
\end{split}
\label{sol0}
\end{align}

\noi
is the solution 
 to \eqref{SNLW3a} with 
 $\vec u_N |_{t=0}= \vec u_0$, satisfying $\Law (\vec u_0) = \rhoo_N$, 
 and the noise $\xi(\o_2)$.
We now 
extend $\Phi^N(t)$ as
\begin{align}
\ft \Phi^N(t) (\vec u_0, \o_2)
= \big( \Phi^N(t)(\vec u_0, \o_2), \tau_t (\o_2) \big).
\label{Phi9}
\end{align}

\noi
Note that 
by the uniqueness of the solution to \eqref{SNLW3a},
we have
\begin{align*}
\Phi^N (t_1+t_2) (\vec u_0, \o_2)
= \Phi^N (t_2) \big( \Phi^N (t_1) (\vec u_0, \o_2), \tau_{t_1}(\o_2) \big)
= \Phi^N (t_2) \big(\ft \Phi^N (t_1) (\vec u_0, \o_2)\big)
\end{align*}

\noi
for $t_1, t_2 \ge 0$
as long as the flow is well defined.

\medskip

By writing the truncated dynamics \eqref{SNLW3a}
as a superposition of the deterministic NLW:
\begin{align}
&\dt^2 u_N   + (1 - \Dl) u_N 
 - \NN_N(u_N) = 0, 
\label{high7}
\end{align}

\noi
where $\NN_N(u_N)$ is as in \eqref{high4}, 
and the Ornstein-Uhlenbeck process (for $\dt u_N$):
\begin{align}
\dt (\dt u_N)  = - \dt u_N + \sqrt{2} \xi, 
\label{high8}
\end{align}

\noi
we see that 
the truncated Gibbs measure $\rhoo_N$ in \eqref{rhooN}
is formally\footnote{Namely, as long as the dynamics is well defined.} 
invariant under the dynamics of \eqref{SNLW3a}, 
since $\rhoo_N$ is invariant under the NLW dynamics \eqref{high7}, 
while the white noise measure $\mu_0$ on $\dt u_N$
(and hence $\rhoo_N = \rho_N \otimes \mu_0$ 
on $(u_N, \dt u_N)$) is invariant under the 
 Ornstein-Uhlenbeck dynamics \eqref{high8}.
Then, by exploiting the formal invariance of the truncated Gibbs measure $\rhoo_N$, 
Bourgain's invariant measure argument \cite{BO94}
yields the following result on 
almost sure global well-posedness of 
 the truncated hyperbolic $\Phi^3_3$-model~\eqref{SNLW3a}
and invariance of the truncated Gibbs measure $\rhoo_N$.
Since the argument is standard (for fixed $N\in \N$), we omit details.
See the proof of Lemma 9.3  in \cite{OOTol1} for details.

\begin{lemma} \label{LEM:GWP4}
Let $N \in \N$.
Then, the truncated hyperbolic $\Phi^3_3$-model~\eqref{SNLW3a} is almost surely globally well-posed
with respect to the random initial data distributed
by the truncated Gibbs measure $\rhoo_N$ in~\eqref{rhooN}.
Furthermore, $\rhoo_N$ is invariant under the resulting dynamics
and, as a consequence,  the measure $\rhoo_N \otimes \PP_2$ is invariant
under the extended stochastic flow map $\ft \Phi^N(t)$ defined in \eqref{Phi9}.
More precisely, there exists $\Si_N \subset \O = \O_1 \times \O_2$ 
with $\rhoo_N \otimes \PP_2 (\Si_N) = 1$ such that the solution 
$u_N = u_N (\vec u_0, \o_2) $ to \eqref{SNLW3a} exists globally in time 
and $\Law (u_N(t), \dt u_N(t)) = \rhoo_N$ for any $t \in \R_+$.
\end{lemma}

Next, we establish 
 uniform  exponential integrability of the truncated enhanced data set 
 $\Xi_N (\vec u_0, \o_2)$ in \eqref{data3x}
 with respect to  the truncated measure $\rhoo_N \otimes \PP_2$.
 We also establish 
 uniform  exponential integrability 
 for the difference of  the truncated enhanced data sets.

\begin{proposition} \label{PROP:exptail2}
Let $T>0$.
Then, 
we have
\begin{align}
\int \E_{\PP_2} \Big[ \exp \Big( \| \Xi_N (\vec u_0, \o_2) \|_{\mathcal{X}^{\eps}_T}^\alpha \Big) 
\Big]d \rhoo_N (\vec u_0)
\le C(T, \eps, \al) < \infty
\label{exp01}
\end{align}
for $0< \al < \frac 13$,
uniformly in $N \in \N$,
where the $\mathcal{X}^\eps_T$-norm and 
the truncated enhanced data set $\Xi_N (\vec u_0, \o_2)$ are as in \eqref{data3} and \eqref{data3x}, respectively.
Here, 
$\E_{\PP_2}$ denotes an expectation 
with respect to 
the  probability measure $\PP_2$ on $\o_2 \in \O_2$
defined in \eqref{PPP}. 

Moreover, there exists small $\be > 0$ such that  
\begin{align}
\int \E_{\PP_2} \Big[ \exp \Big( N_2^\be \| \Xi_{N_1} (\vec u_0, \o_2)
- \Xi_{N_2} (\vec u_0, \o_2)
 \|_{\mathcal{X}^{\eps}_T}^\alpha \Big) \Big]d \rhoo_N (\vec u_0)
\le C(T, \eps, \al) < \infty
\label{exp02}
\end{align}
for $0< \al < \frac 13$,
uniformly in $N, N_1, N_2 \in \N$ with $N \ge N_1 \ge N_2$.

\end{proposition}

\begin{proof}
For simplicity, we only  prove \eqref{exp01} and \eqref{exp02}
for 
 the random operator $\wt \If_{\pl, \pe}^N$ defined in~\eqref{enh0c}.
The other terms in $\Xi_N(\vec u_0, \o_2)$ can be estimated
in an analogous manner.  See Remark~\ref{REM:3}.

\smallskip

\noi
$\bullet$ {\bf Part 1:}
We first  prove the following uniform exponential integrability:
\begin{equation}
\int \E_{\PP_2} \Big[ \exp \Big( \big\| \wt \If_{\pl, \pe}^N\big \|_{ \L_2(q, T)}^\alpha \Big) 
\Big]d \rhoo_N (\vec u_0)
\le C(T, \eps, \al) 
< \infty \label{exp}
\end{equation}
for any $T>0$,   any finite $q>1$,  and $0< \alpha < \frac 12$,
uniformly in $N \in \N$.
Note that the range $0 < \al< \frac 12$  of the exponent  
in \eqref{exp}
comes from the presence
of $\| \ZZ_N \|_{W^{1-\eps, \infty}}^2$ in~\eqref{et5} and \eqref{et7}, 
since  $\ZZ_N$ defined in one line below \eqref{YZ12}
belongs to $\H_{\le 2}$.
Similarly, the overall restriction 
$0 < \al < \frac 13$ in this proposition comes
from the terms involving $\psi_1$ in 
\eqref{et9b}, where $\psi_1$ is defined in 
\eqref{et2b} with 
\eqref{Dr1}.
Namely, the worst contribution in \eqref{et9b}
behaves like
$\| \ZZ_N \|_{W^{1-\eps, \infty}}^{3\al}$ 
which is exponentially integrable only for $\al < \frac 13$;
see \eqref{et9bb}.

From \eqref{enh0b} and \eqref{enh0c}, 
we see that $\wt \If_{\pl, \pe}^N$ depends
on two entries of $\<1>_N = \pi_N\<1>(\vec u_0, \o_2)$.
We now generalize the definition of $\wt \If_{\pl, \pe}^N$ to allow
general entries.
Given $\psi_j \in C(\R_+; \D'(\T^3))$, $j = 1, 2$, we first define
$\wt \If^N_{\pl}[\psi_1]$ by 
\begin{align}
\wt \If^N_{\pl}[\psi_1](w) 
    =  \I \big(\pi_N(w\pl (\pi_N \psi_1))\big).
\label{et1a}
\end{align}

\noi
As  in \eqref{X3} and \eqref{X3a},  
define  $\wt \If^{(2), N}_{\pl}[\psi_1]$  to be 
 the restriction of  $\wt \If^{N}_{\pl}[\psi_1]$
onto  $\{|n_1|\ll |n_2|^\ta\}$:
\begin{align}
\wt \If^{(2), N}_{\pl}[\psi_1](w)
=   \I \big(\pi_N( \K^\ta (w,  \pi_N \psi_1))\big), 
\label{et1aa}
\end{align}

\noi
where $\K^\ta$ is the bilinear Fourier multiplier operator
with the multiplier $\ind_{\{|n_1| \ll |n_2|^\ta\}}$.
More precisely, we have 
\begin{align}
\begin{split}
\wt  \If_{\pl}^{(2), N}[\psi_1] (w) (t)
 &   =  \sum_{n \in \Z^3}\chi_N(n)
e_n  \sum_{n =  n_1 +  n_2}
\sum_{0 \le j  <  \ta k + c_0 }
\varphi_j(n_1) \varphi_k(n_2) 
\chi_N(n_2)
 \\
& \hphantom{XXXXX}
\times 
\int_0^t e^{-\frac{t-t'}2} \frac{\sin ((t - t') \jbb{n})}{\jbb{n}} 
\ft w(n_1, t')\,  \ft{\psi}_1(n_2, t') dt', 
\end{split}
\label{et1ab}
\end{align}

\noi
where $\chi_N$ is as in \eqref{chi} and $c_0 \in \R$ is as in \eqref{XX2a}.
Then, 
 we define
$\wt \If^N_{\pl, \pe}[\psi_1, \psi_2]$ by 
\begin{align}
\wt \If^N_{\pl, \pe}[\psi_1, \psi_2](w) 
   =
\wt  \If_{\pl}^{(2), N}[\psi_1](w)\pe (\pi_N \psi_2).
\label{et1b}
\end{align}

\noi
Note that 
$\wt \If^N_{\pl, \pe}[\psi_1, \psi_2]$ is bilinear in $\psi_1$ and $\psi_2$.
We also set 
\begin{align}
\wt \If_{\pl, \pe}^N [\psi] = \wt \If_{\pl, \pe}^N [\psi,\psi]
\label{et2}
\end{align}
for simplicity.
With this notation, 
we can write  $\wt \If_{\pl, \pe}^N$ in \eqref{exp} as 
$\wt \If_{\pl, \pe}^N[\<1>(\vec u_0,\o_2)]$, 
where  $\vec u_0 =(u_0,u_1)$.
Note that we have
$\wt \If_{\pl, \pe}^N [\pi_N \psi]  = \wt \If_{\pl, \pe}^N [\psi] $.
Before proceeding further, 
we  record the following boundedness
of  $\K^\ta$ defined  in \eqref{et1aa} and \eqref{et1ab};
 a slight modification of the proof of~\eqref{para2a} in Lemma \ref{LEM:para}
yields
\begin{align}
\| \K^\ta(f, g) \|_{B^{s_2}_{p, q}} \les 
\|f \|_{L^{p_1}} 
\|  g \|_{B^{s_2}_{p_2, q}}
\label{para4}
\end{align}

\noi
for any $ s_2 \in \R$ and $1 \leq p, p_1, p_2, q \leq \infty$ such that 
$\frac{1}{p} = \frac 1{p_1} + \frac 1{p_2}$.

By the Bou\'e-Dupuis variational formula (Lemma \ref{LEM:var3})
 with the change of variables \eqref{YZ13},
we have
\begin{align*}
& - \log \int \exp \Big( \big\|\wt  \If_{\pl, \pe}^N [ \<1>(\vec u_0, \o_2)] \big\|_{ \L_2(q, T)}^\alpha \Big) d \rho_N(u_0) \\
&= \inf_{\dot \Ups^N \in \Ha^1}
\E \bigg[ -\big\| \wt \If_{\pl, \pe}^N [ \<1>(Y+\Dr, u_1, \o_2)]\big \|_{ \L_2(q, T)}^\alpha \\
&\hphantom{XXXXXXXX}
+ \ft R_N^\dia (Y+\Ups^N+\s \ZZ_N)
+ \frac 12 \int_0^1 \| \dot \Ups^N(t) \|_{H_x^1}^2 dt \bigg]
+ \log Z_N,
\end{align*}

\noi
where $\ft R_N^\dia$ is as in \eqref{KZ16} and 
\begin{align}
\Dr = \Ups^N + \s \ZZ_N.
\label{Dr1}
\end{align}

\noi
Recall the notation $Y_N = \pi_N Y$ and $\Ups_N = \pi_N \Ups^N$.
Then,
from Lemmas  \ref{LEM:Dr7} and \ref{LEM:Dr8}
with Lemma~\ref{LEM:Dr} and \eqref{YZ15}, 
there exists $\eps_0, C_0 > 0$ such that 
\begin{align}
\begin{split}
& - \log \int \exp \Big( \big\| \wt \If_{\pl, \pe}^N [ \<1>(\vec u_0, \o_2)] \big\|_{ \L_2(q, T)}^\alpha \Big) d \rho_N(u_0) \\
& \ge \inf_{\dot \Ups^N \in \Ha^1}
\E \Big[ -\big\| \wt \If_{\pl, \pe}^N [ \<1>(Y+\Dr, u_1, \o_2)] \big\|_{ \L_2(q, T)}^\alpha
+ \eps_0 \big( \| \Ups^N \|_{H^1}^2 + \| \Ups_N \|_{L^2}^6\big) \Big] - C_0,
\end{split}
\label{et2a}
\end{align}

\noi
uniformly in $u_1$ and $\o_2$.

In view of \eqref{enh0a}, 
we write $\<1>(Y+\Dr, u_1,\o_2)$
as 
\begin{align}
\<1>(Y+\Dr, u_1,\o_2)
= \<1>(Y,u_1,\o_2) + S(t) (\Dr,0)
=: \psi_0 + \psi_1, 
\label{et2b}
\end{align}

\noi
where $S(t)$ is as in \eqref{St0}.
By \eqref{et2},
we have 
\begin{align}
\begin{split}
\big\| \wt \If_{\pl, \pe}^N [ \<1>(Y+\Dr, u_1, \o_2)] \big\|_{ \L_2(q, T)}
&\le
\big\| \wt \If_{\pl, \pe}^N [ \psi_0, \psi_0] \big\|_{ \L_2(q, T)}
+ \big\|\wt  \If_{\pl, \pe}^N [ \psi_0, \psi_1] \big\|_{ \L_2(q, T)} \\
&\quad
+ \big\| \wt \If_{\pl, \pe}^N [ \psi_1, \psi_0] \big\|_{ \L_2(q, T)}
+ \big\| \wt \If_{\pl, \pe}^N [ \psi_1, \psi_1] \big\|_{ \L_2(q, T)}.
\end{split}
\label{et3}
\end{align}

\noi
Under the truncated Gibbs measure $\rhoo_N$, 
we have $\Law(u_1) = \mu_0$
and thus we have $\Law(Y, u_1) = \muu = \mu \otimes \mu_0$. 
Then, 
from
the uniform exponential tail estimates
in 
 Lemmas \ref{LEM:stoconv0} and \ref{LEM:sto4J}
 (see also Remark \ref{REM:decay2})
with~\eqref{YZ12}, 
there exists $K(Y, u_1, \o_2)$ such that 
\begin{align}
\big\| \wt \If_{\pl, \pe}^N [ \psi_0]\big \|_{ \L_2(q, T)}
+
\|  \psi_0 \|_{L_T^\infty W_x^{-\frac 12-\eps,\infty}}^2
+ \| \ZZ_N \|_{W^{1-\eps, \infty}}
\le K(Y, u_1, \o_2)
\label{et4}
\end{align}
and 
\begin{align}
\E_{\muu\otimes \PP_2} \big[ \exp(\delta K(Y, u_1, \o_2)) \big] < \infty
\label{et4a}
\end{align}
for sufficiently small $\delta>0$.

We now estimate the last three terms on the right-hand side of \eqref{et3}.
Let $s_3 < 0$. By Sobolev's inequality, \eqref{et1b}, 
H\"older's inequality,\footnote{To be more precise, this is the Coifman-Meyer
theorem on $\T^3$ to estimate a resonant product. 
The  Coifman-Meyer theorem on $\T^3$
 follows from 
the Coifman-Meyer theorem for functions on $\R^d$
\cite[Theorem 7.5.3]{Graf}
and
the transference principle \cite[Theorem 3]{FS}.
We may equally proceed with \eqref{para3} in Lemma \ref{LEM:para}
with a slight loss of derivative which does not affect the estimate. 
} 
\eqref{et1aa}, Sobolev's inequality, 
Lemma \ref{LEM:Str},  and \eqref{para4} with \eqref{et2b}, 
we have
\begin{align}
\begin{split}
\big\| \wt \If_{\pl, \pe}^N [ \psi_0, \psi_1](w) \big\|_{L_T^\infty H_x^{s_3}}
&\les
\big\| \wt  \If_{\pl}^{(2), N}[\psi_0](w)\pe (\pi_N \psi_1)
 \big\|_{L_T^\infty L_x^{\frac 6{3-2s_3}}} \\
&\les
\big\| \wt  \If_{\pl}^{(2), N}[\psi_0](w)\big\|_{L_T^\infty L_x^{\frac 3{1-s_3-\eps}}}
 \| \pi_N \psi_1 \|_{L_T^\infty L_x^{\frac 6{1+2\eps}}} \\
&\les
\| \I( \K^\ta (w,  \pi_N \psi_0))\|_{L_T^\infty  H_x^{s_3+\frac 12+\eps}} 
\| \psi_1 \|_{L_T^\infty H_x^{1-\eps}} \\
&\les
\|  \K^\ta (w,  \pi_N \psi_0) \|_{L_T^1 H_x^{s_3-\frac 12+\eps}} 
\| \Dr \|_{ H^{1-\eps}} \\
&\les
\| w \|_{L_T^{1} L_x^2} \|  \psi_0 \|_{L_T^\infty W_x^{-\frac 12-2\eps,\infty}}
\| \Dr \|_{ H^{1-\eps}}, 
\end{split}
\label{et4b}
\end{align}

\noi
for $\eps > 0$ sufficiently small such that  $4\eps \le -s_3$.
Hence, by the definition \eqref{L1ast} of the $\L(q, T)$-norm, 
Cauchy's inequality, and \eqref{Dr1}, 
we obtain
\begin{align}
\begin{split}
\big\| \wt \If_{\pl, \pe}^N [ \psi_0, \psi_1] \big\|_{ \L_2(q, T)}
&\les T^\frac{q-1}{q}
\|  \psi_0 \|_{L_T^\infty W_x^{-\frac 12-2\eps,\infty}} \| \Dr \|_{H^{1-\eps}}\\
& \les T^\frac{q-1}{q}\Big(
\|  \psi_0 \|_{L_T^\infty W_x^{-\frac 12-\eps,\infty}}^2
+ \| \Ups^N \|_{H^1}^2
+ \| \ZZ_N \|_{W^{1-\eps, \infty}}^2\Big).
\end{split}
\label{et5}
\end{align}

\noi
Proceeding as in \eqref{et4b} and applying Sobolev's embedding theorem 
with \eqref{Dr1}
and  \eqref{et2b}, we have
\begin{align}
\begin{split}
\big\| \wt \If_{\pl, \pe}^N [ \psi_1, \psi_1] \big\|_{ \L_2(q, T)}
&\les T^\frac{q-1}{q}
\| \psi_1 \|_{L_T^\infty W_x^{-\frac 12-2\eps,\infty}} \| \Dr \|_{H^{1-\eps}}
\les T^\frac{q-1}{q}\| \Dr \|_{H^{1-\eps}}^2 \\
&\les T^\frac{q-1}{q}\Big(
\| \Ups^N \|_{H^1}^2
+ \| \ZZ_N \|_{W^{1-\eps, \infty}}^2\Big).
\end{split}
\label{et6}
\end{align}

\noi
Finally, from 
Lemma \ref{LEM:para}, Lemma \ref{LEM:Str}, Sobolev's inequality, 
and \eqref{para4}, we have 
\begin{align}
\begin{split}
\big\| \wt \If_{\pl, \pe}^N [ \psi_1, \psi_0](w) \big\|_{L_T^\infty H_x^{s_3}}
&\le
\big\| \wt  \If_{\pl}^{(2), N}[\psi_1](w)\pe (\pi_N \psi_0) \big\|_{L_T^\infty L_x^2} \\
&\les
\| \I (\K^\ta (w,  \pi_N \psi_1)) \|_{L_T^\infty H_x^{\frac 12+2\eps}}
\| \psi_0 \|_{L_T^\infty W_x^{-\frac 12-\eps, \infty}} \\
&\les
\|\K^\ta( w,  \pi_N \psi_1) \|_{L_T^1 H_x^{-\frac 12+2\eps}}
\| \psi_0 \|_{L_1^\infty W_x^{-\frac 12-\eps, \infty}} \\
&\les
\| \K^\ta (w,  \pi_N \psi_1) \|_{L_T^1 L_x^{\frac 3{2-2\eps}}}
\| \psi_0 \|_{L_T^\infty W_x^{-\frac 12-\eps, \infty}} \\
&\les
\| w \|_{L_T^{q} L_x^2} \|  \psi_1 \|_{L_T^{q'} B^0_{\frac 6{1-4\eps}, 2}}
\| \psi_0 \|_{L_T^\infty W_x^{-\frac 12-\eps,\infty}}.
\end{split}
\label{et6a}
\end{align}
Note that $( \frac 1{3 \eps}, \frac{6}{1-4\eps})$ is $(1-\eps)$-admissible.
Since $q>1$, 
we can choose  $\eps>0$ sufficiently small
such that  $q' \le \frac 1{3\eps}$.
Then, by Minkowski's integral inequality,   \eqref{et2b},  and Lemma \ref{LEM:Str}, 
we have 
\begin{align}
\|  \psi_1 \|_{L_T^{q'} B^0_{\frac 6{1-4\eps}, 2}}
\le 
\bigg(\sum_{j = 0}^\infty \|  S(t) (\P_j \Dr,0) \|_{L_T^{q'} L^\frac 6{1-4\eps}_x}^2\bigg)^\frac 12
\les  \|\Dr\|_{H^{1-\eps}}, 
\label{et6b}
\end{align}

\noi
where $\P_j$ is the Littlewood-Paley projector
onto the frequencies $\{|n|\sim 2^j\}$.
Hence, from \eqref{L1ast},  \eqref{et6a}, \eqref{et6b}, 
and Cauchy's inequality with \eqref{Dr1}, 
we obtain
\begin{align}
\begin{split}
\big\| \wt \If_{\pl, \pe}^N [ \psi_1, \psi_0] \big\|_{ \L_2(q, T)}
& \le C(T)
\| \psi_0 \|_{L_T^\infty W_x^{-\frac 12-\eps,\infty}} 
\| \Dr \|_{H^{1-\eps}} \\
&\le C(T)
\Big(\|  \psi_0 \|_{L_T^\infty W_x^{-\frac 12-\eps,\infty}}^2
+ \| \Ups^N \|_{H^1}^2
+ \| \ZZ_N \|_{W^{1-\eps, \infty}}^2 \Big).
\end{split}
\label{et7}
\end{align}

By \eqref{et3}, \eqref{et4}, \eqref{et5}, \eqref{et6}, \eqref{et7},
and Young's inequality (with $\al < 1$)
we have
\begin{align}
\begin{split}
\inf_{\dot \Ups^N \in \Ha^1}& 
\E \Big[ -\big\| \wt \If_{\pl, \pe}^N [ \<1>(Y+\Dr, u_1, \o_2)] \big\|_{ \L_2(q,T)}^\alpha
+ \eps_0 \big( \| \Ups^N \|_{H^1}^2 + \| \Ups_N \|_{L^2}^6\big) \Big] \\
&
\ge 
- c \E \Big[ K(Y,u_1,\o_2)^{2\al} \Big]
+ \inf_{\dot \Ups^N \in \Ha^1}
 \Big( - c \| \Ups^N \|_{H^1}^{2\alpha} + \eps_0  \| \Ups^N \|_{H^1}^2 \Big) -C_1 \\
& \gtrsim - \E \Big[ K(Y,u_1,\o_2)^{2\al} \Big] - C_2.
\end{split}
\label{et8}
\end{align}

\noi
Therefore, 
from \eqref{et2a}, \eqref{et8}, Young's inequality, and Jensen's inequality,
we obtain 
\begin{align*}
\int \exp \Big( \big\| \If_{\pl, \pe}^N [ \<1>(\vec u_0, \o_2)] \big\|_{ \L_2(q, 1)}^\alpha \Big) d \rho_N(u_0) 
&\les \exp \Big( C \E \big[ K(Y,u_1,\o_2)^{2\al} \big]  \Big) \\
&\le \exp \Big( \dl \E \big[ K(Y,u_1,\o_2) \big]  \Big) \\
&\le \int  \exp \big( \dl  K(Y,u_1,\o_2)  \big) d\mu(Y)
\end{align*}
for $0<\al< \frac 12$.
Finally, by integrating in $(u_1, \o_2)$
with respect to $\mu_2 \otimes \PP_2$, 
we obtain 
 the desired bound \eqref{exp}
 from 
\eqref{et4a}.

\smallskip

\noi
$\bullet$ {\bf Part 2:}
Next, we briefly discuss how to prove \eqref{exp02}
for the random operator $\wt \If_{\pl, \pe}^N$.
For $N \ge N_1 \ge N_2 \ge 1$, 
proceeding as in Part 1, we arrive at
\begin{align*}
 -&  \log \int \exp \Big( N_2^\be \big\| 
 \wt \If_{\pl, \pe}^{N_1} [ \<1>(\vec u_0, \o_2)]- \wt \If_{\pl, \pe}^{N_2} [ \<1>(\vec u_0, \o_2)] \big\|_{ \L_2(q, T)}^\alpha \Big) d \rho_N(u_0) \\
& \ge \inf_{\dot \Ups^N \in \Ha^1}
\E\Big[ -
N_2^\be 
\big\| \wt \If_{\pl, \pe}^{N_1} [ \<1>(Y+\Dr, u_1, \o_2)]
- \wt \If_{\pl, \pe}^{N_2} [ \<1>(Y+\Dr, u_1, \o_2)] \big\|_{ \L_2(q, T)}^\alpha\\
& \hphantom{XXXXXXi}
+ \eps_0 \big( \| \Ups^N \|_{H^1}^2 + \| \Ups_N \|_{L^2}^6\big) \Big] - C_0,
\end{align*}

\noi
uniformly in $u_1$ and $\o_2$. See \eqref{et2a}.
With $\psi_0$ and $\psi_1$ as in \eqref{et2b}, we write 
\begin{align}
\begin{split}
 N_2^\frac{\be}{\al}  & 
\big\| \wt \If_{\pl, \pe}^{N_1} [ \<1>(Y+\Dr, u_1, \o_2)]
- \wt \If_{\pl, \pe}^{N_2} [ \<1>(Y+\Dr, u_1, \o_2)] \big\|_{ \L_2(q, T)}\\
&\le
N_2^\frac{\be}{\al} 
\big\| \wt \If_{\pl, \pe}^{N_1}  [ \psi_0, \psi_0]  - \wt \If_{\pl, \pe}^{N_2}  [ \psi_0, \psi_0] \big\|_{ \L_2(q, T)}\\
&\quad
+ N_2^\frac{\be}{\al} \big\|\wt \If_{\pl, \pe}^{N_1}  [ \psi_0, \psi_1] - \wt \If_{\pl, \pe}^{N_2}  [ \psi_0, \psi_1] \big\|_{ \L_2(q, T)} \\
&\quad
+ N_2^\frac{\be}{\al} \big\| \wt \If_{\pl, \pe}^{N_1}  [ \psi_1, \psi_0] -  \wt \If_{\pl, \pe}^{N_2} [ \psi_1, \psi_0] \big\|_{ \L_2(q, T)}\\
&\quad
+ N_2^\frac{\be}{\al} \big\| \wt \If_{\pl, \pe}^{N_1}  [ \psi_0, \psi_1]  - \wt \If_{\pl, \pe}^{N_2} [ \psi_1, \psi_1] \big\|_{ \L_2(q, T)}.
\end{split}
\label{Xet3}
\end{align}

\noi
In view of Remark \ref{REM:decay2} 
(see also Lemma \ref{LEM:sto4J} and Remark \ref{REM:decay}), 
we see that 
there exists $K(Y, u_1, \o_2)$ such that 
\begin{align}
\begin{split}
N_2^\frac{\be}{\al} 
\big\| & \wt \If_{\pl, \pe}^{N_1}  [ \psi_0, \psi_0] 
 - \wt \If_{\pl, \pe}^{N_2}  [ \psi_0, \psi_0] \big\|_{ \L_2(q, T)}\\
& +
\|  \psi_0 \|_{L_T^\infty W_x^{-\frac 12-\eps,\infty}}^2
+ \| \ZZ_N \|_{W^{1-\eps, \infty}}
\le \wt K(Y, u_1, \o_2)
\end{split}
\label{Xet4}
\end{align}
and 
\begin{align}
\E_{\muu\otimes \PP_2} \big[ \exp(\delta \wt K(Y, u_1, \o_2)) \big] < \infty
\label{Xet4a}
\end{align}
for sufficiently small $\delta>0$, 
provided that $\be > 0$ is sufficiently small.
The last three terms on the right-hand side of \eqref{Xet3}
can be handled as in 
\eqref{et5},  \eqref{et6}, 
and \eqref{et7}.
By noting that one of the factors comes with $\pi_{N_1} - \pi_{N_2}$, 
we gain a small negative power of $N_2$ by losing 
small regularity
in \eqref{et5},  \eqref{et6}, 
and \eqref{et7}, while keeping the resulting regularities
on the right-hand sides unchanged.
This  allows us to hide 
$N_2^\frac{\be}{\al} $ in \eqref{Xet3}.
The rest of the argument follows
precisely as in Part 1.
\end{proof}

\begin{remark}\label{REM:3} \rm
In the proof of Proposition \ref{PROP:exptail2}, 
we only treated $\wt \If_{\pl, \pe}^N$ from the truncated enhanced data set
$\Xi_N(\vec u_0, \o_2)$ in \eqref{data3x}.
Let us briefly discuss how to treat the other terms
in $\Xi_N(\vec u_0, \o_2)$ to get the exponential integrability bound \eqref{exp01}.
The second bound \eqref{exp02} follows in a similar manner.
The terms 
$\<1>_N$, $\<2>_N$,  $ \<20>_N$, and $\Ab_N$
can be estimated in a similar manner
since they are (at most) quadratic in $\<1>(Y+\Dr, u_1, \o_2)$
and the product $\psi_0 \psi_1$ is well defined, 
where $\psi_j$, $j = 0, 1$, is as in \eqref{et2b}.

As for $\<21p>_N$, with the notation above and \eqref{et2b}, 
we have 
\begin{align}
\begin{split}
\<21p>_N[\<1>(Y+\Dr, u_1, \o_2)]
& = \<21p>_N[\psi_0 + \psi_1]\\
& = \<20>_N[\psi_0 + \psi_1]\pe(\pi_N \psi_0)
+ \<20>_N[\psi_0 + \psi_1]\pe (\pi_N \psi_1).
\end{split}
\label{et9a}
\end{align}

\noi
Let $0 < \al < \frac 13$.
Then, by  Lemma \ref{LEM:para}
and 
Young's inequality, 
we can estimate
the second term on the right-hand side  as 
\begin{align}
\begin{split}
\| \<20>_N[\psi_0 + \psi_1]\pe (\pi_N \psi_1)\|_{C_T H^{-\eps}_x}^\al
& \les 
\| \<20>_N[\psi_0 + \psi_1]\|_{C_T W^{\frac 12 -\eps, \infty}_x}^\al
\| \psi_1\|_{C_T H^{1-\eps}_x}^\al\\
& \les 
\| \<20>_N[\psi_0 + \psi_1]\|_{C_T W^{\frac 12 -\eps, \infty}_x}^{\frac{3}{2}\al}
+ \| \psi_1\|_{C_T H^{1 -\eps}_x}^{3\al}.
\end{split}
\label{et9b}
\end{align}

\noi
Noting that $\frac 32 \al < \frac 12$ and $3\al < 1$, 
we can control the first term on the right-hand side of~\eqref{et9b}
by the exponential integrability bound 
for $\<20>_N$ under $\rhoo_N \otimes \PP_2$,
while by Young's inequality with \eqref{et2b} and \eqref{Dr1}, we can bound the second term by
\begin{align}
\dl \Big(\| \Ups^N \|_{H^1}
+  \| \ZZ_N \|_{W^{1-\eps, \infty}}\Big) + C_\dl.
\label{et9bb}
\end{align}

\noi
for any small $\dl > 0$.

Let us consider the first term on the right-hand side of \eqref{et9a}.
In view of \eqref{enh0}, by writing 
\begin{align}
\begin{split}
\<20>_N[\psi_0 + \psi_1]\pe(\pi_N \psi_0)
& = \<20>_N[\psi_0]\pe(\pi_N \psi_0)
+ 2 \Big(\pi_N \I \big((\pi_N\psi_0) (\pi_N\psi_1)\big)\Big)\pe(\pi_N \psi_0)\\
& \quad + \Big(\pi_N \I \big( (\pi_N\psi_1)^2\big)\Big)\pe(\pi_N \psi_0).
\end{split}
\label{et9c}
\end{align}

\noi
Note that we have 
$\<20>_N[\psi_0]\pe(\pi_N \psi_0) = 
\<21p>_N( (Y, u_1), \o_2)$, 
where the latter term is as in  \eqref{enh0}.
While there is an extra frequency cutoff as compared to $\<21p>_N $
in Lemma \ref{LEM:sto2}, 
the conclusion of Lemma \ref{LEM:sto2}
also holds for $\<20>_N[\psi_0]\pe(\pi_N \psi_0) = 
\<21p>_N( (Y, u_1), \o_2)$.
Hence, we can control 
 the first term on the right-hand side of \eqref{et9c} 
by the exponential tail estimate in Lemma \ref{LEM:sto2}
with $0 < \al < \frac 13$.
The third term on the right-hand side of \eqref{et9c}  causes no issue since 
the resonant product 
of $\pi_N \I \big( (\pi_N\psi_1)^2\big)$ and $\pi_N \psi_0$ is well defined.

Lastly, let us consider    the second term on the right-hand side of \eqref{et9c}. 
In view of 
\eqref{et1a}, \eqref{et1aa}, and \eqref{et1b}, we have 
\begin{align}
\begin{split}
\Big(\pi_N \I \big((\pi_N\psi_0) (\pi_N\psi_1)\big)\Big)\pe(\pi_N \psi_0)
& = \Big(\pi_N \I \big((\pi_N\psi_1) \pge (\pi_N\psi_0)  \big)\Big)\pe(\pi_N \psi_0)\\
& \quad + 
\wt  \If_{\pl}^{(1), N}[\psi_0](\pi_N \psi_1)\pe (\pi_N \psi_0)
+ 
\wt \If^N_{\pl, \pe}[\psi_0](\pi_N \psi_1) , 
\end{split}
\label{et10a}
\end{align}

\noi
where $\wt  \If_{\pl}^{(1), N}[\psi_0]$ is defined by 
\begin{align}
\wt  \If_{\pl}^{(1), N}[\psi_0]  := \wt \If_{\pl}^N[\psi_0] - \wt \If_{\pl}^{(2), N}[\psi_0].
\label{et10b}
\end{align}

\noi
From Lemma \ref{LEM:para} and the one degree of smoothing
from the Duhamel integral operator $\I$, 
we see that 
$\I \big((\pi_N\psi_1)\pge  (\pi_N\psi_0)\big) \in C([0, T]; H^{\frac 32 - 3\eps}(\T^3))$, 
which allows us to handle the first term on the right-hand side of \eqref{et10a}.

Next, we estimate the second term on the right-hand side of \eqref{et10a}.
Recall from \eqref{et2b} that 
$\psi_0 = \<1>(Y,u_1,\o_2) $ with $\Law(Y, u_1) = \muu$.
Namely, $\wt  \If_{\pl}^{(1), N}[\psi_0]$
defined in \eqref{et10b} is nothing but 
$ \If_{\pl}^{(1), N}$ in Lemma \ref{LEM:sto3} 
with an extra frequency cutoff $\chi_N(n)$.
Hence, 
the conclusion of 
Lemma~\ref{LEM:sto3} (in particular \eqref{A00}) holds
true for $\wt  \If_{\pl}^{(1), N}[\psi_0]$.
Then, from Lemma~\ref{LEM:para} and   Lemma~\ref{LEM:sto3}, 
 we have 
\begin{align*}
\big\|\wt  \If_{\pl}^{(1), N}[\psi_0](\pi_N \psi_1)\pe (\pi_N \psi_0)\big\|_{C_T H^{-\eps}_x}^\al
& \les \big\|\wt  \If_{\pl}^{(1), N}[\psi_0](\pi_N \psi_1)\big\|_{C_T H^{\frac 12 + 3\eps}_x}^\al
\| \psi_0\|_{C_T W^{-\frac 12-\eps, \infty}_x}^\al
\\
& 
\le C(T)  \|\psi_1 \|_{C_T H_x^{1-\eps}}^\al
\|\psi_0 \|_{C_T W_x^{-\frac 12 - \eps, \infty}}^{2\al}.
\end{align*}

\noi
Then, Young's inequality allows us to handle this term.

Finally, we treat the third  term on the right-hand side of \eqref{et10a}.
From \eqref{L1ast} and Young's inequality, we have 
\begin{align*}
\big\| \wt \If^N_{\pl, \pe}[\psi_0](\pi_N \psi_1) \big\|_{C_T H^{-\eps}_x}^\al
& \leq 
\big\| \wt \If^N_{\pl, \pe}[\psi_0]\big\|_{\L(\frac{3}{2}, T)}^\al
\|\psi_1 \|_{L^\frac{3}{2}_T L^2_x}^\al\\
& \les C(T) \Big(\big\| \wt \If^N_{\pl, \pe}[\psi_0]\big\|_{\L(\frac{3}{2}, T)}^{\frac 32\al}
+ \|\psi_1 \|_{L^\frac{3}{2}_T L^2_x}^{3\al}\Big),
\end{align*}

\noi
which can be controlled by \eqref{exp} and \eqref{et9bb}.

Therefore, Proposition \ref{PROP:exptail2} holds
for  all the elements in the truncated enhanced data set
$\Xi_N(\vec u_0, \o_2)$ in \eqref{data3x}.

\end{remark}

We conclude this subsection by constructing the full enhanced data set
$\Xi (\vec u_0, \o_2)$ in~\eqref{enh1}  under $\rhoo \otimes\PP_2$
as a limit of the truncated enhanced data set
$\Xi_N (\vec u_0, \o_2)$ in \eqref{data3x}.

\begin{corollary}\label{COR:lim}
Let $T > 0$. Then, 
the truncated enhanced data set $\Xi_N(\vec u_0, \o_2)$ in \eqref{data3x}
converges to 
the enhanced data set $\Xi(\vec u_0, \o_2)$ in \eqref{enh1},  
with respect to the $\Xc^\eps_T$-norm defined in~\eqref{data3},  
almost surely and in measure with respect to 
the limiting measure $\rhoo \otimes \PP_2$.

\end{corollary}

\begin{proof}
Let $0 < \al < \frac 13$ and $\be > 0$ be as in Proposition \ref{PROP:exptail2}.
Then, by 
Fatou's lemma, the weak convergence of $\rhoo_N \otimes \PP_2$
to $\rhoo \otimes \PP_2$, 
and Proposition \ref{PROP:exptail2}, we have 
\begin{align}
\int & \exp \Big( N_2^\be \| \Xi_{N_1} (\vec u_0, \o_2)
- \Xi_{N_2} (\vec u_0, \o_2)
 \|_{\mathcal{X}^{\eps}_T}^\alpha \Big) d (\rhoo \otimes \PP_2)(\vec u_0, \o_2) 
 \notag \\
& \le
\liminf_{L \to \infty}
\int \exp \Big( \min\big(N_2^\be \| \Xi_{N_1} (\vec u_0, \o_2)
- \Xi_{N_2} (\vec u_0, \o_2)
 \|_{\mathcal{X}^{\eps}_T}^\alpha, L\big) \Big) d (\rhoo \otimes \PP_2)(\vec u_0, \o_2) \notag \\
&  = \liminf_{L \to \infty}\lim_{N \to \infty}
\int \exp \Big( \min\big(N_2^\be \| \Xi_{N_1} (\vec u_0, \o_2)\notag \\
& \hphantom{XXXXXXXXXXX}
- \Xi_{N_2} (\vec u_0, \o_2)
 \|_{\mathcal{X}^{\eps}_T}^\alpha, L\big) \Big) d (\rhoo_N \otimes \PP_2)(\vec u_0, \o_2) \notag \\
&  \le \lim_{N \to \infty}
\int \exp \Big( N_2^\be \| \Xi_{N_1} (\vec u_0, \o_2)
- \Xi_{N_2} (\vec u_0, \o_2)
 \|_{\mathcal{X}^{\eps}_T}^\alpha \Big) d (\rhoo_N \otimes \PP_2)(\vec u_0, \o_2) \notag \\
 & \les 1, 
\label{Xex1}
\end{align}

\noi
uniformly in $N_1 \ge N_2 \ge 1$.
Then, by Chebyshev's inequality, we have 
\begin{align*}
\rhoo \otimes \PP_2\Big(  \| \Xi_{N_1} (\vec u_0, \o_2)
- \Xi_{N_2} (\vec u_0, \o_2)
 \|_{\mathcal{X}^{\eps}_T}^\alpha > \ld\Big) 
 \le C e^{-c N_2^\be\ld^\al}
\end{align*}

\noi
for any  $\ld > 0$ and $N_1 \ge N_2 \ge 1$.
This shows that $\{\Xi_{N} (\vec u_0, \o_2)\}_{N \in \N}$
is Cauchy in measure with respect to $\rhoo \otimes \PP_2$
and thus converges in measure to the full enhanced data set 
$\Xi(\vec u_0, \o_2)$ in~\eqref{enh1}.
By Fatou's lemma and \eqref{Xex1}, 
we also have 
\begin{align*}
\int & \exp \Big( N_2^\be \| \Xi (\vec u_0, \o_2)
- \Xi_{N_2} (\vec u_0, \o_2)
 \|_{\mathcal{X}^{\eps}_T}^\alpha \Big) d (\rhoo \otimes \PP_2)(\vec u_0, \o_2) 
 \les 1, 
\end{align*}

\noi
uniformly in $N_1 \ge N_2 \ge 1$, which in turn implies
\begin{align*}
\rhoo \otimes \PP_2\Big(  \| \Xi (\vec u_0, \o_2)
- \Xi_{N_2} (\vec u_0, \o_2)
 \|_{\mathcal{X}^{\eps}_T}^\alpha > \ld\Big) 
 \le C e^{-c N_2^\be\ld^\al}
\end{align*}

\noi
for any  $\ld > 0$ and $ N_2 \in \N$.
By summing in $N_2\in \N$
and invoking the Borel-Cantelli lemma, 
we also conclude almost sure convergence
$ \Xi_{N} (\vec u_0, \o_2)$ to 
$  \Xi (\vec u_0, \o_2)$ with respect to $\rhoo\otimes \PP_2$.
\end{proof}

\subsection{Proof of Theorem \ref{THM:GWP}}
\label{SUBSEC:GWP2}

In this subsection, we present the proof of Theorem \ref{THM:GWP}.
The main task is to prove convergence of 
the solution $(u_N, \dt u_N)$ to the truncated hyperbolic $\Phi^3_3$-model \eqref{SNLW3a}.
We first carry out Steps 2, 3, and 4
described at the beginning of this section.
Namely, 
we first establish 
a stability result (Proposition \ref{PROP:LWPv})
as a slight modification of 
 the local well-posedness argument (Theorem \ref{THM:1}).
Next, we 
 establish a uniform (in $N$) control
on the solution $(X_N, Y_N, \Res_N)$
to the truncated system (see \eqref{Ba4b} below)
with respect to the truncated measure $\rho_N \times \PP_2$
 (Proposition \ref{PROP:tail2}).
Then, by using ideas from theory of optimal transport, 
we study
the convergence property of 
  the pushforward measure
$(\Xi_N)_\#(\rhoo_N\otimes \PP_2)$
to 
$(\Xi)_\#(\rhoo \otimes \PP_2)$
with respect to the 
 Wasserstein-1 distance (Proposition~\ref{PROP:plan}).

Let $\Phi^N_1(t)(\vec u_0, \o_2)$
be the first component    of  $\Phi^N(t)(\vec u_0, \o_2)$ 
 in \eqref{sol0}.
Then, 
by decomposing $\Phi^N_1(t)(\vec u_0, \o_2)$
 as in  \eqref{decomp3a}:
\begin{align}
\Phi_1^N (t) (\vec u_0, \o_2) = \<1> (t; \vec u_0, \o_2) + \s \<20>_N (t; \vec u_0, \o_2) + X_N(t) + Y_N(t), 
\label{Ba4b}
\end{align}

\noi
we see that 
 $X_N$, $Y_N$, and $\Res_N := X_N \pe \<1>_N (\vec u_0, \o_2)$ 
 satisfy  the following system:
\begin{align}
\begin{split}
 (\dt^2&  + \dt  +1 - \Dl) X_{N}\\  
&  =
2\s \pi_{N} \Big(
\big( X_N+Y_N+\s \<20>_N \big) \pl \<1>_{N} \Big) \\
&\phantom{X}
-M (\QxyN + 2\Res_N \, + \s^2 \<20>_N^2+ 2 \s \<21>_N + \<2>_N) \<1>_N ,\\
 (\dt^2 & + \dt +1  - \Dl) Y_{N}\\
&  = \s \pi_{N} \Big( \big( X_N+Y_N+\s \<20>_N  \big)^2
+ 2 \big( \Res_N + Y_N \pe \<1>_N + \s \<21p>_N \big) \\
&\phantom{X} 
+ 2\big( X_N + Y_N + \s \<20>_N \big) \pg \<1>_{N} \Big) \\
&\phantom{X}
-M (\QxyN + 2\Res_N \, + \s^2 \<20>_N^2 + 2 \s \<21>_N + \<2>_N) (X_N+Y_N+\s \<20>_N) ,\\
\Res_{N}
&=  2\s
\wt \If_{\pl}^{(1), N}
\big( X_N+Y_N+\s \<20>_N\big) \pe \<1>_{N}\\
& \hphantom{X}
 + 2\s \wt \If_{\pl, \pe}^{N}
 \big( X_N+Y_N+\s \<20>_N \big) \\
& \hphantom{X}
- \int_0^t  M(\QxyN + 2\Res_N \, + \s^2 \<20>_N^2 + 2 \s \<21>_N + \<2>_N)(t') 
\Ab_N(t, t') dt', \\
(X_{N} & ,  \dt X_{N}, Y_{N} ,   \dt Y_{N} )|_{t = 0}  = (0, 0, 0, 0), 
\end{split}
\label{SNLW9}
\end{align}

\noi
where $M$ is as in \eqref{addM}, 
$\QxyN $ is as in 
\eqref{Pxy} with $\<1>$ replaced by 
$\<1>_N = \<1>_N (\vec u_0, \o_2)$ as in~\eqref{Ba1x}, 
and 
the enhanced data set is  given by $\Xi_N (\vec u_0,\o_2)$  in \eqref{data3x}.

We first establish the following stability result.
The main idea is that 
by introducing a norm 
with an exponential decaying weight in time 
(see \eqref{ZL1}), 
the proof essentially follows
from a straightforward modification
of the local well-posedness argument (Theorem \ref{THM:1}).
A simple, but key observation is \eqref{ZL5} below.

\begin{proposition} \label{PROP:LWPv}
Let $T\gg 1$, $K\gg1$, and $C_0 \gg 1$.
Then, there exist $N_0(T, K, C_0) \in \N$ 
and small  $\kk_0 = \kk_0(T, K, C_0) > 0$
such that the following statements hold.
Suppose that for some $N \ge N_0$, 
we have 
\begin{align}
\|\Xi_N( \vec u_0', \o_2')\|_{\mathcal{X}^{\eps}_T} \le K
\label{bd1}
\end{align}

\noi
and 
\begin{align}
\|(X_N,Y_N,\Res_N)\|_{Z^{s_1, s_2, s_3}(T)} \le C_0
\label{bd2}
\end{align}

\noi
for the solution to $(X_N,Y_N,\Res_N)$ 
to the truncated system \eqref{SNLW9}  on $[0, T]$ with the truncated enhanced data set
 $\Xi_N (\vec u_0', \o_2')$.
Furthermore, suppose that we have 
\begin{align}
 \| \Xi(\vec u_0, \o_2) - \Xi_N(\vec u_0', \o_2')\|_{\mathcal{X}^{\eps}_T} \le \kk
\label{bd3}
\end{align}

\noi
for some $0 < \kk\le \kk_0$ and some $(\vec u_0, \o_2)$, 
where $\Xi(\vec u_0, \o_2)$ denotes the enhanced data set in~\eqref{enh1}.
Then, 
there exists a solution 
$(X, Y, \Res)$ to 
the full  system~\eqref{SNLW6} on $[0, T]$
with the zero initial data 
and the enhanced data set $\Xi (\vec u_0, \o_2)$,   satisfying the bound
\begin{align*}
\|(X, Y, \Res)\|_{Z^{s_1, s_2, s_3}(T)} \le C_0+ 1.
\end{align*}

\smallskip

Conversely,
suppose  that
\[
\|\Xi(\vec u_0, \o_2)\|_{\mathcal{X}^{\eps}_T} \le K
\]
and that 
the full system \eqref{SNLW6} with the zero initial data  and the enhanced data set $\Xi(\vec u_0, \o_2)$
has a solution $(X,Y,\Res)$ on $[0, T]$, satisfying 
\begin{align}
\|(X,Y,\Res)\|_{Z^{s_1, s_2, s_3}(T)} \le C_0.
\notag
\end{align}
Then, if \eqref{bd3} holds
for some  $N \ge N_0$,  $0 < \kk\le \kk_0$, and  $(\vec u_0', \o_2')$, 
then
there exists 
a solution $(X_N, Y_N, \Res_N)$ to the truncated system~\eqref{SNLW9} 
on $[0, T]$ with 
the enhanced data set $\Xi_N (\vec u_0', \o_2')$, 
satisfying 
\begin{align}
\|(X_N, Y_N, \Res_N) - (X, Y, \Res)\|_{Z^{s_1, s_2, s_3}(T)} \le 
   A(T, K, C_0)   ( \kk +    N^{-\dl})
\label{lv3}
\end{align}

\noi
for some $A(T, K, C_0)>0$ and some small $\dl > 0$.

\end{proposition}

\begin{proof}

Fix $T\gg 1$.
Given   $\ld  \ge 1$ (to be determined later), 
we define 
$ \Zc^{s_1, s_2, s_3}_\ld(T)$ by 
\begin{align}
\|(X, Y, \Res)\|_{\Zc^{s_1, s_2, s_3}_\ld(T)}
= \|(e^{-\ld t}X, e^{-\ld t}Y, e^{-\ld t}\Res)\|_{Z^{s_1, s_2, s_3}(T)}.
\label{ZL1}
\end{align}

\noi
For notational simplicity,
we  set $Z = (X, Y, Z)$, $Z_N = (X_N, Y_N, \Res_N)$, 
$\Xi = \Xi(\vec u_0, \o_2)$, 
and  $\Xi_N = \Xi_N (\vec u_0', \o_2')$.

In the following, given $N \in \N$, 
we assume that  \eqref{bd1}, \eqref{bd2}, and \eqref{bd3}
hold.
Without loss of generality, assume that $\kk \le 1$.
Then, from \eqref{bd1} and \eqref{bd3}, we have
\begin{align}
 \| \Xi(\vec u_0, \o_2) \|_{\mathcal{X}^{\eps}_T} \le K + \kk
 \le K + 1 =:K_0.
\label{bd5}
\end{align}

\noi
In the following, 
we study the difference of 
the Duhamel formulation\footnote{Recall that we set $\s =1$ in Section \ref{SEC:LWP} for simplicity
and thus need to insert $\s$ in appropriate locations of~\eqref{SNLW8}.} \eqref{SNLW8} of 
the system~\eqref{SNLW6} with the zero initial data (i.e.~$(X_0, X_1, Y_0, Y_1) = (0, 0, 0, 0)$)
and the Duhamel formulation of the truncated system \eqref{SNLW9}
with respect to the $\Zc^{s_1, s_2, s_3}_\ld(T)$-norm
by choosing appropriate $\ld = \ld(T, K_0, R) \gg1 $.
See \eqref{SNLW12} below.

The main observation is the following bound:
\begin{align}
e^{-\ld t }\| e^{\ld t'} \|_{L^q_{t'}([0, t])} \les  \ld^{-\frac 1q}.
\label{ZL5}
\end{align}

\noi
Let $\I$ be the Duhamel integral operator defined in \eqref{lin1}.
Then, using \eqref{ZL5}, we have 
\begin{align}
\begin{split}
\| e^{-\ld t} \I(F)\|_{C_T H^s_x}
& \leq \bigg\| e^{-\ld t} \int_0^t e^{\ld t'}   \| e^{-\ld t'} F(t') \|_{H^{s-1}_x} dt' \bigg\|_{L^\infty_T}\\
& \les \ld^{-\frac{1}{q}}
\| e^{-\ld t'} F(t') \|_{L^{q'}_T H^{s-1}_x}
\end{split}
\label{ZL5a}
\end{align}

\noi
for any $1 \le q \le \infty$.
Let  $(q_1, r_1)$ be an $s_1$-admissible pair
with $0 < s_1  < 1$.
Then, there exists 
an $s_2$-admissible pair
 $(q_2, r_2)$
with $0 < s_1 < s_2 < 1$
such that 
\[\frac{1}{q_1} = \frac{\ta}{\infty} + \frac{1-\ta}{q_2}, 
\quad 
\frac{1}{r_1} = \frac{\ta}{2} + \frac{1-\ta}{r_2}, 
\quad \text{and}\quad 
s_1 = \ta \cdot 0 + (1-\ta) s_2
\]

\noi
for some  $0 < \ta < 1$.
By the homogeneous Strichartz estimate (\eqref{Str1} with $F = 0$), 
we have 
\begin{align}
\begin{split}
\| e^{-\ld t} \I(F)\|_{L^{q_2}_T L^{r_2}_x}
& \le  \bigg\| \int_0^t  e^{-\ld (t - t')} \D(t - t') (e^{-\ld t'} F(t'))
 dt'\bigg\|_{L^{q_2}_T L^{r_2}_x} \\
& \le  \int_0^T \|  \D(t - t') (e^{-\ld t'} F(t'))
\|_{L^{q_2}_t([0, T];  L^{r_2}_x)}  dt'\\
& \les 
\| e^{-\ld t'} F(t') \|_{L^{1}_T H^{s_2-1}_x}.
\end{split}
\label{ZL5b}
\end{align}

\noi
Thus, 
given any $\dl >0$, 
it follows from 
 interpolating  \eqref{ZL5a}
with large $q \gg1$ and \eqref{ZL5b} that 
there exists small $\ta = \ta(\dl)> 0$ such that 
\begin{align}
\| e^{-\ld t} \I(F)\|_{L^{q_1}_T L^{r_1}_x}
& \le C(T) \ld^{-\ta}
\| e^{-\ld t'} F(t') \|_{L^{1+\dl}_T H^{s_1-1}_x}.
\label{ZL5c}
\end{align}

\noi
Recalling that $(4, 4)$ is $\frac 12$-admissible, 
it follows from  \eqref{ZL5a},  \eqref{ZL5c}, and
Sobolev's inequality that 
\begin{align}
\begin{split}
\| e^{-\ld t} \I(F)\|_{C_T\H^\frac12_x \cap L^{4}_T L^{4}_x}
& \le C(T) \ld^{-\ta}
\| e^{-\ld t'} F(t') \|_{L^{1+\dl}_T H^{-\frac 12 }_x}\\
& \le C(T) \ld^{-\ta}
\| e^{-\ld t'} F(t') \|_{L^{1+\dl}_T L^{\frac 32}_x}.
\end{split}
\label{ZL5d}
\end{align}

By writing \eqref{SNLW9} in the Duhamel formulation, we have
\begin{align}
\begin{split}
X_N 
& = \Phi_{1, N} (X_N, Y_N, \Res_N)\\
: \!&  = 
2\s \pi_{N} \I\Big(
\big( X_N+Y_N+\s \<20>_N \big) \pl \<1>_{N} \Big) \\
&\phantom{X}
-\I\Big(M (\QxyN + 2\Res_N \, + \s^2 \<20>_N^2+ 2 \s \<21>_N + \<2>_N) \<1>_N\Big) ,\\
 Y_{N}
 & = \Phi_{2, N} (X_N, Y_N, \Res_N)\\ 
:\!&  = \s \pi_{N} \I \Big( \big( X_N+Y_N+\s \<20>_N  \big)^2\Big)
+ 2\s \pi_{N}\I \big( \Res_N + Y_N \pe \<1>_N + \s \<21p>_N \big) \\
&\phantom{X} 
+ 2\s \pi_{N}\I\Big(\big( X_N + Y_N + \s \<20>_N \big) \pg \<1>_{N} \Big) \\
&\phantom{X}
-\I\Big(M (\QxyN + 2\Res_N \, + \s^2 \<20>_N^2 + 2 \s \<21>_N + \<2>_N) (X_N+Y_N+\s \<20>_N)\Big) ,\\
\Res_{N} & =\Phi_{3, N}(X_N, Y_N, \Res_N),  \\
:\!&=  2\s
\wt \If_{\pl}^{(1), N}
\big( X_N+Y_N+\s \<20>_N\big) \pe \<1>_{N}\\
& \hphantom{X}
 + 2\s \wt \If_{\pl, \pe}^{N}
 \big( X_N+Y_N+\s \<20>_N \big) \\
& \hphantom{X}
- \int_0^t  M(\QxyN + 2\Res_N \, + \s^2 \<20>_N^2 + 2 \s \<21>_N + \<2>_N)(t') 
\Ab_N(t, t') dt'.
\end{split}
\label{SNLW10}
\end{align}

\noi
Then, $Z - Z_N = (X - X_N, Y - Y_N, \Res - \Res_N)$ satisfies the system
\begin{align}
\begin{split}
X - X_N  & = \Phi_1(X, Y, \Res) - \Phi_{1, N}(X_N, Y_N, \Res_N),  \\
Y - Y_N  & = \Phi_2(X, Y, \Res) - \Phi_{2, N}(X_N, Y_N, \Res_N),  \\
\Res - \Res_N  & = \Phi_3(X, Y, \Res) - \Phi_{3, N}(X_N, Y_N, \Res_N).
\end{split}
\label{SNLW11}
\end{align}

\noi
By setting
\begin{align*}
\dl X_N = X - X_N, \quad 
\dl Y_N = Y - Y_N, \quad 
\text{and} \quad
\dl \Res_N = \Res - \Res_N, 
\end{align*}

\noi
we have 
\begin{align*}
 X = \dl X_N  + X_N, \quad 
Y =\dl  Y_N + Y_N, \quad 
\text{and} \quad 
\Res = \dl  \Res_N + \Res_N.
\end{align*}

\noi
Then, we can view the system \eqref{SNLW11}
for the system for the unknown
\[\dl Z_N = (\dl X_N, \dl Y_N, \dl \Res_N)\]

\noi
with given source terms 
$Z_N = (X_N, Y_N, Z_N)$, $\Xi_N$, and $\Xi$.
We thus rewrite \eqref{SNLW11} as 
\begin{align}
\begin{split}
\dl  X_N  & = \Psi_1(\dl X_N, \dl Y_N, \dl \Res_N),  \\
\dl  Y_N  & =  \Psi_{2}(\dl X_N, \dl Y_N, \dl \Res_N),  \\
\dl \Res_N  & = \Psi_3(\dl X_N, \dl Y_N, \dl \Res_N), 
\end{split}
\label{SNLW12}
\end{align}

\noi
where $\Psi_j$, $j = 1, 2, 3$, is given by 
\begin{align}
\begin{split}
\Psi_j& (\dl X_N, \dl Y_N, \dl \Res_N)\\
& = \Phi_j(\dl X_N + X_N , \dl Y_N + Y_N, \dl \Res_N + \Res_N) - \Phi_{j, N}(X_N, Y_N, \Res_N).
\end{split}
\label{SNLW13}
\end{align}

We now study the system \eqref{SNLW12}.
We basically repeat the computations in Subsection \ref{SUBSEC:LWP3}
 by first multiplying the Duhamel formulation by $e^{-\ld t}$
 and using \eqref{ZL5a}, \eqref{ZL5c}, and \eqref{ZL5d}
 as a replacement of the Strichartz estimates (Lemma \ref{LEM:Str}).
This allows us to place $e^{-\ld t'}$
on one of the factors of 
$\dl X_N(t')$, $\dl Y_N(t')$, or $\dl \Res_N(t')$
 appearing
on the right-hand side of \eqref{SNLW12}
under some integral operator (with integration in the variable $t'$).
Our main goal is to  prove that 
\begin{align}
\vec \Psi = (\Psi_1, \Psi_2, \Psi_3)
\label{ES0}
\end{align}
is a contraction on 
a small ball in 
$\Zc_\ld^{s_1, s_2, s_3}(T)$.
In the following, however, we first establish bounds on $\Psi_j$ in~\eqref{SNLW13}
for 
 $\dl Z_N \in B_1$, 
 where $B_1 \subset Z^{s_1, s_2, s_3}(T)$
  denotes the closed ball of radius 1 (with respect to the 
  $Z^{s_1, s_2, s_3}(T)$-norm) centered at the origin.
For 
 $\dl Z_N \in B_1$, 
 it follows from \eqref{bd2}
that 
\begin{align}
\begin{split}
\|Z \|_{Z^{s_1, s_2, s_3}(T)} 
& \le \|\dl Z_N \|_{Z^{s_1, s_2, s_3}(T)} 
+ \|Z_N \|_{Z^{s_1, s_2, s_3}(T)} \\
& \le 1 + C_0  =:R.
\end{split}
\label{bd6}
\end{align}

 We first study the first equation in \eqref{SNLW12}.
From \eqref{SNLW13} with \eqref{SNLW8}, \eqref{SNLW10}, and \eqref{SNLW13},  
we have
\begin{align}
e^{-\ld t} \Psi_1(\dl X_N, \dl Y_N, \dl \Res_N)(t)
= e^{-\ld t}\1_1 (t)+ e^{-\ld t}\1_2(t)+e^{-\ld t}\1_3(t), 
\label{ZL3}
\end{align}

\noi
where 
(i)  $\1_1$ contains 
the difference of one of the elements
in the enhanced data sets $\Xi$ and $\Xi_N$, 
  (ii) $\1_2$ contains the terms
with the high frequency projection $\pi_N^\perp  = \Id - \pi_N$
onto the frequencies $\{|n|\ges N\}$, 
and (iii) $\1_3$ consists of the rest, which contains at least
one of the differences 
$\dl X_N$, $\dl Y_N$, 
or $\dl \Res_N$ (other than those in $Z = \dl Z_N + Z_N$).

In view of \eqref{bd3}, 
the contribution from $\1_1$ gives a small number $\kk$, 
while 
the contribution from $\1_2$ with $\pi_N^\perp$ gives a small negative power of $N$
by losing a small amount of regularity.\footnote{We have 
sharp inequalities  in \eqref{S1} as compared to the regularity condition in Theorem~\ref{THM:1}.
 This allows us to gain 
 a small negative power of $N$, 
 by losing a small amount of regularity 
 and using $\pi_N^\perp$.}
Proceeding as in~\eqref{M1a}
with \eqref{bd1}, 
\eqref{bd2}, \eqref{bd3},
\eqref{bd5},  and \eqref{bd6}, 
we have 
\begin{align}
\begin{split}
\| e^{-\ld t}\1_1 + e^{-\ld t}\1_2\|_{X^{s_1}(T)}
&\le
C(T) ( \kk +    N^{-\dl} K_0)  ( R^4 +K_0^4 )\\
&\le
C(T) ( \kk +    N^{-\dl}) K_0  ( R^4 +K_0^4 )
\end{split}
\label{ES1}
\end{align}

\noi
for any $\dl Z_N \in B_1$ and some small $\dl > 0$.
As for the last term on the right-hand side of \eqref{ZL3}, 
we use 
\eqref{ZL5a} and   \eqref{ZL5c} in place of Lemma \ref{LEM:Str}.
Then, a slight modification of \eqref{M1a} yields
\begin{align}
\| e^{-\ld t} \1_3 \|_{X^{s_1}(T)}
&\le
C(T)\ld^{-\ta} K_0  \Big( R^3 \|\dl Z_N\|_{\Zc_\ld^{s_1, s_2, s_3}(T)} +K_0^4 \Big)
\label{ES2}
\end{align}

\noi
for any $\dl Z_N \in B_1$.

Next, we study the second equation in \eqref{SNLW12}.
As in \eqref{ZL3}, 
we can write
\begin{align}
e^{-\ld t} \Psi_2(\dl X_N, \dl Y_N, \dl \Res_N)(t)
= e^{-\ld t}\II_1 (t)+ e^{-\ld t}\II_2(t)+e^{-\ld t}\II_3(t), 
\label{ES3}
\end{align}

\noi
where 
(i)  $\II_1$ contains 
the difference of one of the elements
in the enhanced data sets $\Xi$ and $\Xi_N$, 
  (ii) $\II_2$ contains the terms
with the high frequency projection $\pi_N^\perp  = \Id - \pi_N$
onto the frequencies $\{|n|\ges N\}$, 
and (iii) $\II_3$ consists of the rest, which contains at least
one of the differences 
$\dl X_N$, $\dl Y_N$, 
or $\dl \Res_N$ (other than those in $Z = \dl Z_N + Z_N$).
As for the first two terms on the right-hand side of \eqref{ES3}, 
we can proceed as in \eqref{M4}
with \eqref{bd1}, 
\eqref{bd2}, \eqref{bd3},
\eqref{bd5},  and \eqref{bd6}, 
and obtain
\begin{align}
\| e^{-\ld t}\II_1 + e^{-\ld t}\II_2\|_{Y^{s_2}(T)}
&\le
C(T) ( \kk +    N^{-\dl})  ( R^5 +K_0^5 )
\label{ES4}
\end{align}

\noi
for any $\dl Z_N \in B_1$
and some small $\dl > 0$.
Before we proceed to study the last term $e^{-\ld t}\II_3(t)$, 
let us  make a preliminary computation.
By the fractional Leibniz rule (Lemma~\ref{LEM:gko}\,(i))
and Sobolev's inequality, we have 
\begin{align}
\begin{split}
\| \jb{\nb}^{s_2 - \frac 12} (fg) \|_{L^\frac{3}{2}}
& \les \| \jb{\nb}^{s_2 - \frac 12} f \|_{L^{r_1}}
\| g \|_{L^{r_2}}
+ \| f \|_{L^{r_2}}\| \jb{\nb}^{s_2 - \frac 12} g \|_{L^{r_1}}
\\
& \les \| \jb{\nb}^{s_1 - \frac 14} f \|_{L^\frac 83}
\| \jb{\nb}^{s_1 - \frac 14} g \|_{L^\frac 83}, 
\end{split}
\label{ES5}
\end{align}

\noi
provided that $\frac 1{r_1} + \frac 1{r_2} = \frac 23$ with $1 < r_1, r_2 \le \infty$, 
\begin{align}
 \frac{s_1 - s_2 + \frac 14}{3} \geq \frac 38 - \frac 1{r_1}
\qquad \text{and}\qquad 
\frac {s_1 - \frac 14}3\ge \frac 38 - \frac 1{r_2}.
\label{S2}
\end{align}

\noi
This condition is easily satisfied by taking $s_1 < \frac 12 < s_2$
both sufficiently close to $\frac 12$ and $r_1 = r_2 = 3$.
By \eqref{ZL5d},  \eqref{ES5}, and Lemma \ref{LEM:gko}\,(i), 
we have 
\begin{align}
\begin{split}
\big\|  &e^{-\ld t} \I \big( (X_1+Y_1+\Xi_0)(X_2+Y_2+\Xi_0) \big) \big\|_{Y^{s_2} (T)}\\
&  \le C(T) \ld^{-\ta} 
\big\| e^{-\ld t} \jb{\nb}^{s_2 - \frac 12 }
 \big( (X_1+Y_1+\Xi_0)(X_2+Y_2+\Xi_0) \big) \big\|_{L^{1+\dl}_{T}L^{\frac {3}{2}}_x} \\
& \le C(T) \ld^{-\ta} 
\Big(  \|\jb{\nb}^{s_1 - \frac 14}X_1\|_{L^8_{T} L^\frac{8}{3}_x}
+ \|  \jb{\nb}^{s_2 - \frac 12}Y_1\|_{L^4_{T, x}}
+ \| \jb{\nb}^{s_2 - \frac 12}\Xi_0\|_{L^\infty_{T, x}}\Big)\\
& \quad
\times \Big(\|e^{-\ld t}  \jb{\nb}^{s_1 - \frac 14}X_2\|_{L^8_{T} L^\frac{8}{3}_x}
+ \|e^{-\ld t}  \jb{\nb}^{s_2 - \frac 12}Y_2\|_{L^4_{T, x}}
+ \| \jb{\nb}^{s_2 - \frac 12}\Xi_0\|_{L^\infty_{T, x}}\Big), 
\end{split}
\label{ES6}
\end{align}

\noi
provided that
$s_1 < \frac 12 < s_2$
are both sufficiently close to $\frac 12$.
Compare this with 
\eqref{M4a}.
Then, from 
 \eqref{ZL5a}, \eqref{ZL5c}, 
and  \eqref{ES6}
with \eqref{bd1}, 
\eqref{bd2}, 
\eqref{bd5},  and \eqref{bd6}, 
a slight modification of \eqref{M4} yields
\begin{align}
\|  e^{-\ld t}\II_3 \|_{Y^{s_2}(T)}   
\le C(T) \ld^{-\ta}
   \Big( R^4 \|\dl Z_N\|_{\Zc_\ld^{s_1, s_2, s_3}(T)} +K_0^5 \Big)
\label{ES7}
\end{align}

\noi
for any $\dl Z_N \in B_1$.

Finally, we study the third equation in \eqref{SNLW12}.
As in \eqref{ZL3} and \eqref{ES3}, 
we can write
\begin{align}
e^{-\ld t} \Psi_3(\dl X_N, \dl Y_N, \dl \Res_N)(t)
= e^{-\ld t}\III_1 (t)+ e^{-\ld t}\III_2(t)+e^{-\ld t}\III_3(t), 
\label{ES8}
\end{align}

\noi
where 
(i)  $\III_1$ contains 
the difference of one of the elements
in the enhanced data sets $\Xi$ and $\Xi_N$, 
  (ii) $\III_2$ contains the terms
with the high frequency projection $\pi_N^\perp  = \Id - \pi_N$
onto the frequencies $\{|n|\ges N\}$, 
and (iii) $\III_3$ consists of the rest, which contains at least
one of the differences 
$\dl X_N$, $\dl Y_N$, 
or $\dl \Res_N$ (other than those in $Z = \dl Z_N + Z_N$).
Proceeding as in  \eqref{M5}
with \eqref{bd1}, 
\eqref{bd2}, \eqref{bd3},
\eqref{bd5},  and \eqref{bd6}, 
we have 
\begin{align}
\| e^{-\ld t}\III_1 + e^{-\ld t}\III_2\|_{L^3_T H^{s_3}_x}
&\le
C(T) ( \kk +    N^{-\dl}) K_0 ( R^4 +K_0^4 )
\label{ES9}
\end{align}

\noi
for any $\dl Z_N \in B_1$
and some small $\dl > 0$.
As for the last term on the right-hand side of \eqref{ES8}, 
let us fist consider the terms with the random operator $ \If_{\pl, \pe}$.
By  \eqref{bd5} and \eqref{ZL5}, we have 
\begin{align*}
 \big\|  &    e^{-\ld t}\, \If_{\pl, \pe}
 \big(X_1+Y_1+\Xi_0)(t)
 -   e^{-\ld t}\, \If_{\pl, \pe}
 \big(X_2+Y_2+\Xi_0)(t)\big\|_{L^3_{T} H^{s_3}_x}\\
& \leq K_0 \Big\|e^{- \ld t} \|
e^{\ld t'}( e^{-\ld t'}(X_1+Y_1 - X_2 - Y_2))
\|_{L^\frac{3}{2}_{t'}([0, t]; L^2_x)}\Big\|_{L^3_{T}}\\
& \le C(T)\ld^{-\ta}
K_0  \Big( \|e^{-\ld t}(X_1 - X_2)\|_{L^\infty_T H^{s_1}_x}
+ \|e^{-\ld t}(Y_1 - Y_2)\|_{L^\infty_T H^{s_2}_x}\Big)
 \notag 
\end{align*}

\noi
for some $\ta > 0$.
The other terms can be  estimated in a similar manner
and thus we obtain
\begin{align}
\|  e^{-\ld t} \III_3 \|_{L^3_{T} H^{s_3}_x}  
\le C(T) \ld^{-\ta}
K_0  \Big(R^3  \|\dl Z_N\|_{\Zc_\ld^{s_1, s_2, s_3}(T)} +K_0^4 \Big)
\label{ES10}
\end{align}

\noi
for any 
$\dl Z_N  \in B_1$.

Hence, putting 
\eqref{ES1}, \eqref{ES2}, 
\eqref{ES4}, 
\eqref{ES7}, 
\eqref{ES9}, and \eqref{ES10} together, we obtain
\begin{align}
\begin{split}
\|\vec \Psi(\dl Z_N) \|_{\Zc_\ld^{s_1, s_2, s_3}(T)}
& \leq C(T, K_0, R) \ld^{-\ta}
  \|\dl Z_N\|_{\Zc_\ld^{s_1, s_2, s_3}(T)}\\
& \quad
  + C(T, K_0, R)   ( \kk +    N^{-\dl})
\end{split}
\label{ES11}
\end{align}

\noi
\noi
for any 
$\dl Z_N  \in B_1$, where $\vec \Psi$ is as in \eqref{ES0}.
By a similar computation, we also obtain the difference estimate:
\begin{align}
\begin{split}
\|\vec \Psi(\dl Z_N^{(1)}) - \vec \Psi(\dl Z_N^{(2)}) \|_{\Zc_\ld^{s_1, s_2, s_3}(T)}
& \leq C(T, K_0, R) \ld^{-\ta}
  \|\dl Z_N^{(1)} - \dl Z_N^{(2)}\|_{\Zc_\ld^{s_1, s_2, s_3}(T)}
\end{split}
\label{ES12}
\end{align}

\noi
for any 
$\dl Z_N^{(1)}, \dl Z_N^{(2)}  \in B_1$.
We now introduce small $r = r(T, \ld) > 0$ such that, in view of~\eqref{ZL1}, 
we have 
\begin{align}
 \| \dl Z_N\|_{Z^{s_1, s_2, s_3}(T)}
\le e^{\ld T} \|\dl Z_N \|_{\Zc_\ld^{s_1, s_2, s_3}(T)}
\le 
e^{\ld T} r  \le 1
\label{ES13}
\end{align}

\noi
for any $\dl Z_N \in B_r^\ld$, 
where
$B_r^\ld\subset  \Zc_\ld^{s_1, s_2, s_3}(T)$
is  the closed ball 
of radius $r$ 
(with respect to the  $\Zc_\ld^{s_1, s_2, s_3}(T)$-norm) centered at the origin.
From \eqref{ES13}, 
we see that both \eqref{ES11} and \eqref{ES12} hold
on $B_r^\ld$.
Therefore, by choosing large $\ld = \ld(T, K_0, R) \gg1 $, 
 small $\kk = \kk(T, K_0, R) >0 $, 
 and large
$N_0 = N_0(T, K_0, R) \in \N $, 
we conclude that 
$\vec \Psi$ is a contraction on  $B_r^\ld$
for any $N \ge N_0$.
Hence, there exists a unique solution  $\dl Z_N \in B_r^\ld$
to the fixed point problem $\dl Z_N = \vec \Psi(\dl Z_N)$.
We need to check that 
by setting $Z = \dl Z_N + Z_N$, 
$Z$ satisfies the Duhamel formulation
\eqref{SNLW8} of the full system \eqref{SNLW6}
with the zero initial data
and the enhanced data set $\Xi = \Xi(\vec u_0, \o_2)$.
From 
\eqref{SNLW12}
and \eqref{SNLW10}, we have 
\begin{align*}
Z & = \dl Z_N + Z_N
= \vec \Psi (\dl Z_N) + \vec \Phi_N(Z_N)\\
& = \vec \Phi(\dl Z_N + Z_N)
= \vec \Phi(Z), 
\end{align*}

\noi
where
$\vec \Phi_N = 
(\Phi_{1, N}, \Phi_{2, N}, \Phi_{3, N})$.
This shows that $Z$ indeed satisfies  the Duhamel formulation~\eqref{SNLW8} 
with the zero initial data and the enhanced data set $\Xi = \Xi(\vec u_0, \o_2)$.
Lastly, we point out that 
from \eqref{bd5} and 
\eqref{bd6}, we have 
$K_0 = K+1$ and 
$R = C_0 + 1$
and thus the parameters $\ld$, $\kk$, and $N_0$
depend on $T$, $K$, and $C_0$.

As for the second claim in this proposition, 
we write $Z_N = Z - (Z-Z_N)$
and study the system for $\dl Z_N = Z - Z_N$:
\begin{align*}
\dl  Z_N   = \vec \Psi^N(\dl Z_N)
\end{align*}

\noi
where
$\vec \Psi^N = (\Psi_1^N, \Psi_2^N, \Psi_3^N)$
and 
 $\Psi_j^N$, $j = 1, 2, 3$, is given by 
\begin{align*}
\Psi_j^N& (\dl X_N, \dl Y_N, \dl \Res_N)\\
& = 
 \Phi_{j}(X, Y, \Res)
- 
\Phi_{j, N}(X - \dl X_N, Y -  \dl Y_N , \Res -  \dl \Res_N) .
\end{align*}

\noi
Here, we view
$Z = (X, Y, Z)$, $\Xi_N$, and $\Xi$
as  given source terms.
By a slight modification of the computation presented above, 
we obtain
\begin{align}
\begin{split}
\|\vec \Psi^N(\dl Z_N) \|_{\Zc_\ld^{s_1, s_2, s_3}(T)}
& \leq C(T, K_0, R) \ld^{-\ta}
  \|\dl Z_N\|_{\Zc_\ld^{s_1, s_2, s_3}(T)}\\
& \quad
  + C(T, K_0, R)   ( \kk +    N^{-\dl})
\end{split}
\label{ES14}
\end{align}

\noi
and 
\begin{align*}
\|\vec \Psi^N(\dl Z_N^{(1)}) - \vec \Psi^N(\dl Z_N^{(2)}) \|_{\Zc_\ld^{s_1, s_2, s_3}(T)}
& \leq C(T, K_0, R) \ld^{-\ta}
  \|\dl Z_N^{(1)} - \dl Z_N^{(2)}\|_{\Zc_\ld^{s_1, s_2, s_3}(T)}
\end{align*}

\noi
for any 
$\dl Z_N, \dl Z_N^{(1)}, \dl Z_N^{(2)}  \in B_1$.
This shows that there exists
a solution 
\[Z_N = Z - \dl Z_N
= \Phi(Z) - \vec \Psi^N (\dl Z_N) 
= \vec \Phi_N(Z_N)\]

\noi
to the truncated system \eqref{SNLW9} on $[0, T]$.
Furthermore, from \eqref{ES14} with $\ld = \ld(T, K_0, R) \gg1 $, 
we have
\begin{align*}
\|Z - Z_N  \|_{Z^{s_1, s_2, s_3}(T)}
& \le e^{\ld T}\|\vec \Psi^N(\dl Z_N) \|_{\Zc_\ld^{s_1, s_2, s_3}(T)}\\
& \le 
   C(T, K_0, R) e^{\ld T}  ( \kk +    N^{-\dl})
   \too 0, 
\end{align*}

\noi
as $N \to \infty$ and $\kk \to 0$.  This proves \eqref{lv3}.
This concludes the proof of Proposition \ref{PROP:LWPv}.
\end{proof}

Next, 
we prove that the solution $(X_N, Y_N, \Res_N)$ to the 
truncated system \eqref{SNLW9} has a uniform bound with a large probability.
The proof is based on 
the invariance of the truncated Gibbs measure $\rhoo_N$
under the truncated hyperbolic $\Phi^3_3$-model \eqref{SNLW3a} (Lemma \ref{LEM:GWP4})
and a discrete Gronwall argument.

\begin{proposition} \label{PROP:tail2}
Let $T > 0$.
Then, given  any $\delta > 0$, there exists $C_0 = C_0(T, \delta) \gg 1$ such that
\begin{align}
\rhoo_N \otimes \PP_2
\Big(  \|(X_N,Y_N,\Res_N)\|_{Z^{s_1, s_2, s_3}(T)} > C_0 \Big) < \delta, 
\label{F0}
\end{align}

\noi
uniformly in  $N \in \N$,
where $(X_N, Y_N, \Res_N)$ is the solution to the truncated system \eqref{SNLW9} on $[0,T]$
with the truncated enhanced data set $\Xi_N(\vec u_0, \o_2)$ in \eqref{data3x}.
\end{proposition}

\begin{proof}
Let $(u_N, \dt u) = \Phi^N(t)(\vec u_0, \o_2)$ be a global solution to \eqref{SNLW3a} constructed in Lemma \ref{LEM:GWP4}, 
where 
 $\Phi^N(t)(\vec u_0, \o_2)$ is as in \eqref{sol0}.
Then, by the 
 invariance of the truncated Gibbs measure $\rhoo_N$ (Lemma \ref{LEM:GWP4}), 
 we have
 \begin{align}
\int F( \Phi^N(t)(\vec u_0, \o_2))
d(\rhoo_N\otimes \PP_2)(\vec u_0, \o_2)
= \int F(\vec u_0) d\rho_N(\vec u_0)
 \label{F0a}
 \end{align}

\noi
for any bounded continuous function   
 $F: \C^{-100}(\T^3)\times \C^{-100}(\T^3) \to \R$ and  $t \in \R_+$.
By Minkowski's integral inequality, \eqref{F0a}, 
 \eqref{addM}, 
and Proposition~\ref{PROP:exptail2},
we have, for any finite $p \ge 1$,  
\begin{align}
\begin{split}
\bigg\|  \int_0^T & |M(\,:\! (\pi_N u_N)^2\!:\,)(t)| dt 
\bigg\|_{L^p_{\vec u_0, \o_2} (\rhoo_N\otimes \PP_2)} \\
& \le \int_0^T \|M(\,:\! (\pi_N u_0)^2\!:\,)\|_{L^p_{\vec u_0, \o_2} (\rhoo_N\otimes \PP_2)} 
 dt \\
 &\le C(T, p) < \infty, 
\end{split}
\label{F1}
\end{align}

\noi
for any  $0 \le t \le T$ and $p \ge 1$, 
uniformly in $N \in \N$.
By  defining
\[
v_N := u_N - \<1>,
\]

\noi
we see that  $v_N$ satisfies the equation 
\[
(\dt^2 + \dt + 1 - \Dl) v_N
=
\s  \pi_N \big(   :\! (\pi_N u_N)^2 \!:\, \big)
-  M (\,:\! (\pi_N u_N)^2 \!:\,) \pi_N u_N
\]

\noi
with the zero initial data, 
or equivalently
\begin{equation*}
v_N(t) =
\int_0^t e^{- \frac{t-t'}2} \frac{\sin ((t-t') \jbb{\nabla})}{\jbb{\nabla}} \Big(\sigma \pi_N( \,:\! (\pi_N u_N)^2\!:\,) - M(\,:\! (\pi_N u_N)^2\!:\,)  \pi_N u_N\Big)(t') d t'.
\end{equation*}
Thus, we have 
\begin{align*}
\|v_N(t) \|_{W^{-\eps,\infty}_x}
\le \int_0^t \bigg(
& \Big\| \frac{\sin ((t-t') \jbb{\nabla})}{\jbb{\nabla}} \sigma \pi_N( \,:\! (\pi_N u_N)^2\!:\,)(t') \Big\|_{W^{-\eps,\infty}_x} \\
&
+ \Big\|M(\,:\! (\pi_N u_N)^2\!:\,)(t')  \frac{\sin ((t-t') \jbb{\nabla})}{\jbb{\nabla}} \pi_N u_N(t')\Big\|_{W^{-\eps,\infty}_x}
 \bigg)
d t'
\end{align*}

\noi
for any $t>0$.
Then,  by using Minkowski's integral inequality, 
\eqref{F0a}, 
and Proposition \ref{PROP:exptail2} once again,
we have
\begin{align}
\begin{aligned}
\Big\|    \| & v_N(t)\|_{W^{-\eps,\infty}_x} \Big\|_{L^p_{\vec u_0, \o_2} (\rhoo_N\otimes \PP_2)} \\
&\les
\int_0^t  \bigg(\Big\|
 \frac{\sin (\tau \jbb{\nabla})}{\jbb{\nabla}} \pi_N( \,:\! (\pi_N u_0)^2\!:\,) 
 \Big\|_{L^p_{\vec u_0, \o_2} (\rhoo_N\otimes \PP_2; W^{-\eps,\infty}_x)} \\
&\phantom{XXX}
+ \Big\|M(\,:\! (\pi_N u_0)^2 \!:\,) \frac{\sin (\tau \jbb{\nabla})}{\jbb{\nabla}} \pi_N u_0
 \Big\|_{L^p_{\vec u_0, \o_2} (\rhoo_N\otimes \PP_2; W^{-\eps,\infty}_x)} \bigg)
 d \tau\\
&\le C(T, p) < \infty
\end{aligned}
\label{F2}
\end{align}

\noi
for any  $0 \le t \le T$, $p \ge 1$,  and $\eps>0$,
uniformly in $N \in \N$.

We rewrite the system \eqref{SNLW9} as 
\begin{align}
\begin{split}
 (\dt^2 & + \dt  +1 - \Dl) X_N  
 =  2\sigma \pi_N ( v_N \pl \<1>_N )
- M(\,:\! (\pi_N u_N)^2\!:\,)  \<1>_N,\\
 (\dt^2 & + \dt +1  - \Dl) Y_N\\
&  = \sigma \pi_N \Big( v_N \big( X_N+Y_N+\s \<20>_N \big)
+ 2 \big( \Res_N + Y_N\pe \<1>_N + \s \<21p>_N \big) \\
&\quad
+ 2 (X_N+Y_N+\s \<20>_N) \pg \<1>_N \Big) 
- 
M(\,:\! (\pi_N u_N)^2\!:\,)  (X_N+Y_N+\s \<20>_N),\\
\Res_N 
&= 2 \sigma \wt \If_{\pl}^{(1), N}
 \big( X_N+Y_N+\s \<20>_N \big)\pe \<1>_N
+2  \s \wt \If_{\pl, \pe}^N  \big( X_N+Y_N+\s \<20>_N\big)\\
& \hphantom{X}
 - \int_0^t M(\,:\! (\pi_N u_N)^2\!:\,)(t') \Ab_N(t, t') dt', 
\end{split}
\label{F3}
\end{align}

\noi
where we used \eqref{u3} 
(with the frequency truncations
and extra $\s$'s in appropriate places) and $v_N = \s \<20> + X_N + Y_N$
so that the right-hand side is linear in $(X_N, Y_N, \Res_N)$.

Let $\dl > 0$. In view of Proposition \ref{PROP:exptail2}, 
we choose  $K = K(T, \dl)\gg 1$ such that 
\begin{align}
\rhoo_N \otimes \PP_2
\Big(\|\Xi_N (\vec u_0, \o_2) \|_{\mathcal{X}^{\eps}_T} 
>  K \Big) < \frac \dl 3, 
\label{F3x}
\end{align}

\noi
uniformly in $N \in \N$.
We also define $L(t)$ by 
\begin{align}
L(t) = 1 +  \| v_N (t) \|_{W^{-\eps, \infty}_x}
+ |M(\,:\! (\pi_N u_N)^2\!:\,)(t)|.
\label{F3c}
\end{align}

\noi
In view of \eqref{F1} and \eqref{F2}, 
we choose $L_1 = L_1(T, \dl) \gg 1$ such that 
\begin{align}
\rhoo_N \otimes \PP_2
\Big(\| L\|_{L^3_T} 
>  L_1 \Big) < \frac \dl 3.
\label{F3a}
\end{align}

\noi
In the following, 
we work on the set 
\begin{align}
\|\Xi_N (\vec u_0, \o_2) \|_{\mathcal{X}^{\eps}_T} 
\le   K  
\qquad \text{and}\qquad 
\| L\|_{L^3_T} 
\le   L_1 .
\label{F3b}
\end{align}

By applying 
Lemma \ref{LEM:Str} with \eqref{M0}
and  Lemma~\ref{LEM:para}
 to \eqref{F3}
 and using \eqref{data3}, \eqref{F3c}, 
 and 
 \eqref{F3b}, 
we have  
 \begin{align}
\begin{split}
\|  X_N \|_{X^{s_1}(T)}
& \les
\int_0^T \Big( \| v_N  \pl \<1>_N (t)\|_{H^{s_1-1}_x}
+ |M(\,:\! (\pi_N u_N)^2\!:\,)(t)| \cdot \| \<1>_N(t) \|_{H^{s_1-1}_x} \Big)
dt \\
&\les
K 
\int_0^T L(t) dt.
\end{split}
\label{F4}
\end{align}

Since $s_2 < 1$, we can choose sufficiently small $\eps > 0$ 
such that Lemma \ref{LEM:gko}\,(ii) yields
\begin{align*}
 \|v_N(X_N+Y_N+\<20>_N)\|_{H^{s_2-1}_x}
 &\les \|v_N\|_{W^{-\eps,\infty}_x} \| X_N+Y_N+\<20>_N \|_{H^{\eps}_x}\\
 & \les \|v_N\|_{W^{-\eps,\infty}_x}
\Big(\|X_N \|_{H^{s_1}_x} + \| Y_N \|_{H^{s_2}_x} + \|\Xi_N (\vec u_0, \o_2)\|_{\mathcal{X}^{\eps}_T} \Big) .
\end{align*}

\noi
Hence, by \eqref{F3},  Lemma \ref{LEM:Str} with \eqref{M0}, 
 Lemma \ref{LEM:para} (see also \eqref{M4}),
 \eqref{F3c}, and \eqref{F3b}, 
we have
\begin{align}
\| Y_N \|_{Y^{s_2}(T)}
&\les
\int_0^T \Big( \|v_N(t)(X_N(t)+Y_N(t)+\<20>_N(t))\|_{H^{s_2-1}_x}
\notag\\
&\hphantom{XXX}
+ \| \Res_N (t) + Y_N (t)\pe \<1>_N (t) + \s \<21p>_N (t) \|_{H^{s_2-1}_x} \notag\\
&\hphantom{XXX}
+ \| (X_N (t)+Y_N (t)+\s \<20>_N (t)) \pg \<1>_N (t) \|_{H^{s_2-1}_x} \notag\\
&\hphantom{XXX}
+
|M(\,:\! (\pi_N u_N)^2\!:\,)(t)| \cdot \| X_N(t)+Y_N(t)+\s \<20>_N(t) \|_{H^{s_2-1}_x}
\Big)
dt \label{F7}\\
&\le
C(T) K^2
+ K
\int_0^T L(t)
\Big( 1+\| X_N \|_{X^{s_1}(T)} + \| Y_N \|_{Y^{s_2}(t)} \Big) dt \notag\\
&\quad
+ \int_0^T \| \Res_N(t) \|_{H^{s_3}_x} dt.
\notag
\end{align}

Fix $0 < \tau < 1$ and 
set 
\[ L^q_{I_k} = L^q(I_k), \quad \text{where} \quad I_k = [k \tau, (k+1)\tau].\]

\noi
By a computation analogous to that in \eqref{M5}, 
we obtain 
\begin{align}
\begin{split}
\| \Res_N \|_{L^3_{I_k} H^{s_3}_x} 
&\les
\| \wt \If_{\pl}^{(1), N} \big( X_N+Y_N+\s \<20>_N \big)\pe \<1>_N \|_{L^3_{I_k} H^{s_3}_x}
\\
&\hphantom{X}
+ \| \wt \If_{\pl, \pe}^N  \big( X_N+Y_N+\s \<20>_N\big) \|_{L^3_{I_k} H^{s_3}_x}
\\
& \hphantom{X}
+ \int_0^T |M(\,:\! (\pi_N u_N)^2\!:\,)(t')| \cdot \| \Ab_N (t, t')\|_{L^3_t([t', T]; H^{s_3}_x)} dt' \\
&\le
C(T) K^2  \Big(K +  \| X_N \|_{X^{s_1}(T)} + \| Y_N \|_{Y^{s_2}((k+1)\tau)} \Big)
+ K
\int_0^{T} L(t) dt .
\end{split}
\label{F8}
\end{align}

\noi
Given $0 <  t \le T$, 
let $k_*(t)$ be the largest integer such that $k_*(t) \tau \le t$.
Then, from 
\eqref{F7} and \eqref{F8}, 
we have
\begin{align}
\begin{aligned}
\| & Y_N \|_{Y^{s_2}(t)}
\le \| Y_N \|_{Y^{s_2}((k_*(t) + 1)\tau)} 
 \\
&\le C(T)K^2 
+ C_1(T)K^3  
\sum_{k = 0}^{k_*(t)}
\tau^{\frac 23}\Big( 1+ \| L(t) \|_{L^3_{I_k}}\Big)
\Big( 1+\| X_N(t) \|_{X^{s_1}(T)} \Big) \\
& \quad + C_2K
T\sum_{k = 0}^{k_*(t)}
\tau^{\frac 13}\| L(t) \|_{L^3_{I_k}}
+ C_3 K^2
\sum_{k = 0}^{k_*(t)}
\tau^{\frac 23}  \Big(1 + \| L(t)  \|_{L^3_{I_k}}\Big)
 \| Y_N \|_{Y^{s_2}((k+1)\tau)} .
\end{aligned}
\label{F9}
\end{align}

\noi
Now, choose $\tau = \tau(K, L_1) = \tau (T, \dl)>0$ sufficiently small
such that 
\begin{align}
 C_3 K^2 \tau^\frac{2}{3} L_1 \ll1.
\label{F9aa}
 \end{align}

\noi
In view of \eqref{F1} and \eqref{F2}, 
and define $L_2 = L_2(T, \dl) \gg 1$ such that
\begin{align}
\rhoo_N \otimes \PP_2
\bigg(\sum_{k = 0}^{k_*(T)}
\tau^{\frac 13}\Big( 1 + \| L(t) \|_{L^3_{I_k}} \Big)> L_2
 \bigg) < \frac \dl 3.
\label{F9a}
\end{align}

\noi
In the following, 
we work on the set 
\begin{align}
\sum_{k = 0}^{k_*(T)}
\tau^{\frac 13}\Big( 1+ \| L(t) \|_{L^3_{I_k}}\Big)  \le L_2.
\label{F9b}
\end{align}

\noi
It follows from 
\eqref{F9} with \eqref{F3b}, \eqref{F4},
\eqref{F9aa}, 
and \eqref{F9b} that 
\begin{align*}
 \| Y_N \|_{Y^{s_2}((k_*(t) + 1)\tau)} 
&\le C(T)K^4  L_1 L_2  
+ C_4 K^2
\sum_{k = 0}^{k_*(t)-1}
\tau^{\frac 23}  \| L(t)  \|_{L^3_{I_k}}
 \| Y_N \|_{Y^{s_2}((k+1)\tau)} .
\end{align*}

\noi
By applying the discrete Gronwall inequality with \eqref{F9b}, we then obtain
\begin{align}
\begin{split}
\|  Y_N \|_{Y^{s_2}(t)}
& \le \| Y_N \|_{Y^{s_2}((k_*(t) + 1)\tau)} \\
& \le C(T)K^4  L_1 L_2  \exp \bigg(C_4 K^2 \sum_{k = 0}^{k_*(t)-1}
\tau^{\frac 23}  \| L(t)  \|_{L^3_{I_k}}
\bigg)\\
& \le C(T)K^4  L_1 L_2  \exp \big(C_4 K^2 L_2
\big).
\end{split}
\label{F9d}
\end{align}

\noi
Therefore, from \eqref{F4} and \eqref{F9d}, we have 
\begin{align*}
\| X_N \|_{X^{s_1}(T)} + \| Y_N \|_{Y^{s_2}(T)} 
\le C(T) K L_1 + C(T)K^4  L_1 L_2  \exp \big(C_4 K^2 L_2\big).
\end{align*}

\noi
Together with \eqref{F8}, we then obtain 
\begin{align*}
\|(X_N,Y_N,\Res_N)\|_{Z^{s_1, s_2, s_3}(T)} 
\le C_5(T, K, L_1, L_2)
\end{align*}

\noi
under the conditions \eqref{F3b} and \eqref{F9b}.
Hence, by choosing  $C_0 = C_0(T, \dl) > 0$
in \eqref{F0} such that 
$C_0 >  C_5(T, K, L_1, L_2)$, 
we have 
\begin{align}
\begin{split}
\rhoo_N \otimes \PP_2 \bigg(&  \big\{ \|(X_N,Y_N,\Res_N)\|_{Z^{s_1, s_2, s_3}(T)} > C_0\big\}
\cap \big\{\|\Xi_N (\vec u_0, \o_2) \|_{\mathcal{X}^{\eps}_T} 
\le   K\big\}\\
& \cap \big\{\| L\|_{L^3_T}  \le  L_1\big\}
\cap \Big\{\sum_{k = 0}^{k_*(T)}
\tau^{\frac 13}\| L(t) \|_{L^3_{I_k}} \le  L_2\Big\}
 \bigg) 
= 0. 
\end{split}
\label{F11}
\end{align}

\noi
Finally, the bound \eqref{F0} follows from 
\eqref{F3x}, \eqref{F3a}
\eqref{F9a}, and \eqref{F11}.
\end{proof}

Given a map $S$ from a measure space $(X, \mu)$ to a space $Y$, 
we use $S_\#\mu$ to denote the image measure (the pushforward) of $\mu$ under $S$.
Fix $T> 0$ and we set 
\begin{align}
\nu_N = (\Xi_N)_\# (\rhoo_N \otimes \PP_2)
\qquad \text{and}\qquad
 \nu = \Xi_\# (\rhoo \otimes \PP_2), 
\label{OP1}
\end{align}

\noi
where we view $\Xi_N = \Xi_N(\vec u_0, \o_2)$  in \eqref{data3x}
and 
$\Xi = \Xi(\vec u_0, \o_2)$  in \eqref{enh1}
as  maps from 
$\H^{-\frac12 - \eps}(\T^3) \times \O_2$
to $\Xc^\eps_T$ defined in \eqref{data3}. 
In view of the weak convergence of $\rhoo_N \otimes \PP_2$
to $\rhoo \otimes \PP_2$ (Theorem \ref{THM:Gibbs}\,(i))
and the $\rhoo\otimes \PP_2$-almost sure convergence of 
 $\Xi_N(\vec u_0, \o_2)$ to  $\Xi(\vec u_0, \o_2)$ (Corollary \ref{COR:lim}), 
 we see that $\nu_N$ converges weakly to $\nu$.
Indeed, given a bounded continuous function $F: \mathcal{X}^{\eps}_T \to \R$, 
by the dominated convergence theorem, we have
\begin{align*}
\bigg|\int & F(\Xi) d\nu_N  - \int F(\Xi) d\nu\bigg| \\
& = \bigg|\int F(\Xi_N(\vec u_0, \o_2)) d(\rhoo_N\otimes \PP_2) 
- \int F(\Xi(\vec u_0, \o_2)) d(\rhoo\otimes \PP_2) \bigg| \\
& \le \|F\|_{L^\infty}
\bigg|\int 1 \, d\big( (\rhoo_N\otimes \PP_2) - (\rhoo\otimes \PP_2) \big)\bigg|\\
& \quad + \bigg| \int \big( F(\Xi_N(\vec u_0, \o_2)) - F(\Xi(\vec u_0, \o_2)) \big)d(\rhoo\otimes \PP_2) \bigg| \\
& \too 0, 
\end{align*}

\noi
as $N\to \infty$.

Next, we prove that $\nu_N = (\Xi_N)_\# (\rhoo_N \otimes \PP_2)$
converges to 
$\nu = \Xi_\# (\rhoo \otimes \PP_2)$ in the Wasserstein-1 metric.
We view this problem
as 
of Kantorovich's mass optimal transport problem
and study the dual problem under the Kantorovich duality, 
using the Bou\'e-Dupuis variational formula.
This proposition plays a crucial role
in the proof of almost sure global well-posedness
and invariance of the Gibbs measure $\rhoo$
presented at the end of this section.

\begin{proposition} \label{PROP:plan}

Fix $T>0$.
Then,
there exists a sequence $\{\plan_N \}_{N \in \N}$
of probability measures on 
$\mathcal{X}^{\eps}_T \times \mathcal{X}^{\eps}_T$ 
with the first and second marginals $\nu$ and $\nu_N$ on $ \Xc^\eps_T$, respectively, 
namely, 
\begin{align}
  \int_{\Xi^2 \in \Xc^\eps_T} d\plan_N(\Xi^1, \Xi^2) =d \nu(\Xi^1) 
\qquad \text{and}\qquad
 \int_{\Xi^1 \in \Xc^\eps_T} d\plan_N(\Xi^1, \Xi^2) =d \nu_N(\Xi^2),  
 \label{OP2}
\end{align}
such that 
\begin{align*}
\int_{\mathcal{X}^{\eps}_T\times \mathcal{X}^{\eps}_T} \min(\| \Xi^1 - \Xi^2\|_{\mathcal{X}^{\eps}_T}, 1) d \plan_{N}(\Xi^1, \Xi^2) \too 0, 
\end{align*}

\noi
as $N \to \infty$.
Namely, the total transportation cost associated to $\plan_N$
tends to $0$ as $N \to \infty$.

\end{proposition}

\begin{remark}\rm

In view of the weak convergence of the truncated Gibbs measure $\rhoo_N$
to $\rhoo$ (Theorem \ref{THM:Gibbs})
and the almost sure convergence of the truncated enhanced data set $\Xi_N$
to $\Xi$ with respect to  $\rhoo \otimes \PP_2$ (Corollary \ref{COR:lim}), 
it suffices to define 
$\plan_N = (\Xi,\Xi_N)_\# (\rhoo \otimes \PP_2).$
In the following, however, we present the full proof of Proposition \ref{PROP:plan}, 
using the Kantorovich duality and the variational approach
since we believe that such an argument is of general interest.

\end{remark}

\begin{proof}[Proof of Proposition \ref{PROP:plan}]

Define a cost function $c(\Xi^1, \Xi^2)$ on $\mathcal{X}^{\eps}_T\times \mathcal{X}^{\eps}_T$
by setting
\begin{align*}
c(\Xi^1, \Xi^2) = \min(\| \Xi^1 - \Xi^2 \|_{\mathcal{X}^{\eps}_T}, 1).
\end{align*}

\noi
Then, define the Lipschitz norm
for a function $F: \mathcal{X}^{\eps}_T \to \R$ by 
\[
\|F\|_{\Lip} = 
\sup_{\substack{\Xi^1, \Xi^2 \in \mathcal{X}^{\eps}_T\\\Xi^1 \ne \Xi^2}} \frac{|F(\Xi^1) - F(\Xi^2)|}
{c(\Xi^1, \Xi^2)}.
\]

\noi
Note that 
 $\| F \|_{\Lip} \le 1$ implies that
$F$ is bounded and Lipschitz continuous.
From the Kantorovich duality (the Kantorovich-Rubinstein theorem 
\cite[Theorem 1.14]{Villani}),
we have 
\begin{align}
\begin{aligned}
&\inf_{\plan \in \Gamma(\nu,\nu_N)}
\int_{\mathcal{X}^{\eps}_T\times \mathcal{X}^{\eps}_T}
c( \Xi^1, \Xi^2 )d \plan (\Xi^1, \Xi^2)
\\
& \quad 
= \sup_{\|F\|_{\Lip} \le 1}
\bigg( \int F(\Xi) d \nu_N(\Xi) - \int F(\Xi) d \nu(\Xi) \bigg),
\end{aligned}
\label{Ga0}
\end{align}
where $\Gamma(\nu,\nu_N)$ is the set of probability measures
on $\mathcal{X}^{\eps}_T\times \mathcal{X}^{\eps}_T$
with the first and second marginals $\nu$ and $\nu_N$ on $\mathcal{X}^{\eps}_T$, respectively.

For a function $F$ with $\|F\|_{\Lip} \le 1$,
let 
$$ G := F - \inf F + 1. $$
Then, we have
\begin{align}
\begin{aligned}
\int F(\Xi) d \nu_N(\Xi) - \int F(\Xi) d \nu(\Xi)
&= \int G(\Xi) d \nu_N(\Xi) - \int G(\Xi) d \nu(\Xi).
\end{aligned}
\label{Ga1}
\end{align}

\noi
Note that $\|G\|_{\Lip} = \|F\|_{\Lip} \le 1$ and $ 1 \le G \le 2$.
Moreover,
the mean value theorem yields that
\begin{align}
\frac 1e  \le \frac{\log x - \log y}{x-y} \le 1
\label{Ga2}
\end{align}
for any $x,y \in [1,e]$ with $x \neq y$.
Set $\{ a \}_+ = \max (a, 0)$ for any $a \in \R$.
By \eqref{Ga1} and \eqref{Ga2},
we obtain
\begin{align}
\begin{split}
\int & F(\Xi) d \nu_N(\Xi) - \int F(\Xi) d \nu(\Xi)\\
&\les
\bigg\{ -\log \Big( \int G(\Xi) d \nu(\Xi) \Big)
+ \log \Big(\int G(\Xi) d \nu_N(\Xi) \Big) \bigg\}_+
\end{split}
\label{Ga4}
\end{align}
for any $N \in \N$.

Finally, define $H = \log G$.
Then,  from \eqref{Ga2} 
and $1 \le G \le 2$, we have  $\|H\|_{\Lip} \les 1$.
Hence,
it follows from \eqref{Ga0}, \eqref{Ga1},  and \eqref{Ga4} that
\begin{align*}
& \inf_{\plan \in \Gamma(\nu,\nu_N)}
\int_{\mathcal{X}^{\eps}_T\times \mathcal{X}^{\eps}_T} 
c(\Xi^1, \Xi^2)  d \plan (\Xi^1, \Xi^2) \\
& \quad \les \sup_{\substack{0 \le H \le 1 \\ \| H \|_{\Lip} \les 1}}
\bigg\{ -\log\Big(\int \exp(H(\Xi)) d \nu(\Xi)\Big) + \log\Big(\int \exp(H(\Xi)) d \nu_N(\Xi)\Big) \bigg\}_+.
\end{align*}

\noi
Our goal is to show that the right-hand side tends 0 as $N \to \infty.$
Since
$\|H\|_{\Lip} \les 1$,
$H$ is bounded and Lipschitz continuous.
Then, 
by the weak convergence of $\{ \nu_N \}_{N \in \N}$ to $\nu$,
it suffices to show that 
\begin{equation} \label{iv1}
\begin{aligned}
\limsup_{N \to \infty} \sup_{\substack{0 \le H \le 1 \\ \| H \|_{\Lip} \les 1}} \sup_{M \ge N}
\bigg\{
&-\log\Big(\int \exp(H(\Xi)) d \nu_M(\Xi)\Big) \\
&\quad
+ \log\Big(\int \exp(H(\Xi)) d \nu_N(\Xi)\Big) \bigg\}_+
\le 0. 
\end{aligned}
\end{equation}

From \eqref{OP1}, \eqref{rhooN}, 
and \eqref{Ga2} with 
$0 \le H \le 1$,  
we have
\begin{align}
\bigg\{& -\log\Big(\int \exp(H(\Xi)) d \nu_M(\Xi)\Big) + \log\Big(\int \exp(H(\Xi)) d \nu_N(\Xi)\Big)\bigg\}_+ 
\notag \\
&= \bigg\{-\log\Big(\iiint \exp(H(\Xi_M(\vec u_0,  \o_2)))  d \rho_M(u_0) d \mu_0(u_1) d\PP_2(\o_2)\Big) 
\notag \\
&\hphantom{XXX}
+ \log\Big(\iiint \exp(H(\Xi_N(\vec u_0,  \o_2))) d \rho_N(u_0)d \mu_0(u_1) d\PP_2(\o_2)\Big)\bigg\}_+
\notag \\
&\les \bigg\{-\iiint \exp(H(\Xi_M(\vec u_0, \o_2))) d \rho_M(u_0)d \mu_0(u_1) d\PP_2(\o_2)
\notag \\
&\hphantom{XXX}
+ \iiint \exp(H(\Xi_N(\vec u_0, \o_2))) d \rho_N(u_0)d \mu_0(u_1) d\PP_2(\o_2)\bigg\}_+ 
\notag \\
&\les
\iint \bigg[\Big\{ - \int \exp(H(\Xi_M(\vec u_0,\o_2))) d \rho_M(u_0) 
\notag \\
&\hphantom{XXXXXX}
+ \int \exp(H(\Xi_N(\vec u_0, \o_2))) d \rho_N(u_0) \Big\}_+\bigg] d \mu_0(u_1)d\PP_2(\o_2)
\notag \\
&\les \iint \bigg[ \bigg\{ -\log \Big( \int \exp(H(\Xi_M(\vec u_0, \o_2))) d \rho_M(u_0)\Big) 
\notag \\
&\hphantom{XXXXXX}
+ \log\Big(\int \exp(H(\Xi_N(\vec u_0, \o_2))) d \rho_N(u_0) \Big)\bigg\}_+ \bigg] d \mu_0(u_1)d\PP_2(\o_2).
\label{iv1aa}
\end{align}

In the following, we study the integrand of the $(u_1, \o_2)$-integral.
Thus,  we fix $u_1$ and $\o_2$
and write 
 $\Xi_N(\vec u_0, \o_2)
=  \Xi_N( u_0, u_1, \o_2)$ as $\Xi_N(u_0)$
for simplicity of notation.
By the Bou\'e-Dupuis variational formula (Lemma \ref{LEM:var3})
 with the change of variables \eqref{YZ13},
we have
\begin{align}
\begin{aligned}
 - & \log \bigg(
\int \exp(H(\Xi_M(u_0))) d \rho_M(u_0) \bigg)
+ \log \bigg(\int \exp(H(\Xi_N(u_0))) d \rho_N(u_0) \bigg) \\
&= \inf_{\dot \Ups^{M}\in  \mathbb H_a^1}
\E \bigg[ 
- H(\Xi_M(Y+ \Ups^{M}+ \s \ZZ_{M}))
+ \ft R^{\dia}_{M}(Y+ \Ups^{M}+ \s \ZZ_{M})
\\
&\hphantom{XXXXXXX}
+ \frac12 \int_0^1 \| \dot \Ups^{M}(t) \|_{H^1_x}^2dt \bigg]\\
&\phantom{=} - \inf_{\dot \Ups^{N}\in  \mathbb H_a^1}
\E \bigg[ 
- H(\Xi_N(Y+ \Ups^{N}+ \s \ZZ_{N}))
+ \ft R^{\dia}_{N}(Y+ \Ups^{N}+ \s \ZZ_{N})  \\
&\hphantom{XXXXXXX}
+ \frac12 \int_0^1 \| \dot \Ups^{N}(t) \|_{H^1_x}^2dt \bigg] \\
&\quad
+ \log Z_M - \log Z_N,
\end{aligned}
\label{Ga6}
\end{align}
where $\ft R_N^\dia$ is as in \eqref{KZ16}.
Given 
 $\dl>0$,
let $\UUps^N$ be an almost optimizer, namely, 
\begin{align*} 
& \inf_{\dot \Ups^{N}\in  \mathbb H_a^1}
\E \bigg[ 
- H(\Xi_N(Y+ \Ups^{N}+ \s \ZZ_{N}))
+ \ft R^{\dia}_{N}(Y+ \Ups^{N}+ \s \ZZ_{N}) + \frac12 \int_0^1 \| \dot \Ups^{N}(t) \|_{H^1_x}^2dt \bigg] \\
&\quad \ge
\E \bigg[ 
- H(\Xi_N(Y+ \UUps^{N}+ \s \ZZ_{N}))
+ \ft R^{\dia}_{N}(Y+ \UUps^{N}+ \s \ZZ_{N}) 
+ \frac12 \int_0^1 \| \dot \UUps^{N}(t) \|_{H^1_x}^2dt \bigg] -\dl.
\end{align*}

\noi
Then,
by choosing $\Ups^M=\UUps^N$
and the Lipschitz continuity of $H$,
we have
\begin{align}
&\inf_{\dot \Ups^{M}\in  \mathbb H_a^1}
\E \bigg[ 
- H(\Xi_M(Y+ \Ups^{M}+ \s \ZZ_{M}))
+ \ft R^{\dia}_{M}(Y+ \Ups^{M}+ \s \ZZ_{M}) + \frac12 \int_0^1 \| \dot \Ups^{M}(t) \|_{H^1_x}^2dt \bigg]
\notag \\
&\hphantom{XX}
- \inf_{\dot \Ups^{N}\in  \mathbb H_a^1}
\E \bigg[ 
- H(\Xi_N(Y+ \Ups^{N}+ \s \ZZ_{N}))
+ \ft R^{\dia}_{N}(Y+ \Ups^{N}+ \s \ZZ_{N}) + \frac12 \int_0^1 \| \dot \Ups^{N}(t) \|_{H^1_x}^2dt \bigg]
\notag \\
&\hphantom{X}
\le
\dl +\E \Big[ H(\Xi_N(Y+ \UUps^{N}+ \s \ZZ_{N})) - H(\Xi_M(Y+ \UUps^{N}+ \s \ZZ_{M}))
\notag  \\
&\phantom{XXXXX}
+ \ft R^{\dia}_{M}(Y+ \UUps^{N}+ \s \ZZ_{M}) - \ft R^{\dia}_{N}(Y+ \UUps^{N}+ \s \ZZ_{N}) \Big]
\notag \\
&\hphantom{X}
\le
\dl 
+ \| H \|_{\Lip} \cdot \E \Big[ \| \Xi_M(Y+ \UUps^{N}+ \s \ZZ_{N}) 
- \Xi_N(Y+ \UUps^{N}+ \s \ZZ_{M} ) \|_{\mathcal{X}^{\eps}_T} \Big] 
\notag \\
&\phantom{XXXXX}
+ \E \bigg[ 
\ft R^{\dia}_{M}(Y+ \UUps^{N}+ \s \ZZ_{M}) - \ft R^{\dia}_{N}(Y+ \UUps^{N}+ \s \ZZ_{N}) \bigg].
\label{Ga7}
\end{align}

\noi
Proceeding as in 
 Subsection \ref{SUBSEC:wcon} with $0 \le H \le 1$, 
 we have \eqref{KZ20a}.
 Then, 
using the computations from \eqref{KZ18b} to \eqref{KZ21}
we obtain
\begin{align}
\E \bigg[\ft R^{\dia}_{M}(Y+ \UUps^{N}+ \s \ZZ_{M}) - \ft R^{\dia}_{N}(Y+ \UUps^{N}+ \s \ZZ_{N})\bigg] \too 0, 
\label{Ga8}
\end{align}
as $M \ge N \to \infty$.
We also note that as a consequence of \eqref{KZ20a} with \eqref{K10}
and Lemma \ref{LEM:Dr}, 
we have
\begin{equation}
\E \Big[     \, \| \UUps^{N} \|_{H^1}^2 \Big] \les 1,
\label{Ga8a}
\end{equation}
uniformly in $N \in \N$.

Moreover, by slightly modifying (part of)  the proof of Proposition \ref{PROP:exptail2},
we can show that 
\begin{align}
\begin{aligned}
&\E \Big[ \big\| \Xi_M(Y+ \UUps^{N}+ \s \ZZ_{N}) - \Xi_N(Y+ \UUps^{N}+ \s \ZZ_{M}) \big\|_{\mathcal{X}^{\eps}_T} \Big]
\too 0, 
\end{aligned}
\label{Ga9}
\end{align}
as $M\ge N \to \infty$.
Here,
we only consider the contribution from $\wt \If^N_{\pl, \pe}$.  
The other terms 
in the truncated enhanced data sets can be handled in a similar manner.
With the notations \eqref{et1b} and 
 \eqref{et2} (recall that we suppress the dependence on $u_1$ and $\o_2$),
we have
\begin{align}
\begin{split}
&\wt \If_{\pl, \pe}^M [\<1>(Y+ \UUps^{N}+ \s \ZZ_{M})] 
- \wt \If_{\pl, \pe}^N [\<1>(Y+ \UUps^{N}+ \s \ZZ_{N})] \\
&=
\wt \If_{\pl, \pe}^M [\<1>(Y+ \UUps^{N}+ \s \ZZ_{M}), \<1>(\s (\ZZ_{M} - \ZZ_{N}))]\\
& \quad + \wt \If_{\pl, \pe}^M [ \<1>(\s (\ZZ_{M} - \ZZ_{N})), \<1>(Y+ \UUps^{N}+ \s \ZZ_{N})] \\
&\quad
+\Big(\wt  \If_{\pl, \pe}^M [\<1>(Y+ \UUps^{N}+ \s \ZZ_{N})] 
- \wt \If_{\pl, \pe}^N [\<1>(Y+ \UUps^{N}+ \s \ZZ_{N})] \Big)\\
&=:
\1 + \II + \III.
\end{split}
\label{Gc1}
\end{align}
It follows from \eqref{et5}, \eqref{et6}, and \eqref{et7} 
together with Remark \ref{REM:decay} that
there exists small $\dl_0 > 0$ such that 
\begin{align}
\begin{split}
\| \1 & \|_{\L_2(q,T)} + \| \II \|_{\L_2(q,T)}\\
& \le
C(T)
\Big( \| Y \|_{L^\infty_T W^{-\frac 12-\eps, \infty}_x} + \| \UUps^{N} \|_{H^{1}} 
+ \| \ZZ_{N} \|_{W^{1-\eps, \infty}} \Big)
\| \ZZ_{N} - \ZZ_{M} \|_{W^{1-\eps, \infty}}\\
& \le
C(T)
N^{-\dl_0} \Big( \| Y \|_{L^\infty_T W^{-\frac 12-\eps, \infty}_x} + \| \UUps^{N} \|_{H^{1}} 
+ \| \ZZ_{N} \|_{W^{1-\eps, \infty}} \Big)^2\\
& \quad + 
N^{\dl_0} \| \ZZ_{N} - \ZZ_{M} \|_{W^{1-\eps, \infty}}^2
\end{split}
\label{Gc2}
\end{align}

\noi
and 
\begin{align}
\E \Big[N^{\dl_0} \| \ZZ_{N} - \ZZ_{M} \|_{W^{1-\eps, \infty}}^2\Big]\too 0, 
\label{Gc2a}
\end{align}

\noi
as $M\ge N \to \infty$. 
From \eqref{et1aa}
and \eqref{et1b}, we have 
\begin{align*}
\wt \If^N_{\pl, \pe}[\psi_1, \psi_2](w) 
   =
\I \big(\pi_N( \K^\ta (w,  \pi_N \psi_1))\big)\pe (\pi_N \psi_2).
\end{align*}

\noi
Hence, when we consider the difference 
in $\III$, 
we see that 
 one of the factors comes with $\pi_{M} - \pi_{N}$, 
 from which 
we can gain a small negative power of $N$.
Hence,  by repeating the   calculation above with this observation, 
we obtain
\begin{align}
\begin{split}
\big\| \III 
& - (\wt  \If_{\pl, \pe}^N [\<1>(Y)]
- \wt  \If_{\pl, \pe}^M [\<1>(Y)])
\big\|_{\L_2(q,T)}\\
& \les
N^{-\dl_0}
\Big( \| Y \|_{L^\infty_T W^{-\frac 12-\eps, \infty}_x} + \| \UUps^{N} \|_{H^{1}} 
+ \| \ZZ_{N} \|_{W^{1-\eps, \infty}} \Big)^2
\end{split}
\label{Gc3}
\end{align}
for any $M \ge N\ge 1$.
Lastly, from \eqref{Xet4} and \eqref{et2b}, 
there exists $\dl > 0$ such that
\begin{align}
\begin{split}
\big\| \
& \wt  \If_{\pl, \pe}^N [\<1>(Y)]
- \wt  \If_{\pl, \pe}^M [\<1>(Y)]
\big\|_{\L_2(q,T)} 
\le N^{-\dl_0}\wt K(Y, u_1, \o_2)
\end{split}
\label{Gc3a}
\end{align}
for any $M \ge N\ge 1$, 
where, in view of \eqref{Xet4a},  
$\E[ \wt K(Y, u_1, \o_2)] \le C(u_1, \o_2) < \infty$ for almost every $u_1$ and $\o_2$.
Therefore, from \eqref{Gc1}, \eqref{Gc2}, \eqref{Gc2a}, 
\eqref{Gc3}, and \eqref{Gc3a} with the bound~\eqref{Ga8a}, 
we obtain
\[
\E \Big[ \big\| \wt \If_{\pl, \pe}^M [Y+ \Ups_\dl^{N}+ \s \ZZ_{M}] 
- \wt \If_{\pl, \pe}^N [Y+ \Ups_\dl^{N}+ \s \ZZ_{N}] \big\|_{\L_2(q,T)} \Big]
\too 0, 
\]
as $M\ge N \to \infty$.

Note that $\{ Z_N \}_{N\in \N}$ is a convergent sequence and $\dl>0$ was arbitrary.
Hence, it follows from \eqref{Ga6}, \eqref{Ga7}, \eqref{Ga8}, and \eqref{Ga9} that
\begin{align}
\begin{aligned}
\limsup_{N \to \infty}
 \sup_{\substack{0 \le H \le 1 \\ \| H \|_{\Lip} \les 1}} 
 \sup_{M \ge N}
&\bigg\{ -\log \bigg(
\int \exp(H(\Xi_M(u_0, u_1, \o_2))) d \rho_M(u_0) \bigg) \\
&\qquad
+ \log \bigg(\int \exp(H(\Xi_N(u_0, u_1, \o_2))) d \rho_N(u_0) \bigg) \bigg\}_+
\le 0, 
\end{aligned}
\label{Ga10}
\end{align}

\noi
for almost every $u_1$ and $\o_2$, 
where the supremum in $H$ was trivially dropped in the last step of~\eqref{Ga7}.
Therefore, \eqref{iv1} follows from 
\eqref{iv1aa} and \eqref{Ga10}
with 
Fatou's lemma.
This concludes the proof of Proposition \ref{PROP:plan}.
\end{proof}

Finally, we present the proof of Theorem \ref{THM:GWP}.

\begin{proof}[Proof of Theorem \ref{THM:GWP}]
$\bullet$ {\bf Part 1:}
We first prove almost sure global well-posedness
of the hyperbolic $\Phi^3_3$-model.
As in \cite{BO94, CO, BOP2}, 
it suffices to prove ``almost'' almost sure global  well-posedness.
 More precisely, 
 it suffices to prove that given any $T >0$ and small $\dl > 0$, 
there exists $\Si_{T, \dl} \subset  \H^{-\frac{1}{2}-\eps} (\T^3) \times \O_2$
with $\rhoo\otimes\PP_2( \Si_{T, \dl}^c) < \dl$
such that for each 
$(\vec u_0, \o_2) \in  \Si_{T, \dl}$, 
the solution 
 $(X,Y,\Res)$  to  \eqref{SNLW6}, with 
 the zero initial data and 
 the enhanced data  $\Xi(\vec u_0, \o)$ in~\eqref{enh1}, 
 exists on the time interval $[0, T]$.

We assume this 
``almost'' almost sure global  well-posedness
claim for the moment.
Denote by 
 $(X_N,Y_N,\Res_N)$ 
 the  solution to the truncated system \eqref{SNLW9}  with 
 the truncated enhanced data $\Xi_N(\vec u_0, \o)$ in \eqref{data3x}
 and set
 \begin{align}
 u_N(\vec u_0, \o_2) = \<1>(\vec u_0, \o_2) + \s \<20>_N(\vec u_0, \o_2)
 + X_N + Y_N, 
 \label{sol1}
 \end{align}

\noi
which is the solution to the truncated hyperbolic $\Phi^3_3$-model \eqref{SNLW3a}
  with the initial data $(u_N, \dt u_N)|_{t = 0} = \vec u_0 = (u_0, u_1)$
  and the noise $\xi = \xi(\o_2)$.
Here, we used the uniqueness of the solution $u_N$
to \eqref{SNLW3a}; see Remark \ref{REM:uniq}.  
  Then, we 
conclude from Corollary~\ref{COR:lim} (on the almost sure convergence
of $\Xi_N(\vec u_0, \o)$ to $\Xi(\vec u_0, \o)$)
and the second part of Proposition~\ref{PROP:LWPv} 
that $(u_N, \dt u_N)(\vec u_0,  \o_2) $ in \eqref{sol1} converges  to $(u, \dt u)(\vec u_0, \o_2)$ 
in $C([0, T]; \H^{-\frac 12-\eps}(\T^3))$
 for each 
$(\vec u_0, \o_2) \in \Si_{T, \dl}$, 
where  $u(\vec u_0, \o_2)$ is defined by 
 \begin{align}
 u(\vec u_0, \o_2) = \<1>(\vec u_0, \o_2) + \s \<20>(\vec u_0, \o_2)
 + X + Y.
 \label{sol2}
 \end{align}

\noi
Now, we define 
\[\Si = \bigcup_{k = 1}^\infty \bigcap_{j = 1}^\infty \Si_{2^j, 2^{-j} k^{-1} }.\]

\noi
Then, we have $\rhoo\otimes \PP_2(\Si) = 1$
and, for each $(\vec u_0, \o_2) \in \Si$, 
the solution  $(u_N, \dt u_N)(\vec u_0,  \o_2) $ 
to the truncated hyperbolic $\Phi^3_3$-model \eqref{SNLW3a}
converges to $(u, \dt u)(\vec u_0, \o_2)$ in \eqref{sol2}
in $C(\R_+; \H^{-\frac 12-\eps}(\T^3))$
(endowed with the compact-open topology in time).
This proves the almost sure global well-posedness claim in Theorem~\ref{THM:GWP}, 
assuming 
``almost'' almost sure global  well-posedness.

We now prove 
``almost'' almost sure global  well-posedness.
Fix $T>0$ and small $\dl > 0$.
Given $\Xi = (\Xi_1, \dots, \Xi_6)\in \Xc^\eps_T$, 
let  $Z(\Xi)  = (X, Y, \Res)(\Xi)$ be the solution to \eqref{SNLW6}
with the zero initial data and the enhanced data set given by $\Xi$, 
namely,  $\Xi_j$ replacing the $j$th element in~\eqref{Xi1}.
Note that $\Xi$ here
denotes a general element in $\Xc^\eps_T$ and 
is not associated with 
any specific $(\vec u_0, \o_2) \in  \H^{-\frac 12 -\eps} (\T^3) \times \O_2$.
Similarly, given $N \in \N$ and $\Xi \in \Xc^\eps_T$, 
let 
$Z_N (\Xi) = (X_N, Y_N, \Res_N)(\Xi)$ be the solution to \eqref{SNLW9}
with the enhanced data set~$\Xi$, 
namely, $\Xi_j$ replacing the $j$th element of $\Xi_N(\vec u_0, \o_2)$ in \eqref{data3x}.

Given $C_0 > 0$, 
define the set $\Si_{C_0} \subset \Xc^\eps_T$
such that,   
for each 
$\Xi \in \Si_{C_0}$, 
 the solution 
 $Z(\Xi)$  to~\eqref{SNLW6}, with 
 the zero initial data and 
 the enhanced data  $\Xi$, 
 exists on the time interval $[0, T]$, satisfying the bound 
 \begin{align}
\|Z(\Xi)\|_{Z^{s_1, s_2, s_3}(T)} \le C_0+ 1 .
\label{sol3}
\end{align}

Let $N \in \N$.
Given   $K, C_0 > 0$,  
 we set 
 \begin{align}
A_{N, K, C_0} &= \big\{ \Xi' \in \mathcal{X}^{\eps}_T:
\|\Xi'\|_{\mathcal{X}^{\eps}_T} \le K, \,
\|Z_N(\Xi')\|_{Z^{s_1, s_2, s_3}(T)} \le C_0 \big\}
\label{sol3a}
\end{align}

 \noi
 and 
 \begin{align}
 B_{N, K, C_0} 
= 
\big\{ (\Xi, \Xi')  \in \Xc^\eps_T\times \Xc^\eps_T:
 \| \Xi - \Xi'  \|_{\mathcal{X}^{\eps}_T} \le \kk, \ \Xi' \in A_{N, K, C_0}\big\},
\label{sol4}
\end{align}

\noi
where $\kk > 0$ is a small number to be chosen later.
Then, from the stability result (the first claim in Proposition~\ref{PROP:LWPv})
with \eqref{sol3}, \eqref{sol3a}, and \eqref{sol4}, 
there exists
small  $\kk(T, K, C_0)\in (0, 1)$ and $N_0 = N_0(T, K, C_0) \in \N$ such that 
\begin{align}
\Si_{C_0} \times \Xc^\eps_T \supset B_{N, K, C_0}
\label{sol5}
\end{align}

\noi
for any $N \ge N_0$. 

Let $C_0 = C_0(T, \dl) \gg 1$ be as in Proposition \ref{PROP:tail2}
and let $\plan_N$, $N \in \N$, be as in Proposition~\ref{PROP:plan}.
Then, 
from 
\eqref{OP1}, 
\eqref{OP2}, 
and  \eqref{sol5}, we have
\begin{align}
\begin{aligned}
\rhoo  & \otimes \PP_2 \big( \Xi(\vec u_0, \o_2) \in \Si_{C_0}\big) 
 = 
 \int 
\ind_{\Xi \in \Si_{C_0}}(\Xi, \Xi')
d \plan_N (\Xi,\Xi')\\
& \ge  \int 
\ind_{B_{N, K, C_0}}(\Xi, \Xi')
d \plan_N (\Xi,\Xi')\\
&\ge 1 - \int \ind_{\{\| \Xi - \Xi'\|_{\mathcal{X}^{\eps}_T} >  \kk\}} 
d \plan_N(\Xi,\Xi') - \int \ind_{A_{N, K, C_0}^c}(\Xi')d \plan_N(\Xi,\Xi')\\
&\ge 1 - \frac 1 \kk \int \min(\| \Xi - \Xi'\|_{\mathcal{X}^{\eps}_T}, 1) d \plan_N(\Xi,\Xi')
- \rhoo_N \otimes \PP_2 ( \{ \Xi_N (\vec u_0', \o_2')\in A_{N, K, C_0}^c \})\\
&>  1 - \frac 1 \kk \int \min(\| \Xi - \Xi'\|_{\mathcal{X}^{\eps}_T}, 1) d \plan_N(\Xi,\Xi')
- 2\dl , 
\end{aligned}
\label{iv2a}
\end{align}

\noi
where the last step follows from Proposition \ref{PROP:exptail2}
by choosing $K = K(\dl)\gg1 $, together with Proposition~\ref{PROP:tail2}.
By Proposition \ref{PROP:plan},
we have
\begin{align}
\frac 1 \kk \int \min(\| \Xi - \Xi_N'\|_{\mathcal{X}^{\eps}_T}, 1) d \plan_N(\Xi,\Xi_N') \too 0, 
\label{iv2b}
\end{align}

\noi
as $N \to \infty$.
Therefore, we conclude from 
 \eqref{iv2a} and  \eqref{iv2b} that 
$$ \rhoo   \otimes \PP_2 \big( \Xi(\vec u_0, \o_2) \in \Si_{C_0}\big) >  1 - 2\dl.
$$

\noi
This proves ``almost'' almost sure global  well-posedness
with 
\[\Si_{T, \dl} = \{(\vec u_0, \o_2) \in  \H^{-\frac{1}{2}-\eps} (\T^3) \times \O_2: \Xi(\vec u_0, \o_2) \in \Si_{C_0}\},\]

\noi
and hence  almost sure global well-posedness
of the hyperbolic $\Phi^3_3$-model, 
namely, the unique limit $u = u(\vec u_0, \o_2)$
in \eqref{sol2} exists globally in time almost surely with respect to $\rhoo \otimes \PP_2$.

\smallskip

\noi
$\bullet$ {\bf Part 2:}
Next, we prove invariance of the Gibbs measure $\rhoo = \rho \otimes \mu_0$
under the limiting hyperbolic $\Phi^3_3$-dynamics.
In the following, we prove
\begin{align}
\int  F(\Phi (t) (\vec u_0,  \o_2))  d (\rhoo \otimes \PP_2)(\vec u_0, \o_2)
= \int F(\vec u_0) d \rhoo (\vec u_0)
\label{iv30}
\end{align}
for any  bounded Lipschitz functional $F: \C^{-100} (\T^3)
\times \C^{-100}(\T^3)
 \to  \R$ and  $t \in \R_+$, 
 where $\Phi(\vec u_0,  \o_2)$ is the limit
 of the solution 
$(u_N, \dt u_N) = \Phi^N(\vec u_0,  \o_2)$
to the truncated hyperbolic $\Phi^3_3$-model 
defined in  \eqref{sol0}.

As in Part 1, we use the notation 
$(X,Y,\Res) = (X,Y,\Res) (\Xi)$, etc.
Also, 
 let $\plan_N$, $N \in \N$,  be as in Proposition~\ref{PROP:plan}.
Then, by 
the decomposition \eqref{Ba4b} (also for $N = \infty$), 
\eqref{OP1}, \eqref{OP2}, 
and the invariance of  $\rhoo_N $ under the truncated hyperbolic $\Phi^3_3$-model \eqref{SNLW3a} (Lemma \ref{LEM:GWP4}),
we have 
\begin{align*}
\int &  F(\Phi(t) (\vec u_0,  \o_2))   d (\rhoo \otimes \PP_2)(\vec u_0, \o_2)\\
&= \int F(\Phi(t) (\Xi)) d \plan_N(\Xi, \Xi') \nonumber\\
&= \int F(\Phi^N(t) (\Xi') )d \plan_N(\Xi, \Xi') 
+ \int \big[ F(\Phi(t) (\Xi))  - F(\Phi^N(t) (\Xi'))\big] d \plan_N(\Xi, \Xi')
\notag\\
&=
\int F(\vec u_0) d \rhoo_N (\vec u_0) 
+  \int \big[ F(\Phi(t) (\Xi))  - F(\Phi^N(t) (\Xi'))\big] d \plan_N(\Xi, \Xi').
\end{align*}

\noi
By  the weak convergence of $ \rhoo_N $ to $\rhoo$, we have 
\[
\lim_{N \to \infty} \int F(\vec u_0) d \rhoo_N (\vec u_0) 
= \int F(\vec u_0) d \rhoo (\vec u_0).
\]

\noi
Hence, since 
$F$ is bounded and Lipschitz, 
\eqref{iv30} is reduced to showing that
\begin{align}
\begin{aligned}
&\int \min\big(\|\Phi(t) (\Xi) - \Phi^N(t) (\Xi')\|_{\C^{-100}\times \C^{-100}}, 1\big)  d \plan_N(\Xi, \Xi') \too 0, 
\end{aligned}
\label{iv3a}
\end{align}
as $N \to \infty$.

As in \eqref{sol0}, 
we write
\[\Phi(t)(\Xi) = \big(\Phi_1(t)(\Xi), \Phi_2(t)(\Xi)\big)
\qquad \text{and}
\qquad 
\Phi^N(t)(\Xi') = \big(\Phi_1^N(t)(\Xi'), \Phi_2^N(t)(\Xi')\big), 
\]

\noi
where $\Xi = (\Xi_1, \dots, \Xi_6)$ and $\Xi' = (\Xi'_1, \dots, \Xi'_6)$
 (see also \eqref{data3x} and \eqref{enh1}).
With the decomposition as in 
 \eqref{Ba4b},  we have 
\begin{align}
\begin{split}
\Phi_1(t) (\Xi) 
& = \Xi_1 + \s \Xi_3 + X(\Xi) + Y(\Xi),   \\
\Phi_1^N(t) (\Xi') 
& = \Xi_1' + \s \Xi_3' + X_N(\Xi') + Y_N(\Xi'), 
\end{split}
\label{iv33}
\end{align}

\noi
and $ \Phi_2(t)(\Xi) = \dt \Phi_1(t)(\Xi)$ and $\Phi_2^N(t)(\Xi')= \dt \Phi_1^N(t)(\Xi')$
are given by term-by-term differentiation
of the terms on the right-hand sides of \eqref{iv33}.
From the definition \eqref{data3} of the $\Xc^\eps_T$-norm, 
 we clearly have 
\begin{align*}
\| ( & \Xi_1 + \s \Xi_3)(t)
- (\Xi_1' + \s \Xi_3' )(t) \|_{\C^{-100}}\\
& + \|  (\dt \Xi_1 + \s \dt \Xi_3)(t)
- (\dt \Xi_1' + \s \dt \Xi_3' )(t) \|_{\C^{-100}}
\les  \| \Xi - \Xi' \|_{\mathcal{X}^{\eps}_T}.
\end{align*}

 \noi
Hence, in view of \eqref{Z1} with \eqref{M0}, 
\eqref{iv3a} is reduced to showing that
\begin{align}
\int \min\big(\| Z (\Xi) - Z_N(\Xi')\|_{Z^{s_1,s_2,s_3}(T)}, 1 \big) 
d \plan_N(\Xi, \Xi') \too 0, 
\label{iv3c}
\end{align}

\noi
as $N \to \infty$, 
where $Z(\Xi)  = (X, Y, \Res) (\Xi)$ and $Z_N (\Xi') = (X_N, Y_N, \Res_N) (\Xi')$ as in Part 1.

It follows from the second part of 
Proposition \ref{PROP:LWPv} 
(with $\kk = \| \Xi  - \Xi'\|_{\Xc^\eps_T}$)
and 
Proposition~\ref{PROP:plan} that 
\begin{align*}
\int&  \min\big(\| Z (\Xi) - Z_N(\Xi')\|_{Z^{s_1,s_2,s_3}(T)}, 1 \big) 
d \plan_N(\Xi, \Xi')\\ 
& \le A\big(T,  \| \Xi \|_{\Xc^\eps_T}, \|Z(\Xi) \|_{Z^{s_1,s_2,s_3}(T)}\big)
\\
& \quad \times   \int \min\big(\| \Xi - \Xi'\|_{\Xc^\eps_T} + N^{-\dl}, 1 \big) 
d \plan_N(\Xi, \Xi') 
\too 0, 
\end{align*}
\noi
as $N \to \infty$.
This proves \eqref{iv3c}
and therefore, 
we conclude~\eqref{iv30}, 
which proves  invariance of the Gibbs measure $\rhoo$
under the limiting hyperbolic $\Phi^3_3$-model.
\end{proof}

\appendix

\section{Absolute continuity with respect to the shifted measure}
\label{SEC:AC}

\subsection{Preliminary lemmas}
In this appendix, 
we prove that the $\Phi^3_3$-measure $\rho$ in the weakly nonlinear regime
($|\s|\ll 1$), 
constructed in Theorem \ref{THM:Gibbs}\,(i),  
is absolutely continuous
with respect to the shifted measure $\Law (Y(1) +\s \ZZ(1) + \W(1))$,
where $Y$ is as in \eqref{P2}, $\ZZ$ is defined as the limit of 
the antiderivative of 
$\dot \ZZ^N$ in \eqref{YZ12} as $N \to \infty$, and the auxiliary process $\W$ is defined by 
\begin{align}
\W (t) = (1-\Delta)^{-1} \int_0^t \jb{\nabla}^{-\frac 12 - \eps} 
\big(\jb{\nabla}^{-\frac 12 - \eps} Y(t')\big)^{5} dt'
\label{AC0}
\end{align}

\noi
for some small $\eps > 0$.
For the proof,
we construct a drift as in the discussion in Section 3 of \cite{BG2}.
See also Appendix C in \cite{OOTol1}.
The coercive term $\W$ is introduced  to 
guarantee  global existence of a drift on the time interval $[0, 1]$.
See Lemma \ref{LEM:globald} below.
We closely follow the presentation in Appendix C of our previous work \cite{OOTol1}.

First, we recall the following general lemma, 
giving a criterion for absolute continuity.
See Lemma C.1  in \cite{OOTol1} for the proof.

\begin{lemma} \label{LEM:AC1}
Let $\mu_n$ and $\rho_n$ be probability measures on a Polish space $X$.
Suppose that $\mu_n$ and $\rho_n$  converge weakly to $\mu$ and $\rho$, respectively.
Furthermore, suppose  that for every $\eps > 0$,
there exist $\delta(\eps) >0 $ and $\eta(\eps)>0 $ 
with $\delta(\eps)$, $\eta(\eps) \to 0$ as $\eps \to 0$ such that for every continuous function $F: X \to \R$ with $0 < \inf F \le F \le 1$ satisfying  
\[\mu_n(\{F \le \eps\}) \ge 1- \delta(\eps)\] 

\noi
for any $n \geq n_0 (F)$, 
we have
\begin{align}
\limsup_{n \to \infty} \int F(u) d \rho_n(u) \le \eta(\eps).
\notag
\end{align}
Then, $\rho$ is absolutely continuous with respect to $\mu$.
\end{lemma}

By regarding
$\dot \ZZ^N$ in \eqref{YZ12} and $\W$ in \eqref{AC0}
as functions of $Y$,
we write them as 
\begin{align}
\dot \ZZ^N (Y) (t) &= (1-\Delta)^{-1} \!:\!Y_N^2(t)\!:\, , \label{AC01} \\
\W (Y) (t) &= (1-\Delta)^{-1} \int_0^t \jb{\nabla}^{-\frac 12 - \eps} 
\big(\jb{\nabla}^{-\frac 12 - \eps} Y(t')\big)^{5} dt' \notag
\end{align}

\noi
and we set $\dot \ZZ_N (Y)  = \pi_N \dot \ZZ^N (Y)$.
Then, from 
 \eqref{AC01}, we have 
\begin{align}
\dot \ZZ_N (Y + \Theta) - \dot \ZZ_N (Y)
= (1-\Delta)^{-1} \pi_N ( 2\Theta_N Y_N + \Theta_N^2),
\label{Jb0}
\end{align}

\noi
where  $\Dr_N = \pi_N \Dr$.
We also define $\W_N(Y)(t)$ by 
\begin{align}
\W_N (Y) (t) &= (1-\Delta)^{-1} \pi_N \int_0^t \jb{\nabla}^{-\frac 12 - \eps} 
\big(\jb{\nabla}^{-\frac 12 - \eps} Y_N (t')\big)^{5} dt' .
\label{AC04}
\end{align}

Next, we state a lemma on the construction of  a drift $\Dr$.

\begin{lemma} \label{LEM:globald}
Let $\s \in \R$ and  $\dot \Ups \in L^2 ([0,1]; H^1(\T^3))$.
Then, given any $N \in \N$, 
the Cauchy problem for $\Dr$\textup{:}
\begin{equation} \label{ACde}
\begin{cases}
\dot \Theta + \s (1-\Delta)^{-1} \pi_N (2 \Theta_N Y_N + \Theta_N^2) +  \dot \W_N(Y+\Theta) - \dot \Ups = 0 \\
\Theta(0) = 0
\end{cases}
\end{equation}

\noi
is almost surely globally well-posed 
on the time interval $[0, 1]$
such that a solution $\Dr$ belongs to  $ C([0, 1]; H^1(\T^3))$.
Moreover, if  $\|\dot \Ups\|^2_{L^2([0,\tau]; H^1_x)} \le M$ for some $M>0$
and for some stopping time $\tau \in [0, 1]$, 
then, for any $ 1 \le p <\infty$, there exists $C= C(M,p)>0$ such that 
\begin{equation}\label{AC2}
\E \Big[ \|\dot \Theta\|_{L^2([0,\tau]; H^1_x)}^p \Big]
 \le C(M,p),
\end{equation}
where $C(M,p)$ is independent of $N \in \N$.
\end{lemma}

\subsection{Absolute continuity}

In this subsection, we
 prove the absolute continuity of 
the $\Phi^3_3$-measure $\rho$ with respect to $\Law (Y(1) +\s \ZZ(1) + \W(1))$
by assuming Lemma~\ref{LEM:globald}.
 We present the proof of 
Lemma~\ref{LEM:globald} at the end of this appendix.
For simplicity, we use the same short-hand notations as in 
Sections \ref{SEC:Gibbs} and \ref{SEC:non}, 
for instance,
$Y=Y(1)$, $\ZZ = \ZZ(1)$, $\W = \W (1)$, and $\W_N =  \W_N(1)$.

Given $L \gg1$, let $\dl(L)$ and $R(L)$ satisfy $\dl (L) \to 0$ and $R(L) \to \infty$ as $L \to \infty$, which will be specified later.
In view of Lemma \ref{LEM:AC1}, it suffices to show that 
if $G: \C^{-100}(\T^3) \to \R$ is a bounded continuous function
with $G > 0$ and
\begin{align}
\PP \big(\{G(Y +\s \ZZ_N+ \W_N) \ge L \}\big) \ge 1 - \delta(L),
\label{AC4b}
\end{align}

\noi
then
we have
\begin{align}
 \limsup_{N \to \infty} \int \exp(- G(u)) d \rho_N (u) \le \exp(-R (L)), 
\label{AC00}
\end{align}

\noi
where $\rho_N$ denotes the truncated $\Phi^3_3$-measure defined in \eqref{GibbsN}.
Here, think of  $\Law(Y +\s \ZZ_N+ \W_N)$ as the measure  $\mu_N$, weakly 
converging to $\mu = \Law(Y +\s \ZZ+ \W)$.

By the Bou\'e-Dupuis variational formula (Lemma \ref{LEM:var3}) and 
the change of variables \eqref{YZ13},
we have
\begin{align*}
- \log &  \bigg( \int \exp(- G(u) - R^{\dia}_N(u)) d \mu (u) \bigg) \\
&= \inf_{\dot \Ups^N \in  \mathbb H_a^1}\E \bigg[ 
G(Y+ \Ups^N+\s \ZZ_N)
+ \ft R^{\dia}_N (Y+\Ups^N +\s \ZZ_N) + \frac12 \int_0^1 \| \dot \Ups^N(t) \|_{H^1_x}^2dt \bigg],
\end{align*}

\noi
where $\ft R^{\dia}_N$ is as in \eqref{KZ16}.
We proceed as in Subsection \ref{SUBSEC:tight}, using 
 Lemmas \ref{LEM:Dr7} and \ref{LEM:Dr8} with 
Lemma \ref{LEM:Dr}, \eqref{YZ15}, and the smallness of $|\s|$.
See
\eqref{K9}, \eqref{K10}, and \eqref{KZ14a}.
Thus, 
we have
\begin{align}
\begin{split}
- \log &  \bigg( \int \exp(- G(u) - R^{\dia}_N(u)) d \mu (u)\bigg) \\
&\quad \ge \inf_{\dot \Ups^N \in  \mathbb H_a^1}\E \bigg[ 
G(Y+ \Ups^N+\s \ZZ_N)
+ \frac1{20} \int_0^1 \| \dot \Ups^N(t) \|_{H^1_x}^2dt \bigg]
-C_1
\end{split}
\label{AC4a}
\end{align}

\noi
for some constant $C_1 > 0$.
For $\dot \Ups^N \in \Ha^1$,
let $\Theta^N$ be the solution to \eqref{ACde} with $\dot \Ups$ replaced by $\dot\Ups^N$.
For any $M>0$, define the stopping time $\tau_M$ as 
\begin{align}
\begin{split}
\tau_M
&=\min\bigg(1,  \, \min \bigg\{ \tau : \int_0^\tau \|\dot\Ups^N (t) \|_{H^1_x}^2 dt = M \bigg\}, \\
&\hphantom{XXXXX}
\min \bigg\{ \tau : \int_0^\tau \|\dot\Theta^N (t) \|_{H^1_x}^2 dt = 2C(M,2)\bigg\}\bigg),
\end{split}
\label{AC4d}
\end{align}
where $C(M,2)$ is the constant appearing in \eqref{AC2} with $p=2$.
Let
\begin{align}
\Theta^N_M(t) := \Theta^N (\min(t,\tau_M)).
\label{AC4e}
\end{align}
From 
\eqref{P2}, we have  $Y(0) = 0$, while 
$\ZZ^N(0) = 0$ by definition.
Then, from the change of variables~\eqref{YZ13}
with $\Dr (0) = 0$, 
we see that $\Ups^N(0) = 0$.
We also have $\W_N(0) = 0$ from \eqref{AC04}.
Then, substituting  \eqref{Jb0} into \eqref{ACde} 
and integrating from $t = 0$ to $1$ gives 
\begin{align}
Y + \Ups^N +\s \ZZ_N
= Y +  \Theta^N_M +\s \ZZ_N(Y + \Theta^N_M) + \W_N(Y + \Theta^N_M)
\label{Ja9}
\end{align}
on the set $\{\tau_M = 1\}$.

From the definition \eqref{AC4e} with 
\eqref{AC4d}, we have 
\begin{align}
\| \dot \Theta_M^N \|_{L^2_t([0,1];H^1_x)}^2 \le 2C(M,2)
\label{AC4f}
\end{align}

\noi
and thus the Novikov condition is satisfied.
Then, 
Girsanov's theorem 
\cite[Theorem 10.14]{DZ}
yields that $\Law(Y + \Theta_M^N)$ is absolutely continuous with respect to $\Law (Y)$;
see \eqref{AC4c} below.
Let $\Q = \Q^{\dot \Theta_M^N}$  the probability measure
whose  Radon-Nikodym derivative 
with respect to $\PP$
is given by the following stochastic exponential:
\begin{equation}
\frac{d\Q}{d\PP} = e^{ - \int_0^1 \jb{ \dot \Theta_M^N (t),  dY(t)}_{H^1_x} 
- \frac{1}{2} \int_0^1 \| \dot \Theta_M^N (t) \|_{H^1_x}^2dt}
\label{Ja9a}
\end{equation}

\noi
such that, under this new measure $\Q$, 
the process  
\[W^{ \dot  \Theta_M^N}(t) =W(t) +   \jb{\nb}  \dot  \Theta_M^N(t)
=   \jb{\nb} (Y + \dot  \Theta_M^N)(t)\]
is a cylindrical 
 Wiener process on $L^2(\T^3)$.
By setting 
 $Y^{ \dot  \Theta_M^N} (t) = \jb{\nb}^{-1}W^{ \dot  \Theta_M^N} (t)$, 
 we have 
\begin{align}
  Y^{ \dot  \Theta_M^N}(t) = Y(t) + \Theta_M^N(t).
\label{Ja9b}
\end{align}

\noi
Moreover, 
from 
 Cauchy-Schwarz inequality with 
 \eqref{Ja9a}
and 
the bound~\eqref{AC4f}, 
and then~\eqref{Ja9b}, we have 
\begin{align}
\begin{split}
\PP\big(\{Y + \Theta_M^N \in E\}\big) 
& = \int \ind_{\{ Y+  \Theta_M^N\in E\}} \frac{d \PP}{d\Q} d\Q
 \le C_M \Big(\Q\big(\{  Y^{ \dot  \Theta_M^N} \in E\}\big)\Big)^\frac12 \\
& =   C_M \Big(\PP\big(\{Y \in E\}\big)\Big)^\frac12
\end{split}
\label{AC4c}
\end{align}
for any measurable set $E$.

From \eqref{AC4a}, \eqref{Ja9}, and  the non-negativity of $G$, 
we have
\begin{align}
\eqref{AC4a} &\ge \inf_{\dot \Ups^N \in  \mathbb H_a^1}\E \bigg[ 
\Big( G \big( Y+  \Theta^N_M +\s \ZZ_N(Y + \Theta^N_M) + \W_N(Y + \Theta^N_M) \big)\notag \\
&\phantom{XXXXXX}
+ \frac1{20} \int_0^1 \| \dot \Ups^N (t) \|_{H^1_x}^2dt\Big)\ind_{\{ \tau_M = 1 \}}\phantom{]} \notag \\
&\phantom{XXXXXX}
+\Big( G(Y+ \Ups^N+\s \ZZ_N)
+ \frac1{20} \int_0^1 \| \dot \Ups^N (t) \|_{H^1_x}^2dt\Big)\ind_{\{ \tau_M < 1 \}} \bigg]
-C_1 \notag \\
&\ge \inf_{\dot \Ups^N \in  \mathbb H_a^1}\E \bigg[ 
G \big( Y+  \Theta^N_M +\s \ZZ_N(Y + \Theta^N_M) + \W_N(Y + \Theta^N_M) \big)
\cdot \ind_{\{ \tau_M = 1 \}}\phantom{]} \notag \\
&\phantom{XXXXXX}
+ \frac1{20} \int_0^1 \| \dot \Ups^N (t) \|_{H^1_x}^2dt \cdot \ind_{\{ \tau_M < 1 \}} \bigg] - C_1.
\notag
\end{align}

\noi
Then, 
using the definition 
\eqref{AC4d} of the stopping time $\tau_M$ 
and applying 
\eqref{AC4c} and 
\eqref{AC4b}, we have 
\begin{align}
\eqref{AC4a} &\ge \inf_{\dot \Ups^N \in  \mathbb H_a^1}
\E \bigg[ L \cdot \ind_{\{\tau_M = 1\} 
\cap \{ G(Y+  \Theta^N_M +\s \ZZ_N(Y + \Theta^N_M) + \W_N(Y + \Theta^N_M)) \ge L \}} \notag \\
&\phantom{XXXXXX}
+ \frac M {20} \cdot \ind_{\{\tau_M < 1\} \cap \{\int_0^1 \| \dot \Theta^N_M (t) \|_{H^1_x}^2 dt < 2C(M,2)\}} \bigg]  - C_1 \notag \\
&\ge \inf_{\dot \Ups^N \in  \mathbb H_a^1}
\Bigg\{ L \Big( \PP(\{\tau_M = 1\}) - C_M \delta (L)^ \frac 12 \Big) \notag \\
&\hphantom{XXXXX}
+ \frac M {20} \PP\bigg(\{\tau_M < 1\} \cap \bigg\{\int_0^1 \| \dot \Theta^N_M (t) \|_{H^1_x}^2 dt < 2C(M,2)\bigg\}\bigg) \Bigg\} - C_1.
\label{ACX1}
\end{align}

In view of  \eqref{AC2} with \eqref{AC4d} and \eqref{AC4e}, 
Markov's inequality gives 
\[
\PP \bigg( \int_0^1 \| \dot \Theta^N_M (t) \|_{H^1_x}^2 dt
= \int_0^{\tau_M} \| \dot \Theta^N_M (t) \|_{H^1_x}^2 dt \ge 2C(M,2) \bigg) \le  \frac 12,
\]

\noi
which yields
\begin{align}
\PP \bigg( \{\tau_M < 1\} \cap \bigg\{\int_0^1 \| \dot \Theta^N_M (t) \|_{H^1_x}^2dt < 2C(M,2) \bigg\} \bigg) \ge \PP(\{\tau_M < 1\})- \frac 12. 
\label{ACX2}
\end{align}

\noi
Now, we set   $M = 20L$.
Note  from 
\eqref{AC4d} that $\PP(\{\tau_M = 1\})+ \PP(\{\tau_M < 1\}) = 1$.
Then, 
from~\eqref{ACX1} and \eqref{ACX2}, 
we obtain
\begin{align*}
- \log & \bigg( \int \exp(- G(u)- R^{\dia}_N(u)) d \mu (u) \bigg) \\
&\ge  \inf_{\dot \Ups^N \in  \mathbb H_a^1} \bigg\{
L \Big( \PP(\{\tau_M = 1\}) - C'_{L} \delta(L)^ \frac 12 \Big)
+ L \Big( \PP(\{\tau_M < 1\})- \frac 12 \Big) \bigg\} - C_1 \\
&= L \Big( \frac 12 - C_{L}' \delta(L)^\frac 12 \Big) - C_1.
\end{align*}

\noi
Therefore, 
by choosing $\delta(L)>0$ such that $C'_{L} \delta(L)^\frac12 \to 0$ as $L \to \infty$,
this shows \eqref{AC00} with 
\[R(L) = L \Big( \frac 12 - C'_{L} \delta(L)^\frac 12 \Big) - C_1 + \log Z, \]

\noi
where $Z = \lim_{N \to \infty} Z_N$ denotes the 
limit of the partition functions  for the truncated $\Phi^3_3$-measures $\rho_N$.

\smallskip

\subsection{Proof of Lemma~\ref{LEM:globald}}

We conclude this appendix  by presenting the proof of Lemma \ref{LEM:globald}.

\begin{proof}[Proof of Lemma \ref{LEM:globald}]
By Lemma \ref{LEM:gko} (ii) and Sobolev's inequality, 
we have
\begin{align}
\begin{split}
\| (1-\Delta)^{-1} ( 2\Theta_N Y_N + \Theta_N^2)(t) \|_{H^{1}_x}
&\les 
\| ( 2\Theta_N Y_N + \Theta_N^2)(t) \|_{H^{-1}_x} \\
&\les
\| \Dr_N (t) \|_{H^{\frac 12+\eps}_x} \| Y_N (t) \|_{W^{-\frac 12-\eps, \infty}_x}
+ \| \Dr_N^2 (t) \|_{L_x^{\frac 65}} \\
&\les
\| \Dr_N (t) \|_{H^1_x} \| Y_N (t) \|_{W^{-\frac 12-\eps, \infty}_x}
+ \| \Dr_N (t) \|_{H_x^1}^2
\end{split}
\label{Jb8}
\end{align}

\noi
for small $\eps>0$.
Moreover,
from \eqref{AC0}, we have
\begin{align}
\begin{split}
\| \dot \W_N(Y(t)+\Theta(t)) \|_{H^1_x}
&\les
\| \jb{\nb}^{-\frac 12-\eps}Y_N(t) \|_{L^\infty_x}^{5}
+ \| \jb{\nb}^{-\frac 12-\eps} \Dr_N (t)\|_{L^\infty_x}^{5} \\
&\les 
\| Y_N(t) \|_{W^{-\frac 12-\eps, \infty}_x}^{5}
+ \| \Dr_N (t)\|_{H^1_x}^{5}.
\end{split}
\label{Jb9}
\end{align}

\noi
Therefore,
by studying the integral formulation of 
\eqref{ACde}, 
a contraction argument 
in $L^\infty([0,T]; H^1(\T^3))$ for small $T>0$
with  \eqref{Jb8} and \eqref{Jb9}
yields local well-posedness.
Here, the local existence time $T$ depends on 
$\| \Dr(0)\|_{H^1}$, 
$\|\dot \Ups\|_{L^2_T H^1_x}$,   and 
$\| Y_N \|_{L^6_TW^{-\frac 12-\eps, \infty}_x}$, 
where the last term is almost surely bounded in view of Lemma \ref{LEM:Dr}
and \eqref{embed}.

Next, we prove global existence on $[0, 1]$
by establishing an a priori bound on the $H^1$-norm of a solution.
From  \eqref{ACde} with \eqref{AC04}, we have
\begin{align}
\begin{split}
\frac 12 \frac{d}{dt} \| \Theta(t) \|_{H^1}^2
&= -\s \int_{\T^3} ( 2\Theta_N(t) Y_N(t) + \Theta_N^2(t)) \Theta_N(t) dx \\
&\quad
- \int_{\T^3} \big(\jb{\nabla}^{-\frac 12 - \eps} 
(Y_N(t)+ \Theta_N(t))\big)^{5}\cdot \jb{\nabla}^{-\frac 12 - \eps} \Theta_N(t) dx \\
&\quad
+ \int_{\T^3} \jb{\nabla} \Theta(t) \cdot \jb{\nabla} \dot \Ups(t) dx. \label{AC3}
\end{split}
\end{align}

\noi
The second term on the right-hand side of \eqref{AC3}, 
coming from $\W$ is a coercive term, 
allowing us to hide part of the first term on the right-hand side.

From Lemma \ref{LEM:Bes}
and  Young's inequality,
we have
\begin{align}
\begin{split}
\bigg| \int_{\T^3} &  ( 2\Theta_N(t) Y_N(t) + \Theta_N^2(t)) \Theta_N(t) dx \bigg|\\
& \les
\| \Theta_N (t) \|_{H^1}^2 + \|\Dr_N(t)\|_{L^3}^3
+ \| Y_N (t) \|_{\C^{-\frac 12-\eps}}^c
\end{split}
\label{AC3a}
\end{align}

\noi
for small $\eps > 0$ and some $c>0$.
We now estimate the second term on the right-hand side of~\eqref{AC3a}.
By  \eqref{interp}, we have
\begin{align}
\begin{split}
\| \Theta_N(t) \|_{L^3}^3
& \les \| \Theta_N(t) \|_{H^1}^{\frac{3 + 6\eps}{3 + 2\eps}}
 \|  \Theta_N(t) \|_{W^{-\frac 12 - \eps, 6}}^{\frac{6}{3 + 2\eps}}
\\
&\le \| \Theta_N(t) \|_{H^1}^2 + \eps_0 \| \Theta_N(t) \|_{W^{-\frac 12 - \eps, 6}}^{6} 
+ C_{\eps_0}
\end{split}
\label{AC3b}
\end{align}
for small $\eps, \eps_0 > 0$.
As for the coercive term, from \eqref{YY9} and Young's inequality,  we have 
\begin{align}
\begin{split}
\int_{\T^3}&  \big(\jb{\nabla}^{-\frac 12 - \eps} 
(Y_N(t)+ \Theta_N(t))\big)^{5}\jb{\nabla}^{-\frac 12 - \eps} \Theta_N(t) dx \\
&\ge
\frac 12 \int_{\T^3} (\jb{\nabla}^{-\frac 12 - \eps} \Theta_N(t))^{6} dx
- c \int_{\T^3} \big|(\jb{\nabla}^{-\frac 12 - \eps} Y_N(t))^{5}
\jb{\nabla}^{-\frac 12 - \eps} \Theta_N(t)\big| dx \\
& \ge 
\frac 12 \| \Theta_N(t) \|_{W^{-\frac 12 - \eps, 6}}^{6}
- c \| Y_N(t) \|_{W^{-\frac12-\eps, 6}}^{5} \| \Theta_N(t) \|_{W^{-\frac12-\eps, 6}}\\
& \ge
\frac 14  \| \Theta_N(t) \|_{W^{-\frac 12 - \eps, 6}}^{6}
- c \| Y_N(t) \|_{W^{-\frac12-\eps, 6}}^{6}.
\end{split}
\label{AC3c}
\end{align}

\noi
Therefore, putting \eqref{AC3}, \eqref{AC3a}, \eqref{AC3b}, and  \eqref{AC3c} together
we obtain
\begin{align*}
 \frac{d}{dt} \| \Theta(t) \|_{H^1}^2
& \les
\| \Theta(t) \|_{H^1}^2
+ \|\dot \Ups(t)\|_{H^1}^2
 +  \| Y (t) \|_{\C^{-\frac 12-\eps}}^c + \| Y(t) \|_{W^{-\frac12-\eps, 6}}^{6} + 1.
\end{align*}

\noi
By Gronwall's inequality, we then obtain
\begin{align}
\| \Theta(t) \|_{H^1}^2
\les
\|\dot \Ups\|_{L^2([0,t];H^1_x)}^2
+ \| Y_N \|_{L^c([0,1]; \C^{-\frac 12-\eps}_x)}^c
+ \| Y \|_{L^{6}([0,1];W^{-\frac12-\eps, 6}_x)}^{6} + 1, 
\label{AC3ca}
\end{align}

\noi
uniformly in $ 0 \le t \le 1$.
The a priori bound \eqref{AC3ca} together with Lemma \ref{LEM:Dr} allows us to iterate
the local well-posedness argument, 
guaranteeing  existence of the solution $\Dr$ on $[0, 1]$.

Lastly, we prove the bound \eqref{AC2}.
From \eqref{Jb8}, \eqref{Jb9}, 
and \eqref{AC3ca}, we have 
\begin{align}
\begin{split}
\| \s (1 & -\Delta)^{-1} ( 2\Theta_N Y_N + \Theta_N^2) + \dot \W_N(Y+\Theta) \|_{L^2([0,\tau];H^1_x)}\\
& \les \| \dot \Ups \|_{L^2 ([0,\tau]; H^1_x)}^5 +
\| Y_N \|_{L^q([0,1]; \C^{-\frac 12-\frac 12 \eps}_x)}^{c_0} + 1
\end{split}
\label{AC3d}
\end{align}

\noi
for some finite $q, c_0 \geq 1$
and for any $0 \le \tau \leq 1$.
Then, using the equation \eqref{ACde}, the bound~\eqref{AC2}
follows from \eqref{AC3d}, the bound on $\dot \Ups$,
and the following corollary to 
Lemma \ref{LEM:Dr}:
\begin{align*}
\E \Big[ \| Y_N \|_{L^q([0,1]; \C^{-\frac 12-\frac 12 \eps}_x)}^p \Big]
< \infty
\end{align*}

\noi
for any finite $p, q \geq 1$,  uniformly in $N \in \N$.
\end{proof}

\begin{remark}\label{REM:ac}\rm
A slight modification of the argument presented above shows that 
the tamed $\Phi^3_3$-measure $\nu_\dl$ constructed in Proposition
\ref{PROP:ref}
is 
 absolutely continuous
with respect to the shifted measure $\Law (Y(1) +\s \ZZ(1) + \W(1))$.
In this setting, we can use  the analysis
in Subsection~\ref{SUBSEC:refm}
(Step 1 of the proof of Proposition \ref{PROP:ref})
to arrive at  
\eqref{AC4a}.  The rest of the argument remains unchanged.
As a consequence, the $\s$-finite version $\cj \rho_\dl$ of
the $\Phi^3_3$-measure defined in \eqref{ref1a}
is also 
 absolutely continuous
with respect to the shifted measure $\Law (Y(1) +\s \ZZ(1) + \W(1))$
for any $\dl > 0$.

\end{remark}

\begin{ackno}\rm
T.O.~was supported by the European Research Council (grant no.~637995 ``ProbDynDispEq''
and grant no.~864138 ``SingStochDispDyn"). 
M.O.~was supported by JSPS KAKENHI  Grant number JP20K14342.
L.T.~was supported by the Deutsche Forschungsgemeinschaft (DFG, German Research Foundation) through the Hausdorff Center for Mathematics under Germany's Excellence Strategy - EXC-2047/1 - 390685813 and through CRC 1060 - project number 211504053.

\end{ackno}

\end{document}